\documentclass[11pt]{article}

\usepackage{amsmath, amssymb, amsthm, physics, url, bbm, mathrsfs, cite, xcolor, mathtools, dsfont, esint}
\usepackage[shortlabels]{enumitem}
\usepackage{tikz}
\tikzset{every picture/.style={line width=0.75pt}}    

\usepackage[title]{appendix}
\usepackage[T1]{fontenc}

\mathtoolsset{showonlyrefs}

\definecolor{darkblue}{RGB}{32,57,231}
\usepackage[colorlinks,
	citecolor=darkblue, 
	linkcolor=darkblue, 
	urlcolor=darkblue,
	bookmarksopen=true,
	pdfauthor={Aobo Chen and Zhenyu Yu}, 
	pdftitle={Elliptic Harnack inequalities for mixed local and nonlocal {\it p}-energy form on metric measure spaces}
]{hyperref}
\usepackage[open,openlevel=2]{bookmark}

\parskip \medskipamount

\numberwithin{equation}{section}
\numberwithin{figure}{section}

\newtheorem{theorem}{Theorem}[section]
\newtheorem{lemma}[theorem]{Lemma}
\newtheorem{proposition}[theorem]{Proposition}
\newtheorem{corollary}[theorem]{Corollary}

\theoremstyle{definition}
\newtheorem{definition}[theorem]{Definition}

\newtheorem{remark}[theorem]{Remark}

\newtheorem*{notation}{Notation}

\DeclareMathOperator*{\einf}{einf}
\DeclareMathOperator*{\esup}{esup}
\DeclareMathOperator*{\eosc}{eosc}
\DeclareMathOperator*{\sgn}{sgn}
\def\diag{\mathrm{diag}}
\def\od{\mathrm{od}}

\def\diam{{\mathop{{\rm diam }}}}
\newcommand*{\dif}{\mathop{}\!\mathrm{d}}
\newcommand{\supp}{\mathrm{supp}}
\newcommand{\one}{\mathds{1}}
\newcommand{\loc}[0]{\operatorname{loc}}

\def\sA {{\mathcal A}} \def\sB {{\mathcal B}} \def\sC {{\mathcal C}}
\def\sD {{\mathcal D}} \def\sE {{\mathcal E}} \def\sF {{\mathcal F}}
  
  \def\sL {{\mathcal L}}
  \def\sO {{\mathcal O}}
  
  \def\sU {{\mathcal U}}

  \def\rL {{\mathscr L}}

 \def\bN {{\mathbb N}} 
 \def\bQ {{\mathbb Q}} \def\bR {{\mathbb R}}

 \def\bZ {{\mathbb Z}}

\def\ol{\overline}
\def\wt{\widetilde}
\def\wh{\widehat}

\def\la{\langle}
\def\ra{\rangle}

\newcommand{\set}[1]{\left\{ #1 \right\}}

\newcommand{\Sett}[2]{\left\{ #1  : \, #2 \right\}}

\newcommand\restr[2]{{
		\left.\kern-\nulldelimiterspace 
		#1 
		\vphantom{\big|}
		\right|_{#2}
}}

\def\capacity{\operatorname{Cap}}
\def\cutoff{\operatorname{cutoff}}
\newcommand{\BMO}{\mathrm{BMO}}
\renewcommand{\Bbb}{\mathbb}
\def\wcdot{\,\cdot\,}

\title{{\bf Elliptic Harnack inequalities for mixed local and nonlocal \texorpdfstring{$p$}{p}-energy form on metric measure spaces}}

\author{
Aobo Chen\, and\, Zhenyu Yu\thanks{\scriptsize Zhenyu Yu was supported by the Natural
Science Foundation of Hunan Province, China (Grant. No. 2025JJ60039) and Innovation Research Foundation of National University of Defense Technology (Grant. No. ZK25-05).}
}

\date{\today}
\begin{document}

\maketitle

\vspace{-0.5cm}

\begin{abstract}

In the context of metric measure spaces, we introduce an axiomatic formulation of mixed local and nonlocal $p$-energy forms. Within this framework, we use the Poincar\'{e} inequality, the cutoff Sobolev inequality, and mild assumptions on the jump measure to establish the weak and strong elliptic Harnack inequalities for such mixed forms. Our approach is based on the De Giorgi--Nash--Moser method and extends the corresponding results for Dirichlet forms without the killing part, as well as for mixed energy forms on Euclidean spaces. Some consequences of elliptic Harnack inequalities are discussed.


\end{abstract}

\tableofcontents

\section{Introduction}\label{s.intro}

This work concerns the study of mixed local and nonlocal $p$-energy forms on metric measure spaces, with the goal of establishing elliptic Harnack inequalities in this setting. In the classical setting of $n$-dimensional Euclidean space $\bR^{n}$ with standard Laplacian $\Delta$, the elliptic Harnack inequality  
\begin{equation}\label{e.EHI-Rn}
	\text{there exists } C>1 \text{ s.t. } \esup_{B}u\leq C\einf_{B}u, \ \text{for any non-negative harmonic function } u \text{ on } 2B,
\end{equation} 
was first established by Moser \cite{Mos61}, marking a cornerstone in the regularity theory of elliptic equations. For the nonlinear case, the same inequality holds for $\Delta_p$-harmonic functions, where  
\[
	-\Delta_p u := -\mathrm{div}(\lvert \nabla u \rvert^{p-2} \nabla u) = 0   \ (1<p<\infty);
\]
see, for instance, \cite{Tru67,MZ97,Lin19}. 

For nonlocal equations, such as the fractional $p$-Laplace equation on $\mathbb{R}^n$,
\begin{equation}
 (-\Delta_p)^s u(x) := \mathrm{P.V.} \int_{\mathbb{R}^n} \frac{|u(x) - u(y)|^{p-2} (u(x) - u(y))}{|x - y|^{n + sp}} \dif y = 0, \quad 0 < s < 1,
\end{equation}
the situation becomes more subtle. Due to the nonlocality of the operator, the classical form of the Harnack inequality \eqref{e.EHI-Rn} fails unless the harmonic function $u$ is assumed to be globally non-negative; see, for example, \cite[Section 3]{BC10} and \cite[Theorem 2.2]{DK20} in the linear case $p=2$. When $u$ is merely locally non-negative, one must introduce suitable modifications: the so-called \emph{weak} and \emph{strong elliptic Harnack inequalities}, in which a \emph{nonlocal tail term} appears. These variants have been systematically studied in the Euclidean setting, for instance by Kassmann \cite{Kas11} for $p=2$, and by Di Castro, Kuusi, and Palatucci \cite{DCKP14} for all $p \in (1, \infty)$.

The study of mixed local and nonlocal equations has garnered significant interest in recent years, particularly in the form
\[
	-\Delta_p u + (-\Delta_p)^s u = 0.
\]
Such equations arise naturally in models involving both local and nonlocal interactions, such as in random walks with long-range jumps, nonlocal diffusion processes, and in the study of anomalous diffusion. For the case when $p = 2$, these mixed equations were analyzed using probabilistic methods in works such as \cite{Foo09, CK10, CKSV12, AR18}, while \cite{Coz17,GK22} explored the case for all $p \in (1, \infty)$ and $s \in (0, 1)$ using purely analytic methods. The literature on the elliptic Harnack inequality is extensive; here we provide only a selective review.

In the broader context of metric measure spaces, the Harnack inequality for local Laplacians (defined via \emph{Dirichlet forms} or more generally \emph{$p$-energy forms}) is well-understood. In the bilinear case ($p=2$), the Harnack inequality can be derived under a collection of geometric and analytic assumptions on the underlying space. In particular, on complete Riemannian manifolds, the volume doubling property together with a Poincar\'{e} inequality implies a scale-invariant parabolic Harnack inequality (and hence an elliptic one) for the Laplace--Beltrami operator; see \cite{Gri91, SC92, HSC01}. For strongly local, regular Dirichlet forms on volume doubling metric measure spaces, Grigor'yan, Hu, and Lau \cite{GHL15} established the elliptic Harnack inequality using Poincar\'{e} inequality and a generalized capacity condition. In the nonlocal setting, Chen, Kumagai, and Wang \cite{CKW19} proved a weak elliptic Harnack inequality for pure-jump Dirichlet forms on volume doubling metric measure spaces under assumptions including upper bounds on the jump kernel, a Poincar\'{e} inequality, and a cutoff Sobolev inequality. By additionally assuming a lower bound on the jump kernel, they further obtained a strong elliptic Harnack inequality. Subsequent works showed that these upper and lower jump kernel bounds can be replaced by milder conditions such as \eqref{TJ} and \eqref{UJS}; see \cite{HY23, KW24} for the weak elliptic Harnack inequality and \cite{HY24, Che25} for the strong elliptic Harnack inequality, both in the framework of quadratic forms or resurrected Dirichlet forms.

The elliptic Harnack inequality for strongly local Dirichlet forms have also been extended to the nonlinear setting with $p \in (1,\infty)$; see \cite{Yan25a} and \cite[Theorem~6.36]{KS25} (see also \cite{Cap07} in the context of metric fractals). A natural next step is to develop an analogous theory for mixed local and nonlocal $p$-energy forms. Motivated by this problem, we first introduce the notion of mixed local and nonlocal $p$-energy forms on metric measure spaces. We then develop a unified analytic framework for these forms. The main objective of this work is to establish the weak elliptic Harnack inequality \eqref{wEH} and the strong elliptic Harnack inequality \eqref{sEH} in this nonlinear and mixed setting, under a natural collection of assumptions, namely, the volume doubling property \eqref{VD}, the reverse volume doubling property \eqref{RVD}, the cutoff Sobolev inequality \eqref{CS}, and the Poincar\'{e} inequality \eqref{PI}, together with mild analytic conditions on the jump kernel (such as \eqref{TJ} and \eqref{UJS}). Our first main result is stated below; the relevant definitions and assumptions can be found in Section \ref{s.frame}.

\begin{theorem}\label{t.main}
		Let $(M,d,m)$ be a complete metric measure space that is volume doubling \eqref{VD} and reverse volume doubling \eqref{RVD}. Let $(\sE,\sF)$ be a mixed local and nonlocal $p$-energy form on $(M,m)$. Then the following implications hold:
	\begin{align}
		\eqref{CS}+\eqref{PI}+\eqref{TJ}&\Longrightarrow \text{weak elliptic Harnack inequality }\eqref{wEH},\label{e.main-wEH}\\
		\eqref{CS}+\eqref{PI}+\eqref{TJ}+\eqref{UJS}&\Longrightarrow\text{strong elliptic Harnack inequality } \eqref{sEH}\label{e.main-EHI}.
	\end{align}
\end{theorem}
\begin{remark}
We highlight several differences between Theorem~\ref{t.main} and existing results in the literature.
\begin{enumerate}[label=\textup{({\alph*})},align=left,leftmargin=*,topsep=5pt,parsep=0pt,itemsep=2pt]

\item The structural assumptions \eqref{UJS}, \eqref{CS}, and \eqref{PI} are allowed to hold only \emph{locally}, that is, up to a spatial scale bounded by a fixed constant $\ol{R}\in(0,\diam(M,d)]$.

\item In contrast to \cite{HY23,HY24}, our formulations of \eqref{wEH} and \eqref{sEH} do not require the superharmonic or harmonic functions to be bounded a priori.

\item Unlike \cite{CKW19, HY24}, we do not assume the canonical upper bound \eqref{J<} on the jump measure. In Section~\ref{s.cantor}, we construct an example of a purely nonlocal $p$-energy form on a fractal for which \eqref{TJ} holds, \emph{\eqref{J<} fails}, and \eqref{UJS}, \eqref{CS}, and \eqref{PI} are satisfied \emph{locally} (with $\ol{R}=1/2$). Theorem~\ref{t.main} applies in this setting, while most existing results do not.

\item When $(\sE,\sF)$ is strongly local, conditions \eqref{TJ} and \eqref{UJS} hold automatically. By \eqref{e.main-EHI} together with Remarks \ref{r.PI=>RVD} and \ref{r.str}, the assumptions \eqref{VD}, \eqref{CS}, and \eqref{PI} imply the standard elliptic Harnack inequality. In this sense, Theorem~\ref{t.main} extends the result of \cite[Theorem 2.3]{Yan25a}.
\end{enumerate}
\end{remark}

The main challenges in the nonlinear and mixed setting arise due to the lack of a bilinear structure and the need to handle the local and nonlocal terms, which are both nonlinear. In particular, the technique of the \emph{energy of product inequality} used in \cite{GHH24, HY23, HY24} is no longer applicable in our setting. Our framework and results not only extend the theory of Dirichlet forms (corresponding to the case $p=2$) but also provide a unified analytic approach that remains valid for general $p\in(1,\infty)$. In particular, Theorem \ref{t.main} recovers known results for local $p$-energy forms and nonlocal $p$-energy forms as special cases. 


The proof of \cite[Theorem~1.10]{DK20} provides a general mechanism for deriving H\"{o}lder continuity of harmonic functions from the weak Harnack inequality for nonlocal equations driven by integro-differential operators. With only minor adaptations, this approach applies to a broad class of problems on metric spaces, encompassing both local and nonlocal as well as linear and nonlinear settings; see also \cite[Theorem~2.1]{CKW19} for pure-jump Dirichlet forms and \cite[Lemma~6.1]{HY23} for Dirichlet forms without killing part. This includes the framework considered in this paper. We therefore state the following corollary of Theorem~\ref{t.main} without proof.

\begin{corollary}
Let $(M,d,m)$ be a complete metric measure space that is volume doubling \eqref{VD}. Let $(\sE,\sF)$ be a mixed local and nonlocal $p$-energy form on $(M,m)$. If \eqref{wEH} and \eqref{TJ} hold, then there exist two constants $\beta \in (0,1)$ and $C>0$ such that, for any $x\in M$ and any $r\in(0,\sigma \overline{R})$ (where the constant $\sigma$ is from condition \eqref{wEH}) and any harmonic function $u$ on $B(x,r)$,
\begin{equation}
\eosc_{B(x,\rho )}u\leq C\norm{u}_{L^{\infty }}\left( \frac{\rho }{r}%
\right) ^{\beta },\ \ \forall \rho\in(0, r],  \label{600}
\end{equation}
where $\eosc\limits_{\Omega} u:=\esup\limits_{\Omega} u-\einf\limits_{\Omega} u$ for  any $\Omega\subset M$. Consequently, we have 
	\begin{equation}
\eqref{RVD}+\eqref{CS}+\eqref{PI}+\eqref{TJ}\Longrightarrow \text{H\"{o}lder continuity of harmonic functions }\eqref{600}.
	\end{equation}
\end{corollary}

Our second main result establishes a sufficient condition for the cutoff Sobolev inequality \eqref{CS} in terms of the weak elliptic Harnack inequality \eqref{wEH}.
\begin{theorem}
	\label{t.weh=>cs}
	Let $(M,d,m)$ be a complete metric measure space. Let $(\sE,\sF)$ be a mixed local and nonlocal $p$-energy form on $(M,m)$. Then \begin{equation}
		\eqref{VD}+\eqref{FK}+\eqref{wEH}+\eqref{cap<}\Longrightarrow\eqref{CS}.
	\end{equation}
\end{theorem}

 In the framework of regular Dirichlet forms without killing part, such a characterization can be obtained by \cite[Lemmas 6.3 and 6.4]{HY23} and \cite[Theorem 14.1]{GHH24}. For strongly local $p$-energy forms, a related implication \[\eqref{VD}+\eqref{PI}+\eqref{sEH}+\eqref{cap<}\Longrightarrow\eqref{CS}\]was proved in \cite[Theorem 2.1]{Yan25a} via the \emph{Wolff potential estimate}, under additional geometric assumptions such as the \emph{linearly locally connected} condition.

The paper is organized as follows: In Section \ref{s.frame}, we introduce the axiomatic definition of mixed local and nonlocal $p$-energy forms and give a rigorous definition of the conditions appeared in Theorem \ref{t.main}. Section \ref{s.func} discusses the self-improvement property of cutoff Sobolev inequality, and the consequences of the Poincar\'e inequality, proving the equivalence between the Faber--Krahn inequality and Sobolev-type inequalities in Theorem \ref{t.FK=Sob}. In Section \ref{s.Pf-wEH}, we derive key inequalities such as the Caccioppoli inequality (Proposition \ref{prop:cac}), a lemma of growth (Lemma \ref{l.LoG}), and a crossover lemma (Lemma \ref{l.cros}), which together imply the weak elliptic Harnack inequality. Theorem \ref{t.weh=>cs} is proved in Section \ref{s.E><}. In Section \ref{s.Pf-EHI}, we use an estimate on the nonlocal tail (Lemma \ref{L1}) to obtain a mean-value inequality (Lemma \ref{lemma:MV2}), which, combined with the weak elliptic Harnack inequality, leads to the strong elliptic Harnack inequality (Lemma \ref{L32}). Section \ref{s.examples} provides some examples to illustrate Theorem \ref{t.main}. Appendix \ref{A.propEF} collects basic properties of $p$-energy forms, and Appendix \ref{A.ineqs} lists several useful inequalities.

\begin{notation}
In this paper, we use the following notation and conventions.
	\begin{enumerate}[label=\textup{(\arabic*)},align=left,leftmargin=*,topsep=5pt,parsep=0pt,itemsep=2pt]
	
	\item $\bN:=\{1,2,\ldots\}$. That is $0\notin \bN$.
	\item For any $x,y\in\bR$, we write $x\wedge y:=\min(x,y)$, $x\vee y:=\max(x,y)$, $x_{+}:=\max(x,0)$, $x_{-}:=\max(-x,0)$ so that $x=x_{+}-x_{-}$.
	\item Let $U$ and $V$ be two open subsets in a topological space, if $U$ is precompact and the closure of $U$ is contained in $V$, then we write $U\Subset V$.
	\item Let $A$ be a set. We define $A^{2}=A\times A$, $A_\diag:=\Sett{(x,x)}{x\in A}$ and $A_{\od}^{2}:=A^{2}\setminus A_{\diag}$, where ``$\od$'' means \emph{off-diagonal}.
	\item Let $(X,d)$ be a metric space. We define \[B(x,r):=\{y\in X\mid d(x,y)<r\}\text{ for $(x,r)\in X\times(0,\infty)$}\]
and $\diam A:=\sup_{x,y\in A}d(x,y)$ for $A\subset X$. 

Given a ball $B=B(x,r)$ and $\lambda>0$, we denote the ball $B(x,\lambda r)$ by $\lambda B$.
 		\item Let $X$ be a non-empty set. We define $\one_{A}\in\mathbb{R}^{X}$ for $A\subset X$ by
	 $\one_{A}(x):= \begin{dcases}
	 	1 & \mbox{if $x \in A$,}\\
	 	0 & \mbox{if $x \notin A$.}
	 \end{dcases}$
	 
\item Let $X$ be a topological space. We denote $\mathcal{B}(X)$ the Borel $\sigma$-algebra on $X$. Let $\mu$ be a Borel measure on $(X,\sB(X))$. The notations ``$\esup$'' and ``$\einf$'' represent the essential supremum and essential infimum \emph{with respect to the measure $\mu$}, respectively.

For any $A\in\sB(M)$ with $\mu(A)>0$ and any integrable Borel function $u$, we set \begin{equation}
			u_{A}:=\fint_{A}u\dif \mu:=\frac{1}{\mu(A)}\int_{A}u\dif \mu.
		\end{equation}	
		For a Borel measurable function $f$ on $X$, we denote $\supp_{\mu}(f)$ as the support of the measure $\abs{f}\dif\mu$. We write $f_{*}\mu$ the push-forward measure defined by $f_{*}\mu(D)=\mu(f^{-1}(D))$ for all measurable subset $D\subset \bR$.
		\item Let $X$ be a topological space. We set the class of continuous function spaces:\begin{align}
		C(X)&:=\Sett{f}{\textrm{$f$ is a continuous real-valued function on $X$}}\\
		C_c(X)&:=\Sett{f\in C(X)}{\textrm{$X\setminus f^{-1}(0)$ has compact closure in $X$}},\\
		C_c^{+}(X)&:=\Sett{f\in C_{c}(X)}{f(x)\geq0,\ \forall x\in X}.
		\end{align}
		For any $f\in C(X)$, we define $\norm{f}_{\sup}:=\sup_{x\in X}f(x)$.
\end{enumerate}
\end{notation}

\section{Preliminaries}\label{s.frame}

In this paper, we fix a complete separable metric space $(M,d)$ and let $m$ be a Radon measure on $M$. In this section, we first introduce the notion of mixed local and nonlocal $p$-energy forms on $(M,d,m)$, and then present a list of assumptions mentioned in our main result in Subsection \ref{ss.main}.

\subsection{Mixed local and nonlocal \texorpdfstring{$p$}{\em{p}}-energy form}\label{ss.mixed}
\begin{definition}
	Let $p\in(1,\infty)$. Let $\sF$ be a linear subspace of $L^{p}(M,m)$ and let $\sE:\sF\to[0,\infty)$. We say that the pair $(\sE,\sF)$ is a $p$-energy form on $(M,m)$, if $\sE^{1/p}$ is a semi-norm on $\sF$.
\end{definition}

\begin{definition}
	Let $(\sE,\sF)$ be a $p$-energy form on $(M,m)$. \begin{enumerate}[label=\textup{({\arabic*})},align=left,leftmargin=*,topsep=5pt,parsep=0pt,itemsep=2pt]
	\item We say that $(\sE,\sF)$ is \emph{closed}, if $(\sF,\sE_1^{1/p})$ is a Banach space, where $\sE_1(\wcdot):=\sE(\wcdot)+\norm{\wcdot}_{L^{p}(M,m)}^p$.
   \item We say that $(\sE,\sF)$ is \emph{regular}, if $\sF\cap C_{c}(M)$ is dense in $(\sF,\sE_1^{1/p})$ and in $(C_{c}(M),\norm{\cdot}_{\sup})$.
	\item 	We say that $(\sE,\sF)$ have the \emph{Markovian property}, or $(\sE,\sF)$ is Markovian, if for any $\psi\in C(\bR)$ satisfying $\psi(0)=0$ and $\abs{\psi(s)-\psi(t)}\leq\abs{s-t}$ for every $s,t\in\bR$, and for any $u\in\sF$, we have $\psi\circ u\in\sF$ and $\sE(\psi \circ u)\leq\sE(u)$.
	\item We say that $(\sE,\sF)$ is \emph{strongly local}, if for every $u,v\in\sF$ such that both $\supp_{m}(u)$ and $\supp_{m}(v)$ are compact, and $\supp_{m}(u)\cap\supp_{m}(v+a)=\emptyset$ for some $a\in\bR$, then $\sE(u+v)=\sE(u)+\sE(v)$.
	\item We say that the \emph{$p$-Clarkson inequality} \eqref{Cla} holds for $(\sE,\sF)$, if for every $u,v\in \sF$, we have
	\begin{equation}\label{Cla}\tag{$\mathrm{Cla}$}
\begin{aligned}
\mathcal{E}(u+v)+\mathcal{E}(u-v)\geq 2\Big(\mathcal{E}(u)^{\frac{1}{p-1}}+\mathcal{E}(v)^{\frac{1}{p-1}}\Big)^{p-1},\quad&\text{if }p\in(1,2],\\
\mathcal{E}(u+v)+\mathcal{E}(u-v)\leq 2\Big(\mathcal{E}(u)^{\frac{1}{p-1}}+\mathcal{E}(v)^{\frac{1}{p-1}}\Big)^{p-1},\quad&\text{if }p\in[2,\infty).
\end{aligned}
\end{equation}

\item We say the \emph{strong sub-additivity} \eqref{Sub} holds for $(\sE,\sF)$, if for any $u,v\in\sF\cap L^\infty(M,m)$, we have $u\wedge v,\ u\vee v\in\sF\cap L^\infty(M,m)$ and
\begin{equation}\label{Sub}\tag{$\mathrm{SubAdd}$}
	\sE(u\wedge v)+\sE(u\vee v)\leq \sE(u)+\sE(v).
\end{equation}
\end{enumerate}
\end{definition}

Define $M^{2}_{\od}:=(M\times M)\setminus \Sett{(x,x)}{x\in M}$. We call a measure $j$ on $M^{2}_{\od}$ \emph{symmetric}, if  \begin{equation}\label{e.j-Sym}
	\int_{M^{2}_{\od}}f(x)g(y)\dif j(x,y)=\int_{M^{2}_{\od}}f(y)g(x)\dif j(x,y),\ \forall f,g\in C_{c}(M).
\end{equation}
\begin{definition}\label{d.MixedLN}
	We say that $(\sE,\sF)$ is a \emph{mixed local and nonlocal $p$-energy form on $(M,m)$}, if
\begin{enumerate}[label=\textup{(E{\arabic*})},align=left,leftmargin=*,topsep=5pt,parsep=0pt,itemsep=2pt]
	\item\label{lb.EFBN} $(\sE,\sF)$ is a closed and regular $p$-energy form on $(M,m)$ that satisfies \eqref{Sub}.
	\item\label{lb.EFJP} There exists a symmetric Radon measure $j$ on $(M_{\od}^{2},\sB(M_{\od}^{2}))$, such that the form $(\sE^{(L)},\sF\cap C_{c}(M))$ defined by \begin{equation}
	\sE^{(L)}(u):=\sE(u)-\int_{M_{\od}^{2}}\abs{u(x)-u(y)}^{p}\dif j(x,y),\ \forall u\in \sF\cap C_{c}(M),\label{e.BD}
	\end{equation}
	is a strongly local $p$-energy form on $(M,m)$.  
	\item\label{lb.EFNM} The $p$-energy form $(\sE^{(L)},\sF\cap C_{c}(M))$ has the Markovian property.  
	\item\label{lb.EFCI} The $p$-energy form  $(\sE^{(L)},\sF\cap C_{c}(M))$ satisfies the $p$-Clarkson inequality \eqref{Cla}.
\end{enumerate}
If $j\equiv0$, then we say $(\sE,\sF)$ is a \emph{strongly local $p$-energy form} on $(M,m)$; if $\sE^{(L)}(u)=0$ for all $u\in\sF\cap C_{c}(M)$, then we say $(\sE,\sF)$ is a \emph{pure-nonlocal $p$-energy form} on $(M,m)$.
\end{definition}

\begin{remark}
\begin{enumerate}[label=\textup{({\arabic*})},align=left,leftmargin=*,topsep=5pt,parsep=0pt,itemsep=2pt]
	\item For a regular Dirichlet form (see \cite[Chapter 1]{FOT11} for definition), Fukushima's theorem \cite[Theorem 7.2.1]{FOT11} associates a Hunt process (in particular, a strong Markov process) to the form. By the Beurling-Deny decomposition \cite[Theorems 3.2.1 and 5.3.1]{FOT11}, the Dirichlet form splits into a \emph{local part}, a \emph{jump part}, and a \emph{killing part}. The term ``jump part'' reflects that this component governs the discontinuous (jump) behavior of the associated Hunt process. The jump part is encoded by a measure $j$ on $M^{2}_{\od}$, and hence the measure $j$ is naturally referred to as the \emph{jump measure} of the Dirichlet form. It is straightforward to verify that when $p=2$, a mixed local and nonlocal $2$-energy form is precisely a regular Dirichlet form on $(M,m)$ without killing part, and \eqref{e.BD} is exactly its Beurling-Deny decomposition. As a consequence, although no stochastic process is involved in the present paper, we retain this terminology and call $j$ in \eqref{e.BD} the \emph{jump measure}. 
In Lemma \ref{l.j-UNQ}, we will prove that the jump measure of a mixed local and nonlocal $p$-energy form is \emph{unique}.
\item The notions of \emph{capacity}, \emph{relative capacity}, \emph{exceptional set} and \emph{quasi-continuity} etc. for a mixed local and nonlocal $p$-energy form can be developed analogously to that for Dirichlet forms; see Appendix \ref{A.propEF} for details.
\end{enumerate}
\end{remark}

By Proposition \ref{p.QC-ver}, every $u\in\sF$ admits an $\sE$-quasi-continuous version $\wt{u}$. We define \begin{equation}
	\sE^{(J)}(u):=\int_{M^{2}_{\od}}\abs{\wt{u}(x)-\wt{u}(y)}^{p}\dif j(x,y),\ u\in\sF.\label{e.DefEJ}
\end{equation}
By Proposition \ref{p.j-smooth}, the integral in \eqref{e.DefEJ} is well defined. We also define\begin{equation}
	\sE^{(L)}(u):=\sE(u)-\sE^{(J)}(u),\ u\in\sF.\label{e.DefEL}
\end{equation}
For notational convenience, we will often abbreviate $L^{\infty}(M,m)$ as $L^{\infty}$ when no confusion can arise.
\begin{proposition}\label{prop:e}
Let $(\sE,\sF)$ be a {mixed local and nonlocal $p$-energy form}, and let $\sE^{(L)}$ and $\sE^{(J)}$ be as defined in \eqref{e.DefEL} and \eqref{e.DefEJ} respectively. Then the following hold:
\begin{enumerate}[label=\textup{({\arabic*})},align=left,leftmargin=*,topsep=5pt,parsep=0pt,itemsep=2pt]
	\item\label{lb.EFL} The pair $(\sE^{(L)},\sF)$ is a Markovian $p$-energy form on $(M,m)$ that satisfies \eqref{Cla}.
	\item\label{lb.EFJ} The pair $(\sE^{(J)},\sF)$ is a Markovian $p$-energy form on $(M,m)$ that satisfies \eqref{Cla}.
	\item\label{lb.EFfull} The pair $(\sE,\sF)$ is a Markovian $p$-energy form on $(M,m)$ that satisfies \eqref{Cla}.
	\item\label{lb.EFSL1} The pair $(\sE^{(L)},\sF\cap L^{\infty}(M,m))$ is a strongly local $p$-energy form on $(M,m)$. 
	\item\label{lb.EFSL2} If $(\sE,\sF)$ satisfies \begin{equation}
		\lim_{N\to\infty}\sE(u-(-N\vee u)\wedge N)=0, \ \forall u\in\sF.\label{e.Cutconv}
	\end{equation}
	Then $(\sE^{(L)},\sF)$ is also strongly local.
\end{enumerate}
\end{proposition}
\begin{proof} \begin{enumerate}[label=\textup{({\arabic*})},align=left,leftmargin=*,topsep=5pt,parsep=0pt,itemsep=2pt]
	\item[\ref{lb.EFL}] 	Let $u\in\sF$. We claim that for any sequence $\{u_n\}_{n\in\bN}\subset\sF\cap C_{c}(M)$ such that $u_{n}\to u$ in $\sE_{1}$, we have $\sE^{(L)}(u)=\lim_{n\to\infty}\sE^{(L)}(u_{n})$. In fact, let $\widetilde{u}$ be an $\sE$-quasi-continuous version of $u$. By the semi-norm property of $\sE^{(L)}$, we know that $\{\sE^{(L)}(u_{n})^{1/{p}}\}$ is a Cauchy sequence, so the limit $\lim_{n\to\infty}\sE^{(L)}(u_{n})\in\bR$ exists. By Lemma \ref{l.QEconv}, there is a subsequence $\{n_{k}\}$ such that $u_{n_{k}}\to \widetilde{u}$ $\sE$-q.e. (see Subsection \ref{ss.capacity} for the notion of \emph{$\sE$-q.e}). For any $v\in\sF$, we define a function $Tv$ by 
 \begin{equation}
   Tv(x,y):=v(x)-v(y), \ \forall (x,y)\in M_{\od}^{2}.
 \end{equation} 
 Since $\sE^{(J)}$ is dominated by $\sE$, the sequence $\{Tu_{n}\}$ is Cauchy in $L^{p}(M^{2}_{\od},j)$ and by Proposition \ref{p.j-smooth}, we have $\lim_{n\to\infty}Tu_{n}(x,y)=\widetilde{u}(x)-\widetilde{u}(y)=T\widetilde{u}(x,y)$ for $j$-a.e. $(x,y)\in M^{2}_{\od}$. Thus by the completeness of $L^{p}(M^{2}_{\od},j)$, $Tu_{n}$ converges to $T\widetilde{u}$ in $L^{p}(M^{2}_{\od},j)$. Therefore, we obtain \begin{align}
		\lim_{n\to\infty}\sE^{(L)}(u_{n})&=		\lim_{k\to\infty}\sE^{(L)}(u_{n_{k}})=\lim_{k\to\infty}\sE(u_{n_{k}})-\lim_{k\to\infty}\int_{M^{2}_{\od}}\abs{T(u_{n_{k}})(x,y)}^{p}\dif j(x,y)\\
		&=\sE(u)-\int_{M^{2}_{\od}}\abs{T\widetilde{u}(x,y)}^{p}\dif j(x,y)\\
		&=\sE(u)-\int_{M^{2}_{\od}}\abs{\widetilde{u}(x)-\widetilde{u}(y)}^{p}\dif j(x,y)=\sE^{(L)}(u),
	\end{align}
	and the claim holds. As a consequence, using the regularity of $(\sE,\sF)$, we can extend the semi-norm property and the $p$-Clarkson inequality of $(\sE^{(L)},\sF\cap C_{c}(M))$ to $(\sE^{(L)},\sF)$. Therefore, $(\sE^{(L)},\sF)$ is a $p$-energy form that satisfies \eqref{Cla}.

{We now prove the Markovian property of $(\sE^{(L)},\sF)$. Let $\psi\in C(\bR)$ such that $\psi(0)=0$ and $\abs{\psi(s)-\psi(t)}\leq\abs{s-t}$ for every $s,t\in\bR$. For any $u\in\sF$, by the regularity of $(\sE,\sF)$, there exists a sequence $\{u_{n}\}_{n\in\bN}\subset\sF\cap C_{c}(M)$ such that $u_{n}\to u$ in $\sE_{1}$. By \eqref{e.BD} and the Markovian property of $(\sE^{(L)},\sF\cap C_{c}(M))$ in \ref{lb.EFNM}, we have \begin{align}
 		\sE(\psi\circ u_{n})&=\sE^{(L)}(\psi\circ u_{n})+\int_{M_{\od}^{2}}\abs{\psi(u_{n}(x))-\psi(u_{n}(y))}^{p}\dif j(x,y)\\
 		&\leq \sE^{(L)}(u_n)+\int_{M_{\od}^{2}}\abs{u_n(x)-u_n(y)}^{p}\dif j(x,y)=\sE(u_{n}) \\
 &\leq \left(\sE(u_{n}-u)^{1/p}+\sE(u)^{1/p}\right)^p<\infty \ \text{(since $u_{n}\to u$ in $\sE_{1}$)}.\label{e.Mar1}
 	\end{align} 
Since $u_{n}\to u$ in $L^{p}(M,m)$, we have $\psi\circ u_{n}\to \psi\circ u$ in $L^{p}(M,m)$, which combines with \eqref{e.Mar1} and Lemma \ref{lem:A1} gives that $\psi\circ u\in\sF$ and $\psi\circ u_{n}\to \psi\circ u$ weakly in $(\sF,\sE_{1}^{1/p})$. By Mazur's lemma \cite[Theorem 2 in Section V.1]{Yos95}, for each $n\in\bN$, there is $N_{n}\geq n$ and a convex combination $\{\lambda_{k}^{(n)}\}_{k=n}^{N_n}\in[0,1]$ with $\sum_{k=n}^{N_n}\lambda_{k}^{(n)}=1$ such that $\sum_{k=n}^{N_n}\lambda_{k}^{(n)}(\psi\circ u_{k})\to \psi\circ u$ in $\sE_1$. Therefore, by the semi-norm property of $(\sE^{(L)})^{1/p}$,
\begin{align}
 		\sE^{(L)}(\psi\circ u)^{1/p}&=\lim_{n\to\infty}\sE^{(L)}\Big(\sum_{k=n}^{N_n}\lambda_{k}^{(n)}(\psi\circ u_{k})\Big)^{1/p}\leq \liminf_{n\to\infty}\sum_{k=n}^{N_n}\lambda_{k}^{(n)}\sE^{(L)}((\psi\circ u_{k}))^{1/p}\\
 		&\leq \liminf_{n\to\infty}\sum_{k=n}^{N_n}\lambda_{k}^{(n)}\sE^{(L)}(u_{k})^{1/p}=\sE^{(L)}(u)^{1/p} \text{\ (as $u_{k}\to u$ in $\sE_{1}$)}.
 	\end{align}
}

\item[\ref{lb.EFJ}] It is clear that $(\sE^{(J)},\sF)$ is Markovian. 
Let $u,v\in \sF\cap C_{c}(M)$. For $p\in [2,\infty)$, applying \eqref{e.B3-1} on $(M_{\od}^{2},j)$ with $f=Tu$ and $g=Tv$, we have
 \begin{align}
  \sE^{(J)}(u+v)+\sE^{(J)}(u-v) & =\norm{Tu+Tv}_{L^p(M^{2}_{\od},j)}^p+\norm{Tu-Tv}_{L^p(M^{2}_{\od},j)}^p\\
    & \leq 2\left(\norm{Tu}_{L^p(M^{2}_{\od},j)}^{\frac{p}{p-1}}+\norm{Tv}_{L^p(M^{2}_{\od},j)}^{\frac{p}{p-1}}\right)^{p-1} \\
    & =2 \left(\sE^{(J)}(u)^{\frac{1}{p-1}}+\sE^{(J)}(v)^{\frac{1}{p-1}}\right)^{p-1}. \label{e.EJ}
 \end{align}
For $p\in (1,2]$, applying \eqref{e.B3-2} for $f=(Tu+Tv)/2$ and $g=(Tu-Tv)/2$, we have
 \begin{equation}\label{e.EJp<2}
  \sE^{(J)}(u+v)+\sE^{(J)}(u-v)\geq 2\Big(\mathcal{E}^{(J)}(u)^{\frac{1}{p-1}}+\mathcal{E}^{(J)}(v)^{\frac{1}{p-1}}\Big)^{p-1}.
 \end{equation}
 If $u,v\in\sF$, by the regularity of $(\sE,\sF)$, there exist $\{u_{n}\}_{n\in \bN}$, $\{v_{n}\}_{n\in \bN}\subset\sF\cap C_{c}(M)$ such that $u_{n}\to u$ and $v_{n}\to v$ in $\sE_{1}$. Since $\sE^{J}(u_{n}-u)\leq \sE(u_{n}-u)$, we see that $\sE^{J}(u_{n})\to \sE^{J}(u)$ and $\sE^{J}(v_{n})\to \sE^{J}(v)$. Similarly, $\sE^{J}(u_{n}\pm v_{n})\to \sE^{J}(u\pm v)$. So we may extend \eqref{e.EJ} and \eqref{e.EJp<2} to all $u,v\in \sF$ and thus prove the $p$-Clarkson inequality of $(\sE^{(J)},\sF)$.
 \item[\ref{lb.EFfull}] Let $q=({p-1})^{-1}$. For $p\in[2,\infty)$ and any $u,v\in\sF$, by the $p$-Clarkson inequality for $(\sE^{(L)},\sF)$ and for $(\sE^{(J)},\sF)$, we have
 \begin{align}
  \sE(u+v)+\sE(u-v) & =  \sE^{(L)}(u+v)+\sE^{(L)}(u-v)+  \sE^{(J)}(u+v)+\sE^{(J)}(u-v)\\
    & \leq 2 \left(\sE^{(L)}(u)^{q}+\sE^{(L)}(v)^{q}\right)^{1/q}+ 2 \left(\sE^{(J)}(u)^{q}+\sE^{(J)}(v)^{q}\right)^{1/q}\\
    & \leq 2 \left(\sE(u)^{q}+\sE(v)^{q}\right)^{1/q}\  (\text{Minkowski inequality for $q\in(0,1]$}) \\
    &=2 \left(\sE(u)^{\frac{1}{p-1}}+\sE(v)^{\frac{1}{p-1}}\right)^{p-1}. \label{e.E-1}
 \end{align}
Similarly, for $p\in(1,2]$, $\sE(u+v)+\sE(u-v) \geq 2 \left(\sE(u)^{\frac{1}{p-1}}+\sE(v)^{\frac{1}{p-1}}\right)^{p-1}$.

Therefore, $(\sE,\sF)$ satisfies \eqref{Cla}. The Markovian property of $(\sE,\sF)$ follows from the Markovian property of $(\sE^{L},\sF)$ and $(\sE^{(J)},\sF)$.

 \item[\ref{lb.EFSL1}] We now prove that $(\sE^{(L)},\sF\cap L^{\infty}(M,m))$ is strongly local. Assume that $u,v\in\sF\cap L^{\infty}(M,m)$ have the property that both $\supp_{m}(u)$ and $\supp_{m}(v)$ are compact, and $\supp_{m}(u)\cap\supp_{m}(v+a)=\emptyset$ for some $a\in\bR$. Since $\supp_{m}(u)$ is compact, and $(M,d)$ is a metric space, there are two open set $U_1$ and $U_2$ in $M$, such that \begin{equation}
	\supp_{m}(u)\subset U_1\subset \ol{U_1}\subset U_{2}\subset \ol{U_2}\subset M\setminus(\supp_{m}(v+a))
\end{equation} 
and $\restr{v}{U_1}=-a$ $m$-a.e.. By Proposition \ref{p.cut}, there exist $\psi_1,\psi_2\in\sF\cap C_{c}(M)$ such that $0\leq\psi_{1},\psi_{2}\leq1$, $\supp(\psi_1)\subset U_2$, $\restr{\psi_1}{U_1}=1$, $\supp(\psi_2)\subset M\setminus(\supp_{m}(v+a))$ and $\restr{\psi_2}{U_2}=1$. Let $u_{n},v_{n}\in\sF\cap C_{c}(M)$ such that $u_{n}\to u$ in $\sE_1$ and $v_{n}\to v$ in $\sE_1$. Since $u,v\in L^{\infty}(M,m)$, by Lemma \ref{l.cuofbdd} we may assume that $\sup_n\norm{u_{n}}_{L^{\infty}}\leq\norm{u}_{L^{\infty}}$ and $\sup_n\norm{v_{n}}_{L^{\infty}}\leq\norm{v}_{L^{\infty}}$. Define $\ol{u}_{n}:=\psi_1 u_n$ and $\ol{v}_{n}:=-a\psi_2+(1-\psi_2) v_n$. Then $\ol{u}_{n},\ol{v}_{n}\in\sF\cap C_{c}(M)$ by Proposition \ref{p.mar}. Since $\ol{u}_{n}\to \psi_1 u=u$ in $L^{p}(M,m)$ and $\sup_{n}\sE(u_{n})<\infty$. By Lemma \ref{lem:A1}, $\set{\ol{u}_{n}}$ converges to $u$ weakly in $(\sF,\sE_1^{1/p})$. By Mazur's lemma \cite[Theorem 2 in Section V.1]{Yos95}, for each $n\in\bN$, there is a convex combination $\{\lambda_{k}^{(n)}\}_{k=1}^{n}\in[0,1]$ with $\sum_{k=1}^{n}\lambda_{k}^{(n)}=1$ such that $\wh{u}_{n}:=\sum_{k=1}^{n}\lambda_{k}^{(n)}\ol{u}_{k}\in\sF\cap C_{c}(M)$ converges to $u$ in $(\sF,\sE_1)$. Note that $\supp({\wh{u}_{n}})\subset\supp(\psi_1)\subset U_2$. Similarly, since $\ol{v}_{n}\to -a\psi_2+(1-\psi_2) v$ in $L^{p}(M,m)$ and $\sup_{n}\sE(v_{n})<\infty$, there exists a convex combination $\wh{v}_{n}\in\sF\cap C_{c}(M)$ of $\{\ol{v}_{k}\}_{k=1}^{n}$ such that $\wh{v}_{n}\to -a\psi_2+(1-\psi_2) v$ in $(\sF,\sE_1)$ and $\restr{\wh{v}_{n}}{U_2}=-a$. Therefore $\supp(\wh{v}_{n}+a)\cap \supp(\wh{u}_{n})=\emptyset$. By the strong locality of $(\sE^{(L)},\sF\cap C_{c}(M))$, we know that $\sE^{(L)}(\wh{u}_{n}+\wh{v}_{n})=\sE^{(L)}(\wh{u}_{n})+\sE^{(L)}(\wh{v}_{n})$. Letting $n\to\infty$, we obtain $\sE^{(L)}(u+v)=\sE^{(L)}(u)+\sE^{(L)}(v)$.
\item[\ref{lb.EFSL2}] If $u,v\in\sF$ have the property that both $\supp_{m}(u)$ and $\supp_{m}(v)$ are compact, and $\supp_{m}(u)\cap\supp_{m}(v+a)=\emptyset$ for some $a\in\bR$, then we may apply the strongly locality proved in \ref{lb.EFSL1} to $(-N\vee u)\wedge N$ and $(-N\vee v)\wedge N$ for $N>2\abs{a}$, and letting $N\to\infty$ and use the fact that $\lim_{N\to\infty}\sE(u-(-N\vee u)\wedge N)=0$ and $\lim_{N\to\infty}\sE(v-(-N\vee v)\wedge N)=0$ given in \eqref{e.Cutconv}. 
\end{enumerate}
\end{proof}
\begin{remark}
	If $p=2$, the space $(\sF,\sE_{1}^{1/2})$ is a Hilbert space, and therefore the equality \eqref{e.Cutconv} holds automatically, see for example \cite[Theorem 1.4.2-(iii)]{FOT11}.
\end{remark}

\begin{proposition}[{\cite[Theorem 3.7]{KS25}}]\label{p.Varia}
	Let $(\sE,\sF)$ be a $p$-energy form on $(M,m)$ that satisfies \eqref{Cla}. Then for any $u,v\in\sF$, the following limit exists\begin{equation}\label{e.Varia}
		\sE(u;v):=\frac{1}{p}\restr{\frac{\dif}{\dif t}\sE(u+tv)}{t=0}\in\bR
	\end{equation}
	and satisfies the following properties: for any $u,v\in\sF$ and any $a\in\bR$, \begin{align}
	\sE(au;v)&=\sgn(a)\abs{a}^{p-1}\sE(u;v)\\
		\abs{\sE(u;v)}&\leq \sE(u)^{(p-1)/p}\sE(v)^{1/p}. \label{e.sE}
	\end{align}
\end{proposition}

\begin{proposition}
Let $(\sE,\sF)$ be a {mixed local and nonlocal $p$-energy form}. Let $\sE(\wcdot;\wcdot)$, $\sE^{(L)}(\wcdot;\wcdot)$ and $\sE^{(J)}(\wcdot;\wcdot)$ be given in \eqref{e.Varia} with respect to the forms $(\sE,\sF)$, $(\sE^{(L)},\sF)$, and $(\sE^{(J)},\sF)$, respectively. Then, 
\begin{equation}
	\sE(u;v)=\sE^{(L)}(u;v)+\sE^{(J)}(u;v),\ \text{for every $u,v\in\sF$,} \label{e.E-5}
\end{equation}
and we have the following expression for the nonlocal term:\begin{equation}
	\sE^{(J)}(u;v)=\int_{M_{\od}^{2}}\abs{\wt{u}(x)-\wt{u}(y)}^{p-2}(\wt{u}(x)-\wt{u}(y))(\wt{v}(x)-\wt{v}(y))\dif j(x,y),\label{e.E-2} 
\end{equation}
where $\wt{u}$ and $\wt{v}$ are $\sE$-quasi-continuous version of $u$ and $v$, respectively.
\end{proposition}
\begin{proof}
For any $u,v\in\sF$, by Propositions \ref{prop:e} and \ref{p.Varia}, and $\sE(u+tv)=\sE^{(L)}(u+tv)+\sE^{(J)}(u+tv)$, we take differentiation on both sides and obtain \eqref{e.E-5}. To prove \eqref{e.E-2}, by applying Lemma \ref{l.a-blem} with $a=\wt{u}(x)-\wt{u}(y)$ and $b=\wt{v}(x)-\wt{v}(y)$, we know that for all $t\in (0,1)$,
\begin{align}
  &\phantom{\ \leq}\left|\frac{\abs{\wt{u}(x)-\wt{u}(y)+t(\wt{v}(x)-\wt{v}(y))}^p
  -\abs{\wt{u}(x)-\wt{u}(y)}^p}{t}\right| \\
  &\leq  2^{p-1}(p+1)(\abs{\wt{u}(x)-\wt{u}(y)}^p+|\wt{v}(x)-\wt{v}(y)|^p) \in L^p(M_{\od}^{2},j).
\end{align}

By the dominated convergence theorem,
\begin{align}
  \sE^{(J)}(u;v)&=\frac{1}{p}\restr{\frac{\dif}{\dif t}\sE^{(J)}(u+tv)}{t=0}=\lim_{t\rightarrow 0}\frac{\sE^{(J)}(u+tv)-\sE^{(J)}(u)}{pt}  \\
  &=  \lim_{t\downarrow 0^+}\int_{M_{\od}^{2}}\frac{\abs{\wt{u}(x)-\wt{u}(y)+t(\wt{v}(x)-\wt{v}(y))}^p-\abs{\wt{u}(x)-\wt{u}(y)}^p}{pt} \dif j(x,y) \\
   &{=} \int_{M_{\od}^{2}}\lim_{t\downarrow 0^+}\frac{\abs{\wt{u}(x)-\wt{u}(y)+t(\wt{v}(x)-\wt{v}(y))}^p-\abs{\wt{u}(x)-\wt{u}(y)}^p}{pt} \dif j(x,y) \\
   & =\int_{M_{\od}^{2}}\abs{\wt{u}(x)-\wt{u}(y)}^{p-2}(\wt{u}(x)-\wt{u}(y))(\wt{v}(x)-\wt{v}(y))\dif j(x,y),
\end{align}
thus obtaining \eqref{e.E-2}.
\end{proof}
 
\begin{theorem}\label{t.sasEM}
 	Let $(\sE,\sF)$ be a {mixed local and nonlocal $p$-energy form}. Then there exists a unique set $\set{\Gamma^{(L)}\la u\ra }_{u\in\sF}$ of Radon measures on $(M,d)$ with the following properties
 	\begin{enumerate}[label=\textup{(M{\arabic*})},align=left,leftmargin=*,topsep=5pt,parsep=0pt,itemsep=2pt]
	\item\label{lb.M-measu} For every $u\in\sF$, $\Gamma^{(L)}\la u\ra(M)=\sE^{(L)}(u)$.
	\item\label{lb.M-clark} For any Borel set $A\subset M$, $(\Gamma^{(L)}\la\wcdot\ra(A),\sF)$ is a $p$-energy form on $(A,\restr{m}{A})$ and satisfies the $p$-Clarkson inequality \eqref{Cla}.
	\item\label{lb.M-smtt}\textup{(Smoothness)} For any $u\in\sF$, the measure $\Gamma^{(L)}\la{u}\ra $ charges no set of zero capacity. That is, if $N\subset M$ is a Borel set such that $\capacity_{1}(N)=0$, then $\Gamma^{(L)}\la{u}\ra (N)=0$.
	\item\label{lb.M-chain} \textup{(Chain rule)} If $\psi$ is a piecewise $C^{1}$ function on $\bR$ with $\psi(0)=0$, then for every $u,v\in\sF\cap L^{\infty}$, we have $\psi\circ u\in\sF\cap L^{\infty}$ and
\begin{align}
\dif\Gamma^{(L)}\la \psi\circ u;v\ra &=\sgn(\psi^{\prime}\circ \wt{u})\abs{\psi^{\prime}\circ \wt{u}}^{p-1}\dif \Gamma^{(L)}\la u;v\ra,\label{e.Chain1}\\
\dif\Gamma^{(L)}\la v;\psi\circ u\ra &=(\psi^{\prime}\circ \wt{u})\dif \Gamma^{(L)}\la v;u\ra.\label{e.Chain2}
\end{align}

\item\label{lb.M-leibn} \textup{(Leibniz rule)} For every $u,v,w\in\sF\cap L^\infty$, we have $v\cdot w\in\sF$ and 
\begin{equation}\label{e.Leb1}
\dif\Gamma^{(L)}\la u;v\cdot w\ra=\wt{v}\dif \Gamma^{(L)}\la u;w\ra+\wt{w}\dif \Gamma^{(L)}\la u;v\ra .
\end{equation}

 \item\label{lb.M-densi} \textup{(Energy image density property)} For any $u\in\sF\cap L^{\infty}$, we have $\wt{u}_{*}(\Gamma^{(L)}\la u\ra )\ll \rL^{1}$, where $\rL^{1}$ is the Lebesgue measure on $\bR$. 
	\item\label{lb.M-local}\textup{(Locality)} For any $u,v\in\sF\cap L^{\infty}$ and any Borel set $A\subset M$, if $\wt{u}-\wt{v}$ is a constant $\sE$-q.e. on $A$, then $\Gamma^{(L)}\la u\ra (A)=\Gamma^{(L)}\la v\ra (A)$.
\end{enumerate}
\end{theorem}
 \begin{proof}
We first note that, $(\sE^{(L)},\sF)$ is a $p$-energy form that satisfies all assumptions in \cite[(F2)-(F5) in Assumption 1.3]{Sas25}. In fact, by Proposition \ref{prop:e}-\ref{lb.EFL} and \ref{lb.EFSL1}, we know that \cite[(F2), (F3) in Assumption 1.3]{Sas25}\footnote{Note that the $p$-Clarkson inequality in \eqref{Cla} implies \cite[(CI1)-(CI4) in Definition 1.2 for $\sE^{1/p}$]{Sas25} by \cite[Remark 3.3]{KS25}.} holds for $(\sE^{(L)},\sF)$. By \ref{lb.EFJP}, we know that \cite[(F4) in Assumption 1.3]{Sas25} holds. Since $\sF\cap C_{c}(M)$ is dense in $(\sF,\sE_1^{1/p})$, we know that for any $u\in \sF$, there exists a sequence $\{u_n\}\subset \sF\cap C_{c}(M)$ such that  $\sE_1^{1/p}(u_n-u)\rightarrow 0$ as $n\rightarrow \infty$, which implies \[\lim_{n\rightarrow \infty}\left(\sE^{(L)}(u_n-u)+\norm{u_n-u}_{L^p(M,m)}^p\right)^{1/p}=0.\]
Therefore, $\sF\cap C_{c}(M)$ is dense in $\left(\sF,\big(\sE^{(L)}(\wcdot)+\norm{\cdot}_{L^p(M,m)}^p\big)^{1/p}\right)$, thus showing \cite[(F5) in Assumption 1.3]{Sas25}. 

The closedness in \cite[(F1) in Assumption 1.3]{Sas25} may not be satisfied. Nevertheless, the places where (F1) is used can be easily bypassed using the Markovian property of $(\sE^{(L)},\sF)$ and the closedness of $(\sE,\sF)$. Specifically, the approach adopted in \cite{Sas25} where (F1) was used involved constructing piecewise functions $\{f_n\}_{n\in \mathbb{N}}$ and verifying 
\begin{equation}
  \sE(f_{m_3})\leq \sE\left(\frac{f_{m_1}+f_{m_2}}{2}\right)  \label{e.F1}
\end{equation}
 for any $m_1\leq m_2\leq m_3$. However, we observed that all $\{f_n\}$ constructed in \cite{Sas25} has the property that $f_{m_{3}}$ is a truncation of $f_{m_1}$ and $f_{m_2}$, for any $m_1\leq m_2\leq m_3$. Thus, \eqref{e.F1} could be verified using the Markovian property of $(\sE^{(L)},\sF)$, thereby allowing the subsequent steps to rely solely on the closed property of $(\sE,\sF)$.

Therefore, \cite[Theorem 1.4]{Sas25} gives us a unique set $\set{\Gamma^{(L)}\la u\ra }_{u\in\sF}$ of Radon measures on $(M,d)$ such that \ref{lb.M-measu} and \ref{lb.M-clark}  hold for all functions in $\sF$, while \ref{lb.M-chain}, \ref{lb.M-leibn}, \ref{lb.M-densi} and \ref{lb.M-local} hold for functions in $\sF\cap C_{c}(M)$. The rest of the proof is to prove \ref{lb.M-smtt} and to extend the latter to functions in $\sF$.
\begin{enumerate}[label=\textup{(M{\arabic*})},align=left,leftmargin=*,topsep=5pt,parsep=0pt,itemsep=2pt]
\item[\ref{lb.M-smtt}] This fact essentially follows from \cite[Proposition 8.12]{Yan25c}. In our setting, we use \cite[Corollary 5.9]{Sas25} to obtain \cite[(4.35) in Proposition 4.16]{KS25} for functions in $\sF\cap C_{c}(M)$, and we use Proposition \ref{p.mar} in place of \cite[condition (Alg)]{Yan25c}.
\item[\ref{lb.M-chain}] Let $u\in\sF\cap L^{\infty}$. By Proposition \ref{p.mar}, we have $\psi\circ u\in\sF$. We then show \eqref{e.Chain1} holds for $u\in\sF\cap L^{\infty}$ and $v\in\sF\cap C_{c}(M)$. By the regularity of $(\sE,\sF)$, there exists $u_{n}\in\sF\cap C_{c}(M)$ such that $u_{n}\to u$ in $\sE_{1}$. By Lemma \ref{l.cuofbdd}, we may assume that $\sup_{n\in\bN}\norm{u_{n}}_{L^{\infty}}\leq\norm{u}_{L^{\infty}}$. By Lemma \ref{l.QEconv}, we can choose a subsequence and thus may assume further that $u_{n}\to \wt{u}$ $\sE$-q.e.. We claim that \begin{align}
	\lim_{m,n\to\infty}&\sE^{(L)}(\psi(u_{m})-\psi(u_{n}))=0,\label{e.Chain10}\\
	\lim_{n\to\infty}&\sE^{(J)}(\psi(u_n)-\psi(u))=0,\label{e.Chain10+}\\
	\lim_{n\to\infty}&\norm{\psi(u_n)-\psi(u)}_{L^{p}(M,m)}=0.\label{e.Chain10++}
 \end{align}
We first prove \eqref{e.Chain10}. In fact, for any $w\in\sF$ with $\sE(w)\leq 1$, we have \begin{align}
	&\phantom{\ \leq}\abs{\sE^{(L)}(\psi(u_{m});w)-\sE^{(L)}( \psi(u_{n});w)}\\
	&=\abs{\int_{M}\dif\Gamma^{(L)}\la \psi(u_{m});w\ra-\int_{M}\dif\Gamma^{(L)}\la \psi(u_{n});w\ra }=\abs{\int_{M}\big(f_{m}(x)g_{m}(x)-f_{n}(x)g_{n}(x)\big)\dif\nu(x)}\\
	&\leq \abs{\int_{M}(f_{m}(x)-f_{n}(x))g_{m}(x)\dif\nu(x)}+\int_{M}\abs{f_{n}(x)}\abs{(g_m(x)-g_{n}(x))}\dif\nu(x)\label{e.Chain11}
\end{align}
where \begin{align}
	f_{m}:=\sgn(\psi^{\prime}\circ u_{m})\abs{\psi^{\prime}\circ u_{m}}^{p-1}\text{\ and\ }g_{m}:=\frac{\dif\Gamma^{(L)}\la u_{m};w\ra}{\dif\nu},\ m\in\bN.
\end{align}
Let $L:=\esup_{x\in[-\norm{u}_{L^{\infty}},\norm{u}_{L^{\infty}}]}\abs{\psi^{\prime}(x)}$ and $\nu$ be the minimal energy-dominant measure of \( (\sE^{(L)}, \mathcal{F}) \) (see \cite[Definition 5.4]{Sas25}). Then by \ref{lb.M-clark} and \cite[Proposition 4.8]{KS25}, \begin{align}
	&\phantom{\ \leq}\abs{\int_{M}(f_{m}(x)-f_{n}(x))g_{m}(x)\dif\nu(x)}\\
	&=\abs{\int_{M}(f_{m}(x)-f_{n}(x))\dif\Gamma^{(L)}\la u_{m};w\ra}\leq \int_{M}\abs{f_{m}(x)-f_{n}(x)}\dif\abs{\Gamma^{(L)}\la u_{m};w\ra }\\
	&\leq \Big(\int_{M}\abs{f_{m}(x)-f_{n}(x)}\dif\Gamma^{(L)}\la w\ra \Big)^{1/p}\cdot \Big(\int_{M}\abs{f_{m}(x)-f_{n}(x)}\dif\Gamma^{(L)}\la u_{m}\ra\Big)^{(p-1)/p}\\
	&\leq (2L)^{(p-1)/p}\sE^{(L)}(w)\cdot \Big(\int_{M}\abs{f_{m}(x)-f_{n}(x)}\frac{\dif\Gamma^{(L)}\la u_{m}\ra}{\dif \nu}\dif \nu\Big)^{(p-1)/p}\label{e.Chain12}
\end{align}
and by \cite[formula (4.9)]{KS25},\begin{align}
	&\phantom{\ \leq}\int_{M}\abs{f_{n}(x)}\abs{(g_m(x)-g_{n}(x))}\dif\nu(x)\leq L^{p-1}\norm{\frac{\dif\Gamma^{(L)}\la u_{m};w\ra}{\dif\nu}-\frac{\dif\Gamma^{(L)}\la u_{n};w\ra }{\dif\nu}}_{L^{1}(M,\nu)}\\
	&\leq CL^{p-1}(\sE^{(L)}(u_{m})\vee\sE^{(L)}(u_{n}))^{(p-1-\alpha_{p})/p}\sE^{(L)}(u_{m}-u_{n})^{\alpha_{p}/p}\sE^{(L)}(w)\\
	&\leq C_{1}\sE^{(L)}(u_{m}-u_{n})^{\alpha_{p}/p}\sE^{(L)}(w),\label{e.Chain13}
\end{align}
where $C_1$ depends only on $p$ and $\alpha_p=\frac{1}{p}\wedge \frac{p-1}{p}$. Combining \eqref{e.Chain11}, \eqref{e.Chain12}, \eqref{e.Chain13} and \cite[Theorem 3.9]{KS25}, we have 
\begin{align}
	&\phantom{\ \leq}\sE^{(L)}(\psi(u_{m})-\psi(u_{n}))\\
	&\leq C_{2}\Bigg(\Big(\int_{M}\abs{f_{m}(x)-f_{n}(x)}\frac{\dif\Gamma^{(L)}\la u_{m}\ra }{\dif \nu}\dif \nu\Big)^{(p-1)/p}+\sE^{(L)}(u_{m}-u_{n})^{\alpha_{p}/p}\Bigg)^{(p-1)\wedge1}.
\end{align}
Since $u_{n}\to \wt{u}$ $\sE$-q.e., we know that $f_{n}\to \sgn(\psi^{\prime}\circ \wt{u})\abs{\psi^{\prime}\circ \wt{u}}^{p-1} $ $\nu$-a.e. by \ref{lb.M-smtt}. Thus by the generalised dominated convergence theorem \cite[Theorem 19 in Section 4.4]{Roy88}, we have \[\lim_{m,n\to\infty}\int_{M}\abs{f_{m}(x)-f_{n}(x)}\frac{\dif\Gamma^{(L)}\la u_{m}\ra }{\dif \nu}\dif \nu=0.\] Thus we obtain \eqref{e.Chain10}. To show \eqref{e.Chain10+}, we note that \begin{align}
	&\phantom{\ \leq}\abs{\big((\psi\circ u_{n})(x)-(\wt{\psi\circ u})(x)\big)-\big((\psi\circ u_{n})(y)-(\wt{\psi\circ u})(y)\big)}\\
	&\leq 2^{p}L^{p}\big(\abs{{u_{n}}(x)-{u_{n}}(y)}^{p}+ \abs{\wt{u}(x)-\wt{u}(y)}^{p}\big).
\end{align}
By taking integration on $M^{2}_{\od}$ with respect to $j$ and Minkowski inequality
 \begin{align}
	&\phantom{\ \leq}\abs{\Big(\int_{M^{2}_{\od}}\abs{{u_{n}}(x)-{u_{n}}(y)}^{p}\dif j(x,y)\Big)^{1/p}-\Big(\int_{M^{2}_{\od}}\abs{\wt{u}(x)-\wt{u}(y)}^{p}\dif j(x,y)\Big)^{1/p}}\\
	&\leq \Big(\int_{M^{2}_{\od}}\abs{({u_{n}}(x)-{u_{n}}(y))-(\wt{u}(x)-\wt{u}(y))}^{p}\dif j(x,y)\Big)^{1/p} \\
	&=\sE^{(J)}(u_n-u)^{1/p}\to 0\  \text{ as $n\to\infty$}.
\end{align}
Therefore by the generalised dominated convergence theorem \cite[Theorem 19 in Section 4.4]{Roy88}, we obtain \eqref{e.Chain10+}. The convergence in \eqref{e.Chain10++} is a direct consequence of the Lipschitz property of $\psi$ on the compact set $[-\norm{u}_{L^{\infty}},\norm{u}_{L^{\infty}}]$.

Combining \eqref{e.Chain10}, \eqref{e.Chain10+} and \eqref{e.Chain10++}, we see that $\{\psi\circ u_{n}\}$ is a Cauchy sequence in $(\sF,\sE_{1})$ and is convergent to $\psi\circ u$ $m$-a.e.. Therefore $\psi\circ u_{n}\to \psi\circ u\text{ in $\sE_{1}$}$. Since \ref{lb.M-chain} already holds for functions in $\sF\cap C_{c}(M)$, we have \begin{equation}
	\frac{\dif\Gamma^{(L)}\la \psi\circ u_{n};v\ra }{\dif\nu}=\sgn(\psi^{\prime}\circ u_{n})\abs{\psi^{\prime}\circ u_{n}}^{p-1}\frac{\dif \Gamma^{(L)}\la u_{n};v\ra }{\dif\nu},\ \nu\text{-a.e.}.\label{e.Chian14}
\end{equation}
By \cite[Lemma 5.6]{Sas25}, there is a common subsequence $n_{k}$ such that \begin{align}
	\frac{\dif\Gamma^{(L)}\la \psi\circ u_{n_{k}};v\ra }{\dif\nu}&\to \frac{\dif\Gamma^{(L)}\la \psi\circ u;v\ra }{\dif\nu}, \ \nu\text{-a.e.},\label{e.Chian15}\\
	\text{and\ }\frac{\dif \Gamma^{(L)}\la u_{n_{k}};v\ra }{\dif\nu}&\to \frac{\dif \Gamma^{(L)}\la u;v\ra }{\dif\nu}, \ \nu\text{-a.e.}.\label{e.Chian16}
\end{align}
Along $\{n_{k}\}$, since $u_{n}\to \wt{u}$ $\sE$-q.e., we have \begin{equation}
	\sgn(\psi^{\prime}\circ u_{n_{k}})\abs{\psi^{\prime}\circ u_{n_{k}}}^{p-1}\to \sgn(\psi^{\prime}\circ \wt{u})\abs{\psi^{\prime}\circ \wt{u}}^{p-1}, \ \nu\text{-a.e.}.\label{e.Chian17}
\end{equation}
Combining \eqref{e.Chian15}, \eqref{e.Chian16} and \eqref{e.Chian17}, we can take $n\to\infty$ in \eqref{e.Chian14} and obtain that \begin{equation}
	\frac{\dif\Gamma^{(L)}\la \psi\circ u;v\ra }{\dif\nu}=\sgn(\psi^{\prime}\circ \wt{u})\abs{\psi^{\prime}\circ \wt{u}}^{p-1}\frac{\dif \Gamma^{(L)}\la u;v\ra }{\dif\nu}, \ \nu\text{-a.e.},
\end{equation}
which is equivalent to \eqref{e.Chain1} for $u\in\sF\cap L^{\infty}$ and for $v\in\sF\cap C_{c}(M)$. Similar argument allows us to obtain \eqref{e.Chain1} and \eqref{e.Chain2} for all $u,v\in\sF\cap L^{\infty}$. 

	\item[\ref{lb.M-leibn}]  By \eqref{e.Chain2} in \ref{lb.M-chain}, for any \( f \in \mathcal{F} \cap C_c(M) \) and any \( u, v, w \in \mathcal{F} \cap L^\infty \):
\begin{equation}\label{e.Lei-0}
 \int_M f  \dif \Gamma^{(L)}\la u; v^2\ra  = 2 \int_M f \wt{v}  \dif \Gamma^{(L)}\la u; v\ra\text{ and }\int_M f  \dif \Gamma^{(L)}\la u; w^2\ra  = 2 \int_M f \wt{w}  \dif \Gamma^{(L)}\la u; w\ra,
\end{equation}
then consider
\[
\int_M f  \dif \Gamma^{(L)}\la u; (v+w)^2\ra  - \int_M f  \dif \Gamma^{(L)}\la u; (v-w)^2\ra, 
\]
and use \eqref{e.Lei-0}, we will see \eqref{e.Leb1}. 

\item[\ref{lb.M-densi}] The proof follows the same argument in \cite[Theorem 4.17]{KS25} based on Chain rule \ref{lb.M-chain} for functions in $\sF\cap L^{\infty}$.
\item[\ref{lb.M-local}] Though \ref{lb.M-local} for function in $\sF\cap L^{\infty}$ could be obtained by extending the locality for function in $\sF\cap C_{c}(M)$ as in the proof of \ref{lb.M-chain}, we prove the locality using \ref{lb.M-clark} and \ref{lb.M-densi}. Firstly, by the inner regularity, we may assume that $\Gamma^{(L)}\la u\ra (A)\vee \Gamma^{(L)}\la v\ra (A)<\infty$. By the $p$-Clarkson inequality for $(\Gamma^{(L)}(\wcdot)(A),\sF)$ given in \ref{lb.M-clark} and use \cite[(3.10),(3.11)]{KS25} in the second inequality below, we have \begin{align}
	&\phantom{\ \leq}\abs{\Gamma^{(L)}\la u\ra (A)-\Gamma^{(L)}\la v\ra (A)}=\abs{\Gamma^{(L)}\la u;u\ra (A)-\Gamma^{(L)}\la v;v\ra (A)}\\
	&\leq \abs{\Gamma^{(L)}\la u;u-v\ra (A)}+\abs{\Gamma^{(L)}\la u;v\ra (A)-\Gamma^{(L)}\la v;v\ra (A)}\\
	&\leq \Gamma^{(L)}\la u\ra (A)^{(p-1)/p}\Gamma^{(L)}\la u-v\ra (A)^{1/p}\\
	&\qquad+C_{p}\abs{\Gamma^{(L)}\la u\ra (A)\vee \Gamma^{(L)}\la v\ra (A)}\Gamma^{(L)}\la u-v\ra (A)^{\alpha_{p}/p}\Gamma^{(L)}\la v\ra (A)^{1/p}. \label{e.local1}
\end{align}
Since $\restr{(\wt{u}-\wt{v})}{A}$ is $\sE$-q.e. a constant, we have by \ref{lb.M-densi} that $\Gamma^{(L)}\la u-v\ra(A)=0$ and thus by \eqref{e.local1} we have $\Gamma^{(L)}\la u\ra (A)=\Gamma^{(L)}\la v\ra (A)$.
\end{enumerate}
 \end{proof}
 
For a mixed local and nonlocal $p$-energy form $(\sE,\sF)$ on $(M,d,m)$, let
\begin{equation}
\mathcal{F}^{\prime }:=\Sett{u+a}{u\in \mathcal{F},\ a\in \mathbb{R}}
\end{equation}
be the vector space expands $\sF$ and $\bR\one_{M}$. We extend $\sE$ to $\mathcal{F}^{\prime }$ by letting\footnote{The extension is well-defined: if the constant function $a\one_{M}\in\sF$, then $\sE^{(L)}(a\one_{M})=0$ by \ref{lb.M-densi} and $\sE^{(J)}(a\one_{M})=0$ by definition. So \eqref{e.entCST} holds by the semi-norm property.} \begin{equation}\label{e.entCST}
	\sE(u+a):=\sE(u),\ \forall u\in \mathcal{F},\  \forall a\in \mathbb{R} 
\end{equation}
and \begin{equation}
	\sE(u+a;v+b):=\sE(u;v),\ \forall u,v\in \mathcal{F},\  \forall a,b\in \mathbb{R}. 
\end{equation}
For any $v\in\sF^{\prime}$, let $v=u+a$ for some $u\in \sF$ and $a\in \mathbb{R}$. It is clear that $\wt{v}:=\wt{u}+a$ is an $\sE$-quasi-continuous version of $v$. For any open set $\Omega\subset M$, denote the \emph{energy measure} of $v$ by \[\Gamma_{\Omega}\la v\ra :=\Gamma^{(L)}\la v\ra+\Gamma_{\Omega}^{(J)}\la v\ra ,\] where $\Gamma^{(L)}\la{v}\ra:=\Gamma^{(L)}\la{u}\ra$ is the energy measure of $u$ for the strongly local part given in Theorem \ref{t.sasEM}, and $\Gamma_{\Omega}^{(J)}\la{v}\ra:=\Gamma_{\Omega}^{(J)}\la{u}\ra$ is the measure defined by
\begin{equation}
	\int_{M}f\dif\Gamma_{\Omega}^{(J)}\la v\ra:=\int_{x\in M}\int_{y\in\Omega}f(x)\abs{\wt{u}(x)-\wt{u}(y)}^{p}\dif j(x,y), \ \forall f\in C_{c}(M).
\end{equation}
We write $\Gamma\la v\ra :=\Gamma_{M}\la v\ra$ and $\Gamma^{(J)}\la v\ra:=\Gamma_{M}^{(J)}\la v\ra$.
\subsection{List of conditions}\label{ss.main}
Throughout this paper, we fix a mixed local and nonlocal $p$-energy form $(\sE,\sF)$ defined in Definition \ref{d.MixedLN}. We fix a number $\ol{R}\in(0,\diam(M,d)]$, which could be infinity if $\diam(M,d)=\infty$.

\begin{definition}\label{d.vd+rvd}
Let $(M,d)$ be a metric space and let $m$ be a Borel measure on $M$.
\begin{enumerate}[label=\textup{({\arabic*})},align=left,leftmargin=*,topsep=5pt,parsep=0pt,itemsep=2pt]
	\item We say that $m$ satisfies the \emph{volume doubling property} \eqref{VD}, if there exists $C_{\mathrm{VD}} \in (1,\infty)$ such that,
	\begin{equation} \label{VD} \tag{$\operatorname{VD}$}
	0 < m (B(x,2r)) \le C_{\mathrm{VD}} m (B(x,r)) <\infty, \ \text{for all $x \in M$ and all $r\in(0,\infty)$}.
	\end{equation}
	\item We say that $m$ satisfies the \emph{reverse volume doubling property} \eqref{RVD}, if there exist $C_{\mathrm{RVD}}  \in (0,1]$ and $d_{1} \in (0,\infty)$ such that,
	\begin{equation} \label{RVD} \tag{$\operatorname{RVD}$}
	m (B(x,R)) \ge C_{\mathrm{RVD}} \left(\frac{R}{r}\right)^{d_{1}} m (B(x,r)),\ \text{ for all $x\in M$ and all $0<r \le R  <\overline{R}$}.
	\end{equation}
\end{enumerate}
\end{definition}
\begin{remark}
	If condition \eqref{VD} holds, then there exists a
positive number $d_{2}$ such that
\begin{equation}
\frac{m (B(x,R))}{m (B(y,r))}\leq C_{\mathrm{VD}}\left( \frac{d(x,y)+R}{r}\right) ^{d_{2}},\ \text{for all $x,y\in M$ and all $0<r\leq R<\infty $}. \label{eq:vol_3}
\end{equation}
\end{remark}

A function $W:M\times \lbrack 0,\infty )\rightarrow \lbrack 0,\infty)$ is a \emph{scale function}, if
\begin{itemize}
\item for any $x\in M$, $W(x,\wcdot )$ is strictly increasing with $W(x,0)=0$ and $W(x,\infty)=\infty$;
\item there exist constants $C\in[1,\infty)$
and $0<\beta _{1}\leq \beta _{2}<\infty$ such that for all $0<r\leq R<\infty $ and
all $x,y\in M$ with $d(x,y)\leq R$,
\begin{equation}
C^{-1}\left( \frac{R}{r}\right) ^{\beta _{1}}\leq \frac{W(x,R)}{W(y,r)}\leq
C\left( \frac{R}{r}\right) ^{\beta _{2}}.  \label{eq:vol_0}
\end{equation}
\end{itemize}
For simplicity of notation, we write,
\begin{equation}
  W(B):=W(x,R),\ \text{for any metric ball $B=B(x,R)$.}
\end{equation}

Let $(\sE,\sF)$ be a $p$-energy form, and let $U\subset V$ be two open subsets of $M$. Define the set of \emph{cutoff functions} by \begin{equation}
	\cutoff(U,V):=\Sett{\phi\in\sF}{0\leq \phi\leq1,\ \restr{\phi}{U}=1\text{ and }\restr{\phi}{V^{c}}=0 }.
\end{equation} 
By Proposition \ref{p.cut}, if $(\sE,\sF)$ is Markovian and regular, then $\cutoff(U,V)\neq\emptyset$ whenever $U\Subset V$.

\begin{definition}
	We say that the \emph{cutoff Sobolev inequality} \eqref{CS} holds, if there exist $\eta\in(0,\infty)$ and $C\in (0,\infty)$ such that for any $(x_0,R,r)\in M\times (0,\infty)\times(0,\infty)$ such that $R+2r\in(0,\ol{R})$, and any three concentric balls\begin{equation}\label{e.CSball}
		B_0:=B(x_0,R),\ B_1:=B(x_0,R+r),\ \text{and }\Omega:=B(x_0,R+2r),
	\end{equation} 
	there exists a cutoff function $\phi\in\cutoff(B_0,B_1)$, such that 
\begin{equation}\label{CS}\tag{$\operatorname{CS}$}
\int_{\Omega}\abs{\wt{u}}^{p}\dif\Gamma_{\Omega}\la \phi\ra  \leq \eta\int_{\Omega}\dif\Gamma_{\Omega}\la u\ra+\sup_{x\in \Omega}\frac{C}{W(x,r)}\int_{\Omega}\abs{u}^{p}\dif m,\ \forall u\in\sF^{\prime}\cap L^{\infty}(M,m).
\end{equation}
\end{definition}

\begin{definition}
	We say that the \emph{Poincar\'e inequality} \eqref{PI} holds if there exist $C_{\mathrm{PI}}\in(0,\infty)$, $\kappa\in[1,\infty)$ such that for any ball $B$ with radius $r \in\left(0, \kappa^{-1}\ol{R}\right)$ and any $u \in \sF^{\prime}\cap L^{\infty}$, we have
\begin{equation}\label{PI}\tag{$\operatorname{PI}$}
	\int_B\left|u-u_B\right|^p \dif m \leq C_{\mathrm{PI}} W(B)\int_{\kappa B} \dif \Gamma_{\kappa B}\la u\ra, 
\end{equation}
where $u_{B}:=\fint_{B}u\dif m$ is the average of $u$ on $B$.
\end{definition}

\begin{remark}\label{r.PI=>RVD}
If $(M,d,m)$ satisfies \eqref{VD}, and if $(\sE,\sF)$ is strongly local and satisfies \eqref{PI}, then one obtains \eqref{RVD}. Indeed, this follows by a direct adaptation of \cite[Lemma 2.2 and Corollary 2.3]{Mur20} (see also \cite[Proposition 2.7]{Shi25}), combined with \cite[Exercise 13.1]{Hei01}. Alternatively, one may apply \cite[Proposition A.2]{Shi25} together with \cite[Exercise 13.1]{Hei01} to reach the same conclusion.
\end{remark}

\begin{definition}
Let $j$ be the jump measure on $(M_{\od}^{2},\sB(M_{\od}^{2}))$ in \ref{lb.EFJP}. 
\begin{enumerate}[label=\textup{({\arabic*})},align=left,leftmargin=*,topsep=5pt,parsep=0pt,itemsep=2pt]
	\item We say that \eqref{TJ} holds if there exists $C_{\mathrm{TJ}}\in(0,\infty)$ such that for all $x\in M$ and all $r\in(0,\infty)$,
\begin{equation}\label{TJ}\tag{$\operatorname{TJ}$}
	\dif j(x,y)=J(x,\dif y)\dif m(x)\text{ and } J(x,B(x,r))\leq\frac{C_{\mathrm{TJ}}}{W(x,r)}.
\end{equation}
\item  We say that the \emph{upper
jumping smoothness} \eqref{UJS} holds, if the jump kernel $J(x,y)$
exists on $M_{\od}^{2}$, that is \text{$j\ll m \otimes m$, $\dif j(x,y)= J(x,y)\dif m(x)\dif m(y)$,}  and $J$ satisfies that, for any $R\in(0,\sigma \overline{R})$,
\begin{equation}
J(x,y)\leq C\fint_{B(x,R)}J(z,y)\dif m(z)\ \text{for $m \otimes m $-a.e. $(x,y)\in M_{\od}^{2}$ with $d(x,y)\geq 2R$}  \label{UJS}\tag{$\operatorname{UJS}$}
\end{equation}%
with some two positive constants $C$ and $\sigma \in (0,1)$ independent of $%
x$, $y$ and $R$.
\end{enumerate}
\end{definition}
 
\begin{definition}
 Let $U$ be an open set in $M$. We say that $u\in\sF^{\prime}$ is \emph{superharmonic} in $U$ (resp. \emph{subharmonic} in $U$) if and only if \[\text{ $\sE(u;v)\ge0$ (resp. $\sE(u;v)\le0$) for any non-negative $v\in\sF(U)$.}\] We say that $u\in\sF^{\prime}$ is \emph{harmonic} in $U$ if it is both superharmonic and subharmonic in $U$.
\end{definition}

For any two open subsets $U \Subset V$ of $M$ and any Borel measurable function $v$, if $\dif j(x,y)=J(x,\dif y)\dif m(x)$, then we denote by \begin{equation}
	\label{e.Tail}
	T_{U,V}(v):=\esup_{x\in U}\left(\int_{V^{c}}\abs{\wt{v}(y)}^{p-1}J(x,\dif y)\right)^{\frac{1}{p-1}}.
\end{equation}
Again, by Proposition \ref{p.j-smooth}, the integral in \eqref{e.Tail} is well defined.

 We introduce two different Harnack inequalities: the \emph{weak elliptic Harnack inequality} \eqref{wEH} and the \emph{strong elliptic Harnack inequality} \eqref{sEH} for locally non-negative functions.
 \begin{definition}
 	Let $(\sE,\sF)$ be a mixed local and nonlocal $p$-energy form on $(M,m)$. 
 	\begin{enumerate}[label=\textup{({\arabic*})},align=left,leftmargin=*,topsep=5pt,parsep=0pt,itemsep=2pt]
	\item We say that the \emph{weak elliptic Harnack inequality} \eqref{wEH} holds if there exist four universal constants $q$, $\delta$, $\sigma$ in $(0, 1)$ and $C\in[1,\infty)$ such that, for any $x\in M$ and any two concentric balls $B_{r}:= B(x,r) \subset B_R := B(x,R)$ with $r\in (0, \delta {R}]$, $R\in(0, \sigma \ol{R})$, and for any $u \in \sF^{\prime}$ that is non-negative, superharmonic in $B_R$,
\begin{equation}\label{wEH}\tag{$\operatorname{wEH}$}
	\left(\fint_{B_r}u^q\dif m\right)^{1/q}\leq C\left(\einf_{B_r}u+W(B_{r})^{\frac{1}{p-1}}T_{\frac{1}{2}B_{R},B_{R}}(u_{-})
\right).
\end{equation}
\item We say that the \emph{strong elliptic Harnack inequality} \eqref{sEH} holds if there exist three universal constants $\delta$, $\sigma$ in $(0, 1)$ and $C\in[1,\infty)$ such that, for any $x\in M$ and any two concentric balls $B_{r}:= B(x,r) \subset B_R := B(x,R)$ with $r\in (0, \delta {R}]$, $R\in(0, \sigma \ol{R})$, and for $u \in \sF^{\prime}$ that is non-negative, harmonic in $B_R$,
\begin{equation}\label{sEH}\tag{$\operatorname{sEH}$}
	\esup_{B_r}u\leq C\left(\einf_{B_r}u+W(B_{r})^{\frac{1}{p-1}}T_{\frac{1}{2}B_{R},B_{R}}(u_{-})\right).
\end{equation}
\end{enumerate}
 \end{definition}
 \begin{remark}\label{r.str}
 	If $(\sE,\sF)$ itself is strongly local (namely $j\equiv0$), or the harmonic function $u$ considered in \eqref{sEH} is \emph{globally non-negative} in $M$, then since $T_{\frac{1}{2}B_{R},B_{R}}(u_{-})=0$, \eqref{sEH} becomes the \emph{standard elliptic Harnack inequality}: $\esup_{B_r}u\leq C\einf_{B_r}u$.
 \end{remark}
 
For convenience of the reader, we summarize the relationships of the various functional inequalities appeared in this paper in the following diagram. 

\begin{figure}
\centering
\begin{tikzpicture}[x=0.9pt,y=0.9pt,yscale=-1,xscale=1]

\draw    (101.48,119.98) -- (103.29,187.96)(98.48,120.06) -- (100.29,188.04) ;
\draw [shift={(102,196)}, rotate = 268.48] [color={rgb, 255:red, 0; green, 0; blue, 0 }  ][line width=0.75]    (10.93,-3.29) .. controls (6.95,-1.4) and (3.31,-0.3) .. (0,0) .. controls (3.31,0.3) and (6.95,1.4) .. (10.93,3.29)   ;
\draw    (119.48,103.52) -- (159.48,103.51)(119.48,106.52) -- (159.48,106.51) ;
\draw [shift={(167.48,105.01)}, rotate = 179.99] [color={rgb, 255:red, 0; green, 0; blue, 0 }  ][line width=0.75]    (10.93,-3.29) .. controls (6.95,-1.4) and (3.31,-0.3) .. (0,0) .. controls (3.31,0.3) and (6.95,1.4) .. (10.93,3.29)   ;
\draw    (210.29,109.53) -- (287.45,124.96)(209.71,112.47) -- (286.86,127.9) ;
\draw [shift={(295,128)}, rotate = 191.31] [color={rgb, 255:red, 0; green, 0; blue, 0 }  ][line width=0.75]    (10.93,-3.29) .. controls (6.95,-1.4) and (3.31,-0.3) .. (0,0) .. controls (3.31,0.3) and (6.95,1.4) .. (10.93,3.29)   ;
\draw    (111.92,115.85) -- (162.71,156.81)(110.04,118.19) -- (160.83,159.15) ;
\draw [shift={(168,163)}, rotate = 218.88] [color={rgb, 255:red, 0; green, 0; blue, 0 }  ][line width=0.75]    (10.93,-3.29) .. controls (6.95,-1.4) and (3.31,-0.3) .. (0,0) .. controls (3.31,0.3) and (6.95,1.4) .. (10.93,3.29)   ;
\draw    (194,173) -- (221.06,179.53) ;
\draw [shift={(223,180)}, rotate = 193.57] [fill={rgb, 255:red, 0; green, 0; blue, 0 }  ][line width=0.08]  [draw opacity=0] (12,-3) -- (0,0) -- (12,3) -- cycle    ;
\draw    (295,149) -- (224.84,179.21) ;
\draw [shift={(223,180)}, rotate = 336.71] [fill={rgb, 255:red, 0; green, 0; blue, 0 }  ][line width=0.08]  [draw opacity=0] (12,-3) -- (0,0) -- (12,3) -- cycle    ;
\draw    (210,221) -- (222.4,181.91) ;
\draw [shift={(223,180)}, rotate = 107.59] [fill={rgb, 255:red, 0; green, 0; blue, 0 }  ][line width=0.08]  [draw opacity=0] (12,-3) -- (0,0) -- (12,3) -- cycle    ;
\draw    (223.31,178.53) -- (287.47,191.9)(222.69,181.47) -- (286.86,194.84) ;
\draw [shift={(295,195)}, rotate = 191.77] [color={rgb, 255:red, 0; green, 0; blue, 0 }  ][line width=0.75]    (10.93,-3.29) .. controls (6.95,-1.4) and (3.31,-0.3) .. (0,0) .. controls (3.31,0.3) and (6.95,1.4) .. (10.93,3.29)   ;
\draw    (148,226.5) -- (190,226.5)(148,229.5) -- (190,229.5) ;
\draw [shift={(198,228)}, rotate = 180] [color={rgb, 255:red, 0; green, 0; blue, 0 }  ][line width=0.75]    (10.93,-3.29) .. controls (6.95,-1.4) and (3.31,-0.3) .. (0,0) .. controls (3.31,0.3) and (6.95,1.4) .. (10.93,3.29)   ;
\draw    (232.02,227.64) -- (273.03,228.36)(231.97,230.64) -- (272.97,231.36) ;
\draw [shift={(281,230)}, rotate = 181.01] [color={rgb, 255:red, 0; green, 0; blue, 0 }  ][line width=0.75]    (10.93,-3.29) .. controls (6.95,-1.4) and (3.31,-0.3) .. (0,0) .. controls (3.31,0.3) and (6.95,1.4) .. (10.93,3.29)   ;
\draw [shift={(224,229)}, rotate = 1.01] [color={rgb, 255:red, 0; green, 0; blue, 0 }  ][line width=0.75]    (10.93,-3.29) .. controls (6.95,-1.4) and (3.31,-0.3) .. (0,0) .. controls (3.31,0.3) and (6.95,1.4) .. (10.93,3.29)   ;
\draw    (102,214) -- (91.75,239.15) ;
\draw [shift={(91,241)}, rotate = 292.17] [fill={rgb, 255:red, 0; green, 0; blue, 0 }  ][line width=0.08]  [draw opacity=0] (12,-3) -- (0,0) -- (12,3) -- cycle    ;
\draw    (122,228.02) -- (92.84,240.23) ;
\draw [shift={(91,241)}, rotate = 337.28] [fill={rgb, 255:red, 0; green, 0; blue, 0 }  ][line width=0.08]  [draw opacity=0] (12,-3) -- (0,0) -- (12,3) -- cycle    ;
\draw    (92.5,240.91) -- (96.01,297.92)(89.5,241.09) -- (93.01,298.11) ;
\draw [shift={(95,306)}, rotate = 266.48] [color={rgb, 255:red, 0; green, 0; blue, 0 }  ][line width=0.75]    (10.93,-3.29) .. controls (6.95,-1.4) and (3.31,-0.3) .. (0,0) .. controls (3.31,0.3) and (6.95,1.4) .. (10.93,3.29)   ;
\draw    (287,240) -- (280.39,274.03) ;
\draw [shift={(280.01,275.99)}, rotate = 280.99] [fill={rgb, 255:red, 0; green, 0; blue, 0 }  ][line width=0.08]  [draw opacity=0] (12,-3) -- (0,0) -- (12,3) -- cycle    ;
\draw    (280.55,277.39) -- (216.01,302.5)(279.46,274.59) -- (214.92,299.7) ;
\draw [shift={(208.01,304)}, rotate = 338.74] [color={rgb, 255:red, 0; green, 0; blue, 0 }  ][line width=0.75]    (10.93,-3.29) .. controls (6.95,-1.4) and (3.31,-0.3) .. (0,0) .. controls (3.31,0.3) and (6.95,1.4) .. (10.93,3.29)   ;
\draw    (100,318) -- (149.13,335.35) ;
\draw [shift={(151.02,336.01)}, rotate = 199.45] [fill={rgb, 255:red, 0; green, 0; blue, 0 }  ][line width=0.08]  [draw opacity=0] (12,-3) -- (0,0) -- (12,3) -- cycle    ;
\draw    (173,320) -- (152.63,334.83) ;
\draw [shift={(151.02,336.01)}, rotate = 323.93] [fill={rgb, 255:red, 0; green, 0; blue, 0 }  ][line width=0.08]  [draw opacity=0] (12,-3) -- (0,0) -- (12,3) -- cycle    ;
\draw    (151.56,334.61) -- (198.11,352.71)(150.47,337.41) -- (197.02,355.5) ;
\draw [shift={(205.02,357)}, rotate = 201.24] [color={rgb, 255:red, 0; green, 0; blue, 0 }  ][line width=0.75]    (10.93,-3.29) .. controls (6.95,-1.4) and (3.31,-0.3) .. (0,0) .. controls (3.31,0.3) and (6.95,1.4) .. (10.93,3.29)   ;
\draw    (314,209) .. controls (330.83,220.88) and (324.13,255.29) .. (282.28,321) ;
\draw [shift={(281,323)}, rotate = 302.68] [fill={rgb, 255:red, 0; green, 0; blue, 0 }  ][line width=0.08]  [draw opacity=0] (8.93,-4.29) -- (0,0) -- (8.93,4.29) -- cycle    ;
\draw    (245,359) -- (279.59,324.41) ;
\draw [shift={(281,323)}, rotate = 135] [fill={rgb, 255:red, 0; green, 0; blue, 0 }  ][line width=0.08]  [draw opacity=0] (12,-3) -- (0,0) -- (12,3) -- cycle    ;
\draw    (87,202) -- (49.34,160.48) ;
\draw [shift={(48,159)}, rotate = 47.79] [fill={rgb, 255:red, 0; green, 0; blue, 0 }  ][line width=0.08]  [draw opacity=0] (12,-3) -- (0,0) -- (12,3) -- cycle    ;
\draw    (221,370.5) .. controls (105.58,460.55) and (-18.75,239.23) .. (46.99,160.18) ;
\draw [shift={(48,159)}, rotate = 131.08] [fill={rgb, 255:red, 0; green, 0; blue, 0 }  ][line width=0.08]  [draw opacity=0] (12,-3) -- (0,0) -- (12,3) -- cycle    ;
\draw    (203,222) .. controls (194.09,190.32) and (101.87,169.42) .. (49.57,159.3) ;
\draw [shift={(48,159)}, rotate = 10.89] [fill={rgb, 255:red, 0; green, 0; blue, 0 }  ][line width=0.08]  [draw opacity=0] (12,-3) -- (0,0) -- (12,3) -- cycle    ;
\draw    (46.77,158.13) -- (78.16,113.67)(49.23,159.87) -- (80.61,115.4) ;
\draw [shift={(84,108)}, rotate = 125.22] [color={rgb, 255:red, 0; green, 0; blue, 0 }  ][line width=0.75]    (10.93,-3.29) .. controls (6.95,-1.4) and (3.31,-0.3) .. (0,0) .. controls (3.31,0.3) and (6.95,1.4) .. (10.93,3.29)   ;
\draw    (348,153) .. controls (350.97,190.62) and (331.4,251.76) .. (281.53,275.29) ;
\draw [shift={(280.01,275.99)}, rotate = 335.73] [fill={rgb, 255:red, 0; green, 0; blue, 0 }  ][line width=0.08]  [draw opacity=0] (12,-3) -- (0,0) -- (12,3) -- cycle    ;
\draw    (282.3,322.26) -- (303.35,359.3)(279.7,323.74) -- (300.74,360.79) ;
\draw [shift={(306,367)}, rotate = 240.4] [color={rgb, 255:red, 0; green, 0; blue, 0 }  ][line width=0.75]    (10.93,-3.29) .. controls (6.95,-1.4) and (3.31,-0.3) .. (0,0) .. controls (3.31,0.3) and (6.95,1.4) .. (10.93,3.29)   ;
\draw    (231.28,367.53) -- (280.43,377.01)(230.72,370.47) -- (279.86,379.96) ;
\draw [shift={(288,380)}, rotate = 190.92] [color={rgb, 255:red, 0; green, 0; blue, 0 }  ][line width=0.75]    (10.93,-3.29) .. controls (6.95,-1.4) and (3.31,-0.3) .. (0,0) .. controls (3.31,0.3) and (6.95,1.4) .. (10.93,3.29)   ;

\draw (87.98,95.02) node [anchor=north west][inner sep=0.75pt]  [rotate=-359.99] [align=left] {\ref{CS}};
\draw (174.98,95.01) node [anchor=north west][inner sep=0.75pt]  [rotate=-359.99] [align=left] {\ref{iCS}};
\draw (297,131) node [anchor=north west][inner sep=0.75pt]  [rotate=-359.99] [align=left] {Caccioppoli ineq.};
\draw (171.49,158.01) node [anchor=north west][inner sep=0.75pt]  [rotate=-359.99] [align=left] {\ref{TE}};
\draw (296.99,183.99) node [anchor=north west][inner sep=0.75pt]  [rotate=-359.99] [align=left] {\ref{MV}};
\draw (87,196) node [anchor=north west][inner sep=0.75pt]  [rotate=-359.99] [align=left] {\ref{cap<}};
\draw (127,215.02) node [anchor=north west][inner sep=0.75pt]  [rotate=-359.99] [align=left] {\ref{PI}};
\draw (200,221) node [anchor=north west][inner sep=0.75pt]  [rotate=-359.99] [align=left] {\ref{FK}};
\draw (282,220.99) node [anchor=north west][inner sep=0.75pt]  [rotate=-359.99] [align=left] {\ref{SI}};
\draw (210.02,351) node [anchor=north west][inner sep=0.75pt]  [rotate=-359.99] [align=left] {\ref{wEH}};
\draw (294.02,367.98) node [anchor=north west][inner sep=0.75pt]  [rotate=-359.99] [align=left] {\ref{sEH}};
\draw (65.01,305.03) node [anchor=north west][inner sep=0.75pt]  [rotate=-359.99] [align=left] {{\scriptsize Crossover lem.}};
\draw (161.01,309.01) node [anchor=north west][inner sep=0.75pt]  [rotate=-359.99] [align=left] {{\scriptsize Lem.\! of \!Growth}};
\draw (38.23,143.08) node [anchor=north west][inner sep=0.75pt]  [rotate=-304.62] [align=left] {{\scriptsize Thm. \ref{t.weh=>cs}}};
\draw (120,86) node [anchor=north west][inner sep=0.75pt]   [align=left] {{\scriptsize Prop. \ref{p.CS->iCS}}};
\draw (231.58,97.25) node [anchor=north west][inner sep=0.75pt]  [rotate=-11.6] [align=left] {{\scriptsize Prop. \ref{prop:cac}}};
\draw (134.06,112.6) node [anchor=north west][inner sep=0.75pt]  [rotate=-40.31] [align=left] {{\scriptsize Lem. \ref{L1}}};
\draw (250.9,166.53) node [anchor=north west][inner sep=0.75pt]  [rotate=-13.63] [align=left] {{\scriptsize Lem. \ref{lemma:MV2}}};
\draw (148,210) node [anchor=north west][inner sep=0.75pt]   [align=left] {{\scriptsize Prop. \ref{p.PI=>FK}}};
\draw (235,212) node [anchor=north west][inner sep=0.75pt]   [align=left] {{\scriptsize Thm. \ref{t.FK=Sob}}};
\draw (220.29,279.44) node [anchor=north west][inner sep=0.75pt]  [rotate=-339.76] [align=left] {{\scriptsize Lem. \ref{l.LoG}}};
\draw (107.73,249.19) node [anchor=north west][inner sep=0.75pt]  [rotate=-87.82] [align=left] {{\scriptsize Lem. \ref{l.cros}}};
\draw (298.11,319.85) node [anchor=north west][inner sep=0.75pt]  [rotate=-61.98] [align=left] {{\scriptsize Lem. \ref{L32}}};
\draw (115.99,125.17) node [anchor=north west][inner sep=0.75pt]  [rotate=-90.52] [align=left] {{\scriptsize Lem. \ref{l.CS=>cap}}};
\draw (174.68,325.12) node [anchor=north west][inner sep=0.75pt]  [rotate=-25.26] [align=left] {{\scriptsize p.~\pageref{page.wEHpf}}};
\draw (139.38,144.61) node [anchor=north west][inner sep=0.75pt]  [rotate=-35.49] [align=left] {{\scriptsize +\ref{UJS}}};
\draw (157,230) node [anchor=north west][inner sep=0.75pt]   [align=left] {{\scriptsize +\ref{RVD}}};
\draw (226.48,370.38) node [anchor=north west][inner sep=0.75pt]  [rotate=-10.5] [align=left] {{\scriptsize strong locality}};
\draw (246.09,356.77) node [anchor=north west][inner sep=0.75pt]  [rotate=-11.25] [align=left] {{\scriptsize Prop. \ref{P-sl}}};

\end{tikzpicture}
\vspace{-50pt}
\caption{Relationships among functional inequalities under \eqref{VD} and \eqref{TJ}}
\end{figure}

\section{Functional inequalities}\label{s.func}

In this section, we fix a mixed local and nonlocal $p$-energy form $(\sE,\sF)$ on the metric measure space $(M,d,m)$ and prove some consequences of the cutoff Sobolev inequality \eqref{CS} and the Poincar\'e inequality \eqref{PI}. In particular, we prove the equivalence between the Faber--Krahn inequality and Sobolev inequalities in Theorem \ref{t.FK=Sob}.

\subsection{Self-improvement of cutoff Sobolev inequality}\label{ss.self-ipv}
We introduce the following ``improved cutoff Sobolev inequality''.
\begin{definition}
	Let $\eta\in(0,\infty)$. We say that condition \eqref{iCS} holds, if there exists constants $C_{\eta}\in (0,\infty)$ such that for any $(x_0,R,r)\in M\times (0,\infty)\times(0,\infty)$ such that $R+2r<\ol{R}$, and any three concentric balls\begin{equation}
		B_0:=B(x_0,R),\ B_1:=B(x_0,R+r),\ \text{and }\Omega:=B(x_0,R+2r),
	\end{equation} 
	there exists $\phi\in\cutoff(B_0,B_1)$, such that 
\begin{equation}\label{iCS}\tag{$\widetilde{\operatorname{CS}}_{\eta}$}
\int_{\Omega}\abs{\wt{u}}^{p}\dif\Gamma_{\Omega}\la \phi\ra  \leq \eta\int_{\Omega}\phi^{p}\dif\Gamma_{\Omega}\la u\ra +\sup_{x\in \Omega}\frac{C_{\eta}}{W(x,r)}\int_{\Omega}\abs{u}^{p}\dif m,\ \forall u\in\sF^{\prime}\cap L^{\infty}(M,m).
\end{equation}
\end{definition}
The differences between \eqref{iCS} and \eqref{CS} are that, the constant of the first term in the right hand side of \eqref{iCS} is quantitative, and the integrand is the cutoff function $\phi$ instead of constant $\one_{M}$ in \eqref{CS}. Clearly, for any $\eta\in(0,\infty)$, \eqref{iCS} implies \eqref{CS}.
\begin{proposition}\label{p.CS->iCS}
	If \eqref{TJ} and \eqref{CS} hold, then \eqref{iCS} holds for all $\eta\in(0,\infty)$.
\end{proposition}
\begin{proof}
	Fix $(x_0,R,r)\in M\times(0,\infty)\times(0,\infty)$ such that $R+2r<\ol{R}$, and let \begin{equation}
		B_0:=B(x_0,R),\ B_*:=B(x_0,R+r)\ \text{and }\Omega:=B(x_0,R+2r).
	\end{equation}
	Let $q\in(1,9/8)$ to be determined later and define the following non-negative numbers
	\begin{align}
		b_n&:=q^{-\beta n}, \qquad r_n:=(1-q^{-n})r,\ \forall n\in\{0\}\cup\bN; \\
		a_n&:=b_{n-1}-b_n=(q^{\beta}-1)q^{-\beta n},\ \	\ s_n:=r_n-r_{n-1}=(q-1)q^{-n}r, \ \forall n\in\bN,
	\end{align}
then $\sum_{n=1}^\infty a_n=1\ \text{and}\ r_n=\sum_{k=1}^{n}s_k$. Denote 
	\begin{equation}
		B_n:=B(x_0,R+r_n),\ \text{and}\ \Omega_{n}:=B(x_0,R+r_n+2s_{n+1}).
	\end{equation}
	So for $q\in(1,2)$, since $r_{n+1}<r_n+2s_{n+1}< r$, we have for all $n\in\{0\}\cup\bN$ that \begin{equation}
		B(x_0,R)=B_0\subset B_n\subset B_{n+1}\subset \Omega_{n}\subset B(x_0,R+r)=B_{*}.
	\end{equation}
Fix $u\in\sF^{\prime}\cap L^{\infty}$. We represent $u$ by its $\sE$-quasi-continuous version in the following proof. By \eqref{CS}, there is a $\eta_0$ such that for $n\in \bN$, there exists $\phi_n\in\cutoff(B_n,B_{n+1})$ satisfies \begin{equation}
		\int_{\Omega_n}\abs{u}^{p}\dif\Gamma_{\Omega_n}\la \phi_n\ra \leq \eta_0\int_{\Omega_n}\dif\Gamma_{\Omega_n}\la u\ra +\sup_{x\in \Omega}\frac{C}{W(x,s_{n+1})}\int_{\Omega_n}\abs{u}^{p}\dif m.\label{e.siCS}
	\end{equation}
	Let $\phi:=\sum_{n=1}^\infty a_n\phi_n$. By the same proof of \cite[p.137-138]{GHH24}, $\phi\in \sF$ and thus \begin{equation}
	\phi\in\cutoff(B_1,B_*)\subset\cutoff(B_0,B_*).
	\end{equation}
	We claim that $\phi$ is the desired cutoff function for \eqref{iCS} when $q$ is sufficiently close to $1$. 
	
	First by the definition of energy measure \begin{align}
		\int_{\Omega}\abs{u}^{p}\dif\Gamma_{\Omega}\la \phi\ra 
		=\underbrace{\int_{\Omega}\abs{u}^{p}\dif\Gamma^{(L)}\la \phi\ra }_{\text{local part}}+\underbrace{\int_{\Omega\times\Omega}\abs{u(x)}^{p}\abs{\phi(x)-\phi(y)}^{p}\dif j(x,y)}_{\text{nonlocal part}}.
	\end{align}
	
	For the local part, note that if $x\in B_{n+1}\setminus B_{n}$ then $\phi_m(x)=1$ for all $m\geq n+1$, and $\phi_{k}(x)=0$ for all $k\leq n-1$. Thus, for $x\in B_{n+1}\setminus B_{n}$,
\begin{equation}
		\phi(x)=a_n\phi_n(x)+\sum_{m=n+1}^{\infty}a_m=a_n\phi_n(x)+q^{-\beta n}. \label{CS-12}
	\end{equation}
	By the locality of energy measure in Theorem \ref{t.sasEM}-\ref{lb.M-local}, We have 
 \begin{equation}
		\restr{\Gamma^{(L)}\la\phi\ra }{B_{n+1}\setminus B_{n}}=a_n^{p}\restr{\Gamma^{(L)}\la\phi_n\ra }{B_{n+1}\setminus B_{n}}.
\end{equation}
	Therefore
	\begin{align}
		\int_{\Omega}\abs{u}^{p}\dif\Gamma^{(L)}\la \phi\ra &=\int_{B_*}\abs{u}^{p}\dif\Gamma^{(L)}\la \phi\ra =\sum_{n=1}^{\infty}\int_{B_{n+1}\setminus B_{n}}\abs{u}^{p}\dif\Gamma^{(L)}\la \phi\ra \\
		&=\sum_{n=1}^{\infty}\int_{B_{n+1}\setminus B_{n}}a_{n}^{p}\abs{u}^{p}\dif\Gamma^{(L)}\la \phi_n\ra =\sum_{n=1}^{\infty}a_{n}^{p}\int_{\Omega_n}\abs{u}^{p}\dif\Gamma^{(L)}\la \phi_n\ra. \label{e.silocal}
	\end{align}

For the nonlocal part, we define $p_{0}:=\lfloor p\rfloor$ and $\epsilon:=p-\lfloor p\rfloor$,
so that $p=p_0+\epsilon$ and $p_0\geq1>\epsilon\geq0$. First we have
\begin{align}
	\abs{\phi(x)-\phi(y)}^{p}
	&\leq\bigg(\sum_{n=1}^\infty a_{n}\abs{\phi_n(x)-\phi_n(y)}\bigg)^{p}=\lim_{N\to\infty}\bigg(\sum_{n=1}^{N}a_{n}\abs{\phi_n(x)-\phi_n(y)}\bigg)^{p}. \label{cs-16}
\end{align}

We now assume that $p_0\geq2$ and $\epsilon>0$. If $p_0=1$ or if $\epsilon=0$, then the estimates are much easier.
 Fix $N\in\bN$. Define 
 \begin{align} 
\Lambda_k&:=\Sett{(t_1,\ldots,t_k)\in\bN^{k}}{t_1+\cdots+t_k=p_0}\ \text{for} \ 2\leq k\leq p_0, \\
A(x,y)&:=\bigg(\sum_{n=1}^{N}a_{n}^{p_0}\abs{\phi_n(x)-\phi_n(y)}^{p_0}\bigg)
\bigg(\sum_{n=1}^{N}a_{n}\abs{\phi_n(x)-\phi_n(y)}\bigg)^{\epsilon},\\
	B(x,y)&:=\bigg(\sum_{k=2}^{p_0}\sum_{(t_1,\ldots,t_k)\in\Lambda_k}\sum_{1\leq i_1<\cdots<i_k\leq N}\prod_{l=1}^{k}a_{i_l}^{t_l}
\abs{\phi_{i_l}(x)-\phi_{i_l}(y)}^{t_l}\bigg)\bigg(\sum_{n=1}^{N}a_{n}\abs{\phi_n(x)-\phi_n(y)}\bigg)^{\epsilon}.
\end{align}
Then
\begin{align}
\bigg(\sum_{n=1}^{N}a_{n}\abs{\phi_n(x)-\phi_n(y)}\bigg)^{p}&=
\bigg(\sum_{n=1}^{N}a_{n}\abs{\phi_n(x)-\phi_n(y)}\bigg)^{p_0}\bigg(\sum_{n=1}^{N}a_{n}\abs{\phi_n(x)-\phi_n(y)}\bigg)^{\epsilon}\\
&=A(x,y)+B(x,y).\label{e.sidcmp}
\end{align}

\begin{enumerate}[label=\textup{({\arabic*})},align=left,leftmargin=*,topsep=5pt,parsep=0pt,itemsep=2pt]
	\item \textbf{Estimation of $A(x,y)$}
	
	Note that \begin{align}
		A(x,y)&\leq\bigg(\sum_{n=1}^{N}a_{n}^{p_0}\abs{\phi_n(x)-\phi_n(y)}^{p_0}\bigg)
\bigg(\sum_{n=1}^{N}a_{n}^{\epsilon}\abs{\phi_n(x)-\phi_n(y)}^{\epsilon}\bigg)\\
		&= \sum_{n=1}^{N}a_{n}^{p}\abs{\phi_n(x)-\phi_n(y)}^{p}+\sum_{\substack{1\leq i\neq k\leq N \\ {\abs{i-k}=1} }}a_{i}^{p_0}a_{k}^{\epsilon}\abs{\phi_i(x)-\phi_i(y)}^{p_0}\abs{\phi_k(x)-\phi_k(y)}^{\epsilon}\\
&\qquad+\left(\sum_{i=1}^{N-2}\sum_{k=i+2}^{N}+\sum_{i=3}^{N}\sum_{k=1}^{i-2}\right)a_{i}^{p_0}a_{k}^{\epsilon}\abs{\phi_i(x)-\phi_i(y)}^{p_0}\abs{\phi_k(x)-\phi_k(y)}^{\epsilon}\\
		&=:A_1(x,y)+A_2(x,y)+A_3(x,y).    \label{e.Adcmp}
	\end{align}

For the second term, we have by Young's inequality $a^{p_0}b^{\epsilon}\leq \frac{p_0}{p}a^{p}+\frac{\epsilon}{p}b^{p},\ \forall a,b\in[0,\infty)$,
\begin{align}
		A_2(x,y)&\leq \sum_{\substack{1\leq i\neq k\leq N \\ {\abs{i-k}=1} }}\left(\frac{p_0}{p}a_{i}^{p}\abs{\phi_i(x)-\phi_i(y)}^{p}+\frac{\epsilon}{p}a_{k}^{p}\abs{\phi_k(x)-\phi_k(y)}^{p}\right)\\
		&\leq \frac{2p_0}{p}A_1(x,y)+\frac{2\epsilon}{p}A_1(x,y)=2A_1(x,y).\label{e.A2}
	\end{align}
	
	Since $\supp(\phi_{n})\subset B_{n+1}\subset \Omega_{n}$, \begin{align}
		\int_{\Omega\times\Omega}\abs{u(x)}^{p}A_1(x,y)\dif j(x,y)&=\sum_{n=1}^{N}a_{n}^{p}\int_{\Omega\times\Omega}\abs{u(x)}^{p}\abs{\phi_n(x)-\phi_n(y)}^{p}\dif j(x,y)\\
		&=\sum_{n=1}^{N}a_{n}^{p}\int_{\Omega_n\times\Omega_n}\abs{u(x)}^{p}\abs{\phi_n(x)-\phi_n(y)}^{p}\dif j(x,y)\\
		&\qquad +\sum_{n=1}^{N}a_{n}^{p}\int_{\Omega_n\times(\Omega\setminus \Omega_n)}\abs{u(x)}^{p}\abs{\phi_n(x)-\phi_n(y)}^{p}\dif j(x,y)\\
		&\qquad +\sum_{n=1}^{N}a_{n}^{p}\int_{(\Omega\setminus \Omega_n)\times\Omega_n}\abs{u(x)}^{p}\abs{\phi_n(x)-\phi_n(y)}^{p}\dif j(x,y)\\
		&=:I_{1}+I_{2}+I_{3}.\label{e.A1dcmp}
	\end{align}
	Since $\restr{\phi_n}{\Omega\setminus B_{n+1}}=0$, we have 
	\begin{align}
		I_2&=\sum_{n=1}^{N}a_{n}^{p}\int_{B_{n+1}\times(\Omega\setminus \Omega_n)}\abs{u(x)}^{p}\abs{\phi_n(x)}^{p}\dif j(x,y)\leq \sum_{n=1}^{N}a_{n}^{p}\int_{B_{n+1}\times(\Omega\setminus \Omega_n)}\abs{u(x)}^{p}\dif j(x,y)\\
		&\overset{\eqref{TJ}}{\leq}C_{\mathrm{TJ}}\bigg(\sum_{n=1}^{N}a_{n}^{p}\sup_{x\in \Omega}\frac{1}{W(x,s_{n+1})}\bigg)\int_{\Omega}\abs{u(x)}^{p}\dif m(x).  \label{e.I2}
	\end{align}
	Similarly, we have
	\begin{equation}
		I_3\leq C_{\mathrm{TJ}}\bigg(\sum_{n=1}^{N}a_{n}^{p}\sup_{x\in \Omega}\frac{1}{W(x,s_{n+1})}\bigg)\int_{\Omega}\abs{u(x)}^{p}\dif m(x).\label{e.I3}
	\end{equation}
For any $1\leq m< n-1\leq N-1$, note that $\phi_m\phi_n=\phi_m$ and $\restr{\phi_m}{\Omega\setminus B_{m+1}}=0$, $\restr{\phi_n}{B_{n}}=1$, 
\begin{align}
 &\phantom{\ \leq}\int_{\Omega\times\Omega}\abs{u(x)}^{p}\abs{(\phi_m(x)-\phi_m(y))(\phi_n(x)-\phi_n(y))}^{\epsilon} \dif j(x,y) \\
& =  \int_{\Omega\times\Omega}\abs{u(x)}^{p}\abs{\phi_m(x)(1-\phi_n(y))+\phi_m(y)(1-\phi_n(x))}^{\epsilon} \dif j(x,y) \\
   & \leq \left(\int_{B_{m+1}\times(\Omega\setminus B_n)}+\int_{(\Omega\setminus B_n)\times B_{m+1}}\right)\abs{u(x)}^{p}\dif j(x,y) \\
   & \leq 2 \esup_{x\in \Omega}J(x,B(x,s_{m+2})^{c})\int_{\Omega}\abs{u}^{p}\dif m \quad\text{(since $j$ is symmetric)}\\
   & \overset{\eqref{TJ}}{\leq}2 C_{\mathrm{TJ}}
   \sup_{x\in \Omega}\frac{1}{W(x,s_{m+2})}\int_{\Omega}\abs{u}^p\dif m.    \label{cs-11} 
\end{align}

Note that
\begin{align}
		A_3(x,y) &=\left(\sum_{i=1}^{N-2}\sum_{k=i+2}^{N}+\sum_{i=3}^{N}\sum_{k=1}^{i-2}\right)
a_{i}^{p_0}a_{k}^{\epsilon}\underbrace{\abs{\phi_i(x)-\phi_i(y)}^{p_0-\epsilon}}_{\leq1\text{ as }p_0-\epsilon>0}\abs{(\phi_i(x)-\phi_i(y))(\phi_k(x)-\phi_k(y))}^{\epsilon}\\
		&\leq \left(\sum_{i=1}^{N-2}\sum_{k=i+2}^{N}+\sum_{i=3}^{N}\sum_{k=1}^{i-2}\right) a_{i}^{p_0}a_{k}^{\epsilon}\abs{(\phi_i(x)-\phi_i(y))(\phi_k(x)-\phi_k(y))}^{\epsilon}, 
	\end{align}
we have by \eqref{cs-11} that
\begin{align}
		&\phantom{\ \leq}\int_{\Omega\times\Omega}\abs{u(x)}^{p}A_{3}(x,y)\dif j(x,y)\\
		&{\leq} 2C_{\mathrm{TJ}}\bigg(\sum_{i=1}^{N-2}\sum_{k=i+2}^{N}\sup_{x\in \Omega}\frac{a_{i}^{p_0}a_{k}^{\epsilon}}{W(x,s_{i+2})}+\sum_{i=3}^{N}\sum_{k=1}^{i-2}\sup_{x\in \Omega}\frac{a_{i}^{p_0}a_{k}^{\epsilon}}{W(x,s_{k+2})}\bigg)\int_{\Omega}\abs{u}^p\dif m.\label{e.A32}
	\end{align}

	Combining \eqref{e.Adcmp}, \eqref{e.A2}, \eqref{e.A1dcmp}, \eqref{e.I2}, \eqref{e.I3}  and \eqref{e.A32}, we obtain
	\begin{align}
		&\phantom{\ \leq}\int_{\Omega\times\Omega}\abs{u(x)}^{p}A(x,y)\dif j(x,y)\\
		&\leq 3I_{1}+6C_{\mathrm{TJ}}\bigg(\sum_{n=1}^{N}a_{n}^{p}\sup_{x\in \Omega}\frac{1}{W(x,s_{n+1})}\bigg)\int_{\Omega}\abs{u}^{p}\dif m\\
		&\qquad +2C_{\mathrm{TJ}}\bigg(\sum_{i=1}^{N-2}\sum_{k=i+2}^{N}(a_{i}^{p_0}a_{k}^{\epsilon}+a_{k}^{p_0}a_{i}^{\epsilon})
\sup_{x\in \Omega}\frac{1}{W(x,s_{i+2})}\bigg)\int_{\Omega}\abs{u}^p\dif m.\label{e.A}
	\end{align} 
\item \textbf{Estimation of $B(x,y)$}

 Decompose the summation in $B(x,y)$ as follows: 
	\begin{align}
	 	&\phantom{\ \leq} B(x,y)\\
	 	&=\bigg(\sum_{k=2}^{p_0}\sum_{(t_1,\ldots,t_k)\in\Lambda_k}\sum_{\substack{1\leq i_1<\cdots<i_k\leq N \\ {\abs{i_h-i_{h+1}}=1}\\ \forall 1\leq h\leq k-1 }}\prod_{l=1}^{k}a_{i_l}^{t_l}\abs{\phi_{i_l}(x)-\phi_{i_l}(y)}^{t_l}\bigg)\bigg(\sum_{n=1}^{N}a_{n}\abs{\phi_n(x)-\phi_n(y)}\bigg)^{\epsilon}\\
	 	&\quad+\bigg(\sum_{k=2}^{p_0}\sum_{(t_1,\ldots,t_k)\in\Lambda_k}\sum_{\substack{1\leq i_1<\cdots<i_k\leq N \\ {\abs{i_h-i_{h+1}}\geq2}\\ \text{for some } 1\leq h\leq k-1 }}\prod_{l=1}^{k}a_{i_l}^{t_l}\abs{\phi_{i_l}(x)-\phi_{i_l}(y)}^{t_l}\bigg)\underbrace{\bigg(\sum_{n=1}^{N}a_{n}\abs{\phi_n(x)-\phi_n(y)}\bigg)^{\epsilon}}_{\leq1}\\
	 	&\leq \bigg(\sum_{k=2}^{p_0}\sum_{(t_1,\ldots,t_k)\in\Lambda_k}\sum_{\substack{1\leq i_1<\cdots<i_k\leq N \\ {\abs{i_h-i_{h+1}}=1}\\ \forall 1\leq h\leq k-1 }}\prod_{l=1}^{k}a_{i_l}^{t_l}\abs{\phi_{i_l}(x)-\phi_{i_l}(y)}^{t_l}\bigg)\bigg(\sum_{n=1}^{N}a_{n}\abs{\phi_n(x)-\phi_n(y)}\bigg)^{\epsilon}\\
	 	&\quad+\sum_{k=2}^{p_0}\sum_{(t_1,\ldots,t_k)\in\Lambda_k}\sum_{\substack{1\leq i_1<\cdots<i_k\leq N \\ {\abs{i_h-i_{h+1}}\geq2}\\ \text{for some } 1\leq h\leq k-1 }}\prod_{l=1}^{k}a_{i_l}^{t_l}\abs{\phi_{i_l}(x)-\phi_{i_l}(y)}^{t_l}\\
	 	&=:B_1(x,y)+B_2(x,y).\label{e.Bdcmp}
	\end{align}
	
	We first estimate $B_{1}(x,y)$. By Young's inequality in the first inequality below, 
	\begin{align}
		&\phantom{\ \leq} \sum_{k=2}^{p_0}\sum_{(t_1,\ldots,t_k)\in\Lambda_k}\sum_{\substack{1\leq i_1<\cdots<i_k\leq N \\ {\abs{i_h-i_{h+1}}=1}\\ \forall 1\leq h\leq k-1 }}\prod_{l=1}^{k}a_{i_l}^{t_l}\abs{\phi_{i_l}(x)-\phi_{i_l}(y)}^{t_l}\\
		&\leq \sum_{k=2}^{p_0}\sum_{(t_1,\ldots,t_k)\in\Lambda_k}\sum_{\substack{1\leq i_1<\cdots<i_k\leq N \\ {\abs{i_h-i_{h+1}}=1}\\ \forall 1\leq h\leq k-1 }}\sum_{l=1}^{k}\frac{t_l}{p_0}a_{i_l}^{p_0}\abs{\phi_{i_l}(x)-\phi_{i_l}(y)}^{p_0}\\
		&\leq \sum_{k=2}^{p_0}\binom{p_0-1}{k-1}\cdot k \cdot \sum_{n=1}^{N}a_{n}^{p_0}\abs{\phi_{n}(x)-\phi_{n}(y)}^{p_0}\leq  2^{p_0-1}p_0\sum_{n=1}^{N}a_{n}^{p_0}\abs{\phi_{n}(x)-\phi_{n}(y)}^{p_0},
	\end{align}
where in the second inequality we have used that fact that $\#\Lambda_k=\binom{p_0-1}{k-1}$. Therefore 
	\begin{align}
		B_1(x,y)&\leq 2^{p_0-1}p_0 \bigg(\sum_{n=1}^{N}a_{n}^{p_0}\abs{\phi_{n}(x)-\phi_{n}(y)}^{p_0}\bigg) \bigg(\sum_{n=1}^{N}a_{n}\abs{\phi_n(x)-\phi_n(y)}\bigg)^{\epsilon} \\
&=2^{p_0-1}p_0 A(x,y).\label{e.B1}
	\end{align}
	
	We then estimate $B_{2}$. For each $2\leq k\leq p_0$ and each $(t_1,\ldots,t_k)\in\Lambda_{k}$, we have 
	\begin{align}
		&\phantom{\ \leq}\int_{\Omega\times\Omega}\abs{u(x)}^{p}\sum_{\substack{1\leq i_1<\cdots<i_k\leq N \\ {\abs{i_h-i_{h+1}}\geq2}\\ \text{for some } 1\leq h\leq k-1 }}\prod_{l=1}^{k}a_{i_l}^{t_l}\abs{\phi_{i_l}(x)-\phi_{i_l}(y)}^{t_l}\dif j(x,y)\\
		&\leq \sum_{h=1}^{k-1}\sum_{\substack{1\leq i_1<\cdots<i_k\leq N \\ {\abs{i_h-i_{h+1}}\geq2}}}\left(\prod_{l=1}^{k}a_{i_l}^{t_l}\right)\int_{\Omega\times\Omega}\abs{u(x)}^{p}\abs{\phi_{i_h}(x)-\phi_{i_h}(y)}\abs{\phi_{i_{h+1}}(x)-\phi_{i_{h+1}}(y)}\dif j(x,y)\\
		&\leq 2C_{\mathrm{TJ}}\sum_{h=1}^{k-1}\sum_{\substack{1\leq i_1<\cdots<i_k\leq N \\ {\abs{i_h-i_{h+1}}\geq2}}}\left(\prod_{l=1}^{k}a_{i_l}^{t_l}\right)\sup_{x\in \Omega}\frac{1}{W(x,s_{2+i_h})}\int_{\Omega}\abs{u}^{p}\dif m
\ \text{(as in \eqref{cs-11})}. \label{e.B2}
	\end{align}
	\end{enumerate}
By \eqref{e.sidcmp}, \eqref{e.A}, \eqref{e.Bdcmp}, \eqref{e.B1} and \eqref{e.B2}, we have 
	\begin{align}
		&\phantom{\ \leq}\int_{\Omega\times\Omega}\abs{u(x)}^{p}\bigg(\sum_{n=1}^{N}a_{n}\abs{\phi_n(x)-\phi_n(y)}\bigg)^{p}\dif j(x,y)\\
		&\leq 3(1+p_{0}2^{p_0-1})I_1+6(1+p_02^{p_0-1})C_{\mathrm{TJ}}\bigg(\sum_{n=1}^{N}a_{n}^{p}\sup_{x\in \Omega}\frac{1}{W(x,s_{n+1})}\bigg)\int_{\Omega}\abs{u}^{p}\dif m\\
		&\qquad +2(1+p_{0}2^{p_0-1})C_{\mathrm{TJ}}\bigg(\sum_{i=1}^{N-2}\sum_{k=i+2}^{N}(a_{i}^{p_0}a_{k}^{\epsilon}+a_{k}^{p_0}a_{i}^{\epsilon})\sup_{x\in \Omega}\frac{1}{W(x,s_{i+2})}\bigg)\int_{\Omega}\abs{u}^p\dif m\\
		&\qquad +2C_{\mathrm{TJ}}\bigg(\sum_{k=2}^{p_0}\sum_{(t_1,\ldots,t_k)\in\Lambda_k}\sum_{h=1}^{k-1}\sum_{\substack{1\leq i_1<\cdots<i_k\leq N \\ {\abs{i_h-i_{h+1}}\geq2}}}\left(\prod_{l=1}^{k}a_{i_l}^{t_l}\right)\sup_{x\in \Omega}\frac{1}{W(x,s_{2+i_h})}\bigg)\int_{\Omega}\abs{u}^{p}\dif m.
	\end{align}
Letting $N\rightarrow\infty$ in above inequality and combining it with \eqref{cs-16} and \eqref{e.silocal}, we obtain
	\begin{align}	&\int_{\Omega}\abs{u}^{p}\dif\Gamma_{\Omega}\la \phi \ra =\int_{\Omega}\abs{u}^{p}\dif\Gamma^{(L)}\la \phi \ra +\int_{\Omega\times\Omega}\abs{u(x)}^{p}\abs{\phi(x)-\phi(y)}^{p}\dif j(x,y)\\
		&\leq 3(1+p_{0}2^{p_0-1})\sum_{n=1}^{\infty}a_{n}^{p}\int_{\Omega_n}\abs{u}^{p}\dif\Gamma_{\Omega_n}(\phi_n)\\
		&\quad+6(1+p_02^{p_0-1})C_{\mathrm{TJ}}\bigg(\sum_{n=1}^{\infty}a_{n}^{p}\sup_{x\in \Omega}\frac{1}{W(x,s_{n+1})}\bigg)\int_{\Omega}\abs{u}^{p}\dif m.\\
		&\quad +2(1+p_{0}2^{p_0-1})C_{\mathrm{TJ}}\bigg(\sum_{i=1}^{\infty}\sum_{k=i+2}^{\infty}(a_{i}^{p_0}a_{k}^{\epsilon}+a_{k}^{p_0}a_{i}^{\epsilon})\sup_{x\in \Omega}\frac{1}{W(x,s_{i+2})}\bigg)\int_{\Omega}\abs{u}^p\dif m\\
		&\quad +2C_{\mathrm{TJ}}\bigg(\sum_{k=2}^{p_0}\sum_{(t_1,\ldots,t_k)\in\Lambda_k}\sum_{h=1}^{k-1}\sum_{\substack{1\leq i_1<\cdots<i_k< \infty \\ {\abs{i_h-i_{h+1}}\geq2}}}\sup_{x\in \Omega}\frac{\prod_{l=1}^{k}a_{i_l}^{t_l}}{W(x,s_{2+i_h})}\bigg)\int_{\Omega}\abs{u}^{p}\dif m. \label{e.sip10}
	\end{align}
Next, we estimate the first term in the right hand side of \eqref{e.sip10}. By \eqref{e.siCS},
\begin{align}	\sum_{n=1}^{\infty}a_{n}^{p}\int_{\Omega_n}\abs{u}^{p}\dif\Gamma_{\Omega_n}\la \phi_n\ra 
	\leq\sum_{n=1}^{\infty}a_{n}^{p}\bigg(\eta_0\int_{\Omega_n}\dif\Gamma_{\Omega_n}\la u\ra +\sup_{x\in \Omega}\frac{C\int_{\Omega_n}\abs{u}^{p}\dif m}{W(x,s_{n+1})}\bigg), \label{e.sip11}
\end{align}
where \begin{align}
	\sum_{n=1}^{\infty}a_{n}^{p}\int_{\Omega_n}\dif\Gamma_{\Omega_n}\la u\ra &=\sum_{n=1}^{\infty}a_{n}^{p}
\bigg(\int_{B_1}\dif\Gamma_{\Omega_n}\la u\ra +\int_{\Omega_n\setminus B_1}\dif\Gamma_{\Omega_n}\la u\ra \bigg)\\
	&\leq\sum_{n=1}^{\infty}a_{n}^{p}\bigg(\int_{B_1}\abs{\phi}^{p}\dif\Gamma_{\Omega}\la u\ra+\int_{\Omega_n\setminus B_1}\dif\Gamma_{\Omega}\la u\ra \bigg)\\
	&=\frac{q^{-\beta p}(q^\beta-1)^{p}}{1-q^{-\beta p}}\int_{B_1}\abs{\phi}^{p}\dif\Gamma_{\Omega}\la u\ra +\sum_{n=1}^{\infty}a_{n}^{p}\int_{\Omega_n\setminus B_1}\dif\Gamma_{\Omega}\la u\ra . \label{cs-13}
\end{align}
Note that $(q^\beta-1)\restr{\phi}{B_{k+1}\setminus B_k}\geq a_k$ by \eqref{CS-12}, and since $\Omega_n\subset B_{n+3}$ for $q\in(1,9/8)$, we have 
\begin{align}
	&\phantom{\ \leq}\sum_{n=1}^{\infty}a_{n}^{p}\int_{\Omega_n\setminus B_1}\dif\Gamma_{\Omega}\la u\ra \leq \sum_{n=1}^{\infty}\int_{B_{n+3}\setminus B_{n+1}}a_{n}^{p}\dif\Gamma_{\Omega}\la u\ra +\sum_{n=1}^{\infty}a_{n}^{p}\sum_{k=1}^{n}\int_{B_{k+1}\setminus B_{k}}\dif\Gamma_{\Omega}\la u\ra \\
	&\leq q^{2\beta p}(q^\beta-1)^{p}\sum_{n=1}^{\infty}\int_{B_{n+3}\setminus B_{n+1}}\phi^{p}\dif\Gamma_{\Omega}\la u\ra +\sum_{k=1}^{\infty}\sum_{n=k}^{\infty}q^{-(n-k)\beta p}\int_{B_{k+1}\setminus B_k}a_k^{p}\dif\Gamma_{\Omega}\la u\ra \\
	&\leq 2q^{2\beta p}(q^\beta-1)^{p}\sum_{n=1}^{\infty}\int_{B_{n+1}\setminus B_{n}}\phi^{p}\dif\Gamma_{\Omega}\la u\ra +\sum_{k=1}^{\infty}\frac{1}{1-q^{-\beta p}}\int_{B_{k+1}\setminus B_k}(q^\beta-1)^{p}\phi^{p}\dif\Gamma_{\Omega}\la u\ra \\
	&\leq \frac{q^{-\beta p}(q^\beta-1)^{p}}{1-q^{-\beta p}}\left(2q^{2\beta p}(q^{\beta p}-1)+q^{\beta p}\right)
\int_{\Omega\setminus B_1}\phi^{p}\dif\Gamma_{\Omega}\la u\ra . \label{cs-14}
\end{align}
Therefore, by choosing $q$ close to $1$, we obtain from \eqref{cs-13} and \eqref{cs-14} that
\begin{align}
	&\phantom{\ \leq}\sum_{n=1}^{\infty}a_{n}^{p}\int_{\Omega_n}\dif\Gamma_{\Omega_n}\la u\ra \leq \frac{3q^{-\beta p}(q^\beta-1)^{p}}{1-q^{-\beta p}}\int_{\Omega}\phi^{p}\dif\Gamma_{\Omega}\la u\ra .\label{e.sip12}
\end{align}
Inequalities \eqref{e.sip10}, \eqref{e.sip11}, \eqref{e.sip12} together gives 	\begin{align}
		\int_{\Omega}\abs{u}^{p}\dif\Gamma_{\Omega}\la \phi \ra &\leq 9\eta_{0}(1+p_{0}2^{p_0-1})\frac{q^{-\beta p}(q^\beta-1)^{p}}{1-q^{-\beta p}}\int_{\Omega}\phi^{p}\dif\Gamma_{\Omega}\la u\ra \\
		&\qquad+(3C+6C_{\mathrm{TJ}})(1+p_{0}2^{p_0-1})\bigg(\sum_{n=1}^{\infty}\sup_{x\in \Omega}\frac{a_{n}^{p}}{W(x,s_{n+1})}\bigg)\int_{\Omega}\abs{u}^{p}\dif m.\\
		&\qquad +2(1+p_{0}2^{p_0-1})C_{\mathrm{TJ}}\bigg(\sum_{i=1}^{\infty}\sum_{k=i+2}^{\infty}
(a_{i}^{p_0}a_{k}^{\epsilon}+a_{k}^{p_0}a_{i}^{\epsilon})\sup_{x\in \Omega}\frac{1}{W(x,s_{i+2})}\bigg)\int_{\Omega}\abs{u}^{p}\dif m\\
		&\qquad +2C_{\mathrm{TJ}}\bigg(\sum_{k=2}^{p_0}\sum_{(t_1,\ldots,t_k)\in\Lambda_k}\sum_{h=1}^{k-1}\sum_{\substack{1\leq i_1<\cdots<i_k<\infty \\ {\abs{i_h-i_{h+1}}\geq2}}}\sup_{x\in \Omega}\frac{\prod_{l=1}^{k}a_{i_l}^{t_l}}{W(x,s_{2+i_h})}\bigg)\int_{\Omega}\abs{u}^{p}\dif m. \label{e.sip13}
	\end{align}
	Let us compute the constants appeared in \eqref{e.sip13}. Firstly, by \eqref{eq:vol_0}, for any $x\in \Omega$,
\begin{align}
\frac{1}{W(x,s_{n+1})}=\frac{W(x,r)}{W(x,s_{n+1})}\frac{1}{W(x,r)}\leq \bigg(\frac{r}{s_{n+1}}\bigg)^{\beta}\frac{C}{W(x,r)}=\frac{C(q-1)^{-\beta}q^{(n+1)\beta}}{W(x,r)}, \label{cs-15}
\end{align}
which implies
\begin{align}
	\sum_{n=1}^{\infty}\sup_{x\in \Omega}\frac{a_{n}^{p}}{W(x,s_{n+1})}&\leq 
\sum_{n=1}^{\infty}(q^\beta-1)^{p}q^{-n\beta p}\cdot (q-1)^{-\beta}q^{(n+1)\beta}\sup_{x\in \Omega}\frac{C}{W(x,r)}\\
	&=\frac{q^{\beta}(q^\beta-1)^{p}q^{-\beta(p-1)}}{(q-1)^{\beta}(1-q^{-\beta(p-1)})}\sup_{x\in \Omega}\frac{C }{W(x,r)}.\label{e.sipct1}
\end{align}
Recall that we are now assuming that $\epsilon>0$, so\begin{align}
	&\phantom{\ \leq}\sum_{i=1}^{\infty}\sum_{k=i+2}^{\infty}(a_{i}^{p_0}a_{k}^{\epsilon}+a_{k}^{p_0}a_{i}^{\epsilon})
\sup_{x\in \Omega}\frac{1}{W(x,s_{i+2})}\\
	&\leq \sum_{i=1}^{\infty}\sum_{k=i+2}^{\infty} (q^\beta-1)^{p}(q^{-i\beta p_0}q^{-k\beta\epsilon}+q^{-k\beta p_0}q^{-i\beta\epsilon})(q-1)^{-\beta}q^{(i+2)\beta}\sup_{x\in \Omega}\frac{C}{W(x,r)}\\
	&=\frac{(q^{\beta}-1)q^{\beta(3-p)}}{(q-1)^{\beta}(1-q^{-\beta(p-1)})}
\left(\frac{q^{-2\beta \epsilon}}{1-q^{-\beta\epsilon}}+\frac{q^{-2\beta p_0}}{1-q^{-\beta p_0}}\right)\sup_{x\in \Omega}\frac{C}{W(x,r)}.\label{e.sipct2}
\end{align}
For the third constant, note that for $2\leq k\leq p_0$, $(t_1,\ldots, t_k)\in\Lambda_{k}$ and $1\leq h\leq k-1$, 
\begin{align}
	&\phantom{\ \leq}\sum_{\substack{1\leq i_1<\cdots<i_k<\infty \\ {\abs{i_h-i_{h+1}}\geq2}}}\sup_{x\in \Omega}\frac{\prod_{l=1}^{k}a_{i_l}^{t_l}}{W(x,s_{2+i_h})} \overset{\eqref{cs-15}}{\leq}\frac{q^{2\beta}(q^{\beta}-1)^{p_0}}{(q-1)^{\beta}}
\sum_{i_1=1}^{\infty}q^{-i_{1}t_{1}\beta}\sum_{i_2=i_1+1}^{\infty}q^{-i_{2}t_{2}\beta}\times\cdots\\
	&\qquad\times \sum_{i_h=i_{h-1}+1}^{\infty}q^{-i_{h}t_{h}\beta}q^{i_{h}\beta}\sum_{i_{h+1}=i_{h}+1}^{\infty}q^{-i_{h+1}t_{h+1}\beta}\ldots \sum_{i_k=i_{k-1}+1}^{\infty}q^{-i_{k}t_{k}\beta}\cdot\sup_{x\in \Omega}\frac{C}{W(x,r)}\\
	&\leq\frac{q^{2\beta}(q^{\beta}-1)^{p_0}}{(q-1)^{\beta}}\cdot \frac{q^{-2\beta(p_0-1)}}{(1-q^{-\beta})^{k}}\cdot \sup_{x\in \Omega}\frac{C}{W(x,r)}.
\end{align}
Therefore, since $\#\Lambda_k=\binom{p_0-1}{k-1}$,
\begin{align}
	&\phantom{\ \leq}\sum_{k=2}^{p_0}\sum_{(t_1,\ldots,t_k)\in\Lambda_k}\sum_{h=1}^{k-1}\sum_{\substack{1\leq i_1<\cdots<i_k<\infty \\ {\abs{i_h-i_{h+1}}\geq2}}}\sup_{x\in \Omega}\frac{\prod_{l=1}^{k}a_{i_l}^{t_l}}{W(x,s_{2+i_h})}\\
	&\leq \frac{(q^{\beta}-1)^{p_0}q^{-2\beta(p_0-2)}}{(q-1)^{\beta}}\sum_{k=2}^{p_0}\binom{p_0-1}{k-1}\frac{(k-1)}{(1-q^{-\beta})^{k}} \cdot\sup_{x\in \Omega}\frac{C}{W(x,r)} \\
&= \frac{(q^{\beta}-1)^{p_0}q^{-2\beta(p_0-2)}}{(q-1)^{\beta}}\sum_{k=2}^{p_0}\binom{p_0-2}{k-2}\frac{(p_0-1)}{(1-q^{-\beta})^{k}} \cdot \sup_{x\in \Omega}\frac{C}{W(x,r)} \\
&=\frac{(p_0-1)(q^{\beta}-1)^{p_0}q^{-2\beta(p_0-2)}}{(q-1)^{\beta}(1-q^{-\beta})^2} \left(1+\frac{1}{1-q^{-\beta}}\right)^{p_0-2}\sup_{x\in \Omega}\frac{C}{W(x,r)} \\
&=\frac{(p_0-1)q^{-(p_0-4)\beta}(2-q^{-\beta})^{p_0-2}}{(q-1)^{\beta}} \sup_{x\in \Omega}\frac{C}{W(x,r)}.\label{e.sipct3}
\end{align}
Combining \eqref{e.sipct1}, \eqref{e.sipct2}, \eqref{e.sipct3} with \eqref{e.sip13}, we finally have
\begin{equation}
	\int_{\Omega}\abs{u}^{p}\dif\Gamma_{\Omega}\la \phi \ra \leq 3\eta_{0}(1+p_{0}2^{p_0-1})\frac{q^{-\beta p}(q^\beta-1)^{p}}{1-q^{-\beta p}}\int_{\Omega}\phi^{p}\dif\Gamma_{\Omega}\la \phi \ra + \sup_{x\in \Omega}\frac{C(\eta_0,C_{\mathrm{TJ}},p,\beta,q)}{W(x,r)}\int_{\Omega}\abs{u}^{p}\dif m.
\end{equation}
As $\frac{q^{-\beta p}(q^\beta-1)^{p}}{1-q^{-\beta p}}\to 0^+$ as $q\to 1^{+}$, we obtain \eqref{iCS} under the assumption $p_0\geq2$ and $\epsilon>0$.

If $p_0=1$, then $B(x,y)$ in \eqref{e.sidcmp} vanishes so we only need to estimate the term involved with $A(x,y)$; if $\epsilon=0$, the term $A_3$ vanishes in \eqref{e.Adcmp} so the proofs are easier. The details are omitted.
\end{proof}

\subsection{Poincar\'e, Faber--Krahn and Sobolev inequalities}\label{ss.PIFKSI}

Recall that we define a linear subspace $\sF(\Omega)$ of $\sF$ for any open set $\Omega$ in \eqref{e.partEF}.
\begin{definition}
	We say that $(\sE,\sF)$ satisfies \emph{Faber--Krahn inequality} \eqref{FK}, if there exist three constants $\sigma\in(0,1]$, $C\in(0,\infty)$ and $\nu\in(0,\infty)$ such that for any metric ball $B=B(x,r)$ with $r<\sigma\ol{R}$ and for any non-empty open set $\Omega\subset B$, \begin{equation}\label{FK}\tag{$\operatorname{FK}$}
		\sE(u)\geq\frac{C}{W(B)}\bigg(\frac{m(B)}{m(\Omega)}\bigg)^{\nu}\norm{u}^p_{L^{p}(\Omega,m)},\ \forall u\in\sF(\Omega).
	\end{equation}
\end{definition}
\begin{remark}\label{r.FK}
	\begin{enumerate}[label=\textup{({\arabic*})},align=left,leftmargin=*,topsep=5pt,parsep=0pt,itemsep=2pt]
	\item Define for each open subset $\Omega\subset M$ a quantity $\lambda_{1}(\Omega):=\inf_{u\in\sF(\Omega)\setminus\{0\}}\frac{\sE(u)}{\norm{u}^p_{L^{p}(\Omega)}}$. Then \eqref{FK} can be equivalently formulated as \begin{equation}
		\lambda_{1}(\Omega)\geq\frac{C}{W(B)}\bigg(\frac{m(B)}{m(\Omega)}\bigg)^{\nu}\ \text{for any non-empty open set $\Omega\subset B$}.
	\end{equation}
	For $p=2$, by the variational principle for Hilbert space, $\lambda_{1}(\Omega)$ is the \emph{principal Dirichlet eigenvalue} of the domain $\Omega$.
	\item If \eqref{FK} holds for $\nu\in(0,\infty)$, then it holds for any $\nu^{\prime}\in(0,\nu\wedge1)$.
\end{enumerate}

\end{remark}
 An obvious modification of \cite[Lemmas 2.2 and 2.3]{HY23} gives

\begin{proposition}\label{p.PI=>FK}
Let $(\sE,\sF)$ be a mixed local and nonlocal $p$-energy form on $(M,d,m)$. Then 
\begin{equation}
 \eqref{VD}+\eqref{RVD}+\eqref{PI}\Longrightarrow\eqref{FK}. 
\end{equation}
\end{proposition}
In the remaining of this subsection, we prove the equivalence between Faber--Krahn inequality \eqref{FK} and various Sobolev inequalities.

\begin{definition}
	We say that $(\sE,\sF)$ satisfies \emph{Sobolev's inequality} \eqref{SI}, if there exist three constants $\sigma\in(0,1]$, $C\in(0,\infty)$ and $\nu\in(0,1)$ such that for any metric ball $B=B(x,r)$ with radii $r\in(0,\sigma\ol{R})$, \begin{equation}
		\label{SI}\tag{$\operatorname{\mathrm{SI}}$}
		\sE(u)\geq C\frac{m(B)^{\nu}}{W(B)}\norm{u}_{L^{\nu^*}(B)}^{p},\ \forall u\in\sF(B),
	\end{equation}
	where $\nu^*:=\frac{p}{1-\nu}$ is called the \emph{Sobolev conjugate} of $\nu$.
\end{definition}
\begin{definition}
	We say that $(\sE,\sF)$ satisfies \emph{Sobolev type inequality} \eqref{STI}, if there exist three constants $\sigma\in(0,1]$, $C>0$ and $\nu>0$ such that for any metric ball $B=B(x,r)$ with $r\in(0,\sigma\ol{R})$,
	\begin{equation}
		\label{STI}\tag{$\operatorname{\mathrm{STI}}$}
		\int_{B}\abs{u}^{p(1+\nu)}\dif m\leq C\frac{W(B)}{m(B)^{\nu}}\sE(u)\bigg(\int_B \abs{u}^{p}\dif m\bigg)^{\nu},\ \forall u\in\sF(B).
	\end{equation}
\end{definition}

\begin{theorem}\label{t.FK=Sob}
Let $(\sE,\sF)$ be a mixed local and nonlocal $p$-energy form on $(M,d,m)$. Then\begin{equation}
		\eqref{FK}\Longleftrightarrow\eqref{SI}\Longleftrightarrow\eqref{STI}.
\end{equation}
\end{theorem}

Suppose $B\subset M$ is a metric ball with radii less than $\sigma\ol{R}$ and $f\in\sF(B)\cap C_{c}^{+}(B)$, where $C^{+}_{c}(B)$ is the set of all continuous functions in $C_{c}(B)$ with non-negative values. For each $k\in\bZ$, we write \begin{equation}\label{e.f-dcmp}
	u_{k}(x):=\big(\min(u(x),2^{k+1})-2^{k}\big)_{+},\ x\in M
\end{equation}
and\begin{equation}\label{e.Om-dcmp}
	\Omega_{k}:=\Sett{x\in B}{u(x)>2^{k}}\subset B,
\end{equation}
which is open in $B$ and in $M$. 

\begin{proposition}\label{p.p4}
Let $B\subset M$ be a metric ball with radii less than $\sigma\ol{R}$ and $u\in\sF(B)\cap C_{c}^{+}(B)$. Let $\set{u_{k}}_{k\in\bZ}$ be defined as in \eqref{e.f-dcmp}. Then \begin{equation}\label{e.p40}
			 u(x)=\sum_{k\in \bZ}u_{k}(x),\ \forall x\in M.
		\end{equation} Moreover, for each $k\in\bZ$, the followings are true:
	\begin{enumerate}[label=\textup{({\arabic*})},align=left,leftmargin=*,topsep=5pt,parsep=0pt,itemsep=2pt]
		\item\label{lb.p4-1} $u_k\in\sF(\Omega_k)\subset\sF(B)$.
		\item\label{lb.p4-2} For any $j<k$, $u_j$ is constant over $\Omega_k$.
		\item\label{lb.p4-3} For any $j\leq k$, if we denote $T_{j,k}(s):= ((s\wedge 2^{k+1})-2^j)_{+}$, $s\in\bR$.
		Then $T_{j,k}$ is a $1$-Lipschitz function on $\bR$ and \begin{equation}\label{e.p41}
			\sum_{i=j}^{k}u_{k}(x)=T_{j,k}(u(x)),\ \forall x\in M.
		\end{equation}
	\end{enumerate}
\end{proposition}
\begin{proof}
	\begin{enumerate}[label=\textup{({\arabic*})},align=left,leftmargin=*,topsep=5pt,parsep=0pt,itemsep=2pt]
	\item[\ref{lb.p4-1}] Let $T_{k}(s):=((s\wedge 2^{k+1})-2^{k})_{+},\ s\in\bR$.
		Then $u_{k}(x)=T_{k}(f(x))$ for all $x\in M$. Note that $T_k$ is a $1$-Lipschitz function on $[0,\infty)$, we see that $u_{k}\in\sF$ by the Markovian property. Since $\restr{u_{k}}{B\setminus\Omega_{k}}=0$ and $f=0$ $\sE$-q.e. on $M\setminus	B$, we have $u_{k}=0$ $\sE$-q.e. on $M\setminus \Omega_k$ and thus $f\in\sF(\Omega_k)$.
	\item[\ref{lb.p4-2}] By the definitions of $u_j$ and $\Omega_{k}$ we see that $\restr{u_j}{\Omega_k}=2^j$ for each $k>j$.
	\item[\ref{lb.p4-3}] By a direct computation, $\sum_{i=j}^{k}T_{i}(s)=T_{j,k}(s)$ for all $s\in[0,\infty)$, which gives \eqref{e.p41}. For each $x\in M$, let $j\to-\infty$ and $k\to\infty$ in \eqref{e.p41} we obtain \eqref{e.p40}.
\end{enumerate}
\end{proof}
\begin{lemma}\label{l.l6}
Let $B\subset M$ be a metric ball with radii less than $\sigma\ol{R}$ and $u\in\sF(B)\cap C^{+}_{c}(B)$. Let $\set{u_{k}}_{k\in\bZ}$ be defined as in \eqref{e.f-dcmp}, then 
	\begin{equation}\label{e.p61}
	\sE(u)\geq \frac{1}{6}\sum_{j\in\bZ}\sE(u_j).
 	\end{equation}	
\end{lemma}
\begin{proof}
 We will deal with the local part $\sE^{(L)}$ and the nonlocal part $\sE^{(J)}$ respectively. Some of our proof is motivated by \cite{BCLSC95}.
	\begin{enumerate}[label=\textup{({\arabic*})},align=left,leftmargin=*,topsep=5pt,parsep=0pt,itemsep=2pt]
	\item (Local part) Recall that in Proposition \ref{p.p4}-\ref{lb.p4-3}, we have obtained that for each $i\leq k$, $\sum_{i=j}^{k}f_{i}=T_{j,k}\circ u$ and $T_{j,k}$ is $1$-Lipschitz. Therefore by the Markovian property and the strong locality of $(\sE^{(L)},\sF\cap L^{\infty}(M,m))$, and Proposition \ref{p.SL},
\begin{align}
		\sE^{(L)}(u)\geq \sE^{(L)}(T_{j,k}\circ u)=\sE^{(L)}\big(\sum_{i=j}^{k}u_{i}\big)=\sum_{i=j}^{k}\sE^{(L)}(u_i).
	\end{align}
	Letting $k\to\infty$ and $j\to-\infty$ we have \begin{equation}\label{e.l6l}
		\sE^{(L)}(u)\geq\sum_{j\in\bZ}\sE^{(L)}(u_j).
	\end{equation}
	\item (Nonlocal part) Denote $B_k:=\Sett{x\in M}{2^{k}<u(x)\leq 2^{k+1}}$ for $k\in\bZ$.	Then $\bigcup_{k\in\bZ} B_k$ is a disjoint union. Decompose $\sE^{(J)}(u_{k})$ as the following:\begin{align}
		\sE^{(J)}(u_{k})&=\int_{M_{\od}^{2}}\abs{u_{k}(x)-u_{k}(y)}^{p}\dif j(x,y)\\
		&=\bigg(\int_{B_{k}\times M}+\int_{(M\setminus B_{k})\times B_{k}}+\int_{(M\setminus B_{k})\times (M\setminus B_{k})}\bigg)\abs{u_{k}(x)-u_{k}(y)}^{p}\dif j(x,y)\\
		&:=J^{(1)}_{k}+J^{(2)}_{k}+J^{(3)}_{k}.\label{e.l6j0}
	\end{align}
	\begin{enumerate}[label=\textup{(\roman*)},align=left,leftmargin=*,topsep=5pt,parsep=0pt,itemsep=2pt]
	\item By Proposition \ref{p.p4}-\ref{lb.p4-3}, we have $\abs{u_{k}(x)-u_{k}(y)}\leq\abs{u(x)-u(y)}$, $\ \forall x,y\in M$. Therefore, \begin{align}
		\sum_{k\in\bZ}J^{(1)}_{k}&\leq \sum_{k\in\bZ}\int_{B_{k}\times M}\abs{u(x)-u(y)}^{p}\dif j(x,y)\\
		&\leq\int_{M_{\od}^{2}}\abs{u(x)-u(y)}^{p}\dif j(x,y)=\sE^{(J)}(u).\label{e.l6j1}
	\end{align}
	\item By the symmetry of $j$, \begin{align}
		J^{(2)}_{k}&=\int_{(M\setminus B_{k})\times B_{k}}\abs{u_{k}(x)-u_{k}(y)}^{p}\dif j(x,y)\leq \int_{M\times B_{k}}\abs{u_{k}(x)-u_{k}(y)}^{p}\dif j(x,y)=J^{(1)}_{k}
	\end{align}
	Therefore \begin{equation}\label{e.l6j2}
		\sum_{k\in\bZ}J^{(2)}_{k}\leq\sum_{k\in\bZ}J^{(1)}_{k}\overset{\eqref{e.l6j1}}{\leq}\sE^{(J)}(u).
	\end{equation}
	\item For $x,y\in M\setminus B_k$, we have $\abs{u_{k}(x)-u_{k}(y)}\leq 2^{k}$ and the value of $\abs{u_{k}(x)-u_{k}(y)}$ is non-zero if and only one of the following conditions hold \begin{itemize}
		\item $u(x)>2^{k+1}$ and $u(y)\leq 2^{k}$;
		\item $u(x)\leq 2^{k}$ and $u(y)>2^{k+1}$.
	\end{itemize}
	
	Define 	$Z_{k}:=\Sett{(x,y)\in M_{\od}^{2}}{u(y)\leq 2^{k}<\frac{1}{2}u(x)}$,
then \begin{align}
		J^{(3)}_{k}&=\int_{(M\setminus B_{k})\times (M\setminus B_{k})}\abs{u_{k}(x)-u_{k}(y)}^{p}\dif j(x,y)\\
		&\leq \int_{M_{\od}^{2}}2^{pk}\big(\one_{Z_k}(x,y)+\one_{Z_{k}}(y,x)\big)\dif j(x,y).\label{e.l6j31}
	\end{align}
	For any $(x,y)\in \bigcup_{k\in\bZ}Z_{k}$, we can always choose $k_2\leq k_1$ such that \begin{equation}
		2^{k_1}<\frac{1}{2}u(x)\leq 2^{k_1+1}\text{ and }2^{k_2-1}<u(y)\leq 2^{k_2}.
	\end{equation} 
	Therefore $(x,y)\in Z_k$ if and only if $k_2\leq k\leq k_1$. In this case, $u(x)-u(y)>2^{k_1}$, and
 \begin{equation}
		\sum_{k\in\bZ}2^{pk}\one_{Z_k}(x,y)=\sum_{k=k_2}^{k_1}2^{pk}\leq \sum_{k=-\infty}^{k_1}2^{pk}= \frac{2^{p}}{2^{p}-1}2^{pk_1}\leq\frac{2^{p}}{2^{p}-1}\abs{u(x)-u(y)}^{p}.\label{e.l6j32}
	\end{equation}
Equations \eqref{e.l6j31} and \eqref{e.l6j32} together imply that 
	\begin{align}
	\sum_{k\in\bZ}J^{(3)}_{k}\leq \frac{2^{p+1}}{2^{p}-1}\int_{M_{\od}^{2}}\abs{u(x)-u(y)}^{p}\dif j(x,y)=\frac{2^{p+1}}{2^{p}-1}\sE^{(J)}(u).\label{e.l6j3}
	\end{align}
\end{enumerate}
Combining \eqref{e.l6j1}, \eqref{e.l6j2} and \eqref{e.l6j3} we have \begin{align}
	\sum_{k\in\bZ}\sE^{(J)}(u_{k})&\overset{\eqref{e.l6j0}}{\leq} \sum_{k\in\bZ}J^{(1)}_{k}+\sum_{k\in\bZ}J^{(2)}_{k}+\sum_{k\in\bZ}J^{(3)}_{k}\\
	&\leq \sE^{(J)}(u)+\sE^{(J)}(u)+\frac{2^{p+1}}{2^{p}-1}\sE^{(J)}(u)=\frac{2^{p+2}-2}{2^{p}-1}\sE^{(J)}(u).\label{e.l6j}
\end{align}
\end{enumerate}
Finally, combining \eqref{e.l6j} with the previous estimates on local part \eqref{e.l6l}, we have \begin{align}
	\sE(u)&=\sE^{(L)}(u)+\sE^{(J)}(u)\geq \sum_{k\in\bZ}\sE^{(L)}(u_{k})+\frac{2^{p}-1}{2^{p+2}-2}\sum_{k\in\bZ}\sE^{(J)}(u_{k})\\
	&\geq \frac{2^{p}-1}{2^{p+2}-2}\sum_{k\in\bZ}\big(\sE^{(L)}(u_{k})+\sE^{(J)}(u_{k})\big)=\frac{2^{p}-1}{2^{p+2}-2}\sum_{k\in\bZ}\sE(u_{k})\geq \frac{1}{6}\sum_{k\in\bZ}\sE(u_{k}),
\end{align}
which is \eqref{e.p61}.
\end{proof}

\begin{proof}[Proof of Theorem \ref{t.FK=Sob}]
We will prove \eqref{FK}$\Longrightarrow$\eqref{SI}$\Longrightarrow$\eqref{STI}$\Longrightarrow$\eqref{FK}.
\begin{enumerate}[label={AAA}, align=left,leftmargin=*,topsep=5pt,parsep=0pt,itemsep=2pt]
	\item[\eqref{FK}$\Longrightarrow$\eqref{SI}:] Let $(\sigma,C,\nu)$ be the constants in \eqref{FK}. We may assume $\nu\in(0,1)$. Let $B\subset M$ be a metric ball with radii less than $\sigma\ol{R}$ and $u\in\sF(B)\cap C_{c}^{+}(B)$. Let $\set{u_{k}}_{k\in\bZ}$ and $\set{\Omega_k}_{k\in\bZ}$ be defined as in \eqref{e.f-dcmp} and \eqref{e.Om-dcmp}, respectively. Let $m_k:=m(\ol{\Omega_k})$. Then $\Sett{x\in B}{u(x)>0}=\bigcup_{k\in\bZ}\Omega_k$ and \begin{equation}
		\int_{B}\abs{u}^{\frac{p}{1-\nu}}\dif m=\sum_{k\in\bZ}\int_{\Omega_{k}\setminus\Omega_{k+1}}{u}^{\frac{p}{1-\nu}}\dif m\leq \sum_{k\in\bZ}2^{\frac{(k+1)p}{1-\nu}}m_k.\label{e.pf0}
	\end{equation}
	On the other hand, \begin{align}
		\sE(u)&\overset{\eqref{e.p61}}{\geq}\frac{1}{6}\sum_{k\in\bZ}\sE(u_{k})\overset{\eqref{FK}}{\geq} \frac{1}{6}\sum_{k\in\bZ}\frac{C}{W(B)}\bigg(\frac{m(B)}{m(\Omega_k)}\bigg)^\nu\int_{\Omega_k}u_{k}^{p}\dif m\\
		&\geq \frac{C}{6}\frac{m(B)^{\nu}}{W(B)}\sum_{k\in\bZ}\frac{2^{pk}m_{k+1}}{m_k^{\nu}}=\frac{C}{6}\frac{m(B)^{\nu}}{2^{p}W(B)}\sum_{k\in\bZ}\frac{2^{\frac{p(k+1)}{1-\nu}}m_{k+1}}{2^{\frac{p(k+1)\nu}{1-\nu}}m_k^{\nu}}.\label{e.pf1}
	\end{align}
	Applying the following inequality \begin{equation}
		\sum_{k\in\bZ}\frac{x_k}{y_k}\geq\frac{\big(\sum_{k\in\bZ}x_k^{1/r}\big)^r}{\big(\sum_{k\in\bZ}y_k^{s/r}\big)^{r/s}}\geq\frac{\sum_{k\in\bZ}x_k}{\big(\sum_{k\in\bZ}y_k^{s/r}\big)^{r/s}},\ \forall \set{x_k},\set{y_k}\subset(0,\infty),\ \frac{1}{r}+\frac{1}{s}=1\label{e.pf2}
	\end{equation}
	with $x_k=2^{\frac{p(k+1)}{1-\nu}}m_{k+1}$, $y_k=2^{\frac{p(k+1)\nu}{1-\nu}}m_k^{\nu}$, $r=1+\nu$ and $s=(1+\nu)/\nu$, we obtain \begin{align}
		\sE(u)&\overset{\eqref{e.pf1},\eqref{e.pf2}}{\geq}\frac{C}{6}\frac{m(B)^{\nu}}{2^{p}W(B)}\bigg(\sum_{k\in\bZ}2^{\frac{p(k+1)}{1-\nu}}m_{k+1}\bigg)\bigg(\sum_{k\in\bZ}\big(2^{\frac{p(k+1)\nu}{1-\nu}}m_k^{\nu}\big)^{1/\nu}\bigg)^{-\nu}\\
		&=\frac{C}{6}\frac{m(B)^{\nu}}{2^{\frac{p(2-\nu)}{1-\nu}}W(B)}\bigg(\sum_{k\in\bZ}2^{\frac{p(k+1)}{1-\nu}}m_{k}\bigg)^{1-\nu}\overset{\eqref{e.pf0}}{\geq}\frac{C}{6}\frac{m(B)^{\nu}}{2^{\frac{p(2-\nu)}{1-\nu}}W(B)}\bigg(\int_{B}\abs{u}^{\frac{p}{1-\nu}}\dif m\bigg)^{\frac{1-\nu}{p}\cdot p}\\
&=2^{-\frac{p(2-\nu)}{1-\nu}}\frac{C}{6}\frac{m(B)^{\nu}}{W(B)}\norm{u}_{L^{\nu^*}(B)}^{p}.\label{e.pf3}
	\end{align}
	For $u\in \sF(B)\cap C_{c}(B)$, we may apply \eqref{e.pf3} to $\abs{u}\in \sF(B)\cap C_{c}^{+}(B)$ and obtain that \begin{equation}\label{e.pf4}
		\sE({u})\geq\sE(\abs{u})\overset{\eqref{e.pf3}}{\geq} 2^{-\frac{p(2-\nu)}{1-\nu}}\frac{C}{6}\frac{m(B)^{\nu}}{W(B)}\norm{u}_{L^{\nu^*}(B)}^{p},\ \forall u\in \sF(B)\cap C_{c}(B).
	\end{equation}
	For general $u\in \sF(B)$, by the regularity of $(\sE_{B},\sF(B))$ in Lemma \ref{l.partEF}-\ref{lb.part-reg}, there is a sequence $\{u_n\}\subset\sF(B)\cap C_c(B)$ such that $u_n\to u$ $m$-a.e. and $\sE(u_n)\to \sE(u)$, thus by Fatou's lemma \begin{align}
		2^{-\frac{p(2-\nu)}{1-\nu}}C\cdot \frac{1}{6}\frac{m(B)^{\nu}}{W(B)}\norm{u}_{L^{\nu^*}(B)}^{p}&\leq \liminf_{n\to\infty}2^{-\frac{p(2-\nu)}{1-\nu}}C\cdot \frac{1}{6}\frac{m(B)^{\nu}}{W(B)}\norm{u_n}_{L^{\nu^*}(B)}^{p} \\
		&\overset{\eqref{e.pf4}}{\leq}\liminf_{n\to\infty}\sE(u_n)=\sE(u).
	\end{align}
\item[\eqref{SI}$\Longrightarrow$\eqref{STI}:]Let $(\sigma,C,\nu)$ be the constants in \eqref{SI}. Let $B\subset M$ be a metric ball with radii less than $\sigma\ol{R}$ and let $u\in\sF(B)$. Note that \begin{equation}
	\frac{1}{p(1+\nu)}=\frac{\theta}{\nu^*}+\frac{1-\theta}{p},\ \text{with } \theta=\frac{1}{1+\nu},
\end{equation}
we can use interpolation theorem \cite[Proposition 6.10]{Fol99} and obtain \begin{align}
	\norm{u}_{L^{p(1+\nu)}(B)}&\leq\norm{u}_{L^{\nu^*}(B)}^{\theta}\norm{u}_{L^{p}(B)}^{1-\theta}\\
	&\overset{\eqref{SI}}{\leq}\bigg(\frac{W(B)}{C\cdot m(B)^{\nu}}\sE(u)\bigg)^{1/(p(1+\nu))}\bigg(\int_B \abs{u}^{p}\dif m\bigg)^{\nu/(p(1+\nu))},
	\end{align}
	which is equivalent to \eqref{STI} by rearranging the terms.
\item[\eqref{STI}$\Longrightarrow$\eqref{FK}:]Let $(\sigma,C,\nu)$ be the constants in \eqref{SI}. Let $B\subset M$ be a metric ball with radii less than $\sigma\ol{R}$ and let $u\in\sF(B)$. Let $\Omega\subset B$ be any non-empty open subset of $B$. Then for $u\in\sF(\Omega)\subset\sF(B)$, we have
\begin{align}
	m(\Omega)^{-\nu}\bigg(\int_\Omega \abs{u}^{p}\dif m\bigg)^{1+\nu}&\leq \int_\Omega \abs{u}^{p(1+\nu)}\dif m\overset{\eqref{STI}}{\leq}C\frac{W(B)}{m(B)^{\nu}}\sE(u)\bigg(\int_B \abs{u}^{p}\dif m\bigg)^{\nu},
\end{align}
which implies that for any $u\in\sF(B)$, \begin{equation}
	\sE(u)\geq \frac{C^{-1}}{W(B)}\bigg(\frac{m(B)}{m(\Omega)}\bigg)^{\nu}\norm{u}_{L^{p}(B)}^{p}.
\end{equation}
\end{enumerate}
\end{proof}
\begin{remark} \begin{enumerate}[label=\textup{({\arabic*})},align=left,leftmargin=*,topsep=5pt,parsep=0pt,itemsep=2pt]
	\item By checking the proof, we see that the equivalences in Theorem \ref{t.FK=Sob} holds on each individual balls.
	\item If $p$-energy form $(\sE,\sF)$ contains a \emph{potential} (or \emph{killing part} following the terminology in \cite{FOT11}), namely, there is a Radon measure $k$ on $M$ such that $(\sE,\sF)$ admits the following decomposition:\begin{equation}
		\sE(u)=\sE^{(L)}(u)+\int_{M^{2}_{\od}}\abs{u(x)-u(y)}^{p}\dif j(x,y)+\int_{M}\abs{u}^{p}\dif k.
	\end{equation}
	Then Theorem \ref{t.FK=Sob} also holds for $(\sE,\sF)$. In fact, since every $f_j$ we deal with in Lemma \ref{l.l6} is non-negative, we have \begin{equation}\label{e.l6k1}
		u(x)^{p}=\bigg(\sum_{j\in\bZ}u_{j}(x)\bigg)^{p}\geq\sum_{j\in\bZ}u_{j}(x)^{p}.
	\end{equation}
	Integrating \eqref{e.l6k1} on both sides with respect to $k$, we obtain that 
	\begin{equation}
		\label{e.l6k}
		\int_{M}\abs{u}^{p}\dif k\geq\sum_{j\in\bZ}\int_{M}\abs{u_{j}}^{p}\dif k,
	\end{equation}
so Lemma \ref{l.l6} holds, and so does Theorem \ref{t.FK=Sob}.
\end{enumerate}

\end{remark}
\section{Proof of the weak elliptic Harnack inequality}\label{s.Pf-wEH}
In this section, we prove Theorem \ref{t.main}--\eqref{e.main-wEH}, the weak elliptic Harnack inequality. The proof relies on three fundamental estimates, namely the \emph{Caccioppoli inequality}, the \emph{lemma of growth}, and the \emph{crossover lemma}.
\subsection{Caccioppoli inequality}

We first show that the property of subharmonicity is preserved under truncations of functions.
\begin{lemma}\label{l.sub}
 Let $(\mathcal{E},\mathcal{F})$ be a mixed local and nonlocal regular $p$-energy form. Let $\Omega$ be a bounded open set and $u\in\sF^{\prime}\cap L^\infty$ subharmonic in $\Omega$. Then $u_+\in\sF^{\prime}\cap L^\infty$ is also subharmonic in $\Omega$.
\end{lemma}
\begin{proof}
By Proposition \ref{prop:e}, $(\mathcal{E},\mathcal{F})$ is Markovian, which implies that $u_+\in\sF^{\prime}\cap L^\infty$. We are going to prove $\sE(u_+;\varphi)\leq 0$ for every non-negative $\varphi \in \sF(\Omega)$. Owing to the density of $\sF(\Omega)\cap C_{c}(\Omega)$ in $\sF(\Omega)$ by \eqref{e.sF}, we may assume that $0\leq\varphi\in\sF(\Omega)\cap C_{c}(\Omega)$. For each $n\in\bN$, let $F_{n}\in C^{2}(\bR)$ be a function defined by%
\begin{equation}
F_{n}(r):=\frac{1}{2}\left( r+\sqrt{r^{2}+\frac{1}{n^{2}}}\right) -\frac{1}{2n},\ r\in \bR.  \label{F1}
\end{equation}%
Clearly, $F_{n}(0)=0$, and for any $u\in\sF^{\prime}\cap L^\infty$,
\begin{align}
&0<\frac{1}{2}\left( 1-\frac{\norm{u}_{L^\infty}}{\sqrt{\norm{u}_{L^\infty}^{2}+n^{-2}}}
\right)\leq \restr{F_{n}^{\prime}}{[-\norm{u}_{L^\infty},\norm{u}_{L^\infty}]} 
\leq 1,  \\
&0\leq F_{n}^{\prime \prime }(r)=\frac{1}{2n^{2}(r^{2}+n^{-2})^{3/2}}\leq
\frac{n}{2},   \\
&F_{n}(\cdot )\to \max(\wcdot,0)\ \ \text{uniformly on }\bR\
\text{as}\ n\rightarrow \infty .  \label{700}
\end{align}
Using Proposition \ref{P61}, the functions $F_{n}(u)$, $F_{n}^{\prime }(u)^{p-1}$ belong to $
\mathcal{F}^{\prime}\cap L^{\infty }$ for each $n\in\bN$. Using Proposition \ref{p.A3-1}, 
$F_{n}^{\prime }(u)^{p-1}\varphi\in \mathcal{F}(\Omega)\cap L^{\infty }$.

Write $u=v+a$ for some $v\in \mathcal{F}$ and $a\in\mathbb{R}$. Let $f_{n}(t):=F_{n}(t+a)-F_{n}(a)$.
Since $f_n(0)=0$ and $|f_{n}(s)-f_{n}(t)|\leq |s-t|$ for every $s,t\in \mathbb{R}$ by $0<F_{n}^{\prime }\leq 1$, we have by Markovian property that
\begin{equation}
F_{n}(v+a)-F_{n}(a)=f_n(v)\in \mathcal{F}\text{ \ and \ }%
\mathcal{E}(f_{n}(v))\leq \mathcal{E}(v).
\end{equation}

Since $(v+a)_{+}-a_{+}$ is also a normal contraction of $v\in \mathcal{F}$,
we have $u_{+}-a_{+}=(v+a)_{+}-a_{+}\in \mathcal{F}$. On the other hand, by the dominated convergence theorem, $f_{n}(v)\to u_{+}-a_{+}$ in $L^{p}(M,m)$ as $n\to\infty$.
Since $f_{n}(v)\in \mathcal{F}$ and $\sup_{n}\mathcal{E}(f_{n}(v))\leq \mathcal{E}(v)<\infty$, by Lemma \ref{lem:A1}, $f_{n}(v)\to u_{+}-a_{+}$ weakly in $(\sF,\sE_{1}^{1/p})$. By Mazur's lemma \cite[Theorem 2 in Section V.1]{Yos95}, for each $n\in\bN$, there is a convex combination $\{\lambda_{k}^{(n)}\}_{k=1}^{n}\in[0,1]$ with $\sum_{k=1}^{n}\lambda_{k}^{(n)}=1$ such that $\sum_{k=1}^{n}\lambda_{k}^{(n)} f_{k}(v)\to  u_{+}-a_{+}$ in $\sE_1$. Therefore $ \lim_{n\to\infty}\mathcal{E}(g_n(v);\varphi )= \mathcal{E}(u_{+}-a_{+};\varphi )$ by \cite[formula (3.11)]{KS25}, where $g_n:=\sum_{k=1}^{n}\lambda_{k}^{(n)} f_{k}$. Since $g_n$ satisfies all hypotheses in Proposition \ref{P61}, we see by \eqref{eq61} that 
$\mathcal{E}(g_{n}(u);\varphi ){\leq} \mathcal{E}(u;g_{n}^{\prime }(u)^{p-1}\varphi )\leq 0$ for all $n\in \bN$ since $u$ is subharmonic in $\Omega$, therefore $\sE(u_+;\varphi)=\mathcal{E}(u_{+}-a_{+};\varphi )\leq 0$.
\end{proof}

The following Caccioppoli inequality (also known as the \emph{reverse Poincar\'{e} inequality}) shows that the local energy of a subharmonic function can be controlled by its $L^{p}$-norm.
\begin{proposition} \label{prop:cac}
  Assume \eqref{VD}, \eqref{CS} and \eqref{TJ}. Then there exists $C>0$ such that for any $x_0\in M$ and $0<R < R+2r  <\overline{R}$, there exists a cutoff function $\phi$ for $B(x_0,R)\subset B(x_0,R+r)$ such that for any $u\in \mathcal{F}^{\prime}\cap L^{\infty}(M,m)$ which is subharmonic on $B(x_0,R+r)$ and for any $\theta>0$, we have
  \begin{equation}\label{Cacci}
    \int_\Omega \dif \Gamma\la \phi(u-\theta)_+\ra \leq C\left(\sup_{x\in \Omega}\frac{1}{W(x,r)}
   +\left(\frac{T_{B_1,\Omega}((u-\theta)_+)}{\theta}\right)^{p-1}\right)\int_{\Omega} |u|^p \dif m,
  \end{equation}
where 
$\Omega:=B(x_0,R+2r)$.
\end{proposition}
\begin{proof}
For notational convenience, we represent every function in $\mathcal{F}$ and in $\sF^{\prime}\cap L^{\infty}$ by its {$\sE$-quasi-continuous} version. By Proposition \ref{p.CS->iCS}, \eqref{iCS} holds for every $\eta>0$. Let $\phi\in \mathrm{cutoff}(B_0,B_1)$ be the cutoff function in \eqref{iCS} for some $\eta>0$ to be determined later in the proof. Let $w:=(u-\theta)_+$, then $w\in \mathcal{F}^{\prime}\cap L^{\infty}$ is subharmonic on $B(x_0,R+r)$ by Lemma \ref{l.sub}. By Proposition \ref{p.mar}, $\phi^p\in \sF\cap L^\infty$ and then $\phi ^{p}\in \sF(B_1)\cap L^\infty$ by noting that  $\widetilde{\phi ^{p}} =0 $ $\sE$-q.e. on $M\setminus B_1$. By Proposition \ref{p.A3-1}, we have $\phi^pw\in\sF(B_1)$. Therefore, by the subharmonicity of $w$,
\begin{equation}
 0\geq \mathcal{E}(w;\phi^pw)=\mathcal{E}^{(L)}(w;\phi^pw)+\mathcal{E}^{(J)}(w;\phi^pw). \label{subhar}
\end{equation}
\begin{enumerate}[label=\textup{({\arabic*})},align=left,leftmargin=*,topsep=5pt,parsep=0pt,itemsep=2pt]
	\item 
For the local term, by Theorem \ref{t.sasEM}-\ref{lb.M-leibn} and \ref{lb.M-chain},
\begin{equation}
\mathcal{E}^{(L)}(w;\phi^pw)=\int_{M}\phi^p\dif\Gamma^{(L)}\la w\ra +p\int _{M}w\phi^{p-1}\dif \Gamma^{(L)}\la w;\phi \ra .
\end{equation}
By \cite[Proposition 4.8]{KS25} and Young's inequality,
\begin{align}
   \int_{M} w\phi^{p-1}\dif \Gamma^{(L)}\la w;\phi \ra  &\geq -\left(\int_{M} w^p\dif \Gamma^{(L)}\la w\ra \right)^{\frac{p-1}{p}}\left( \int_{M}\phi^{p}\dif \Gamma^{(L)}\la \phi \ra \right)^{\frac{1}{p}}\\
   & \geq -\frac{p-1}{p}2^{\frac{1}{p-1}}\int_{M} w^p \dif \Gamma^{(L)}\la \phi \ra -\frac{1}{2p}\int_{M} \phi^p \dif \Gamma^{(L)}\la w\ra . \label{elt-1}
\end{align}
Therefore,
\begin{equation}
\mathcal{E}^{(L)}(w;\phi^pw)\geq \frac{1}{2}\int_{M}\phi^p\dif\Gamma^{(L)}\la w\ra -2^{\frac{1}{p-1}}(p-1)\int_{M} w^p \dif \Gamma^{(L)}\la \phi \ra . \label{elt}
\end{equation}

\item For the nonlocal term, by the symmetry of $j$ in \eqref{e.j-Sym} and since $\phi(y)=0$ for $y\in M\setminus \Omega$, we have
\begin{align}
\phantom{\ \leq}\mathcal{E}^{(J)}(w;\phi^pw)&=\int_\Omega \int_\Omega 
  \abs{w(x)-w(y)}^{p-2}(w(x)-w(y))(w(x)\phi(x)^p-w(y)\phi(y)^p)\dif j(x,y)\\
  &\qquad +2\int_{x\in \Omega}\int_{y\in M\setminus  \Omega}\abs{w(x)-w(y)}^{p-2}(w(x)-w(y))w(x)\phi(x)^p\dif j(x,y) \\
  &=:I_1 +2I_2. \label{spj}
\end{align}
Next, we consider the integrands of two terms above separately. For $I_{1}$, note that for any $x,y\in\Omega$, if $w(x)\geq w(y)$ and $\phi(x)\leq\phi(y)$, then by taking \begin{equation}
	a:=\phi(y),\ b:=\phi(x),\ \epsilon:=\max(1,2c_{p})^{-1}(w(x)-w(y))w(x)^{-1}\in(0,1]
\end{equation}in Lemma \ref{l.p-dif}, we have for some universal constant $c$ that \begin{align}
	&\phantom{\ \leq}\abs{w(x)-w(y)}^{p-2}(w(x)-w(y))(w(x)\phi(x)^p-w(y)\phi(y)^p)\\
	&\geq \frac{1}{2}\abs{w(x)-w(y)}^p\max(\phi(x)^p,\phi(y)^p)-c\max(w(x)^p,w(y)^p)\abs{\phi(x)-\phi(y)}^p, 
\end{align} which is trivial when $\phi(x)>\phi(y)$. The estimate above is also true for the general settings by interchanging the roles of $x$ and $y$ (see for example \cite[p.~1286]{DCKP16}). Therefore 
\begin{align}
 I_1 &\geq \frac{1}{2}\int_{\Omega^{2}}\abs{w(x)-w(y)}^p\max(\phi(x)^p,\phi(y)^p)\dif j(x,y) \\
   &\quad -c \int_{\Omega^{2}}\max(w(x)^p,w(y)^p)\abs{\phi(x)-\phi(y)}^p\dif j(x,y)\\
   & \geq \frac{1}{2}\int_{\Omega^{2}}\abs{w(x)-w(y)}^p\phi(x)^p\dif j(x,y)  -2c \int_{\Omega^{2}} w(x)^p\abs{\phi(x)-\phi(y)}^p\dif j(x,y). \label{spj1}
\end{align}

For $I_{2}$, we have for all $x\in\Omega$ and all $y\in M\setminus \Omega$ that \begin{align}
	\abs{w(x)-w(y)}^{p-2}(w(x)-w(y))w(x)\phi(x)^p\geq -w(y)^{p-1}w(x)\phi(x)^p,
\end{align} which implies 
\begin{align}
  I_2 
  \geq -\Big(\esup_{x\in B_1}\int_{ M\setminus  \Omega}w(y)^{p-1}J(x,\dif y)\Big)\int_{B_1}w\dif m. \label{spj2-0}
\end{align}
Since
\begin{align}
 \int_{B_1}w\dif m \leq \int_{B_1}(u-\theta)_+\abs{\frac{u}{\theta}}^{p-1}\dif m \leq \theta^{1-p}\int_{B_1}|u|^{p}\dif m,
\end{align}
we deduce from \eqref{spj2-0} that
\begin{equation}\label{spj2}
  I_2\geq -\left(\frac{T_{B_1,\Omega}(w)}{\theta}\right)^{p-1}\int_{B_1}|u|^{p}\dif m.
\end{equation}
It follows from \eqref{spj}, \eqref{spj1} and \eqref{spj2} that
\begin{align}
  \mathcal{E}^{(J)}(w;\phi^pw)&\geq \frac{1}{2}\int_{\Omega^{2}}\abs{w(x)-w(y)}^p\phi(x)^p\dif j(x,y)  \\
  &-2c \int_{\Omega^{2}} w(x)^p\abs{\phi(x)-\phi(y)}^p\dif j(x,y) -2\left(\frac{T_{B_1,\Omega}(w)}{\theta}\right)^{p-1}\int_{B_1}|u|^{p}\dif m. \label{spj-1}
\end{align}
\end{enumerate}
Combining \eqref{subhar}, \eqref{elt} and \eqref{spj-1}, and using the locality $\Gamma^{(L)}\la \phi \ra(M\setminus\Omega)=0$, we have
\begin{equation}\label{cac-1}
 \int_\Omega\phi^p\dif\Gamma_{\Omega}\la w\ra\leq  (2^{\frac{2p-1}{p-1}}(p-1)+2c)\int_\Omega w^p \dif \Gamma_{\Omega}\la \phi \ra +4\left(\frac{T_{B_1,\Omega}(w)}{\theta}\right)^{p-1}\int_{B_1}|u|^{p}\dif m.
\end{equation}
Now we estimate the left hand side in \eqref{Cacci}. By \cite[Lemma 6.1]{Yan25b},
\begin{equation}\label{cac-2}
  \int_\Omega\dif\Gamma^{(L)}\la w\phi\ra \leq 2^{p-1}\Big(\int_\Omega\phi^p\dif\Gamma^{(L)}\la w\ra 
  +\int_\Omega w^p\dif\Gamma^{(L)}\la \phi \ra  \Big).
\end{equation}
On the other hand, using the inequality $|a+b|^p\leq 2^{p-1}(|a|^p+|b|^p)$, we have
\begin{align}
 & \phantom{\ \leq}\int_\Omega\dif\Gamma^{(J)}\la w\phi\ra \\
 &=\int_{\Omega^{2}} |w(x)\phi(x)-w(y)\phi(y)|^p\dif j(x,y)  +\int_{x\in\Omega}\int_{y\in M \setminus  \Omega} |w(x)\phi(x)-w(y)\phi(y)|^p\dif j(x,y) \\
   & \leq 2^{p-1}\Big(\int_{\Omega^{2}} w(x)^p\abs{\phi(x)-\phi(y)}^p\dif j(x,y)+\int_{\Omega^{2}} \phi(y)^p\abs{w(x)-w(y)}^p\dif j(x,y) \Big)\\
   &\qquad\quad +\int_{B_1}w^p\dif m\cdot \esup_{z\in B_1}\int_{y\in M\setminus  \Omega}J(z,\dif y)\\
   & \overset{\eqref{TJ}}{\leq }2^{p-1}\Big(\int_\Omega w^p\dif\Gamma^{(J)}_\Omega \la \phi\ra +\int_\Omega\phi^p\dif\Gamma^{(J)}_\Omega \la w\ra  \Big)+\sup_{x\in B_1}\frac{C_{\mathrm{TJ}}}{W(x,r)}\int_{B_1}|u|^p\dif m,
\end{align}
which combines with \eqref{cac-2} gives
\begin{equation}
\int_\Omega\dif\Gamma_{\Omega}\la w\phi \ra \leq 2^{p-1}\Big(\int_\Omega w^p\dif\Gamma_{\Omega}\la \phi \ra +\int_\Omega\phi^p\dif\Gamma_{\Omega}\la w \ra  \Big)+\sup_{x\in B_1}\frac{C_{\mathrm{TJ}}}{W(x,r)}\int_{B_1}|u|^p\dif m.\label{e.cac-3}
\end{equation}
Hence, by firstly multiplying \eqref{e.cac-3} by $a>0$ on both sides and then adding \eqref{cac-1}, we have
\begin{align}
 a\int_\Omega\dif\Gamma_{\Omega}\la w\phi \ra  &\overset{\eqref{e.cac-3}}{\leq} 2^{p-1}a\Big(\int_\Omega w^p\dif\Gamma_{\Omega}\la \phi \ra +\int_\Omega\phi^p\dif\Gamma_{\Omega}\la w \ra  \Big)
+a\sup_{x\in B_1}\frac{C_{\mathrm{TJ}}}{W(x,r)}\int_{B_1}|u|^p\dif m \\
   & \overset{\eqref{cac-1}}{\leq} \big(2^{p-1}a+(2^{\frac{2p-1}{p-1}}(p-1)+2c)\big)\int_\Omega w^p\dif\Gamma_{\Omega}\la \phi \ra +(2^{p-1}a-1)\int_\Omega\phi^p\dif\Gamma_{\Omega}\la w \ra  \\
   & \qquad\quad+ \left(a\sup_{x\in B_1}\frac{C_{\mathrm{TJ}}}{W(x,r)}+4\left(\frac{T_{B_1,\Omega}(w)}{\theta}\right)^{p-1}\right)\int_{B_1}|u|^p\dif m   \\
   &\overset{\eqref{iCS}}{\leq} \Big(\eta\big(2^{p-1}a+(2^{\frac{2p-1}{p-1}}(p-1)+2c)\big)+2^{p-1}a-1\Big)\int_\Omega\phi^p\dif\Gamma_{\Omega}\la w \ra  \\
   &\qquad\quad+ \left(C(a,\eta,p,C_{\mathrm{TJ}})\sup_{x\in \Omega}\frac{1}{W(x,r)}
   +4\left(\frac{T_{B_1,\Omega}(w)}{\theta}\right)^{p-1}\right)\int_{\Omega}|u|^p\dif m. \label{cac-3}
\end{align}
Choosing $\eta=(2^{(p-1)^{-1}(3p-2)}(p-1)+2c)^{-1}$
 and $a={2^{-p}(1+\eta)^{-1}}$ in \eqref{cac-3}, we have that \[\eta\big(2^{p-1}a+(2^{\frac{2p-1}{p-1}}(p-1)+2c)\big)+2^{p-1}a-1=0,\] and therefore
\begin{equation}
\int_\Omega\dif\Gamma_{\Omega}\la w\phi \ra \leq C\left(\sup_{x\in \Omega}\frac{1}{W(x,r)}
   +\left(\frac{T_{B_1,\Omega}(w)}{\theta}\right)^{p-1}\right)\int_{\Omega}|u|^p\dif m,
\end{equation}
thus showing \eqref{Cacci}.
\end{proof}

\subsection{Lemma of growth}
In this subsection, we prove the \emph{lemma of growth}, which establishes a relation between the measure of a super-level set of a function in a ball and the minimum value of the function near the center of the ball. The lemma of growth was originally introduced by Landis \cite{Lan98} in the study of elliptic second-order PDEs and was later further developed in \cite{GHL15, GHH18}. Before proving this lemma, we first establish a useful estimate on the nonlocal tail term.

\begin{lemma}
	Let $B$ be a metric ball with radii strictly less than $\frac{1}{2}\ol{R}$. If $v\in \sF^{\prime}$ and $\restr{v}{2B}\geq0$ $m$-a.e., then for any $x\in M$ and any $0<r<s<\infty$ with $B(x,r)\subset B$, we have \begin{equation}\label{e.cpTail}
		T_{B(x,r),B(x,s)}(v_{-})\leq T_{B,2B}(v_{-}).
	\end{equation}
\end{lemma}
\begin{proof}
	By definition, \begin{align}
		T_{B(x,r),B(x,s)}(v_{-})^{p-1}&=\esup_{z\in B(x,r)}\int_{B(x,s)^{c}}\wt{v}_{-}(y)J(z,\dif y)\leq\esup_{z\in B(x,r)}\int_{(2B)^{c}}\wt{v}_{-}(y)J(z,\dif y)\\
		&\leq \esup_{z\in B}\int_{(2B)^{c}}\wt{v}_{-}(y)J(z,\dif y) \ \text{ (as $B(x,r)\subset B$)}\\	&=T_{B,2B}(v_{-})^{p-1}.
	\end{align}
\end{proof}

\begin{lemma}[Lemma of growth] \label{l.LoG}
Assume \eqref{VD}, \eqref{SI}, \eqref{CS} and \eqref{TJ}. For any two fixed numbers $\varepsilon ,\delta $ in $(0,1)$, there exist four constants $\sigma \in (0,1),\varepsilon _{0}\in (0,\frac{1}{2})$ and $\theta,C_{L}\in(0,\infty)$ such that, for any ball $B:=B(x_{0},R)$ with radii $R\in (0,\frac{1}{2}\sigma \overline{R})$, 
and for any $u\in \mathcal{F}^{\prime}\cap L^{\infty }$ that is 
superharmonic and non-negative in $2B$, if for some $a>0$,
\begin{equation}
\frac{m (B\cap \{u<a\})}{m (B)}\leq \varepsilon _{0}(1-\varepsilon
)^{(2p-1)\theta }(1-\delta )^{C_{L}\theta }\left( 1+W(B)\left(\frac{ T_{B,2B}(u_{-}) }{a}\right)^{p-1}
\right) ^{-\theta },  \label{eq:vol_317}
\end{equation}
then
\begin{equation}
\einf_{{\delta B}}u\geq \varepsilon a.  \label{eq:a/2}
\end{equation}
\end{lemma}
\begin{proof}
For any $r\in(0,\infty)$, denote $B_{r}:=B(x_{0},r)$ so that $B_{R}=B=B(x_{0},R)$. Fix four numbers $a,b$ and $r_{1},r_{2}$ such that $0<a<b<\infty$ and $\frac{1}{2}r_2\leq r_{1}<r_{2}<R$, we set
\[
m_{1}:=\frac{m (B_{r_{1}}\cap \{u<a\})}{m \left( B_{r_{1}}\right)}\ \text{and}\ m_{2}:=\frac{m
(B_{r_{2}}\cap \{u<b\})}{m \left( B_{r_{2}}\right)}.  \]

Let $\theta=(b-a)/2$ and $w:=( b-\theta-u) _{+}$. We first note that for any $x\in B_{r_2}$,
  \begin{equation}
  \frac{1}{W(x,\frac{1}{2}(r_2-r_1))}=\frac{1}{W(x_0,r_2)}\frac{W(x_0,r_2)}{W(x,\frac{1}{2}(r_2-r_1))}\leq \frac{C}{W(B_{r_2})}
  \Big(\frac{2r_2}{r_2-r_1}\Big)^{\beta_2}. \label{e.LG2}
  \end{equation}
 Since $u\in \mathcal{F}^{\prime}\cap L^{\infty}$ is superharmonic in $2B$, we have by Lemma \ref{l.sub} that $(b-u)_{+}\in\mathcal{F}^{\prime}\cap L^{\infty}$ is subharmonic in $2B$.  By Proposition \ref{prop:cac}, there exists $\phi\in\cutoff(B_{r_1},B_{\frac{1}{2}(r_1+r_2)})$ such that
 \begin{equation}
      \int_{B_{r_2}} \dif \Gamma\la \phi w \ra \leq C\left(\sup_{x\in B_{r_2}}\frac{1}{W(x,\frac{1}{2}(r_2-r_1))}
   +\left(\frac{T_{B_{\frac{1}{2}(r_1+r_2)},B_{r_2}}(w)}{\theta}\right)^{p-1}\right)\int_{B_{r_2}} (b-u)_+^p \dif m, \label{e.LG1}
  \end{equation}
  where the constant $C$ is independent of $\phi,\theta, x_0,r_1$ and $r_2$. Consequently,
\begin{align}
&\phantom{\ \leq}\left(\int_{B_{r_2}}|\phi w|^{\nu^*}\dif m\right)^{1-\nu}
\overset{\eqref{SI}}{\leq} C_{1}^{-1}\frac{W(B_{r_2})}{m(B_{r_2})^{\nu}}\mathcal{E}(\phi w)\\
&=C_{1}^{-1}\frac{W(B_{r_2})}{m(B_{r_2})^{\nu}}\Big(\int_{B_{r_2}}\dif\Gamma\la \phi w \ra +\int_{B_{r_{2}}^{c}\times M}\abs{\phi(x)w(x)-\phi(y)w(y)}^{p}\dif j(x,y)\Big)\\
&\overset{\eqref{TJ}}{\leq} C_{1}^{-1}\frac{W(B_{r_2})}{m(B_{r_2})^{\nu}}\Big(\int_{B_{r_2}}\dif\Gamma\la \phi w \ra +\sup_{x\in B_{r_2}}\frac{1}{W(x,\frac{1}{2}(r_2-r_1))}\int_{B_{r_2}}w^{p}\dif m\Big)\\
&\leq C\frac{W(B_{r_2})}{m(B_{r_2})^{\nu}}\left(\sup_{x\in B_{r_2}}\frac{1}{W(x,\frac{1}{2}(r_2-r_1))}
   +\left(\frac{T_{B_{\frac{1}{2}(r_1+r_2)},B_{r_2}}(w)}{\theta}\right)^{p-1} \right)\int_{B_{r_2}} (b-u)_+^p \dif m\\
   &\overset{\eqref{e.LG2}}{\leq}\frac{C}{m(B_{r_2})^{\nu}}
   \Big(\frac{2r_2}{r_2-r_1}\Big)^{\beta_2}\left(1+W(B_{r_2})\left(\frac{T_{B_{\frac{1}{2}(r_1+r_2)},B_{r_2}}(w)}{\theta}\right)^{p-1}\right)
    \int_{B_{r_2}} (b-u)_+^p \dif m,\label{e.LG14}
\end{align}
where in the first equality we use the locality of $\Gamma^{(L)}$, and in the third inequality we use \eqref{e.LG1} and the fact that $0\leq w\leq (b-u)_{+}$. Hence
{\small
\begin{align}
&\phantom{\ \leq}\int_{B_{r_1}}(b-\theta-u)_+^p\dif m\leq\int_{B_{r_2}}|\phi\cdot (b-\theta-u)_+|^p\one_{\{u<b-\theta\}}\dif m\\
&\leq\left(\int_{B_{r_2}}|\phi w|^{\nu^*}\dif m\right)^{1-\nu}m(B_{r_2}\cap\{u<b-\theta\})^{\nu} \text{ (H\"older's inequality)}\\
&\overset{\eqref{e.LG14}}{\leq} \frac{C}{m(B_{r_2})^{\nu}}\Big(\frac{2r_2}{r_2-r_1}\Big)^{\beta_2}
m(B_{r_2}\cap\{u<b\})^{\nu}  \left(1+W(B_{r_2})\left(\frac{T_{B_{\frac{1}{2}(r_1+r_2)},B_{r_2}}(w)^{p-1}}{\theta^{p-1}}\right)\right)
    \int_{B_{r_2}} (b-u)_+^p \dif m  \\
&\leq\frac{2^{\beta_2}C'}{m(B_{r_2})^{\nu}}\Big(\frac{r_2}{r_2-r_1}\Big)^{\beta_2}\left(1+W(B_{r_2})\left(\frac{T_{B_{\frac{1}{2}(r_1+r_2)},B_{r_2}}(w)}{\theta}\right)^{p-1}\right)b^pm(B_{r_2}\cap\{u<b\})^{1+\nu}.\label{e.LG15}
\end{align}
}
On the other hand,
\begin{equation}
\int_{B_{r_1}}(b-\theta-u)_+^p\dif m\geq \int_{B_{r_1}\cap\{u<a\}}(b-\theta-u)_+^p\dif m\ge 2^{-p}(b-a)^p m(B_{r_1}\cap\{u<a\}),
\end{equation}
which combines with \eqref{e.LG15} implies
\begin{align}
{m}_{1}&\le C\frac{m(B_{r_2})}{m(B_{r_1})}\left(\frac{r_2}{r_2-r_1}\right)^{\beta_2}\left(\frac{b}{b-a}\right)^p
\left(1+W(B_{r_2})\left(\frac{T_{B_{\frac{1}{2}(r_1+r_2)},B_{r_2}}(w)}{\theta}\right)^{p-1}\right){m}_{2}^{1+\nu}\\
&\overset{\eqref{VD}}{\leq} C\left(\frac{r_2}{r_2-r_1}\right)^{\beta_2}\left(\frac{b}{b-a}\right)^p
\left(1+W(B_{r_2})\left(\frac{T_{B_{\frac{1}{2}(r_1+r_2)},B_{r_2}}(w)}{\theta}\right)^{p-1}\right){m}_{2}^{1+\nu}. \label{LG-1}
\end{align}
Note that \begin{equation}
	w^{p-1}=(b-\theta-u)_+^{p-1}\leq (b-\theta+u_-)^{p-1}\leq 2^{p-1}((b-\theta)^{p-1}+u_-^{p-1}).\label{e.LG16}
\end{equation} As $\theta=(b-a)/2$, we have
\begin{align}
  T_{B_{\frac{1}{2}(r_1+r_2)},B_{r_2}}(w)^{p-1}  &=\esup_{x\in B_{\frac{1}{2}(r_1+r_2)}}\int_{B_{r_2}^c}w^{p-1} J(x,\dif y) \\
   & \overset{\eqref{e.LG16}}{\leq} 2^{p-1}\esup_{x\in B_{\frac{1}{2}(r_1+r_2)}}\int_{B_{r_2}^c}((b-\theta)^{p-1}+u_-^{p-1}) J(x,\dif y)\\
   & \overset{\eqref{TJ}}{\leq} (b-\theta)^{p-1}\sup_{x\in B_{\frac{1}{2}(r_1+r_2)}}\frac{2^{p-1}C_{\mathrm{TJ}} }{W(x,\frac{1}{2}(r_2-r_1))}+2^{p-1}T_{B_{\frac{1}{2}(r_1+r_2)},B_{r_2}}(u_-)^{p-1} \\
   & \overset{\eqref{e.LG2},\eqref{e.cpTail}}{\leq} 2^{\beta_2}C\frac{(b-a)^{p-1}}{W(B_{r_2})}\Big(\frac{r_2}{r_2-r_1}\Big)^{\beta_2}+2^{p-1}T_{B,2B}(u_-)^{p-1}.
\end{align}
It follows from \eqref{LG-1} that
\begin{equation}
{m}_{1}\leq C\left(\frac{r_2}{r_2-r_1}\right)^{2\beta_2}\left(\frac{b}{b-a}\right)^p
\left(1+\frac{W(x_0,r_2)T_{B,2B}(u_-)^{p-1}}{(b-a)^{p-1}}\right){m}_{2}^{1+\nu}, \label{eq:tildem1m2}
\end{equation}
where $C>0$ depends only on the constants from the
hypotheses (but is independent of $a,b,x_0,r_{1},r_{2}$ and
the functions $u$). We will apply \eqref{eq:tildem1m2} to show this lemma.

In fact, let $\delta ,\varepsilon $ be any two fixed numbers in $(0,1)$.
Consider the following sequences%
\begin{equation}
R_{k}:=\left( \delta +2^{-k}(1-\delta )\right) R\quad \text{ and }\quad
a_{k}:=\left( \varepsilon +2^{-k}(1-\varepsilon )\right) a\text{ \ for }%
k\in\{0\}\cup\bN.
\end{equation}%
Clearly, $R_{0}=R$, $a_{0}=a$, $R_{k}\downarrow \delta R$, and $a_{k}\downarrow
\varepsilon a$ as $k\rightarrow \infty $, and $\frac{1}{2}R_{k-1}<R_{k}<R_{k-1}$ for any $k\in\bN$.
Set $m_{k}:={m (B_{R_{k}}\cap \{u<a_{k}\})}/{m (B_{R_{k}})}$. 

Applying \eqref{eq:tildem1m2} with $a=a_{k}$, $b=a_{k-1}$, $r_{1}=R_{k}$ and $r_{2}=R_{k-1}$, we know that for all $k\in \bN$,
\begin{equation}
m_{k}\leq CA_k\left( \frac{a_{k-1}}{a_{k-1}-a_{k}}\right) ^{p}\left( \frac{%
R_{k-1}}{R_{k-1}-R_{k}}\right) ^{2\beta _{2}}m_{k-1}^{1+\nu }, 
\label{72}
\end{equation}
where $A_k$ is given by $A_k:=1+(a_{k-1}-a_{k})^{1-p}W(x_0,R_{k-1})T_{B,2B}(u_-)$.

Define $A:=1+a^{1-p}W(B) T_{B,2B}(u_{-})^{p-1}$. Since $a_{k-1}-a_k=2^{-k}(1-\varepsilon)a$ and $W(x_0,R_{k-1})\leq W(x_0,R)=W(B)$, it follows that
\begin{equation}
A_{k}\leq \left(\frac{2^{k}}{1-\varepsilon}\right)^{p-1}\left(1+\frac{W(B) T_{B,2B}(u_{-})^{p-1} }{ a^{p-1}}\right)=\left(\frac{2^{k}}{1-\varepsilon}\right)^{p-1}A
\text{ \ for any }k\in \bN.  \label{71}
\end{equation}
Note that
\begin{equation}
\frac{a_{k-1}}{a_{k-1}-a_{k}}=\frac{\varepsilon +2^{-\left( k-1\right)
}(1-\varepsilon )}{\left( 2^{-\left( k-1\right) }-2^{-k}\right)
(1-\varepsilon )}\leq \frac{2^{k}}{1-\varepsilon }\text{ \ and \ }\frac{%
R_{k-1}}{R_{k-1}-R_{k}}\leq \frac{2^{k}}{1-\delta }
\end{equation}%
for all $k\in\bN$, by \eqref{72}, \eqref{71},
\begin{equation}
m_{k}\leq CA\left( \frac{2^{k}}{1-\varepsilon }\right) ^{2p-1}\left( \frac{2^{k}%
}{1-\delta }\right) ^{2\beta _{2}}m_{k-1}^{1+\nu }\eqqcolon %
DA\cdot 2^{\lambda k}\cdot m_{k-1}^{q},  \label{73}
\end{equation}%
where the constants $D$, $\lambda $, $q$ are respectively given by $D:=C(1-\varepsilon )^{-2p+1}(1-\delta )^{-2\beta _{2}}$, $\lambda :=2p+2\beta _{2}-1$ and $q:=1+\nu$. It follows from Proposition \ref{prop:A3} that $m_{k}\leq \Big((DA)^{\frac{1}{q-1}}\cdot 2^{\frac{\lambda q}{\left(
q-1\right) ^{2}}}\cdot m_{0}\Big)^{q^{k}}$, from which and the fact that $q>1$, we conclude that 
\begin{align}
&\phantom{\Longleftrightarrow}\lim_{k\rightarrow \infty }m_{k}=0\\
	&{\Longleftarrow\ }2^{\frac{\lambda q}{\left( q-1\right) ^{2}}}\cdot (DA)^{\frac{1}{q-1}}\cdot m_{0}\leq \frac{1}{2}\Longleftrightarrow m_{0}\leq 2^{-\frac{\lambda q}{\left( q-1\right) ^{2}}-1}\cdot (DA)^{-\frac{1}{q-1}}\label{e.75-}\\
	&\Longleftrightarrow \frac{m (B\cap \{u<a\})}{m (B)} \leq \varepsilon _{0}(1-\varepsilon )^{(2p-1)\theta }(1-\delta
)^{C_{L}\theta }\left( 1+\frac{W(B) T_{B,2B}(u_{-})^{p-1} }{ a^{p-1}}\right) ^{-\theta},\label{75}
\end{align}
where $\varepsilon _{0},\theta ,C_{L}$ are universal constants given by
\begin{equation}
\varepsilon _{0}:=2^{-\lambda q/\left( q-1\right) ^{2}-1}C^{-1/(q-1)}<1/2%
\text{, \ }\theta :=1/\nu \text{, \ and \ }C_{L}:=2\beta _{2},\label{eq:epsdef}
\end{equation}
since 
\begin{equation}
2^{-\frac{\lambda q}{\left( q-1\right) ^{2}}-1}D^{-\frac{1}{q-1}} =2^{-\frac{\lambda q}{\left( q-1\right) ^{2}}-1}\left( C(1-\varepsilon
)^{-2p+1}(1-\delta )^{-2\beta _{2}}\right) ^{-\frac{1}{q-1}} =\varepsilon _{0}\left( (1-\varepsilon )^{2p-1}(1-\delta )^{2\beta
_{2}}\right) ^{1/\nu}.
\end{equation}
Note that the constants $\varepsilon _{0},\theta ,C_{L}$ are all universal,
all of which are independent of the numbers $\varepsilon ,\delta ,$ the ball
$B$ and the functions $u$. The inequality \eqref{75} is just the hypothesis \eqref{eq:vol_317}. With a choice of $\varepsilon _{0},\theta ,C_{L}$ in \eqref{eq:epsdef}, the assumption \eqref{75} is satisfied. Hence $\lim_{k\rightarrow \infty }m_{k}=0$, then 
\begin{equation}
\frac{m (\delta B\cap \{u<\varepsilon a\})}{m (\delta B)}=0,
\end{equation}%
thus showing that \eqref{eq:a/2} is true. 
\end{proof}
\subsection{Logarithmic Lemma} \label{s.log}
In this subsection, we first establish the \emph{logarithmic lemma}. We then combine the Poincar\'e inequality \eqref{PI} with the cutoff Sobolev inequality \eqref{CS} to show that the logarithm of a superharmonic function has \emph{bounded mean oscillation (BMO)}. Finally, by applying the John--Nirenberg inequality, we derive the weak elliptic Harnack inequality stated in Theorem \ref{t.main}-\eqref{e.main-wEH}.

\begin{lemma}Let $U\subset\Omega$ be two open subsets of $M$ and 
	Let $u\in\sF^{\prime}\cap L^{\infty}$ be non-negative in $\Omega$ and $\phi\in\sF^{\prime}\cap L^{\infty}$ be such that $\restr{\wt{\phi}}{M\setminus U}=0$ $\sE$-q.e.. Let $\lambda\in (0,\infty)$ and set $x:=u+\lambda$. Then we have $\abs{\phi}^{p}\abs{u_{\lambda}}^{-(p-1)}\in \sF(U)$ and 
	\begin{align}
		\phantom{\ \leq}\sE^{(J)}\big(u;\abs{\phi}^{p}\abs{u_{\lambda}}^{-(p-1)}\big)&\leq -C_{1}\int_{U^{2}}\abs{\log\frac{u_{\lambda}(x)}{u_{\lambda}(y)}}^{p}\big(\abs{\phi(x)}\wedge \abs{\phi(y)}\big)^{p}\dif j(x,y)
		\\
		&\qquad +C_{2}\int_{U}\dif\Gamma^{(J)}\la \phi\ra +C_3
\norm{\phi}_{L^{p}(U,m)}^{p}\left(\frac{T_{U,\Omega}(u_{-})}{\lambda}\right)^{p-1}\label{e.cros0-}
	\end{align}
	for some positive constants $C_1$, $C_2$ and $C_3$ depending only on $p$.
\end{lemma}
\begin{proof}
	That $\abs{\phi}^{p}\abs{u_{\lambda}}^{-(p-1)}\in \sF(U)$ follows from Proposition \ref{p.mar}. Note that 
	\begin{align}
		\sE^{(J)}\big(u;\abs{\phi}^{p}\abs{u_{\lambda}}^{-(p-1)}\big)&=\int_{U^{2}}\abs{u(x)-u(y)}^{p-2}(u(x)-u(y))\bigg(\frac{\abs{\phi(x)}^{p}}{\abs{u_{\lambda}(x)}^{p-1}}-\frac{\abs{\phi(y)}^{p}}{\abs{u_{\lambda}(y)}^{p-1}}\bigg)\dif j(x,y)\\
		&\qquad+2\int_{U\times (M\setminus U)}\abs{u(x)-u(y)}^{p-2}(u(x)-u(y))\frac{\abs{\phi(x)}^{p}}{\abs{u_{\lambda}(x)}^{p-1}}\dif j(x,y)\\
		&=:I_1+2I_2.\label{e.cors0}
	\end{align}
	\begin{itemize}
	\item We first estimate the integrand in $I_1$. Suppose $x,y\in U$, then $u(x)\geq0$ and $u(y)\geq0$.
	
	Assume that $u(x)>u(y)\geq0$, then $0<u(x)-u(y)\leq u(x)<u(x)+\lambda$ and thus \begin{equation}
		\frac{u(x)-u(y)}{u(x)+\lambda}\in(0,1).\label{e.cors1}
	\end{equation} 
	Let $\delta\in(0,1)$ to be determined later. Applying Lemma \ref{l.p-dif} with $a=\phi(x)$, $b=\phi(y)$ and \begin{equation}
		\epsilon:=\delta\cdot \frac{u(x)-u(y)}{u(x)+\lambda}\in(0,1),
	\end{equation}
	we have
\begin{align}
		\abs{\phi(x)}^{p}\leq \bigg(1+c_{p}\delta\frac{u(x)-u(y)}{u(x)+\lambda}\bigg)\bigg(\abs{\phi(y)}^{p}+\delta^{1-p} \bigg(\frac{u(x)-u(y)}{u(x)+\lambda}\bigg)^{1-p}\abs{\phi(x)-\phi(y)}^{p}\bigg)\label{e.cors1+}
	\end{align}
	and therefore 
	\begin{align}
		&\phantom{\ \leq}\frac{\abs{\phi(x)}^{p}}{\abs{u(x)+\lambda}^{p-1}}-\frac{\abs{\phi(y)}^{p}}{\abs{u(y)+\lambda}^{p-1}}\\
		&\overset{\eqref{e.cors1+}}{\leq }\bigg(\frac{\abs{\phi(y)}^{p}}{\abs{u(x)+\lambda}^{p-1}}+c_{p}\delta\frac{u(x)-u(y)}{u(x)+\lambda}\frac{\abs{\phi(y)}^{p}}{\abs{u(x)+\lambda}^{p-1}}-\frac{\abs{\phi(y)}^{p}}{\abs{u(y)+\lambda}^{p-1}}\bigg)\\
		&\qquad +\bigg(1+c_{p}\delta\frac{u(x)-u(y)}{u(x)+\lambda}\bigg)\delta^{1-p} \big(u(x)-u(y)\big)^{1-p}\abs{\phi(x)-\phi(y)}^{p}\\
		&\overset{\eqref{e.cors1}}{\leq} \frac{\abs{\phi(y)}^{p}}{\abs{u(x)+\lambda}^{p-1}}\bigg(1+c_{p}\delta\frac{u(x)-u(y)}{u(x)+\lambda}-\bigg(\frac{u(y)+\lambda}{u(x)+\lambda}\bigg)^{1-p}\bigg)\\
		&\qquad +(1+c_{p})\delta^{1-p} \big(u(x)-u(y)\big)^{1-p}\abs{\phi(x)-\phi(y)}^{p}.\label{e.cros2}
		\end{align}
		Take $\delta:=(p-1)c_{p}^{-1}2^{-(p+1)}$, then \eqref{e.cros2} gives
\begin{align}
		&\phantom{\ \leq}\abs{u(x)-u(y)}^{p-2}(u(x)-u(y))\bigg(\frac{\abs{\phi(x)}^{p}}{\abs{u(x)+\lambda}^{p-1}}-\frac{\abs{\phi(y)}^{p}}{\abs{u(y)+\lambda}^{p-1}}\bigg)\\
		&\overset{\eqref{e.cros2}}{\leq} \bigg(1-\frac{u_{\lambda}(y)}{u_{\lambda}(x)}\bigg)^{p}\bigg(\frac{p-1}{2^{p+1}}
+\frac{1-\big(u_{\lambda}(y)u_{\lambda}(x)^{-1}\big)^{1-p}}{1-u_{\lambda}(y)u_{\lambda}(x)^{-1}}\bigg) \abs{\phi(y)}^{p}+(1+c_{p})\delta^{1-p}\abs{\phi(x)-\phi(y)}^{p}.
\end{align}
An application of Lemma \ref{l.LL} with $s=u_{\lambda}(y)u_{\lambda}(x)^{-1}\in(0,1)$ gives
		\begin{align}
			&\phantom{\ \leq}\abs{u(x)-u(y)}^{p-2}(u(x)-u(y))\bigg(\frac{\abs{\phi(x)}^{p}}{\abs{u(x)+\lambda}^{p-1}}-\frac{\abs{\phi(y)}^{p}}{\abs{u(y)+\lambda}^{p-1}}\bigg)\\
			&\leq-C_{1}\abs{\log\frac{u_{\lambda}(x)}{u_{\lambda}(y)}}^{p}(\abs{\phi(x)}^{p}\wedge \abs{\phi(y)}^{p})+C\abs{\phi(x)-\phi(y)}^{p},\label{e.cros3}
		\end{align}
where $C_1=\frac{(p-1)^2}{p2^{p+1}}$ and $C=(1+c_{p})\delta^{1-p}$.
	If $u(x)=u(y)$, then \eqref{e.cros3} trivially holds. If $u(x)<u(y)$, we can exchange the roles of $x$ and $y$ in \eqref{e.cros3}. Therefore \eqref{e.cros3} holds for all $x,y\in U$ and \begin{align}\label{e.cors3+}
		I_1&\leq-C_{1}\int_{U^{2}}\abs{\log\frac{u_{\lambda}(x)}{u_{\lambda}(y)}}^{p}(\abs{\phi(x)}\wedge \abs{\phi(y)})^{p}\dif j(x,y)+C\int_{U^{2}}\abs{\phi(x)-\phi(y)}^{p}\dif j(x,y).
	\end{align}
	\item We then estimate the integrand in $I_2$. Since $\restr{u}{\Omega}\geq0$, we have
	\begin{align}
		\abs{u(x)-u(y)}^{p-2}\frac{u(x)-u(y)}{\abs{u_{\lambda}(x)}^{p-1}}\leq \frac{(u(x)-u(y))_{+}^{p-1}}{(u(x)+\lambda)^{p-1}}\leq1,\ \forall x\in U,\ \forall y\in\Omega\setminus U.\label{e.cros4}
	\end{align}
If $x\in U$ and $y\in M\setminus \Omega$, then \begin{align}
		\abs{u(x)-u(y)}^{p-2}\frac{u(x)-u(y)}{\abs{u_{\lambda}(x)}^{p-1}}&\leq \frac{(u(x)-u(y))_{+}^{p-1}}{(u(x)+\lambda)^{p-1}}\leq 2^{p-1}\frac{u(x)^{p-1}+u_{-}(y)^{p-1}}{(u(x)+\lambda)^{p-1}}\\
		&\leq 2^{p-1}(1+\lambda^{1-p}u_{-}(y)^{p-1}). \label{e.cros5}
		\end{align}
		Therefore 
{\small
\begin{align}
			I_2&=\bigg(\int_{U\times (\Omega\setminus U)}+\int_{U\times (M\setminus \Omega)}\bigg)\abs{u(x)-u(y)}^{p-2}(u(x)-u(y))\frac{\abs{\phi(x)}^{p}}{\abs{u_{\lambda}(x)}^{p-1}}\dif j(x,y)\\
			&\overset{\eqref{e.cros4},\eqref{e.cros5}}{\leq}C^{\prime}\int_{U\times (M\setminus U)}\abs{\phi(x)}^{p}\dif j(x,y)+C_3\lambda^{1-p}\int_{U\times (M\setminus\Omega)}\abs{\phi(x)}^{p}\abs{u_{-}(y)}^{p-1}\dif j(x,y)\\
			&\leq C^{\prime}\int_{U\times (M\setminus U)}\abs{\phi(x)-\phi(y)}^{p}\dif j(x,y)+C_3\lambda^{1-p}\int_{U}\abs{\phi(x)}^{p}\bigg(\esup_{z\in U}\int_{M\setminus \Omega}\abs{u_{-}(y)}^{p-1}J(z,\dif y)\bigg)\dif m(x)\\
			&=C^{\prime}\int_{U\times (M\setminus U)}\abs{\phi(x)-\phi(y)}^{p}\dif j(x,y)+C_3\norm{\phi}_{L^{p}(U,m)}^{p}\left(\frac{T_{U,\Omega}(u_{-})}{\lambda}\right)^{p-1}.\label{e.cors6}
		\end{align}
}
		\end{itemize}
Combining \eqref{e.cors0}, \eqref{e.cors3+} and \eqref{e.cors6}, we obtain \eqref{e.cros0-} with $C_2=C+C^{\prime}$.
\end{proof}

\begin{definition}
	We say that the \eqref{cap<} holds if there exists a constant $C\in(0,\infty)$ such that for any $B:=B(x,r)$ with $r\in(0,2\ol{R}/3)$,
	\begin{equation}\label{cap<}\tag{$\operatorname{cap}_{\leq}$}
		\capacity(\frac{1}{2}B,B)\leq C\frac{m(B)}{W(B)},
	\end{equation}
	 where the $\capacity(U,V)$ for any two open sets $U\subset V$ is defined by \begin{equation}
	 	\capacity(U,V):=\inf\Sett{\sE(\phi)}{\phi\in \cutoff(U,V)}.
	 \end{equation}
	 \end{definition}
The following implication is standard.
\begin{lemma}\label{l.CS=>cap}
	Assume \eqref{VD}, \eqref{CS} and \eqref{TJ} hold, then for any two concentric balls $B_0:=B(x_0,R),\ B_1:=B(x_0,R+r)$ with
$0<R<R+2r<\ol{R}$, we have
\begin{equation}
\capacity(B_0,B_1)\leq \sup_{x\in B(x_0,R+2r)}\frac{C}{W(x,r)}m(B(x_0,R+2r)). \label{e.cap2}
\end{equation} 
In particular, \eqref{cap<} holds.
\end{lemma}
\begin{proof}
Let $\Omega:=B(x_0,R+2r)$. Let $\phi\in \cutoff(B_0,B_1)$ be the cutoff function in \eqref{CS}. Since \eqref{CS} holds for all $u\in\sF^{\prime}\cap L^{\infty}(M,m)$, we may simply take $u=\one_{M}$ in \eqref{CS} and obtain 
\begin{align}
		\sE(\phi)&=\int_{\Omega}\dif\Gamma_{\Omega}\la\phi\ra +2\int_{\Omega \times \Omega^{c}}\abs{\phi(x)-\phi(y)}^{p}\dif j(x,y)\\
		&\overset{\eqref{CS}}{\leq}\sup_{z\in \Omega}\frac{C}{W(z,r)}m(\Omega)+2\int_{B_1\times \Omega^{c}}\abs{\phi(x)}^{p}\dif j(x,y)\overset{\eqref{TJ}}{\leq}\sup_{z\in \Omega}\frac{C_{1}}{W(z,r)}m(\Omega),
	\end{align}
thus showing \eqref{e.cap2}. Now for any $B=B(x,s)$, taking $R=r=s/2$ in \eqref{e.cap2}, we have
\begin{equation}
 \capacity(\frac{1}{2}B,B)\leq C\frac{m((3/2)B)}{W(B)}\sup_{z\in B(x,s)}\frac{W(x,s)}{W(z,s/2)}\overset{\eqref{VD}, \eqref{eq:vol_0}}{\leq}C_{2}\frac{m(B)}{W(B)},
\end{equation}
which gives \eqref{cap<}.
\end{proof}

\begin{definition}[$\text{BMO}$ function]
\label{df:BMO} Let $\Omega $ be an open subset in $M$. We define 
\begin{equation}
\norm{u}_{\mathrm{BMO}(\Omega )}:=\sup\Sett{\fint_{B}|u-u_{B}|\dif m}{B\text{ is a ball contained in $\Omega $}},\ \forall u\in L^{1}_{\loc}(\Omega),
\end{equation}
where $L^{1}_{\loc}(\Omega):=\Sett{f}{f\in L^1(K) \text{ for all }K\Subset \Omega}$, and define the space of \emph{bounded mean oscillation}  \begin{equation}
	\mathrm{BMO}(\Omega):=\Sett{u\in L^{1}_{\loc}(\Omega)}{\norm{u}_{\mathrm{BMO}(\Omega )}<\infty }.
\end{equation}
\end{definition}

We show the following \emph{crossover lemma}.
\begin{lemma}[Crossover lemma]\label{l.cros}
	Assume \eqref{VD}, \eqref{cap<}	and \eqref{PI}. Let $(x_0,R)\in M\times (0,\overline{R})$, $u\in \mathcal{F^{\prime }}\cap L^{\infty }$ be superharmonic and non-negative in $B_{R}:=B(x_0,R)$. If \begin{equation}\label{e.COlam}
		\lambda \geq W(B_R)^{\frac{1}{p-1}}T_{\frac{1}{2}B_{R},B_{R}}(u_{-}),
	\end{equation}
	then \begin{equation}
		\log(u+\lambda)\in \BMO\bigg(B\bigg(x_0,\frac{R}{2(4\kappa+1)}\bigg)\bigg),\label{e.COBMO}
	\end{equation}
	and \begin{equation}
		\bigg(\fint_{B_r}(u+\lambda)^{q}\dif m\bigg)^{1/q}\bigg(\fint_{B_r}(u+\lambda)^{-q}\dif m\bigg)^{1/q}\leq C\label{e.CO}
	\end{equation}
	for any $B_r=B(x_0,r)$ with $r<\frac{R}{24(4\kappa+1)}$, where $C>0$ and $q\in(0,1)$ are two constants independent of $x_0$, $R$, $r$, $u$, $\lambda$, and the constant $\kappa \geq 1$ comes from condition \eqref{PI}. 
\end{lemma}
\begin{proof}
Let $B=B(z,r)$ be an arbitrary ball contained in $B\big(x_0,\frac{R}{2(4\kappa+1)}\big)$. By \cite[Remark 3.16]{BB11}, we may assume without loss of generality that $r\leq 2\cdot \frac{R}{2(4\kappa +1)}=\frac{ R}{4\kappa +1}<R$. Then $2\kappa B\subset \frac{1}{2}B_{R}=B(x_{0},\frac{1}{2}R)$.
By the triangle inequality, for any point $x\in 2\kappa B=B(z,2\kappa r)$, we have $d(x,x_{0})\leq d(x,z)+d(z,x_{0})<R/2$. Also, for any $x\in 4\kappa B=B(z,4\kappa r)$, we have $d(x,x_{0})\leq 4\kappa \cdot \frac{1}{(4\kappa +1)}R+\frac{1}{2(4\kappa +1)}R<R$,
which implies $4\kappa B\subset B_{R}$.

We represent $u$ by its $\sE$-quasi-continuous version in the following proof. Write $u_\lambda=u+\lambda$. Applying \eqref{PI} to $\log(\frac{u_{+}+\lambda}{\lambda})\in\sF$ and use the fact that $u=u_{+}$ on $B(x_0,R)$ we obtain 
\begin{align}
	&\phantom{\ \ \leq}\int_{B}|\log(u+\lambda)-(\log(u+\lambda))_{B}|^{p}\dif m=\int_{B}|\log(u_{+}+\lambda)-(\log(u_{+}+\lambda))_{B}|^{p}\dif m\\
	&\overset{\eqref{PI}}{\leq} C_{\mathrm{PI}}W(B)\bigg(\int_{\kappa B}\dif\Gamma^{(L)}\la \log(u_{+}+\lambda)\ra +\int_{(\kappa B)^{2}}\abs{\log(u_{+}(x)+\lambda)-\log(u_{+}(y)+\lambda)}^{p}\dif j(x,y)\bigg)\\
	&=C_{\mathrm{PI}}W(B)\bigg(\int_{\kappa B}\dif\Gamma^{(L)}\la \log u_{\lambda}\ra +\int_{(\kappa B)^{2}}\abs{\log\frac{u_{\lambda}(x)}{u_{\lambda}(y)}}^{p}\dif j(x,y)\bigg)\label{e.CO-2}
\end{align}
By \eqref{cap<}, there exists $\phi\in\cutoff(\kappa B,2\kappa B)$ such that \begin{equation}
	\sE(\phi)\leq C\frac{m(\kappa B)}{W(\kappa B)}\leq C\frac{m(B)}{W(B)}.\label{e.CO-1}
\end{equation}
By Theorem \ref{t.sasEM},
\begin{align}
	\int_{M}{\phi}^{p}\dif\Gamma^{(L)}\la \log u_{\lambda}\ra &\overset{\ref{lb.M-chain}}{=}\int_M \phi(x)^{p}u_{\lambda}(x)^{-p}\dif\Gamma^{(L)}\la u\ra \overset{\ref{lb.M-chain}}{=}-\frac{1}{p-1}\int_{M}\phi^{p}\dif\Gamma^{(L)}\la u;u_{\lambda}^{1-p}\ra \\
	&=-\frac{1}{p-1}\int_{M}\dif\Gamma^{(L)}\la u;\phi^{p}\cdot u_{\lambda}^{1-p}\ra +\frac{p}{p-1}\int_{M}\phi^{p-1}u_{\lambda}^{1-p}\dif\Gamma^{(L)}\la u;\phi\ra , \label{e.cros6}
\end{align}
where in the last equality we have used Theorem \ref{t.sasEM}-\ref{lb.M-leibn}, \ref{lb.M-chain}. Now, we estimate the last term in the right-hand side of \eqref{e.cros6}. Indeed, using \cite[Proposition 4.8]{KS25} and Young's inequality, 
\begin{align}
\frac{p}{p-1}\int_{M}\phi^{p-1}u_{\lambda}^{1-p}\dif\Gamma^{(L)}\la u;\phi\ra &\leq \frac{p}{p-1}\bigg(\int_{M}\phi^{p} u_{\lambda}^{-p}\dif\Gamma^{(L)}\la u\ra \bigg)^{1-\frac{1}{p}}\bigg(\int_M \dif\Gamma^{(L)}\la \phi \ra \bigg)^{\frac{1}{p}} \\
	&\leq\frac{1}{2}\int_{M}\phi^{p} u_{\lambda}^{-p}\dif\Gamma^{(L)}\la u\ra +\frac{2^{p-1}}{p-1}\int_M \dif\Gamma^{(L)}\la \phi \ra \\
	&=\frac{1}{2}\int_{M}\phi^{p} \dif\Gamma^{(L)}\la \log u_{\lambda}\ra +\frac{2^{p-1}}{p-1}\int_M \dif\Gamma^{(L)}\la \phi \ra ,
\end{align}
combining which with \eqref{e.cros6},
\begin{equation}
	\int_{M}\phi^{p} \dif\Gamma^{(L)}\la \log u_{\lambda}\ra \leq-\frac{2}{p-1}\sE^{(L)}(u;\abs{\phi}^{p}\abs{u_{\lambda}}^{1-p})+\frac{2^{p}}{p-1}\int_M \dif\Gamma^{(L)}\la \phi \ra .\label{e.CO1}
\end{equation}
Since $\restr{\phi}{\kappa B}=1$, we have
\begin{align}
	\int_{\kappa B}\dif\Gamma^{(L)}\la \log u_{\lambda}\ra &\leq\int_{M}\phi^{p} \dif\Gamma^{(L)}\la \log u_{\lambda}\ra \\
	&\overset{\eqref{e.CO1}}{\leq}-\frac{2}{p-1}\sE^{(L)}(u;\abs{\phi}^{p}\abs{u_{\lambda}}^{1-p})+\frac{2^{p}}{p-1}\int_M \dif\Gamma^{(L)}\la \phi \ra .\label{e.CO2}
	\end{align}
Noting that $u$ is superharmoic in $4\kappa B $ as $4\kappa B\subset B_{R}$, applying Lemma \ref{l.cros} with $U=2\kappa B$ and $\Omega=4\kappa B$ in the second inequality below, we have 
\begin{align}
	&\phantom{\ \leq}-\sE^{(L)}(u;\abs{\phi}^{p}\abs{u_{\lambda}}^{1-p})=-\sE(u;\abs{\phi}^{p}\abs{u_{\lambda}}^{1-p})+\sE^{(J)}(u;\abs{\phi}^{p}\abs{u_{\lambda}}^{1-p})\leq\sE^{(J)}(u;\abs{\phi}^{p}\abs{u_{\lambda}}^{1-p})\\
	&\overset{\eqref{e.cros0-}}{\leq }-C_{1}\int_{(2\kappa B) \times (2\kappa B)}\abs{\log\frac{u_{\lambda}(x)}{u_{\lambda}(y)}}^{p}(\abs{\phi(x)}^{p}\wedge \abs{\phi(y)}^{p})\dif j(x,y)\\
	&\qquad +C_{2}\int_{2\kappa B}\dif\Gamma^{(J)}\la \phi\ra 
+C_3\norm{\phi}_{L^{p}(2\kappa B,m)}^{p}\left(\frac{T_{2\kappa B,4\kappa B}(u_{-})}{\lambda}\right)^{p-1}\\
	&{\leq}-C_{1}\int_{(\kappa B) \times (\kappa B)}\abs{\log\frac{u_{\lambda}(x)}{u_{\lambda}(y)}}^{p}\dif j(x,y) +C_{2}\sE^{(J)}(\phi)+C_4 m(B)\left(\frac{T_{\frac{1}{2}B_{R},B_{R}}(u_{-})}{\lambda}\right)^{p-1},
\end{align}
where in the last inequality we have used $2\kappa B\subset B(x_{0},\frac{1}{2}R)$, \eqref{VD} and \eqref{e.cpTail}.
Combining this, \eqref{e.CO-1} and \eqref{e.CO2}, we obtain that there exists some $C>0$ depending only on $p$ such that, if $\lambda$ satisfies \eqref{e.COlam}, then \begin{align}
	&\phantom{\ \leq}\int_{\kappa B}\dif\Gamma^{(L)}\la \log u_{\lambda}\ra +\int_{(\kappa B)\times(\kappa B)}\abs{\log\frac{u_{\lambda}(x)}{u_{\lambda}(y)}}^{p}\dif j(x,y)\\
	&\leq C\bigg(\frac{m(B)}{W(B)}+m(B)\left(\frac{T_{\frac{1}{2}B_{R},B_{R}}(u_{-})}{\lambda}\right)^{p-1}
\bigg)\overset{\eqref{e.COlam}}{\leq} C\bigg(\frac{m(B)}{W(B)}+\frac{m(B)}{W(B_R)}\bigg)\leq C^{\prime}\frac{m(B)}{W(B)}.\label{e.CO4}
\end{align}

By H\"older's inequality, \begin{align}
	\fint_B\abs{\log u_{\lambda}-(\log u_{\lambda})_B}\dif m&\leq \fint_B\abs{\log u_{\lambda}-(\log u_{\lambda})_B}^{p}\dif m\\
	&\overset{\eqref{e.CO-2},\eqref{e.CO4}}{\leq} \frac{1}{m(B)}\cdot C_{\mathrm{PI}}W(B)\cdot C^{\prime}\frac{m(B)}{W(B)}=C^{\prime\prime},
\end{align}
which gives \eqref{e.COBMO}. By using the John-Nirenberg inequality in doubling metric measure spaces \cite[Theorem 5.2]{ABKY11},
inequality \eqref{e.CO} is a consequence of \cite[proof of Corollary 5.6]{BM95}.
\end{proof}

Recall that in Remark \ref{r.FK}, for any open subset $\Omega\subset M$, we define a quantity \begin{equation}
	\lambda_{1}(\Omega):=\inf_{u\in\sF(\Omega)\setminus\{0\}}\frac{\sE(u)}{\norm{u}^{p}_{L^{p}(\Omega)}}.
\end{equation}
The following \emph{comparison principle} will be used later in the proof of Theorem \ref{t.main}.

\begin{proposition}\label{p.comp}
Let $(\mathcal{E},\mathcal{F})$ be a mixed local and nonlocal $p$-energy form on $(M,m)$. Let $\Omega$ be an open subset of $M$ such that $\lambda_{1}(\Omega)>0$. Let $u,v \in \sF^{\prime}$ satisfying that $(v -u)_+\in\sF(\Omega)$ and 
\begin{equation*}
 \sE(u;(v -u)_+)+\lambda \int_M|u|^{p-2}u(v -u)_+ \dif m \geq \sE(v;(v -u)_+)+\lambda \int_M|v|^{p-2}v(v -u)_+ \dif m
\end{equation*}
for some $\lambda\in [0,\infty)$, then $u\geq v$ $m$-a.e. in $\Omega$.
\end{proposition}
\begin{proof}
 Denote $w=v-u$ and $\Psi(t) :=\Gamma^{(L)}(u+tw;w)(\{\wt{v}>\wt{u}\})$. By \cite[Theorem 1.4-(d)]{KS25} and \cite[proof of Proposition 3.6]{Yan25b},
\begin{align}
  \sE^{(L)}(u;w_+)&=\Gamma^{(L)}(u;v-u)(\{\wt{v}>\wt{u}\})=\Psi(0) \\
  & \leq \Psi(1)=\Gamma^{(L)}(v;v-u)(\{\wt{v}>\wt{u}\})=\sE^{(L)}(v;w_+)  \label{e.L1}
\end{align}
with the equality if and only if $\Gamma^{(L)}(v-u)(\{\wt{v}>\wt{u}\})=0$, that is $\sE^{(L)}(w_+)=0$. 
For the nonlocal part, note that
\begin{align}
 &\phantom{\ \leq} \abs{\wt{v}(x) - \wt{v}(y)}^{p-2} (\wt{v}(x) - \wt{v}(y)) - \abs{\wt{u}(x) - \wt{u}(y)}^{p-2} (\wt{u}(x) - \wt{u}(y)) \\
 & =\int_0^1 \frac{\dif}{\dif t} \Big( \abs{\wt{u}(x) - \wt{u}(y) + t ( \wt{w}(x) - \wt{w}(y))}^{p-2} \big(\wt{u}(x) - \wt{u}(y)+ t ( \wt{w}(x) - \wt{w}(y))\big)\Big)  \dif t \\
 & =(p-1) \left(\int_0^1 \abs{\wt{u}(x) - \wt{u}(y) + t(\wt{w}(x) - \wt{w}(y))}^{p-2} \dif t\right)(\wt{w}(x) - \wt{w}(y)) .
\end{align}
Denote $I(x,y):=(p-1) \left(\int_0^1 |\wt{u}(x) - \wt{u}(y) + t(\wt{w}(x) - \wt{w}(y))|^{p-2} \dif t\right)$. Since $I(x,y)\geq 0$ and $(X-Y)(X_+-Y_+)\geq 0 $ for all $X,Y\in \mathbb{R}$, we have
\begin{align}
 &\phantom{\ \leq}  \sE^{(J)}(v;w_+) - \sE^{(J)}(u;w_+) \\
 &=\int_{M_{\od}^{2}}\Big(|\wt{v}(x) - \wt{v}(y)|^{p-2} (\wt{v}(x) - \wt{v}(y)) - |\wt{u}(x) - \wt{u}(y)|^{p-2} (\wt{u}(x) - \wt{u}(y))\Big)( \wt{w}_+(x) - \wt{w}_+(y) )\dif j(x,y) \\
&= \int_{M_{\od}^{2}} I(x,y)\left( \wt{w}(x) - \wt{w}(y) \right) ( \wt{w}_+(x) - \wt{w}_+(y) )\dif j(x,y)\geq0. \label{e.J1}
\end{align}
 with equality if and only if 
\begin{align}
&\phantom{\Longleftrightarrow\ } I(x,y)\left( \wt{w}(x) - \wt{w}(y) \right) ( \wt{w}_+(x) - \wt{w}_+(y) )=0 \ \text{$j$-a.e. $(x,y)\in M_{\od}^{2}$},\\
&\Longleftrightarrow \text{$ \wt{w}_+(x)- \wt{w}_+(y)=0$\ $j$-a.e. $(x,y)\in M_{\od}^{2}$} \Longleftrightarrow \sE^{(J)}(w_+)=0.
\end{align}
Therefore, we conclude from \eqref{e.L1} and \eqref{e.J1} that
\begin{equation}
\sE(u;w_+)=\sE^{(L)}(u;w_+)+\sE^{(J)}(u;w_+)\leq \sE^{(L)}(v;w_+)+\sE^{(J)}(v;w_+)=\sE(v;w_+).
\end{equation}
On the other hand, by assumption and the elementary inequality 
\begin{equation*}
 \left(|x|^{p-2}x-|y|^{p-2}y\right)(x -y)\geq 0 \ \ \text{for any }x,y\in \mathbb{R},
\end{equation*}  
we have
\begin{align*}
\sE(u;w_+)-\sE(v;w_+)&\geq \lambda \int_M \left(|v|^{p-2}v-|u|^{p-2}u\right)(v -u)_+\dif m \\
&=\lambda \int_{\{v>u\}} \left(|v|^{p-2}v-|u|^{p-2}u\right)(v -u)\dif m\geq 0
\end{align*}
hence $\sE(u;w_+)=\sE(v;w_+)$, and then it must be 
\begin{equation}
 \sE^{(L)}(u;w_+)=\sE^{(L)}(v;w_+), \ \ \sE^{(J)}(u;w_+)=\sE^{(J)}(v;w_+),
\end{equation}
that is, the equalities hold for \eqref{e.L1} and for \eqref{e.J1}. Therefore $\sE^{(L)}(w_+)=\sE^{(J)}(w_+)=0$, combine which with the facts that $w_{+}=(v -u)_+\in\sF(\Omega)$ and that $\lambda_{1}(\Omega)>0$, we have $\norm{(v -u)_+}_{L^p(\Omega,m)}=0$, i.e. $u\geq v$ $m$-a.e. in $\Omega$. 
\end{proof}

\begin{proposition}[{\cite[Proposition 4.2]{Yan25b}}]\label{p.exten}
Let $\Omega$ be a bounded open subset of $M$ such that $\lambda_{1}(\Omega)>0$ and $u\in\sF^{\prime}$. Then there exists a unique function $h\in\sF^{\prime}$ such that $h$ is harmonic in $\Omega$ and $\tilde{h}=\tilde{u}$ $\sE$-q.e. on $M \setminus\Omega$. We denote this function by $H^\Omega u$. Moreover, if $0\le u\le a$ $m$-a.e. in $M$, where $a\in(0,\infty)$ is some constant, then $0\le H^\Omega u\le a$ $m$-a.e. in $M$.
\end{proposition}
\begin{proof}
  The proof is actually the same as \cite[Proposition 4.2]{Yan25b}. Note that conditions volume doubling and Poincar\'e inequality used in the proof of \cite[Proposition 4.2]{Yan25b} are only to obtain  $\lambda_{1}(\Omega)>0$ (see \cite[Lemma 4.1]{Yan25b}).
\end{proof}

With the Lemma \ref{l.cros}, Proposition \ref{p.comp}, Proposition \ref{p.exten} and the lemma of growth at hand, we are now ready to prove the weak elliptic Harnack inequality.

\begin{proof}[Proof of Theorem \ref{t.main}-\eqref{e.main-wEH}]\label{page.wEHpf}
Let $B_{R}:=B(x_{0},R)$ be a metric ball
in $M$ with $0<R<\sigma \overline{R}$, where constant $\sigma $ comes from Lemma \ref{l.LoG}. Let $B_{r}:=B(x_{0},r)$ with $r\in(0,\delta R)$, where $\delta={(96\kappa+24)^{-1}}$. First, we prove \eqref{wEH} for $u\in \mathcal{F^{\prime }}\cap L^{\infty }$. Let $u\in \mathcal{F^{\prime }}\cap L^{\infty }$ be a function that is
non-negative, superharmonic in $B_{R}$. We need to show 
\begin{equation}\label{24}
\left( \frac{1}{m(B_{r})}\int_{B_{r}}u^{q}\dif m\right) ^{1/q }\leq
C\left( \einf_{{B}_{r}}u+W(B_{r})^{\frac{1}{p-1}}T_{\frac{1}{2}B_{R},B_{R}}(u_{-}) \right)
\end{equation}
for some universal numbers $q\in (0,1)$ and $C\geq 1$, both of which are
independent of $x_{0},\ R,\ r$ and $u$.

To do this, let $\lambda:= W(B_{R})^{\frac{1}{p-1}}T_{\frac{1}{2}B_{R},B_{R}}(u_{-})$. We claim that
\begin{equation}
\left( \fint_{B_{r}}u_{\lambda }^{q }\dif m\right) ^{1/q }\leq
C\einf_{{B_{r}}}u_{\lambda }\text{ with }u_{\lambda }:=u+\lambda, \ \forall r\in (0,\delta R]
\label{eq:vol_318}
\end{equation}%
for some constant $C$ independent of $x_{0}, R, r$ and $u$. 

Indeed, by Lemmas \ref{l.CS=>cap} and \ref{l.cros}, there exist two positive constants $q\in(0,1)$ and $c^{\prime }$, independent of $x_{0},\ R,\ r$ and $u$, such that 
\begin{equation}
\left( \fint_{B_{r}}u_{\lambda }^{q }\dif m\right) ^{1/q
}\left( \fint_{B_{r}}u_{\lambda }^{-q}\dif m\right) ^{1/q
}\leq c^{\prime },\  \forall r\in(0,2\delta R].\label{eq:vol_329}
\end{equation}%
Let $\theta$ be the constant from \eqref{eq:vol_317} in Lemma \ref{l.LoG} with $\epsilon=\delta=1/2$. Without loss of generality, assume $\theta \geq 1$ so that $q/\theta\in (0,1)$. Let $b:= W(B_{r})^{\frac{1}{p-1}}T_{B_{r},2B_{r}}\big((u_{\lambda })_{-}\big)$.
Define a function $g$ by
\begin{equation}
g(a):=a\Big(1+\frac{b^{p-1}}{a^{p-1}}\Big)^{\theta/q}, \ a\in (0,\infty ). \label{gg}
\end{equation}
Using the facts that $(u_{\lambda })_{-}\leq u_{-}$ in $M$\ and $2B_{r}\subset B_{R}$, we have 
\begin{equation}
0\leq b\leq W(B_{R})^{\frac{1}{p-1}}T_{B_{r},2B_{r}}\big((u_{\lambda })_{-}\big) 
=W(B_{R})^{\frac{1}{p-1}}T_{B_{r},B_{R}}\big((u_{\lambda })_{-}\big)\leq \lambda .  \label{26}
\end{equation}%
Clearly, for any $a>\lambda $, by Chebyshev's inequality
\begin{equation}
\frac{m(B_{r}\cap \{u_{\lambda }<a\})}{m(B_{r})}=\frac{m(B_{r}\cap
\{u_{\lambda }^{-q }>a^{-q }\})}{m(B_{r})}\leq a^{q}%
\fint_{B_{r}}u_{\lambda }^{-q }\dif m.  \label{27}
\end{equation}

Note that $u_{\lambda }\in \mathcal{F^{\prime }}\cap L^{\infty }$ is
superharmonic and non-negative in $2B_{r}\subset B_{R}$. To look at whether the
hypotheses \eqref{eq:vol_317} in Lemma \ref{l.LoG} is satisfied or not, we consider
two cases.
\begin{enumerate}[label=\textit{{Case \arabic*}.},align=left,leftmargin=*,topsep=5pt,parsep=0pt,itemsep=2pt]
	\item Assume that there exists a number $a\in(\lambda,\infty)\overset{\eqref{26}}{\subset}(b,\infty)$ such that 
\begin{equation}
g(a)=\varepsilon _{1}^{1/q }\left( \fint_{B_{r}}u_{\lambda
}^{-q }\dif m\right) ^{-1/q },  \label{eq:vol_319}
\end{equation}%
where $\varepsilon _{1}:=\varepsilon _{0}2^{-(2p-1+C_{L})\theta}$, and $C_{L}$ comes from Lemma \ref{l.LoG}. In this case, we have 
{\small
\begin{align}
\frac{m(B_{r}\cap \{u_{\lambda }<a\})}{m(B_{r})}& \overset{\eqref{27}}{\leq}a^{q}\fint%
_{B_{r}}u_{\lambda }^{-q }\dif m\overset{\eqref{eq:vol_319}}{=}\varepsilon
_{0}2^{-(2p-1+C_{L})\theta }\left( 1+\frac{W(B_{r})T_{B_{r},2B_{r}}\big(%
(u_{\lambda })_{-}\big)^{p-1}}{a^{p-1}}\right) ^{-\theta } \\
& =\varepsilon _{0}(1-1/2)^{(2p-1)\theta }(1-1/2)^{C_{L}\theta }\left( 1+%
\frac{W(B_{r})T_{B_{r},2B_{r}}\big((u_{\lambda })_{-}\big)^{p-1}}{a^{p-1}}\right)
^{-\theta }.
\end{align}}
Therefore, we see that the assumption \eqref{eq:vol_317}, with $B$ being
replaced by $B_{r}$ and $u$ replaced by $u_{\lambda }$, is true with $\varepsilon =\delta =1/2$. Thus, by
Lemma \ref{l.LoG}, 
\begin{align}
\einf_{{\frac{1}{2}B_{r}}}u_{\lambda }& \geq \frac{1}{2}a\overset{\eqref{eq:vol_319}}{=}\frac{1}{2}\left(
1+\frac{b^{p-1}}{a^{p-1}}\right) ^{-\theta/q}\varepsilon _{1}^{1/q }\left( %
\fint_{B_{r}}u_{\lambda }^{-q }\dif m\right) ^{-1/q }\\
& \overset{\eqref{eq:vol_329}}{\geq} \frac{1}{2}\left( 1+\frac{b^{p-1}}{a^{p-1}}\right) ^{-\theta/q}\varepsilon
_{1}^{1/q }{c^{\prime }}^{-1}\left( \fint_{B_{r}}u_{\lambda
}^{q }\dif m\right) ^{1/q }\\
& \geq \frac{1}{2}2^{-\theta/q}\varepsilon _{1}^{1/q }{c^{\prime }}%
^{-1}\left( \fint_{B_{r}}u_{\lambda }^{q }\dif m\right) ^{1/q }%
\text{ \ (as }a>b \text{),}
\end{align}%
which gives that $\left( \fint_{B_{r}}u_{\lambda }^{q }\dif m\right) ^{1/q }\leq
2c^{\prime }\varepsilon _{1}^{-1/q}2^{\theta/q}\einf_{{\frac{1}{2}B_{r}}%
}u_{\lambda }$.
\item  Assume that \eqref{eq:vol_319} is not
satisfied for any $a\in (\lambda ,\infty )$. Since $g$ is continuous on $(0,\infty )$ and $\lim_{a\rightarrow \infty }g(a)=\infty $, we have  
\begin{equation}
g(a)>\varepsilon _{1}^{1/q }\left( \fint%
_{B_{r}}u_{\lambda }^{-q}\dif m\right) ^{-1/q },  \ \forall a\in(\lambda,\infty). \label{g2}
\end{equation}%
\begin{itemize}
	\item If $\lambda =0$. Then $b=0$ by \eqref{26}. So by definition \eqref{gg}, we have $g(a)\equiv a$ for all $a>0$. Letting $a\rightarrow 0^{+}$ in \eqref{g2}, we have $\left(\fint_{B_{r}}u_{\lambda
}^{-q }\dif m\right) ^{-1/q }=0$, which combines with \eqref{eq:vol_329} gives $\left( \fint_{B_{r}}u_{\lambda }^{q }\dif m\right) ^{1/q}=0$.

\item If $\lambda >0$. Since $g$ is continuous on $%
(0,\infty )$, we can let $a\downarrow \lambda $ in \eqref{g2}, and obtain $g(\lambda )\geq \varepsilon _{1}^{1/q }\left( \fint_{B_{r}}u_{\lambda }^{-q }\dif m\right) ^{-1/q}$, from which we see that 
\begin{equation}
2^{\theta/q}\lambda \overset{\eqref{26}}{\geq} \lambda \left( 1+\frac{b}{\lambda }\right) ^{\theta/q}=
g(\lambda )\geq \varepsilon _{1}^{1/q }\left( \fint%
_{B_{r}}u_{\lambda }^{-q }\dif m\right) ^{-1/q }.
\label{eq:vol_321}
\end{equation}%
Thus, we have 
\begin{align}
\left( \fint_{B_{r}}u_{\lambda }^{q }\dif m\right) ^{1/q }
&\overset{\eqref{eq:vol_329}}{\leq} c^{\prime }\left( \fint_{B_{r}}u_{\lambda }^{-q }\dif m\right)
^{-1/q } \overset{\eqref
{eq:vol_321}}{\leq} c^{\prime }\varepsilon _{1}^{-1/q}2^{\theta/q}\lambda. \label{eq:vol_324}
\end{align}
\end{itemize}
\end{enumerate}
Combining the above two cases, we always have 
\begin{equation}
\left( \fint_{B_{r}}u_{\lambda }^{q}\dif m \right) ^{1/q}\leq
C(\lambda +\einf_{{\frac{1}{2}B_{r}}}u_{\lambda })\leq 2C\einf_{{\frac{1}{2}%
B_{r}}}u_{\lambda },\ \forall r\in(0,2\delta R].\label{eq:vol_325}
\end{equation}%
Therefore for all $r\in(0,2\delta R]$,
\begin{equation}
\left( \fint_{\frac{1}{2}B_{r}}u_{\lambda }^{q }\dif m\right)
^{1/q }\overset{\eqref{VD}}{\leq}\Big(\frac{C_{\mathrm{VD}}}{m(B_{r})}\int_{%
\frac{1}{2}B_{r}}u_{\lambda }^{q }\dif m\Big)^{1/q}\leq C_{\mathrm{VD}}^{1/q}\left( \fint_{B_{r}}u_{\lambda }^{q }\dif m\right)
^{1/q } \overset{\eqref{eq:vol_325}}{\leq} C^{\prime }\einf_{{\frac{1}{2}B_{r}}}u_{\lambda }.\label{e.v326}
\end{equation}%
Our claim \eqref{eq:vol_318} is thus proved by renaming $r/2$ by $r$ in \eqref{e.v326}. As a consequence of our claim,
\begin{align}
\left( \fint_{B_{r}}u^{q }\dif m\right) ^{1/q
}& \leq \left( \fint_{B_{r}}u_{\lambda }^{q }\dif %
m\right) ^{1/q }\overset{\eqref{eq:vol_318} }{\leq} C^{\prime }\einf_{{B_{r}}}u_{\lambda }   \\
& =C^{\prime }\left( \einf_{{B_{r}}}u+ W(B_{R})^{\frac{1}{p-1}}T_{\frac{1}{2}
B_{R},B_{R}}(u_{-})\right),\ \forall r\in(0,\delta R].\label{eq:vol_327}
\end{align}%

The rest of the proof is to show that, for any $r\in(0,
\delta R)$, the term $W(B_{R})$ on the right-hand side of \eqref
{eq:vol_327} can be replaced by a smaller one $W(B_{r})$, by adjusting the value of constant $C^{\prime }$. Indeed, fix $r\in(0,\delta R)$. Let $r_{k}:=\delta ^{k}R$, $k\in\{0\}\cup\bN$. Let $i\in\bN$ such that $r_{i+1}\leq r <r_{i}$. By \eqref{eq:vol_0},
\begin{equation}
W(B_{r_{i-1}})=\frac{W(x_{0},\delta ^{i-1}R)}{W(x_{0},r)}W(x_{0},r)\leq
C_{2}\left( \frac{\delta ^{i-1}R}{r}\right) ^{\beta _{2}}W(B_{r})\leq
C_{2}\delta ^{-2\beta _{2}}W(B_{r}).  \label{41}
\end{equation}

Since $u$ is superharmonic in $B_{r_{i-1}}$, applying \eqref{eq:vol_327} in the first inequality below with $R$ being replaced by $r_{i-1}$ and noting that $T_{\frac{1}{2}B_{r_{i-1}},B_{r_{i-1}}}(u_{-})\leq T_{\frac{1}{2}%
B_{R},B_{R}}(u_{-})$, we have 
\begin{align}
\left(\fint_{B_{r}}u^{q}\dif m\right) ^{1/q }& \leq C^{\prime }\left( \einf_{{B_{r}}}u+ W(B_{r_{i-1}})^{\frac{1}{p-1}}T_{\frac{1}{2}
B_{r_{i-1}},B_{r_{i-1}}}(u_{-}) \right)  \\
& \overset{\eqref{41}}{\leq} C^{\prime }\left( \einf_{{B_{r}}}u+\left( C\delta ^{-2\beta
_{2}}W(B_{r})\right) ^{\frac{1}{p-1}}T_{\frac{1}{2}B_{R},B_{R}}(u_{-})\right)  \\
&\leq C\left( \einf_{{B_{r}}}u+ W(B_{r})^{\frac{1}{p-1}}T_{\frac{1}{2}
B_{R},B_{R}}(u_{-})\right) , \label{e.bdw}
\end{align}%
thus obtaining the weak elliptic Harnack inequality \eqref{wEH} for $u\in \mathcal{F^{\prime }}\cap L^{\infty }$. 

Next, we prove \eqref{wEH} for general $u\in \mathcal{F^{\prime }}$. For any $n\ge1$, let $u_n:=H^{B(x_0,R)}(u\wedge n)$, then by Proposition \ref{p.exten}, we have $u_n$ is non-negative bounded harmonic on $B(x_0,R)$, $\widetilde{u}_n=\widetilde{u}\wedge n\le \widetilde{u}$ $\sE$-q.e. on $M\setminus B(x_0,R)$, and $\mathcal{E}(u_n)\le \mathcal{E}(u\wedge n)\le \mathcal{E}(u)$. By \eqref{e.bdw}, we have
\begin{equation}
\left(\fint_{B_{r}}u_n^{q}\dif m\right) ^{1/q }\le C\left( \einf_{{B_{r}}}u_n+ W(B_{r})^{\frac{1}{p-1}}T_{\frac{1}{2}
B_{R},B_{R}}((u_n)_{-})\right). \label{e.wEH-1}
\end{equation}
Since $u_n,u$ are both harmonic in $B_R$, by Proposition \ref{p.comp}, we have $u_n\le u$ in $B(x_0,R)$, which gives $u_n\le u$ in $M$, $\lVert u_n\rVert_{L^p(M;m)}\le \lVert u\rVert_{L^p(M;m)}$. Similarly, since $\widetilde{u}_n=\widetilde{u}\wedge n\le\widetilde{u}\wedge (n+1)=\widetilde{u}_{n+1}$ $\sE$-q.e. on $M \setminus B_R$, we have $u_n\le u_{n+1}$ on $M$. Hence there exists non-negative $v\in L^p(M;m)$ with $v\le u$ in $M$ such that $u_n\uparrow v$ in $M$. 
Note that
\begin{align}
\abs{T_{\frac{1}{2}B_{R},B_{R}}((u_n)_{-})^{p-1}-T_{\frac{1}{2}
B_{R},B_{R}}(v_{-})^{p-1}}
&\leq \esup_{x\in\frac{1}{2}B_{R}}\left(\int_{B_{R}^{c}}
\abs{|\wt{u_n}(y)_-|^{p-1}-|\wt{v}(y)_-|^{p-1}}J(x,\dif y)\right) \\
&\overset{\eqref{TJ}, \eqref{eq:vol_0}}{\leq} \frac{C}{W(B_{R})}\norm{((u_n)_{-})^{p-1}-(v_{-})^{p-1}}_{L^\infty(M,m)}.
\end{align}
Letting $n\uparrow\infty$ in \eqref{e.wEH-1}, we see that
\begin{equation}
\left(\fint_{B_{r}}v^{q}\dif m\right) ^{1/q }\le C\left( \einf_{{B_{r}}}v+ W(B_{r})^{\frac{1}{p-1}}T_{\frac{1}{2}
B_{R},B_{R}}(v_{-})\right).\label{e.wEH-2}
\end{equation}
By the same argument as in \cite[Proof of EHI in Theorem 2.3 on p.~28-29]{Yan25a}, we know that
$u,v\in \sF^{\prime}$ are both harmonic on $B_R$ and $\wt{u}=\wt{v}$ $\sE$-q.e. in $M\setminus B_R$. By Proposition \ref{p.comp}, we conclude that $u=v$ a.e. in $B_R$, thus showing \eqref{wEH} by \eqref{e.wEH-2}.
\end{proof}

\subsection{Consequence of \texorpdfstring{(\ref{wEH})}{(wEH)}}\label{s.E><}
In this section, we prove Theorem \ref{t.weh=>cs}. We first prove the existence of weak solutions for $p$-energy form, along with some monotonicity properties.
\begin{lemma}\label{l.exist-weak}
	Let $\Omega$ be an open subset of $M$ such that $\lambda_{1}(\Omega)>0$. Let $\lambda\in[0,\infty)$.
 \begin{enumerate}[label=\textup{({\arabic*})},align=left,leftmargin=*,topsep=5pt,parsep=0pt,itemsep=2pt]
	\item \label{it.G1} If $f\in L^{\infty}(\Omega,m)$, then there exists a unique function in $\sF(\Omega)$, denoted by $G_{\lambda}^{\Omega}[f]$, such that 
\begin{equation}\label{e.defG}
		\sE(G_{\lambda}^{\Omega}[f];\phi)+\lambda\int_{\Omega}
\abs{G_{\lambda}^{\Omega}[f]}^{p-2}G_{\lambda}^{\Omega}[f]\cdot \phi\dif m=\int_{\Omega}f\cdot\phi\dif m,\ \forall\phi\in\sF(\Omega).
	\end{equation}
We write $G^{\Omega}[f]=G_{0}^{\Omega}[f]$ for $\lambda=0$.
\item \label{it.G2} For any $f_1, f_2\in L^{\infty}(\Omega,m)$ with $f_1\geq f_2  $ $m$-a.e. in $\Omega$, we have
    \begin{equation*}
     G_{\lambda}^{\Omega}[f_1]\geq G_{\lambda}^{\Omega}[f_2] \ m\text{-a.e. on }\Omega.
    \end{equation*}
Particular, $G_{\lambda}^{\Omega}[f]\geq 0 \ m\text{-a.e. on }\Omega$ for any $0\leq f \in L^{\infty}(\Omega,m)$ since $G_{\lambda}^{\Omega}[0]=0$.
\item \label{it.G3} Let $U$ be any non-empty open subset of $\Omega$. Then for any nonnegative function $f\in L^{\infty}(\Omega,m)$,
  \begin{equation*}
     G_{\lambda}^{\Omega}[f]\geq G_{\lambda}^{U}[f] \ m\text{-a.e. on }\Omega.
    \end{equation*}
\end{enumerate}
\end{lemma}
\begin{proof}
 \begin{enumerate}[label=\textup{({\arabic*})},align=left,leftmargin=*,topsep=5pt,parsep=0pt,itemsep=2pt]
	\item[\ref{it.G1}]
By the assumption, we know that $(\sF(\Omega),(\sE(\cdot)+\lambda\norm{\cdot}_{L^{p}(\Omega,m)}^{p})^{1/p})$ is a reflexive Banach space for all $\lambda\in[0,\infty)$.
Note that
\begin{align*}
\left|\int_\Omega f\cdot \phi \dif m\right|&\leq m(\Omega)^{1-1/p}\norm{f}_{L^{\infty}(\Omega,m)}\left(\int_\Omega |\phi|^p\dif m\right)^{1/p} \\
&\leq m(\Omega)^{1-1/p}\norm{f}_{L^{\infty}(\Omega,m)}\lambda_{1}(\Omega)^{-1}\cdot \sE(\phi)^{1/p},
\end{align*}
which implies that $\phi \mapsto \int_\Omega f\cdot\phi \dif m$ is a bounded linear functional on the Banach space $(\sF(\Omega),(\sE(\cdot)+\lambda\norm{\cdot}_{L^{p}(\Omega,m)}^{p})^{1/p})$. Therefore there exists a unique function $G_{\lambda}^{\Omega}[f] \in \sF(\Omega)$ such that \eqref{e.defG} holds.
\item[\ref{it.G2}] Note that $(G^{\Omega}_{\lambda}[f_2]-G^{\Omega}_{\lambda}[f_1])_{+}\in\sF(\Omega)$ by \ref{it.G1}, and
\begin{align*}
 &\phantom{\ \leq}\sE(G^{\Omega}_{\lambda}[f_1];(G^{\Omega}_{\lambda}[f_2]-G^{\Omega}_{\lambda}[f_1])_{+})+\lambda \int_M|G^{\Omega}_{\lambda}[f_1]|^{p-2}G^{\Omega}_{\lambda}[f_1]
 (G^{\Omega}_{\lambda}[f_2]-G^{\Omega}_{\lambda}[f_1])_{+} \dif m \\
 &= \int_{\Omega}f_1\cdot(G^{\Omega}_{\lambda}[f_2]-G^{\Omega}_{\lambda}[f_1])_{+}\dif m \geq \int_{\Omega}f_2\cdot(G^{\Omega}_{\lambda}[f_2]-G^{\Omega}_{\lambda}[f_1])_{+}\dif m\\
& =\sE(G^{\Omega}_{\lambda}[f_2];(G^{\Omega}_{\lambda}[f_2]-G^{\Omega}_{\lambda}[f_1])_{+})
+\lambda \int_M|G^{\Omega}_{\lambda}[f_2]|^{p-2}G^{\Omega}_{\lambda}[f_2]
(G^{\Omega}_{\lambda}[f_2]-G^{\Omega}_{\lambda}[f_1])_{+} \dif m,
\end{align*}
 we have from Proposition \ref{p.comp} that $G^{\Omega}_{\lambda}[f_1]\geq G^{\Omega}_{\lambda}[f_2]$ $m$-a.e. on $\Omega$.
\item[\ref{it.G3}] First we note that $\lambda_{1}(U)\geq\lambda_{1}(\Omega)>0$.  Since $(G^{U}_{\lambda}[f]-G^{\Omega}_{\lambda}[f])_{+}\in\sF(\Omega)$, and note that
\begin{align*}
 &\phantom{\ \leq}\sE(G^{\Omega}_{\lambda}[f];(G^{U}_{\lambda}[f]-G^{\Omega}_{\lambda}[f])_{+})+\lambda \int_M|G^{\Omega}_{\lambda}[f]|^{p-2}G^{\Omega}_{\lambda}[f]
 (G^{U}_{\lambda}[f]-G^{\Omega}_{\lambda}[f])_{+}\dif m \\
 &= \int_{\Omega}f\cdot(G^{U}_{\lambda}[f]-G^{\Omega}_{\lambda}[f])_{+}\dif m \geq \int_{U}f\cdot(G^{U}_{\lambda}[f]-G^{\Omega}_{\lambda}[f])_{+}\dif m\\
& =\sE(G^{U}_{\lambda}[f];(G^{U}_{\lambda}[f]-G^{\Omega}_{\lambda}[f])_{+})+\lambda \int_M|G^{U}_{\lambda}[f]|^{p-2}G^{U}_{\lambda}[f]
 (G^{U}_{\lambda}[f]-G^{\Omega}_{\lambda}[f])_{+}\dif m,
\end{align*}  
     we have by Proposition \ref{p.comp} that $ G_{\lambda}^{\Omega}[f]\geq G_{\lambda}^{U}[f]$ $m$-a.e. on $\Omega$.
\end{enumerate}
\end{proof}

\begin{remark}
In the case $p=2$, for each $\lambda\in(0,\infty)$, the mapping $f\mapsto G_{\lambda}^{\Omega}f$ defines a bounded linear operator on $L^{2}(\Omega,m)$, called the \emph{resolvent operator} associated with the part Dirichlet space $(\sE,\sF(\Omega))$. The family $\{G_\lambda^\Omega\}_{\lambda\in(0,\infty)}$ is referred as the \emph{resolvent}. For $\lambda=0$, the operator $G^{\Omega}=G_{0}^{\Omega}$ is called the \emph{Green operator}. See \cite[Chapter 1]{FOT11} and \cite{GH14} for further details.
\end{remark}

\begin{proposition}[\text{c.f. \cite[Proposition 9.3]{GHL15}}]
Let $(\mathcal{E}, \mathcal{F})$ be a mixed local and nonlocal $p$-energy form on $(M,m)$
and let $\Omega \subset M$ be an open set with $m(\Omega) < \infty$. If \(u \in \mathcal{F}(\Omega)\) is non-negative and satisfying $\sE(u;\phi)\leq \int_\Omega f\cdot\phi \dif m$ for some \(f \in L^q(\Omega,m)\) with $q \geq \frac{p}{p-1}$, then 
\begin{equation}\label{e.Es}
 \|(u - s)_+\|_{L^{1}(\Omega,m)}^{p-1} \leq \frac{m(E_s)^{p-1-1/q}}{\lambda_{1}(E_s')} \|f\|_{L^{q}(\Omega,m)}\quad  \text{for any $s \geq 0$},
\end{equation}
where \(E_s := \{x \in \Omega : u(x) \geq s\}\) and \(E_s'\) is any open neighborhood of \(E_s\).
\end{proposition}
\begin{proof}
Without loss of generality, we can assume that $u$ is $\sE$-quasi-continuous. Note that
\begin{equation}\label{e.Es1}
|a-b|^{p-2}(a-b)(a_+-b_+)\geq |a_+-b_+|^p \ \ \ \forall \ a,b\in \mathbb{R}.
\end{equation}
We have 
\begin{align}
  \sE^{(J)}(u;(u - s)_+)&=\int_{M_{\od}^{2}} |u(x)-u(y)|^{p-2}(u(x)-u(y))((u - s)_+(x)-(u - s)_+(y))\dif j(x,y) \\
  &\overset{\eqref{e.Es1}}{\geq} \int_{M_{\od}^{2}} |(u - s)_+(x)-(u - s)_+(y)|^p \dif j(x,y)=\sE^{(J)}((u - s)_+). \label{e.se-j}
\end{align}
Since $\psi(t)=(t-s)_+$ is a piecewise $C^1$ function on $\mathbb{R}$ with $\psi(0)=0$. By Theorem \ref{t.sasEM}-\ref{lb.M-chain},
\begin{align}
\sE^{(L)}((u - s)_+)&=\int_M \dif\Gamma^{(L)}\la \psi\circ u;(u - s)_+\ra=\int_{\{u>s\}}\dif\Gamma^{(L)}\la u;(u - s)_+\ra \\
&=\int_{\{u>s\}}\dif\Gamma^{(L)}\la u;u\ra= \sE^{(L)}(u;(u - s)_+).\label{e.se-l}
\end{align}
Therefore, we conclude that
\begin{align}
 \mathcal{E}((u - s)_+)  &\overset{\eqref{e.se-j}, \eqref{e.se-l}}{ \leq} \mathcal{E}(u; (u - s)_+)\leq \int_\Omega (u - s)_+ f \dif m
\leq \|(u - s)_+\|_{L^{q^{\prime}}(\Omega,m)} \|f\|_{L^{q}(\Omega,m)} \\
   & \leq m(E_s)^{1/q^{\prime}-1/p} \|(u - s)_+\|_{L^{p}(\Omega,m)} \|f\|_{L^{q}(\Omega,m)} \text{\ (by H\"{o}lder inequality)}\label{e.es-2}
\end{align}
where $q':= q/(q - 1)$ is the H\"{o}lder conjugate of $q$. Since $\restr{(u - s)_+}{M\setminus E_s} = 0$ $\sE$-q.e., we have
\begin{align}
\|(u - s)_+\|_{L^{p}(\Omega,m)}^p&\leq \lambda_{1}(E_s')^{-1}\mathcal{E}((u - s)_+)\ \text{ (definition of \(\lambda_{1}(E_s')\))}\\
&\overset{\eqref{e.es-2}}{\leq} \lambda_{1}(E_s')^{-1} m(E_s)^{1/q^{\prime}-1/p} \|(u - s)_+\|_{L^{p}(\Omega,m)} \|f\|_{L^{q}(\Omega,m)},
\end{align}
that is, \begin{equation}\label{e.es-3}
	\|(u - s)_+\|_{L^{p}(\Omega,m)}^{p-1}\leq \lambda_{1}(E_s')^{-1} m(E_s)^{1/q^{\prime}-1/p} \|f\|_{L^{q}(\Omega,m)}.
\end{equation}
By the H\"{o}lder inequality, we have
\begin{equation}
  \|(u - s)_+\|_{L^{1}(\Omega,m)}\leq m(E_s)^{1-1/p} \|(u - s)_+\|_{L^{p}(\Omega,m)}.
\end{equation}
combining which with \eqref{e.es-3} gives \eqref{e.Es}.
\end{proof}
In the following lemma, we will use a corollary of Faber--Frahn inequality \eqref{FK}:\begin{equation}
	\lambda_{1}(\Omega)>0\ \text{for any ball $B$ whose radii is less than $\sigma \ol{R}$ and any open subset $\Omega\subset B$},
\end{equation}
so that the existence of weak solution is ensured by Lemma \ref{l.exist-weak}.
\begin{lemma}\label{l.E<>}
 Assume \eqref{VD}, \eqref{FK}, \eqref{wEH}, \eqref{cap<}. Let $\sigma, \delta$ be the constants from condition \eqref{wEH}. Let $B:=B(x_{0},R)$ be a ball in $M$ with $R<\sigma \overline{R}$.
 \begin{enumerate}[label=\textup{({\arabic*})},align=left,leftmargin=*,topsep=5pt,parsep=0pt,itemsep=2pt]
	\item\label{it.e<>} There exists $C>0$, independent of $B$, such that \begin{align}
\esup_B G^{B}[\one_{B}] & \leq CW(B)^{1/(p-1)} \label{e.E<} \\
  \einf_{\delta B}G^{B}[\one_{B}] & \geq C^{-1}W(B)^{1/(p-1)}.\label{e.E>} 
\end{align}
\item\label{it.el<>} There exist $C_{1},C_{2}>0$, independent of $B$, such that for all $\lambda\in(0,\infty)$, \begin{align}
	0\leq G^{B}_{\lambda}[\one_{B}]&\leq \lambda^{-{1}/({p-1})} \ m\text{-a.e. on }B,\label{e.resol2}\\
		 \einf_{\delta B}G^{B}_{\lambda}[\one_{B}]&\geq C_{1}\left(\frac{W(B)}{C_{2}+\lambda W(B)}\right)^{1/(p-1)} .
		 \label{e.resol3}
\end{align}
\end{enumerate}
\end{lemma}
\begin{proof}
 \begin{enumerate}[label=\textup{({\arabic*})},align=left,leftmargin=*,topsep=5pt,parsep=0pt,itemsep=2pt]
	\item[\ref{it.e<>}]
  By \eqref{e.Es} and the proof of \cite[Theorem 9.4]{GHL15}, we know that \eqref{e.E<} is a consequence of \eqref{FK}.

By Lemma \ref{l.exist-weak}-\ref{it.G2}, we have \begin{equation}
	G^{B}[\one_{B}]\geq G^{B}[\one_{\delta B}]\geq0\ m\text{-a.e. on }B.
\end{equation} 

Since $G^{B}[\one_{\delta B}]$ is non-negative on $M$ and is superharmonic in $B$, we can apply \eqref{wEH} to $G^{B}[\one_{\delta B}]$ on $B$ and obtain
\begin{equation}
\left( \fint_{\delta B}G^{B}[\one_{\delta B}]^{q}\dif m \right) ^{1/q}\overset{\eqref{wEH}}{\leq} C\einf_{\delta B}G^{B}[\one_{\delta B}].
\label{614}
\end{equation}
On the other hand, by \eqref{cap<},
\begin{equation}
\mathcal{E}(\phi )\leq C\frac{m (B)}{W(B)}\text{ for some $\phi \in \cutoff(2^{-1}B,B) $}.  \label{612}
\end{equation}
By the definition of $G^{B}[\one_{\delta B}]$ and \eqref{VD},
we see that
\begin{equation}
\mathcal{E}(G^{B}[\one_{\delta B}];\phi )=\int_{\delta B}\phi \dif m \geq m(
(2^{-1}\wedge \delta)B)\geq C_{\mathrm{VD} }^{-1}(2^{-1}\wedge \delta)^{d_{2}}m (B).  \label{615}
\end{equation}
Using the Cauchy--Schwarz inequality
and \eqref{612}, it follows that
\begin{align}
\mathcal{E}(G^{B}[\one_{\delta B}];\phi )&\leq \mathcal{E}(G^{B}[\one_{\delta B}])^{(p-1)/p}\mathcal{E}(\phi)^{1/p}=\left(\int_{\delta B}G^{B}[\one_{\delta B}]\dif m\right)^{(p-1)/p}\mathcal{E}(\phi)^{1/p}\\
&\leq C^{1/p}\left(\int_{\delta B}G^{B}[\one_{\delta B}]\dif m\right)^{(p-1)/p}\left(\frac{m (B)}{W(B)}\right)^{1/p}.  \label{616}
\end{align}
Combining \eqref{615} and \eqref{616}, we obtain
\begin{equation}
\int_{\delta B}G^{B}[\one_{\delta B}]\dif m \geq C_{1}m (B)W(B)^{1/(p-1)}.  \label{617}
\end{equation}
We conclude by \eqref{e.E<} that
\begin{align}
C_{1}m (B)W(B)^{1/(p-1)}&\overset{\eqref{617}}{\leq} \int_{\delta B}G^{B}[\one_{\delta B}]\dif m  =\int_{\delta B}G^{B}[\one_{\delta B}]^{q}\cdot G^{B}[\one_{\delta B}]^{1-q}\dif m \\
&\overset{\eqref{e.E<}}{\leq} \left(
CW(B)^{1/(p-1)}\right) ^{1-q}\int_{\delta B}G^{B}[\one_{\delta B}]^{q}\dif m\\
&=C_{2}W(B)^{(1-q)/(p-1)} m
(\delta B)\fint_{\delta B}G^{B}[\one_{\delta B}]^{q}\dif m \\
& \overset{\eqref{VD},\eqref{614}}{\leq} C_{3}W(B)^{(1-q)/(p-1)}m (B)\left(\einf_{\delta B}G^{B}[\one_{\delta B}]\right)^{q}.\label{e.618}
\end{align}%
By rearranging \eqref{e.618}, we obtain \eqref{e.E>}.
\item[\ref{it.el<>}] 
	By Lemma \ref{l.exist-weak}-\ref{it.G2}, we have $G^{B}_{\lambda}[\one_{B}]\geq 0$ $m$-a.e. on $B$. Similar to the proof of Lemma \ref{l.exist-weak}-\ref{it.G2}, since $(G^{B}_{\lambda}[\one_{B}]-\lambda^{-{1}/({p-1})})_{+}\in\sF(B)$, we have by Proposition \ref{p.comp} that $G^{B}_{\lambda}[\one_{B}]\leq \lambda^{-{1}/({p-1})}$ $m$-a.e. on $B$. This proves \eqref{e.resol2}.
	
	 To prove \eqref{e.resol3}, we first observe \ref{it.e<>} that \begin{equation}\label{e.resol4}
		\text{$0\leq G^{B}[\one_{B}]\leq CW(B)^{1/(p-1)}$ and $G^{B}[\one_{B}]\geq C^{-1} W(B)^{1/(p-1)}$ $m$-a.e. on $\delta B$.}
	\end{equation} Let $a\in(0,\infty)$ be a constant to be determined later. Then for any $\phi\in\sF(B)$, we have \begin{align}
&\phantom{\ \leq}\sE(aG^{B}[\one_{B}];\phi)+\lambda\int_{M}\abs{aG^{B}[\one_{B}]}^{p-2}\cdot aG^{B}[\one_{B}]\cdot \phi\dif m\\
&{=}a^{p-1}\int_{B}\phi\dif m+a^{p-1}\lambda\int_{B}G^{B}[\one_{B}]^{p-1}\phi\dif m \\
&\overset{\eqref{e.resol4}}{\leq}a^{p-1}\left(1+\lambda C^{p-1}W(B)\right)\int_{B}\phi\dif m. \label{e.resol5}
	\end{align}
	Let $a:=\left(1+\lambda C^{p-1}W(B)\right)^{-1/(p-1)}$. Then \eqref{e.resol5} gives \begin{align}
		&\phantom{\ \leq}\sE(aG^{B}[\one_{B}];\phi)+\lambda\int_{M}\abs{aG^{B}[\one_{B}]}^{p-2}\cdot aG^{B}[\one_{B}]\cdot \phi\dif m\\
		&\leq \sE(G^{B}_{\lambda}[\one_{B}];\phi)+\lambda\int_{M}\abs{G^{B}_{\lambda}[\one_{B}]}^{p-2}\cdot G^{B}_{\lambda}[\one_{B}]\cdot \phi\dif m.
	\end{align} 
	Since $G^{B}[\one_{B}]$ and $G^{B}_{\lambda}[\one_{B}]$ both belong to $\sF(B)$, we see that $(aG^{B}[\one_{B}]-G^{B}_{\lambda}[\one_{B}])_{+}\in\sF(B)$. By Proposition \ref{p.comp}, $G^{B}_{\lambda}[\one_{B}]\geq aG^{B}[\one_{B}]$ $m$-a.e. on $B$, which combines with \eqref{e.resol4} gives \eqref{e.resol3}.
	\end{enumerate}
\end{proof}

We are now ready to prove Theorem \ref{t.weh=>cs}.
\begin{proof}[Proof of Theorem \ref{t.weh=>cs}]
	Fix $(x_0,R,r)\in M\times (0,\infty)\times(0,\infty)$ such that $R+2r\in(0,\ol{R})$, and any three concentric balls\begin{equation}
		B_0:=B(x_0,R),\ B_1:=B(x_0,R+r),\ \text{and }\Omega:=B(x_0,R+2r),
	\end{equation} 
	Let $\lambda\in(0,\infty)$ to be determined later. Consider $G^{B_{1}}_{\lambda}[\one_{B_{1}}]\in\sF(B_{1})$. For any $x\in B_{0}$, consider $\emptyset\neq\wt{B}:=B(x,r)\subset B_{1}$. By Lemma \ref{l.exist-weak}-\ref{it.G2} and \ref{it.G3}, $G^{B_{1}}_{\lambda}[\one_{B_{1}}]\geq G^{\wt{B}}_{\lambda}[\one_{\wt{B}}]$ $m$-a.e. on $\wt{B}$. By Lemma \ref{l.E<>}-\ref{it.el<>}, we have \begin{equation}
		G^{\wt{B}}_{\lambda}[\one_{\wt{B}}]\geq C_{1}\left(\frac{W(x,r)}{C_{2}+\lambda W(x,r)}\right)^{1/(p-1)}\ m\text{-a.e. on $\delta \wt{B}$}.
	\end{equation}  Therefore, \begin{equation}\label{e.weh=>cs1}
		G^{B_{1}}_{\lambda}[\one_{B_{1}}]\geq \inf_{x\in B_{0}}C_{1}\left(\frac{W(x,r)}{C_{2}+\lambda W(x,r)}\right)^{1/(p-1)}=C_{1}\left(\frac{\inf_{x\in B_{0}} W(x,r)}{C_{2}+\lambda\inf_{x\in B_{0}}  W(x,r)}\right)^{1/(p-1)}\ m\text{-a.e. on }B_{0}.
	\end{equation}
	Take $\lambda:=(\inf_{x\in B_{0}} W(x,r))^{-1}$. By Lemma \ref{l.E<>}-\ref{it.el<>} and \eqref{e.weh=>cs1}, we know that there is a constant $c\in(0,1)$ such that \begin{equation}\label{e.weh=>cs1.1}
		0\leq G^{B_{1}}_{\lambda}[\one_{B_{1}}]\leq \inf_{x\in B_{0}} W(x,r)^{1/(p-1)}\ m\text{-a.e. on }B_{1}
	\end{equation}  \begin{equation}
		G^{B_{1}}_{\lambda}[\one_{B_{1}}]\geq c \inf_{x\in B_{0}} W(x,r)^{1/(p-1)}\ m\text{-a.e. on }B_{0}.
	\end{equation}
	Define \begin{equation}
		g:=\frac{G^{B_{1}}_{\lambda}[\one_{B_{1}}]}{c \inf_{x\in B_{0}} W(x,r)^{1/(p-1)}}\ \text{and }\phi:=g\wedge1\in\cutoff(B_{0},B_{1}).
	\end{equation}
	We will prove that $\phi$ is the suitable function for \eqref{CS}. In fact, for any $u\in\sF^{\prime}\cap L^{\infty}(M,m)$, by the proof of \cite[Proof of Proposition 3.1]{Yan25a}, we have \begin{equation}\label{e.weh=>cs2}
		\int_{\Omega}\abs{\wt{u}}^{p}\dif\Gamma^{(L)}\langle \phi\rangle\leq 2\sE^{(L)}(g;\abs{u}^{p}{g})+C(p)\int_{\Omega}\abs{\wt{g}}^{p}\dif\Gamma^{(L)}\langle u\rangle.
	\end{equation}
	By applying Lemma \ref{l.4real} with $a=\wt{u}(x)$, $b=\wt{u}(y)$, $c=\wt{g}(x)$ and $d=\wt{g}(y)$, we have \begin{align}
		&\phantom{\ \leq}\int_{\Omega}\abs{\wt{u}}^{p}\dif\Gamma^{(J)}_{\Omega}\langle \phi\rangle=\frac{1}{2}\int_{\Omega^{2}\setminus B_{0}^{2}}(\abs{\wt{u}(x)}^{p}+\abs{\wt{u}(y)}^{p})\abs{\wt{\phi}(x)-\wt{\phi}(y)}^{p}\dif j(x,y)\\
		&\leq \frac{1}{2}\int_{\Omega^{2}\setminus B_{0}^{2}}(\abs{\wt{u}(x)}^{p}+\abs{\wt{u}(y)}^{p})\abs{\wt{g}(x)-\wt{g}(y)}^{p}\dif j(x,y) \text{  (the map $z\mapsto z\wedge 1$ is $1$-Lipschitz)}\\
		&\overset{\eqref{e.4real}}{\leq }2\int_{\Omega^{2}\setminus B_{0}^{2}}\abs{\wt{g}(x)-\wt{g}(y)}^{p-2}(\wt{g}(x)-\wt{g}(y))(\abs{\wt{u}(x)}^{p}\wt{g}(x)-\abs{\wt{u}(y)}^{p}\wt{g}(y))\dif j(x,y)\\
		&\qquad +C(p)\int_{\Omega^{2}\setminus B_{0}^{2}}(\abs{\wt{g}(x)}^{p}+\abs{\wt{g}(x)}^{p})\abs{\wt{u}(x)-\wt{u}(y)}^{p}\dif j(x,y)\\
		&\leq 2\sE^{(J)}(g;\abs{u}^{p}{g})+C(p)\int_{\Omega}\abs{\wt{g}}^{p}\dif\Gamma^{(J)}\langle u\rangle,\label{e.weh=>cs3}
	\end{align}
	Combining \eqref{e.weh=>cs2} with \eqref{e.weh=>cs3}, we obtain 
	{\small 
	\begin{align}
		&\phantom{\ \leq}\int_{\Omega}\abs{\wt{u}}^{p}\dif\Gamma_{\Omega}\langle \phi\rangle=\int_{\Omega}\abs{\wt{u}}^{p}\dif\Gamma^{(L)}\langle \phi\rangle+\int_{\Omega}\abs{\wt{u}}^{p}\dif\Gamma^{(J)}_{\Omega}\langle \phi\rangle\\
		&\leq2\sE(g;\abs{u}^{p}{g})+C(p)\int_{\Omega}\abs{\wt{g}}^{p}\dif\Gamma_{\Omega}\langle u\rangle\\
		&\leq { \frac{C_{1}}{\inf_{x\in B_{0}}W(x,r)}\left(\sE(G^{B_{1}}_{\lambda}[\one_{B_{1}}];\abs{u}^{p}{g})+\lambda\int_{\Omega}\abs{G^{B_{1}}_{\lambda}[\one_{B_{1}}]}^{p-2}G^{B_{1}}_{\lambda}[\one_{B_{1}}]\cdot \abs{u}^{p}g\right)+C(p)\int_{\Omega}\abs{\wt{g}}^{p}\dif\Gamma_{\Omega}\langle u\rangle}\\
		&\overset{\eqref{e.weh=>cs1.1}}{\leq}\sup_{x\in \Omega}\frac{C_{1}c^{-1}}{W(x,r)}\int_{\Omega}\abs{u}^{p}\dif m+C(p)\int_{\Omega}\abs{\wt{g}}^{p}\dif\Gamma_{\Omega}\langle u\rangle,
	\end{align}
	}
	which completes the proof of \eqref{CS}.
\end{proof}

\section{Proof of the strong elliptic Harnack inequality}\label{s.Pf-EHI}
In this section, we shall derive the strong elliptic Harnack inequality \eqref{sEH} by combining a \emph{mean-value inequality} \eqref{MV} with the weak elliptic Harnack inequality \eqref{wEH}.

We first look at the simple case when $(\mathcal{E},\mathcal{F})$ is
strongly local. We will see that the weak elliptic Harnack inequality \eqref{wEH} is enough to derive the strong elliptic Harnack inequality \eqref{sEH} in this situation. Before that, we present a maximum principle, which is a consequence of the comparison principle presented in Proposition \ref{p.comp}.

\begin{proposition}[\text{c.f. \cite[Proposition 4.3]{GH14}}]\label{p.max}
Let \((\mathcal{E}, \mathcal{F})\) be a strongly local $p$-energy form. Let \(\Omega\) be a open subset of $M$ such that \(\lambda_{1}(\Omega) > 0\), and let \(A \subset \Omega\) be compact. Let \(0 \leq u \in \mathcal{F}(\Omega) \cap L^\infty\), where \(u\) is subharmonic in \(\Omega \setminus A\) and continuous in some neighborhood of \(\partial U\) for some open \(U\) with \(A \Subset U \Subset \Omega\). Then $\esup_{\Omega\setminus U} u = \sup_{\partial U} u$.
\end{proposition}

\begin{proof}
Since $\esup_{\Omega \setminus U} u \geq \sup_{\partial U} u $, assume on the contrary that $s := \sup_{\partial U} u < \esup_{\Omega\setminus U} u$. Then we can choose a small \(\epsilon > 0\) such that
\begin{equation}\label{e.esup}
\esup_{\Omega\setminus U} u \geq s + \varepsilon.
\end{equation}
Furthermore, by the continuity of $u$ near $\partial U$, we can select an open set \(V\) such that \(A \subset V \subset U\) and \(\sup_{U \setminus V} u \leq s + \epsilon/2\). Let \(\varphi\in\cutoff(V, U)\) and \(u^* := u - u\varphi\). Clearly, \(u^* \in \mathcal{F} \cap L^\infty(M,m)\), \(\restr{u^*}{V} = 0\) $m$-a.e., $u=u^{*}$ on $\Omega\setminus U$ and $u^* \leq u \leq s + \epsilon/2 \text{ in } U \setminus V$. Therefore, the function \(v := (u^* - (s + \epsilon/2))_+\) satisfies that \(\restr{v}{U} = 0\). Since \(v \in \mathcal{F}(\Omega)\), by \eqref{e.partEF}, we have that \(v \in \mathcal{F}(\Omega \setminus A)\). Since \(u\) is subharmonic in \(\Omega \setminus A\), we know that 
\begin{equation}\label{e.max2}
 \mathcal{E}\left(u; \left(u - u\varphi - \left(s + \frac{\epsilon}{2}\right)\right)_+\right) = \mathcal{E}(u; v) \leq 0.
\end{equation}
Combining the strong locality of \((\mathcal{E}, \mathcal{F})\), Proposition \ref{p.SL} and the fact that \(\varphi v = 0\), we have 
\begin{align*}
\mathcal{E}\left(u\varphi + s + \frac{\epsilon}{2}; \left(u - u\varphi - \left(s + \frac{\epsilon}{2}\right)\right)_+\right)  
= \mathcal{E}\left(u\varphi + s+ \frac{\epsilon}{2}; v\right)
= \mathcal{E}(u\varphi; v) = 0,
\end{align*}
combining which with \eqref{e.max2} gives
\[
\mathcal{E}\left(u\varphi + s+ \frac{\epsilon}{2}; \left(u - u\varphi - \left(s+ \frac{\epsilon}{2}\right)\right)_+\right) \geq \mathcal{E}\left(u; \left(u - u\varphi - \left(s + \frac{\epsilon}{2}\right)\right)_+\right).
\]
By applying Proposition \ref{p.comp}, we have $u\varphi + s + 2^{-1}\epsilon \geq u$ in $\Omega\setminus A$, which implies that $u^* = u - u\varphi \leq s + 2^{-1}\epsilon$ $m$-a.e. in $\Omega\setminus U$. This leads to a contradiction by $\eqref{e.esup}$ and the fact that $u=u^{*}$ on $\Omega\setminus U$.
\end{proof}

We say that condition \hypertarget{wEH0} $\hyperlink{wEH0}{(\mathrm{wEH}_0)}$ holds if \eqref{wEH} holds for any $u\in \mathcal{F}^{\prime }$ that is non-negative, \emph{harmonic} in $B_{R}$ (instead of being \emph{superharmonic}). We will show that $\hyperlink{wEH0}{(\mathrm{wEH}_0)}$ is equivalent to \eqref{sEH} whenever $(\mathcal{E},\mathcal{F})$ is strongly local. 

\begin{proposition}\label{P-sl}
Let \((\mathcal{E}, \mathcal{F})\) be a strongly local $p$-energy form, then
\begin{equation}
\eqref{wEH}\Longrightarrow\hyperlink{wEH0}{(\mathrm{wEH}_0)}\Longleftrightarrow \eqref{sEH}.
\end{equation}
\end{proposition}
\begin{proof}
The implications $\eqref{wEH}\Longrightarrow\hyperlink{wEH0}{(\mathrm{wEH}_0)}$ and $\eqref{sEH}\Longrightarrow \hyperlink{wEH0}{(\mathrm{wEH}_0)}$ are evident. It remains to show $\hyperlink{wEH0}{(\mathrm{wEH}_0)}\Longrightarrow \eqref{sEH}$. Since $(\mathcal{E},\mathcal{F})$ is strongly local and $u$ is
harmonic, non-negative in $B_{R}$, $\hyperlink{wEH0}{(\mathrm{wEH}_0)}$ becomes
\begin{equation}
\left( \fint_{B_{r}}u^{q}\dif m \right) ^{1/q}\leq C\einf_{{B}_{r}}u, \label{12-2}
\end{equation}
and \eqref{sEH} reads
\begin{equation}\label{e.sEH1}
 \esup_{B_r}u\leq C\einf_{B_r}u
\end{equation}
It suffices to show that $\eqref{12-2}\Longrightarrow\eqref{e.sEH1}$. In fact, by the proof of \cite[Theorem 1.9]{HY23}, \eqref{12-2} is equivalent to 
\begin{equation}
\einf_{B_{r}}u\geq a\exp \left( -\frac{C}{\omega _{B_{r}}(\{u\geq a\})}%
\right),\ \forall a\in(0,\infty) , \label{12-3}
\end{equation}
where 
\begin{equation*}
 \omega _{B_r}(\{u\geq a\}):=\frac{m (\{u\geq a\}\cap B_r)}{m (B_r)}
\end{equation*}
is the occupation density of the set $\{u\geq a\}$ in $B_r$. Furthermore, by combining the proof of \cite[from Corollary 7.3 to Theorem 7.8 on p.~1525-1535]{GHL15}, Proposition \ref{p.max} and the standard argument at the end of Subsection \ref{s.log}, we see that $\eqref{12-3}\Longrightarrow\eqref{e.sEH1}$.
\end{proof}

\subsection{Tail estimate}

To derive the mean-value inequality \eqref{MV}, we first need the following \emph{tail estimate controlled by the supremum on a ball} \eqref{TE}, which was partially inspired by \cite[Lemma 5.2]{KW24} and \cite{Che25}.
\begin{definition}
\label{def-te} We say that the condition \eqref{TE} is satisfied if
there exist three positive constants $C\in(1,\infty)$, $c_{1}\in(0,\infty)$ and $\sigma \in (0,1)$ such
that, for any two concentric balls $B_{R}=B(x_{0},R)$, $B_{R+r}=B(x_{0},R+r)$
with $0<R<R+r<\sigma \overline{R}$, for any function $%
u\in \mathcal{F}^{\prime }\cap L^{\infty }$ that is non-negative and superharmonic in $%
B_{R+r} $ with $\esup_{{B_{R+r}}}u>0$, we have 
\begin{equation}\label{TE}\tag{$\mathrm{TE}$}
T_{B_{R},B_{R+r}}(u_+)^{p-1}\leq C\left( \frac{R+r}{r}\right) ^{c_{1}}\left( \sup_{x\in B_{R+r}}\frac{1}{W(x,r)} \esup_{{B_{R+r}}}u^{p-1} +T_{B_{R+\frac{r}{2}},B_{R+r}}(u_-)^{p-1}\right).
\end{equation}%
where $B_{R+\frac{r}{2}}=B(x_{0},R+\frac{r}{2})$. We emphasize that the constants $C,c_{1},\sigma $ are all independent of $x_{0}, R, r$ and $u$.
\end{definition}
In the following Lemma \ref{L1}, the condition \eqref{UJS} is used to obtain \eqref{TE}.
\begin{lemma}
\label{L1}Assume that $(\mathcal{E},\mathcal{F})$ is a {mixed local and nonlocal $p$-energy form}, then
\begin{equation}
\eqref{VD}+\eqref{CS}+\eqref{TJ}+\eqref{UJS}\Longrightarrow \eqref{TE}.
\end{equation}
\end{lemma}
\begin{proof}
The proof is motivated by \cite[Lemma 4.2]{DCKP14} and \cite[Lemma 2.6]{CKW19}. As usual, we represent every function in $\mathcal{F}$ and in $\sF^{\prime}\cap L^{\infty}$ by its {$\sE$-quasi-continuous} version. Set $B_{r}:=B(x_{0},r)$ for a point $x_{0}\in M$ and any $r>0$. Let $u\in \mathcal{F}^{\prime }\cap L^{\infty }$ be a function
that is non-negative and superharmonic in $B_{R+r}$ for $0<R<R+r<\sigma 
\overline{R}$, where $\sigma \in (0,1)$ comes from condition \eqref{UJS}. Set $s:=\esup_{{B_{R+r}}}u>0$. Applying Lemma \ref{l.CS=>cap} to two concentric balls $B_{R+\frac{r}{4}}$ and $B_{R+\frac{r}{2}}$, there exists $\phi \in \text{cutoff}(B_{R+\frac{r}{4}},B_{R+\frac{r}{2}})$ such that
\begin{equation}
\mathcal{E}(\phi)\leq \sup_{x\in B_{R+3r/4}}\frac{C}{W(x,r/4)}m(B_{R+3r/4})
\leq \sup_{x\in B_{R+r}}\frac{C'}{W(x,r/2)}m(B_{R+r}).  \label{211-1}
\end{equation}
Let $w:=u-2s$. By Proposition \ref{p.mar}, $\phi ^{p}\in \sF\cap L^\infty$. Note that ${\phi ^{p}} =0 $  $\sE$-q.e. on $M\setminus B_{R+r}$, $\phi ^{p}\in \sF(B_{R+r})\cap L^\infty$. Since $u\in \mathcal{F}^{\prime }\cap L^\infty$, we know that $w\phi ^{p}\in \sF(B_{R+r})$ by Proposition \ref{p.A3-1}. Therefore, since $w\phi ^{p}=\phi^{p}\cdot (u-2s)\leq 0$ in $B_{R+r}$ and $u $ is superharmonic in $B_{R+r}$, we have by Theorem \ref{t.sasEM}-\ref{lb.M-local} that 
\begin{equation}
\mathcal{E}(w;w\phi ^{p})=\mathcal{E}(u-2s;w\phi ^{p})=\mathcal{E}(u;w\phi ^{p})\leq 0. \label{e.te-0}
\end{equation}

For the local term $\mathcal{E}^{(L)}(w;\phi^pw)$, we have
\begin{align}
\mathcal{E}^{(L)}(w;\phi^pw)&=\int_{M}\phi^p\dif\Gamma^{(L)}\la w\ra+p\int_{M} w\phi^{p-1}\dif \Gamma^{(L)}\la w;\phi\ra \text{\  (by Theorem \ref{t.sasEM}-\ref{lb.M-leibn}, \ref{lb.M-chain})}
\\
&\geq \frac{1}{2}\int_{M}\phi^p\dif\Gamma^{(L)}\la w \ra-2^{\frac{1}{p-1}}(p-1)\int_{M} |w|^p\dif \Gamma^{(L)}\la \phi \ra \text{\ (by Young's inequality)}\\
&\geq -2^{\frac{1}{p-1}}(p-1)s^p\mathcal{E}^{(L)}(\phi)\text{\quad (since $0\leq u\leq s$ in $B_{R+r}$)}.\label{e.te-1}
\end{align}

For the nonlocal term $\mathcal{E}^{(J)}(w;\phi^pw)$, by the symmetry of $j$, we have
\begin{align}
\mathcal{E}^{(J)}(w;\phi^pw)& =\int_{(B_{R+r})^{2}}\abs{w(x)-w(y)}^{p-2}(w(x)-w(y))(\phi^p(x)w(x)-\phi^p(y)w(y))\dif j(x,y)   \\
  &\quad  +2\int_{B_{R+r}\times B_{R+r}^{c}}\abs{w(x)-w(y)}^{p-2}(w(x)-w(y))w(x)\phi^p(x)\dif j(x,y)  
  \\
 &=:I_1+2I_2.\label{e.te-2}
\end{align}
\begin{itemize}
	\item 

For any $x,y\in B_{R+r}$, if $\phi(x)\geq \phi(y)$, 
\begin{align}
& \phantom{\ \leq}\abs{w(x)-w(y)}^{p-2}(w(x)-w(y))(\phi^p(x)w(x)-\phi^p(y)w(y))  \\
   & =\abs{w(x)-w(y)}^{p}\phi^p(x)+\abs{w(x)-w(y)}^{p-2}(w(x)-w(y))w(y)(\phi^p(x)-\phi^p(y)) \\
   & \geq \abs{w(x)-w(y)}^{p}\phi^p(x)-\abs{w(x)-w(y)}^{p-1}|w(y)| \phi(x)^{p-1}|\phi(x)-\phi(y)|\\
   &\geq \abs{w(x)-w(y)}^{p}\phi^p(x)-\left(\frac{p-1}{p}\abs{w(x)-w(y)}^{p} \phi(x)^{p}+\frac{1}{p}|w(y)|^p|\phi(x)-\phi(y)|^p\right) \\
   &{\geq} -\frac{2^p}{p}s^p|\phi(x)-\phi(y)|^p, \label{e.te-6+}
\end{align}
where in the second inequality we have used Young's inequality and in the last inequality we have used $0\leq u\leq s$ in $B_{R+r}$. If $\phi(y)> \phi(x)$, \eqref{e.te-6+} also holds by using the same argument. Therefore,
\begin{align}
  I_1  \geq -\frac{2^p}{p}\int_{(B_{R+r})^{2}}s^p\abs{\phi(x)-\phi(y)}^p \dif j(x,y) \geq -\frac{2^p}{p} s^p \mathcal{E}^{(J)}(\phi;\phi). \label{e.te-3}
\end{align}

\item For any $x\in B_{R+r}$ and $y\in B_{R+r}^{c}$, we first have \begin{equation}\label{e.te-3.1}
	u(y)^{p-1}\leq (s+(u(y)-s)_+)^{p-1}\leq 2^{p-1}((u(y)-s)_+^{p-1}+s^{p-1})
\end{equation}
and \begin{equation}\label{e.te-3.2}
	(u(x)-u(y))_+^{p-1}\leq (s-u(y))_+^{p-1}\leq (s+u_-(y))^{p-1}\leq 2^{p-1}(s^{p-1}+u_-(y)^{p-1})
\end{equation}
As a consequence,
\begin{align}
 &\phantom{\ \leq} \abs{w(x)-w(y)}^{p-2}(w(x)-w(y))w(x)=\abs{u(x)-u(y)}^{p-2}(u(y)-u(x))(2s-u(x))\\
 &=\abs{u(x)-u(y)}^{p-2}(u(y)-u(x))_+(2s-u(x))-\abs{u(x)-u(y)}^{p-2}(u(y)-u(x))_-(2s-u(x))\\
  & \geq s(u(y)-s)_+^{p-1}-2s(u(x)-u(y))_+^{p-1}\\
  & \overset{\eqref{e.te-3.1},\eqref{e.te-3.2}}{\geq} \frac{s}{2^{p-1}}u(y)^{p-1}-(2^p+1)s^p-2^psu_-(y)^{p-1}
  \label{e.te-3+}
\end{align}
 therefore
\begin{align}
  I_2 & \overset{\eqref{e.te-3+}}{\geq} \int_{B_{R+r}\times B_{R+r}^{c}}
  \left(\frac{s}{2^{p-1}}u_+(y)^{p-1}-(2^p+1)s^{p}-2^psu_-(y)^{p-1}\right)\phi^p(x)\dif j(x,y)\\
   & \geq \frac{s}{2^{p-1}}\int_{B_{R+\frac{r}{4}}\times B_{R+r}^{c}}u_+(y)^{p-1}\dif j(x,y)
   -(2^p+1)s^{p}\int_{B_{R+\frac{r}{2}}\times B_{R+r}^{c}}\dif j(x,y) \\
   &\qquad-2^ps\int_{B_{R+\frac{r}{2}}\times B_{R+r}^{c}}u_-(y)^{p-1}\dif j(x,y) \\
   & \overset{\eqref{TJ}}{\geq} \frac{s}{2^{p-1}}\int_{B_{R+\frac{r}{4}}\times B_{R+r}^{c}}u_+(y)^{p-1}\dif j(x,y)
   -(2^p+1)s^pC_{\mathrm{TJ}}\sup_{x\in B_{R+r}}\frac{m(B_{R+r})}{W(x,r/2)}\\
   &\qquad-2^{p}s\cdot m(B_{R+r})T_{B_{R+\frac{r}{2}},B_{R+r}}(u_-)^{p-1}. \label{e.te-4}
\end{align}
\end{itemize}
Combining \eqref{e.te-0}, \eqref{e.te-1}, \eqref{e.te-2}, \eqref{e.te-3} and \eqref{e.te-4}, we conclude that 
\begin{align}
0&\geq -2^{\frac{1}{p-1}}(p-1)s^p\mathcal{E}^{(L)}(\phi;\phi)-\frac{2^p}{p} s^p \mathcal{E}^{(J)}(\phi;\phi) +\frac{2s}{2^{p-1}}\int_{B_{R+\frac{r}{4}}\times B_{R+r}^{c}}u_+(y)^{p-1}\dif j(x,y)\\
&\qquad -2(2^p+1)C_{\mathrm{TJ}}s^p \sup_{x\in B_{R+r}}\frac{m(B_{R+r})}{W(x,r/2)}-2^{p+1}s\cdot m(B_{R+r})T_{B_{R+\frac{r}{2}},B_{R+r}}(u_-)^{p-1} \\
   &\geq -2^{\frac{p^2}{p-1}}ps^{p}\sE(\phi)+\frac{2s}{2^{p-1}}\int_{B_{R+\frac{r}{4}}\times B_{R+r}^{c}}u_+(y)^{p-1}\dif j(x,y)   \\
 &\qquad  -2(2^p+1)C_{\mathrm{TJ}}s^p \sup_{x\in B_{R+r}}\frac{m(B_{R+r})}{W(x,r/2)}
 -2^{p+1}s\cdot m(B_{R+r})T_{B_{R+\frac{r}{2}},B_{R+r}}(u_-)^{p-1}\\
   &\overset{\eqref{211-1}}{\geq}-C_1s^p\sup_{x\in B_{R+r}}\frac{m(B_{R+r})}{W(x,r/2)}+\frac{2s}{2^{p-1}}\int_{B_{R+\frac{r}{4}}\times B_{R+r}^{c}}u_+(y)^{p-1}\dif j(x,y)
   \\
   &\qquad\quad -2^{p+1}s\cdot m(B_{R+r})T_{B_{R+\frac{r}{2}},B_{R+r}}(u_-)^{p-1}.
\end{align}
Therefore,
\begin{equation}\label{e.te-5}
\int_{B_{R+\frac{r}{4}}\times (M\setminus B_{R+r})}u_+(y)^{p-1}\dif j(x,y)\leq C_2s^{p-1} \sup_{x\in B_{R+r}}\frac{m(B_{R+r})}{W(x,r)}+2^{2p-1}m(B_{R+r})T_{B_{R+\frac{r}{2}},B_{R+r}}(u_-)^{p-1}.
\end{equation}

On the other hand, by \eqref{VD}, 
\begin{equation}
\sup_{x\in B_{R}}\frac{m (B_{R+r})}{m (B(x,r/4))}\overset{\eqref{eq:vol_3}}{\leq}\sup_{x\in B_{R}}
C_{\mathrm{VD}}\left( \frac{d(x_{0},x)+R+r}{r/4}\right) ^{d_{2}}\leq C_{\mathrm{VD}
}8^{d_{2}}\left( \frac{R+r}{r}\right) ^{d_{2}}.  \label{215}
\end{equation}%
Therefore, noting that $u_{+}\geq 0$ globally in $M$, we have by \eqref{UJS}
that for any open set $U\subset B_{R}$, 
\begin{align}
&\phantom{\ \leq}\int_{U}\left( \int_{B_{R+r}^{c}}u_+(y)^{p-1}J(x,y)\dif m (y)\right) \dif m (x) \\
&\overset{\eqref{UJS}}{\leq} \int_{U}\left( \int_{B_{R+r}^{c}}u_+(y)^{p-1}\left( \frac{C}{m (B(x,r/4))}
\int_{B(x,r/4)}J(z,y)\dif m (z)\right) \dif m (y)\right) \dif m (x) \\
&\overset{\eqref{215}}{\leq} \int_{U}\frac{C_3}{m (B_{R+r})}\left( \frac{R+r}{r}\right)
^{d_{2}}\left( \int_{B_{R+r}^{c}}u_+(y)^{p-1}\left(
\int_{B_{R+r/4}}J(z,y)\dif m (z)\right) \dif m (y)\right) \dif m (x) \\
&=\int_{U}\frac{C_3}{m (B_{R+r})}\left( \frac{R+r}{r}\right)
^{d_{2}}\left( \int_{B_{R+r/4}\times B_{R+r}^{c}}u_+(y)^{p-1}\dif j(x,y)\right) \dif m (x),
\end{align}%
As $U$ is arbitrary, we know that for $m$-a.e. $x\in B_{R}$,
\begin{equation}
\int_{B_{R+r}^{c}}u_+(y)^{p-1}J(x,y)\dif m (y)\leq \frac{C_3}{m (B_{R+r})}%
\left( \frac{R+r}{r}\right) ^{d_{2}}\int_{B_{R+r/4}\times
B_{R+r}^{c}}u_+(y)^{p-1}\dif j(x,y).\label{e.te-6}
\end{equation}%
Hence, 
\begin{align}
T_{B_{R},B_{R+r}}(u_+)^{p-1}&=\esup_{x\in B_{R}}\int_{B_{R+r}^{c}}u_+(y)^{p-1}J(x,y)\dif m(y) \\
&\overset{\eqref{e.te-6}}{\leq} \frac{C_3}{m (B_{R+r})}\left( \frac{R+r}{r}\right)
^{d_{2}}\int_{B_{R+r/4}\times B_{R+r}^{c}}u_+(y)^{p-1}\dif j(x,y) \\
&\overset{\eqref{e.te-5}}{\leq} C_4\left( \frac{R+r}{r}\right)
^{d_{2}}\left( s^{p-1} \sup_{x\in B_{R+r}}\frac{1}{W(x,r)}+T_{B_{R+\frac{r}{2}},B_{R+r}}(u_-)^{p-1}\right),
\end{align}
thus showing \eqref{TE} with $c_1=d_2$, $C=C_4$.
\end{proof}

\subsection{A mean-value inequality}
In this subsection, we derive a mean-value inequality \eqref{MV} for any globally non-negative harmonic function in a ball, using condition \eqref{TE}. To do so, we apply the De Giorgi iteration technique to the $L^p$-norms of truncated functions of the form $(u - b)_+$ for real positive numbers $b$ (see, for example, \cite{GHH18}). Before proceeding, we give the definition of \eqref{MV}, a \emph{mean-value inequality} for any function that is non-negative and harmonic in a ball.

\begin{definition}
\label{def:MV} We say that condition \eqref{MV} holds if there exist
three constants $C,\theta \in[1,\infty)$ and $\sigma \in (0,1)$ such that, for any $%
\epsilon \in (0,1]$ and any function $u\in \mathcal{F}%
^{\prime }\cap L^{\infty }$ that is non-negative and harmonic in any ball $B_{R}:=B(x_{0},R)$
with $R\in(0,\sigma \overline{R})$, 
\begin{equation}
\esup_{\frac{1}{2}B_{R}}u\leq C\left(\epsilon ^{-\theta }\Big(\fint
_{B_{R}}u^{p}\dif m \Big)^{1/p}+\epsilon \Big(\esup_{B_{R}}u+W(x_{0},R)^{\frac{1}{p-1}}T_{\frac{15}{16}B_{R},B_{R}}(u_{-})\Big)\right).  \label{MV}\tag{$\operatorname{MV}$}
\end{equation}%
We emphasize that the constants $C,\theta ,\sigma $ are independent of $x_{0},\ R,\ \epsilon $, and $u$.
\end{definition}

\begin{lemma}
\label{lemma:MV2}Let $(\mathcal{E},\mathcal{F})$ be a mixed local and nonlocal $p$-energy form.
Then
\begin{equation}
\eqref{VD}+\eqref{CS}+\eqref{FK}+\eqref{TJ}+\eqref{TE} \Longrightarrow \eqref{MV}. 
\end{equation}
\end{lemma}
\begin{proof}
As usual, we represent every function in $\mathcal{F}$ and in $\sF^{\prime}\cap L^{\infty}$ by its {$\sE$-quasi-continuous} version. Let $u\in \mathcal{F}^{\prime }\cap L^{\infty }$ be a function that is
non-negative and harmonic in $B_{R}$ with $0<R<\sigma \overline{R}$, where constant $\sigma \in (0,1)$ is the
minimal of the same labels in conditions \eqref{TE} and \eqref{FK}. Without loss of
generality, assume that $\esup_{\frac{1}{2}B_{R}}u>0$.

Let $b>0$ be a number to be determined later, and set for $i\in\{0\}\cup\bN$,
\begin{align}
	b_{i}:=(1-2^{-i})b,\ r_{i}:=\frac{1}{2}(1+2^{-i})R,\ B_{i}:= B(x_{0},r_{i}),\\
\widehat{B}_i:=B\left( x_{0},\frac{3r_{i}+r_{i+1}}{4}\right) \text{ and }\widetilde{B}_{i}:=B\left( x_{0},\frac{r_{i}+r_{i+1}}{2}\right),
\end{align}
so that $b_{i}\leq b_{i+1}$, $r_{i}\geq r_{i+1}$ and $B_{i+1}\subset 
\widetilde{B}_{i}\subset B_{i}$. Define 
\begin{equation}
v_{i}:=\left( u-b_{i}\right) _{+}\text{ and } I_{i}:=\int_{B_{i}}v_{i}^{p}\dif m =\int_{B_{i}}\left( u-b_{i}\right)
_{+}^{p}\dif m,\ i\in\{0\}\cup\bN,
\end{equation}%
 so that $v_{i+1}\leq v_{i}$ in $M$ and $I_{i+1}\leq I_{i}$
for any $i\in\{0\}\cup\bN$. Our goal
is further to show that $I_{i+1}$ can be controlled by $I_{i}^{1+\nu }$ with $\nu $ given by \eqref{FK}.

Since $u\in \mathcal{F}^{\prime }\cap L^{\infty }$ is harmonic in $B_{R}$, it is also subharmonic, thus the function $v_{i}=(u-b_{i})_{+}\in 
\mathcal{F}^{\prime }\cap L^{\infty }$ is also subharmonic in $B_{R}$ by Lemma \ref{l.sub}.
Note that $v_{i+1}=(u-b_{i+1})_{+}=(u-b_{i}-(b_{i+1}-b_{i}))_{+}=(v_{i}-(b_{i+1}-b_{i}))_{+}$. Applying the Caccioppoli inequality in Proposition \ref{prop:cac} to the triple $B_{i+1}\subset \widetilde{B}_{i}\subset B_{i}$, we know that there exists some function $\phi \in \cutoff(B_{i+1},\widetilde{B}_{i})$ such that 
  \begin{equation}
    \int_{B_{i}} \dif \Gamma\la \phi v_{i+1}\ra \leq C\Big(\sup_{x\in B_{i}}\frac{1}{W(x,(r_{i}-r_{i+1})/2)}
   +\frac{T_{\widetilde{B}_{i},B_{i}}(v_{i+1})^{p-1}}{(b_{i+1}-b_{i})^{p-1}}\Big)\int_{B_{i}} v_i^p \dif m. \label{e.MV-2}
  \end{equation}
Note that for all $x\in B_{i}=B(x_{0},r_{i})$, by (\ref{eq:vol_0}) 
\begin{equation}
\frac{1}{W\left( x,(r_{i}-r_{i+1})/2\right) }=\frac{1}{W(x_{0},r_{i})}\cdot 
\frac{W(x_{0},r_{i})}{W\left( x,(r_{i}-r_{i+1})/2\right) }\leq \frac{C}{%
W(B_{i})}\left( \frac{2r_{i}}{r_{i}-r_{i+1}}\right) ^{\beta _{2}}.
\label{303}
\end{equation}
  
Since $\restr{(\phi\cdot v_{i+1})}{B_i^c}=0$, we have $\sE^{(J)}(\phi  v_{i+1})\leq 2\int_{B_i}\dif\Gamma^{(J)}\la \phi v_{i+1} \ra$ (see for example \cite[Page 40]{CKW21}). It follows from the locality of $\Gamma^{(L)}$ and \eqref{e.MV-2},
\begin{align}
  \sE(\phi v_{i+1})&= \sE^{(L)}(\phi v_{i+1})+ \sE^{(J)}(\phi v_{i+1})\leq 2 \int_{B_{i}} \dif \Gamma\la \phi v_{i+1}\ra \\
  &\leq   2C\Big(\sup_{x\in B_{i}}\frac{1}{W(x,(r_{i}-r_{i+1})/2)}
   +\frac{T_{\widetilde{B}_{i},B_{i}}(v_{i+1})^{p-1}}{(b_{i+1}-b_{i})^{p-1}}\Big)I_i.\label{e.MV-1}
\end{align}

Since $v_{i+1}\leq v_{i}=\left( u-b_{i}\right) _{+}\leq u$ in $M$ as $u\geq 0$ in $M$, we see that 
\begin{align}
T_{\widetilde{B}_{i},B_{i}}(v_{i+1})^{p-1}&=\esup_{x\in \widetilde{B}_{i}}\int_{B_{i}^{c}}v_{i+1}(y)^{p-1}J(x,\dif y) \leq  \esup_{x\in \widetilde{B}%
_{i}}\int_{B_{i}^{c}}u_+(y)^{p-1}J(x,\dif y)=T_{\widetilde{B}_{i},B_{i}}(u_+)^{p-1}   \\
& \overset{\eqref{TE}}{\leq} C\left( \frac{r_{i}}{(r_{i}-r_{i+1})/2}
\right) ^{d_{2}}\left( \sup_{x\in B_i} \frac{1}{W\left(
x,(r_{i}-r_{i+1})/2\right) } \esup_{B_{i}}u^{p-1}+T_{\widehat{B}_{i},B_{i}}(u_-)^{p-1} \right),  \label{30-30}
\end{align}
where we have used the inequality \eqref{TE} by noting that $\esup_{B_{i}}u\geq \esup_{\frac{1}{2}B_{R}}u>0$ (also recall in the last line of the proof of Lemma \ref{L1}, the constant $c_{1}$ in \eqref{TE} can be choose as $d_{2}$ appeared in \eqref{eq:vol_3}). Therefore, by \eqref{e.MV-1},
\eqref{30-30} and \eqref{303},
\begin{align}
  \sE(\phi v_{i+1})&\leq \frac{C}{W\left( B_{i}\right) }\left( \frac{r_{i}}{
r_{i}-r_{i+1}}\right) ^{d_{2}+\beta_2}\left(1+
\frac{\esup\limits_{B_i} u^{p-1}+W(B_{i})T_{\widehat{B}_{i},B_{i}}(u_{-})^{p-1}}{(b_{i+1}-b_{i})^{p-1}}\right)I_i\\
&\leq \frac{C}{W\left( B_{i}\right) }\left( \frac{r_{i}}{
r_{i}-r_{i+1}}\right) ^{d_{2}+\beta_2}\left(1+\frac{\esup\limits_{B_R} u^{p-1}+W(B_{R})T_{\frac{15}{16}B_{R},B_{R}}(u_{-})^{p-1}}{(b_{i+1}-b_{i})^{p-1}}\right)I_i, \label{e.MV-3}
\end{align}
where in the last inequality we have used the fact that ${(3r_{i}+r_{i+1})}/{4}\leq {(3R+\frac{3}{4}R)}/{4}=\frac{15}{16}R$ and $u$ is non-negative in $B_R$.

On the other hand, observe that 
\begin{align}
I_{i} &=\int_{B_{i}}v_{i}^{p}\dif m  \geq \int_{B_{i}\cap \{u>b_{i+1}\}}(u-b_{i})_{+}^{p}\dif m \geq
(b_{i+1}-b_{i})^{p}m (D_{i}),\label{e.MV3+}
\end{align}
where $D_{i}:=B_{i}\cap \{u>b_{i+1}\}$. By the outer regularity of $m $,
for any $\epsilon_0 >0$, there exists an open set $\Omega _{i}$ such that $%
D_{i}\subset \Omega _{i}\subset B_{i}$ and $m (\Omega _{i})\leq m (D_{i})+\epsilon_0$. Thus, 
\begin{equation}
m (\Omega _{i})\leq m(D_{i})+\epsilon_0\overset{\eqref{e.MV3+}}{\leq} \frac{I_{i}}{%
(b_{i+1}-b_{i})^{p}}+\epsilon_0 .  \label{305}
\end{equation}

Since $\phi =0$ $\sE$-q.e. outside $\widetilde{B}_{i}\subset B_{i}$ and $v_{i+1}=0$
outside $\{u>b_{i+1}\}$, we see that $v_{i+1}\phi =0$ $\sE$-q.e. in $D_{i}^{c}\supset \Omega
_{i}^{c}$. By Proposition \ref{p.A3-1}, we have $v_{i+1}\phi \in \mathcal{F}(\Omega _{i})$. Note that $\phi =1$ in $B_{i+1}$, we have
\begin{align}
I_{i+1}&=\int_{B_{i+1}}v_{i+1}^{p}\dif m =\int_{B_{i+1}\cap
\{u>b_{i+1}\}}(v_{i+1}\phi )^{p}\dif m \\
 &\leq \int_{D_{i}}(v_{i+1}\phi )^{p}\dif m \leq \int_{\Omega
_{i}}(v_{i+1}\phi )^{p}\dif m \overset{\eqref{FK}}{\leq}C \mathcal{E}(v_{i+1}\phi )W(B_{i})\left( \frac{m
(B_{i})}{m (\Omega _{i})}\right) ^{-\nu } \\
&\overset{\eqref{305}}{\leq}C \mathcal{E}(v_{i+1}\phi )W(B_{i})\left( \frac{(b_{i+1}-b_{i})^{-p}I_{i}+\epsilon_0 }{m (B_{i})}%
\right) ^{\nu}.  \label{eq:I<}
\end{align}
Letting $\epsilon_0 \downarrow 0$ in \eqref{eq:I<} and using the fact that $m (B_{i})\geq
m (B_{R/2})\overset{\eqref{VD}}{\geq} C_{\mathrm{VD}}^{-1}m(B_{R})$, we obtain 
\begin{equation}
I_{i+1} \leq \frac{C}{(b_{i+1}-b_{i})^{p\nu }}\left( \frac{I_{i}}{m
(B_{i})}\right) ^{\nu } W(B_{i})\mathcal{E}(v_{i+1}\phi )  \leq \frac{C'}{(b_{i+1}-b_{i})^{p\nu }}\left( \frac{I_{i}}{m (B_{R})}%
\right) ^{\nu }W(B_{i}) \mathcal{E}(v_{i+1}\phi ),  \label{MV1}
\end{equation}
where $C'>0$ is a constant depending only on the constants from condition \eqref{FK}.

Therefore, using the fact that
\begin{equation}
\frac{r_{i}}{r_{i}-r_{i+1}}=\frac{(1+2^{-i})R/2}{2^{-i-2}R}\leq 2^{i+2}\text{
\ and \ }b_{i+1}-b_{i}=2^{-i-1}b\text{\ \ for }i\in\{0\}\cup\bN,\label{e.MV5}
\end{equation}
we obtain by \eqref{MV1} and \eqref{e.MV-3} that 
\begin{align}
I_{i+1}& \leq \frac{C}{(b_{i+1}-b_{i})^{p\nu }}\left( \frac{I_{i}}{m
(B_{R})}\right) ^{\nu }\left( \frac{r_{i}}{
r_{i}-r_{i+1}}\right) ^{d_{2}+\beta_2} \left(1+\frac{\esup\limits_{B_R} u^{p-1}+W(B_{R})T_{\frac{15}{16}B_{R},B_{R}}(u_{-})^{p-1}}{(b_{i+1}-b_{i})^{p-1}}\right)I_i\\
&=\frac{C}{(b_{i+1}-b_{i})^{p\nu }m (B_{R})^{\nu }}\left( \frac{r_{i}
}{r_{i}-r_{i+1}}\right) ^{d_{2}+\beta_2}\left(1+\frac{\esup\limits_{B_R} u^{p-1}+W(B_{R})T_{\frac{15}{16}B_{R},B_{R}}(u_{-})^{p-1}}{(b_{i+1}-b_{i})^{p-1}}\right) I_{i}^{1+\nu }  \\
&\overset{\eqref{e.MV5}}{\leq} 
\frac{C^{\prime }2^{(p\nu
+d_{2}+\beta_2+p-1)(i+1)}}{b^{p\nu }m (B_{R})^{\nu }}\left(1+{b^{-p+1}}\left({\esup\limits_{x\in B_R} u^{p-1}+W(B_{R})T_{\frac{15}{16}B_{R},B_{R}}(u_{-})^{p-1}}\right)\right) I_{i}^{1+\nu } \\
& := D\cdot \lambda ^{i+1}\cdot I_{i}^{1+\nu },  \label{306}
\end{align}
where $C^{\prime }>0$ depends only on the constants from the hypotheses, and
the constants $D$, $\lambda $ are respectively given by 
\begin{equation}
D:=\frac{C^{\prime }}{b^{p\nu }m (B_{R})^{\nu }}\left(1+{b^{-p+1}}\left({\esup\limits_{x\in B_R} u^{p-1}+W(B_{R})T_{\frac{15}{16}B_{R},B_{R}}(u_{-})^{p-1}}\right)\right),\ \lambda :=2^{p\nu+d_{2}+\beta_2+p-1}.
\label{e.74}
\end{equation}%
Applying Proposition \ref{prop:A3}, we have $I_{i}\leq D^{-\frac{1}{\nu }}\left( D^{\frac{1}{\nu }}\lambda ^{\frac{1+\nu 
}{\nu ^{2}}}I_{0}\right) ^{(1+\nu )^{i}}$, from which we know that
\begin{equation}\label{307}
	\esup_{B_{R/2}}u\leq b  {\ \Longleftarrow\ } \lim_{i\rightarrow \infty }I_{i}=0{\ \Longleftarrow\ }D^{\frac{1}{\nu }}\lambda ^{\frac{1+\nu }{\nu ^{2}}}I_{0}\leq \frac{1}{2}{\ \Longleftrightarrow\ } I_{0}\leq \frac{1}{2}D^{-\frac{1}{\nu }}\lambda ^{-\frac{1+\nu }{\nu ^{2}}},
\end{equation}
where \begin{equation}
	\frac{1}{2}D^{-\frac{1}{\nu }}\lambda ^{-\frac{1+\nu }{\nu ^{2}}}=\frac{1}{2}\lambda^{-\frac{1+\nu}{\nu^{2}}}b^{p}m(B_{R})(C^{\prime})^{-\frac{1}{\nu}}
\left(1+{b^{-p+1}}\left({\esup\limits_{x\in B_R} u^{p-1}+W(B_{R})T_{\frac{15}{16}B_{R},B_{R}}(u_{-})^{p-1}}\right)\right)^{-\frac{1}{\nu }}.
\end{equation}

Now choosing \begin{equation}
	b=\epsilon \left(\esup_{B_{R}}u+W(B_{R})^{\frac{1}{p-1}}T_{\frac{15}{16}B_{R},B_{R}}(u_{-})\right)+\left(\frac{C^{^{\prime \prime }}I_{0}}{m (B_{R})}\right)^{\frac{1}{p}}
\end{equation} 
with constant $C^{\prime \prime }=2\lambda ^{\frac{1+\nu }{\nu ^{2}}%
}(C^{^{\prime }}(1+2\epsilon ^{1-p}))^{\frac{1}{\nu }}$, we know that the last condition in \eqref{307} is guaranteed and therefore, 
\begin{equation}
\esup_{B_{R/2}}u\leq b\leq C\left( \epsilon ^{-\frac{p-1}{p\nu}}\left( \fint%
_{B_{R}}u^{p}\dif m \right) ^{1/p}+\epsilon \left(\esup_{B_{R}}u+W(B_{R})^{\frac{1}{p-1}}T_{\frac{15}{16}B_{R},B_{R}}(u_{-})\right)\right) ,
\end{equation}%
where $C$ is some positive constant independent of $x_{0},\ R,\ \epsilon $, and $u$. So \eqref{MV} holds
with $\theta =\frac{p-1}{p\nu}$.
\end{proof}

\subsection{From \texorpdfstring{(\ref{wEH})}{(wEH)} to \texorpdfstring{(\ref{sEH})}{(sEH)}}
\label{sec:EHI}
In this subsection, we derive \eqref{sEH} using the conditions \eqref{MV} and \eqref{wEH}.

\begin{lemma}
\label{L32}
Let $(\mathcal{E},\mathcal{F})$ be a mixed local and nonlocal $p$-energy form. Then%
\begin{equation}
\eqref{VD}+\eqref{MV}+\eqref{wEH}\Longrightarrow \eqref{sEH}.  \label{308}
\end{equation}%
\end{lemma}

\begin{proof}
The proof is motivated by those in \cite[Theorem 2.8]{CKW19}, \cite[p.~1828-1829]{DCKP14} and \cite[Lemma 9.2]{GHL15}. Without loss of generality, assume that the constants $\sigma $ appearing in
conditions $\eqref{MV}$ and \eqref{wEH} are the same, otherwise we
take the minimum of them. First, we prove \eqref{sEH} for $u\in \mathcal{F^{\prime }}\cap L^{\infty }$. Let $u\in \mathcal{F}^{\prime }\cap L^{\infty }$ be non-negative and
harmonic in $B(x_{0},R)$ with $0<R<\sigma \overline{R}$. 

For any $r >0$, denote by $B_{r }:=B(x_{0},r )$. Firstly, by \eqref{wEH},
\begin{equation}
\left( \fint_{B_{r}}u^{q}\dif m \right) ^{1/q}\leq c_{1}\left(\einf_{B_r}u+W(B_{r})^{\frac{1}{p-1}}T_{\frac{1}{2}B_{R},B_{R}}(u_{-})\right),\ \forall r\in(0,\delta R]  \label{313}
\end{equation}%
for some constants $\delta, q\in (0,1)$ and $c_1\geq 1$. Without loss of generality, we can assume $\delta\in (0,1/2)$. Let $\epsilon \in (0,1)$ be a number to be specified and $r\in(0,\delta R]$. Set $\rho =(\chi -\chi ^{\prime })r$ with $1/2\leq \chi ^{\prime }<\chi\leq 1$ be any two numbers. For any point $z\in B_{\chi ^{\prime }r}$, applying \eqref{MV} for $B(z,{(\chi -\chi ^{\prime })r})\subset B(x_{0},{\chi r})=B_{\chi r}$, 
{\small
\begin{align}
 &\phantom{\ \leq}c_{2}^{-1}\esup_{B(z,(\chi -\chi ^{\prime })r/2)}u\\
& \leq \epsilon
^{-\theta }\left( \fint_{B(z,(\chi-\chi ^{\prime })r)}u^{p}\dif m \right) ^{1/p}+\epsilon \left(\esup_{B(z,(\chi-\chi ^{\prime })r)}u+W(z,(\chi -\chi ^{\prime })r)^{\frac{1}{p-1}}T_{B(z,\frac{15(\chi-\chi ^{\prime })r}{16}),B(z,(\chi-\chi ^{\prime })r)}(u_{-})\right).
\end{align}
}
By using the facts that%
\begin{equation}
\frac{m(B(x_{0},\chi r))}{m(B(z,(\chi -\chi ^{\prime })r))}\overset{\eqref{VD}}{\leq} C_{\mathrm{VD}}\left( 
\frac{d(x_{0},z)+\chi r}{(\chi -\chi ^{\prime })r}\right) ^{d_{2}}\leq
C_{\mathrm{VD}}\left( \frac{\chi ^{\prime }+\chi }{\chi -\chi ^{\prime }}\right)
^{d_{2}}\leq 2^{d_{2}}C_{\mathrm{VD}}\left( \frac{1}{\chi -\chi ^{\prime }}\right)
^{d_{2}},
\end{equation}
and
\begin{equation}
	\frac{W(z,(\chi -\chi ^{\prime })r)}{W(x_{0},r)}\leq \frac{W(z,r)}{W(x_{0},r)}\overset{\eqref{eq:vol_0}}{\leq} C\ \text{and } T_{B(z,\frac{15(\chi-\chi ^{\prime })r}{16}),B(z,(\chi-\chi ^{\prime })r)}(u_{-})\leq T_{B_{\chi r},B_{R}}(u_{-}),
\end{equation}
and by varying point $z$ in the ball $B_{\chi ^{\prime }r}$, we have
\begin{equation}
\esup_{B_{\chi ^{\prime }r}}u\leq c_{3}\left( \frac{\epsilon ^{-\theta }}{%
(\chi -\chi ^{\prime })^{d_{2}/p}}\left( \fint_{B_{\chi r}}u^{p}\dif m \right)
^{1/p}+\epsilon \left(\esup_{B_{\chi r}}u+W(B_{r})^{\frac{1}{p-1}} T_{B_{\chi r},B_{R}}(u_{-})\right)\right).  \label{311}
\end{equation}%

Using Young's inequality, we know that
\begin{align}
\left( \fint_{B_{\chi r}}u^{p}\dif m \right) ^{1/p} \leq \big(\esup_{B_{\chi r}}u\big)^{\frac{p-q
}{p}}\left( \fint_{B_{\chi r}}u^{q}\dif m \right) ^{1/p}   
\leq \eta \esup_{B_{\chi r}}u+\eta ^{-\frac{p-q}{p}}\left( \fint_{B_{\chi
r}}u^{q}\dif m \right) ^{1/q}.  \label{31-2}
\end{align}%
Now choosing $\epsilon =({4c_{3}})^{-1}$ in \eqref{311} and choosing the
number $\eta $ by 
\begin{equation}
\eta =\frac{(\chi -\chi ^{\prime })^{d_{2}/p}}{4c_{3}\epsilon ^{-\theta }}%
\Longleftrightarrow \frac{c_{3}\epsilon ^{-\theta }}{(\chi -\chi
^{\prime })^{d_{2}/p}}\eta =\frac{1}{4},\label{e.choseeta}
\end{equation}%
it follows that for any $1/2\leq \chi ^{\prime }<\chi \leq 1$, 
\small 
\begin{align}
&\phantom{\ \leq}\esup_{B_{\chi ^{\prime }r}}u\\
& \overset{\eqref{311},\eqref{31-2}}{\leq} \frac{c_{3}\epsilon ^{-\theta }}{%
(\chi -\chi ^{\prime })^{d_{2}/p}}\left( \eta \esup_{B_{\chi r}}u+\eta
^{-(p-q)/q}\left( \fint_{B_{\chi r}}u^{q}\dif m \right) ^{1/q}\right) +\frac{1}{4}\left(\esup_{B_{\chi r}}u+W(B_{r})^{\frac{1}{p-1}} T_{B_{\chi r},B_{R}}(u_{-})\right)   \\
&\overset{\eqref{e.choseeta}}{\leq}\frac{1}{2}\left(\esup_{B_{\chi r}}u+W(B_{r})^{\frac{1}{p-1}} T_{B_{\chi r},B_{R}}(u_{-})\right) +\frac{c_{3}\epsilon ^{-\theta }}{(\chi
-\chi ^{\prime })^{d_{2}/p}}\left( \frac{4c_{3}\epsilon ^{-\theta }}{(\chi -\chi ^{\prime })^{d_{2}/p}}\right) ^{(p-q)/q}\left( \fint_{B_{\chi
r}}u^{q}\dif m \right) ^{1/q}   \\
& \overset{\eqref{VD}}{\leq} \frac{1}{2}\left(\esup_{B_{\chi r}}u+W(B_{r})^{\frac{1}{p-1}} T_{B_{\chi r},B_{R}}(u_{-})\right) +\frac{c_{4}}{(\chi -\chi ^{\prime
})^{d_{2}/q}}\left( \fint_{B_{r}}u^{q}\dif m \right) ^{1/q}  \label{31-3}
\end{align}
\normalsize
for some constant $c_{4}$ independent of $\chi ,\chi ^{\prime
},x_0,r, R$, and $u$.

Set $\left\{ \chi_{k}\right\}
_{k=0}^{\infty }$ be a sequence of positive numbers given by $\chi_{k}:=1-\frac{1}{2}\chi ^{k}$, $k\in\{0\}\cup\bN$, with $\chi =\left(3/4\right)^{q/d_{2}}$. Clearly, $\left\{ \chi_{k}\right\}
_{k=0}^{\infty }$ is non-decreasing, $\chi _{0}=1/2$, $\chi
_{k}\uparrow 1$ as $k\uparrow \infty$ and $\chi _{k+1}-\chi _{k}=\frac{1}{2}(1-\chi )\chi ^{k}$. Applying \eqref{31-3} with $\chi ^{\prime }=\chi _{k}$ and $\chi =\chi
_{k+1} $, we obtain for all $k\in\{0\}\cup\bN$,
\begin{align}
	\esup_{B_{\chi _{k}r}}u &\leq \frac{1}{2}\left(\esup_{B_{\chi_{k+1} r}}u+W(B_{r})^{\frac{1}{p-1}} T_{B_{\chi_{k+1} r},B_{R}}(u_{-})\right)+\frac{c_{4}}{(\chi
_{k+1}-\chi _{k})^{d_{2}/q}}\left( \fint_{B_{r}}u^{q}\dif m \right) ^{1/q}\\
&=Aa^{-k}+\frac{1}{2}\esup_{B_{\chi_{k+1}r}}u+T, \label{e.ietEHI}
\end{align}

where \begin{equation}
	a:=\chi ^{d_{2}/q}=3/4,\ A:=c_{4}(%
{2}/{(1-\chi)})^{d_{2}/q}\left( \fint_{B_{r}}u^{q}\dif m \right) ^{1/q}  \text{ and } T:=\frac{1}{2}W(B_{r})^{\frac{1}{p-1}} T_{B_{r},B_{R}}(u_{-})
\end{equation} Therefore,
\begin{align}
\esup_{\frac{1}{2}B_{r}}u 
&=\esup_{B_{\chi_{0}r}}u\overset{\eqref{e.ietEHI}}{\leq} A+\frac{1}{2}\esup_{B_{\chi _{1}r}}u+T
\overset{\eqref{e.ietEHI}}{\leq} A+\frac{1}{2}\left( Aa^{-1}+\frac{1}{2}\esup_{B_{\chi _{2}r}}u+T\right)+T \\
&=A\left( 1+(2a)^{-1}\right) +2^{-2}\esup%
_{B_{\chi _{2}r}}u+T(1+2^{-1})\\
&\overset{\eqref{e.ietEHI}}{\leq} \cdots \overset{\eqref{e.ietEHI}}{\leq} A \sum_{i=0}^{k}(2a)^{-i}+\left( \frac{1}{2}%
\right) ^{k+1}\esup_{B_{\chi _{k+1}r}}u+T\sum_{i=0}^{k}2^{-i} \\
&\leq \frac{A}{1-(2a)^{-1}}+\left( \frac{1}{2}\right) ^{k+1}\norm{u}_{L^{\infty}(M,m)}+2T\\
&\rightarrow \frac{A}{1-(2a)^{-1}}+2T=C\left( \fint_{B_{r}}u^{q}\dif m \right)^{1/q}+W(B_{r})^{\frac{1}{p-1}} T_{B_{r},B_{R}}(u_{-})\label{e.pEHI2}
\end{align}
as $k\rightarrow \infty $, where $C$ is a positive constant independent of $x_{0}, R ,r$, and $u$. Therefore, we conclude by using \eqref{313} that for any $r\in(0,\delta R]$,
\begin{align}
\esup_{\frac{1}{2}B_{r}}u&\overset{\eqref{e.pEHI2}}{\leq} C\left( \fint_{B_{r}}u^{q}\dif m \right)
^{1/q}+W(B_{r})^{\frac{1}{p-1}} T_{B_{r},B_{R}}(u_{-})\\
&\overset{\eqref{313}}{\leq} C^{\prime }\left(\einf_{B_{r}}u+W(B_{r})^{\frac{1}{p-1}} T_{\frac{1}{2}B_{ R},B_{R}}(u_{-})\right)\\
&\overset{\eqref{eq:vol_0}}{\leq} C^{\prime \prime }\left(\einf_{\frac{1}{2}B_{r}}u+W\left(\frac{1}{2}B_{r}\right)^{\frac{1}{p-1}} T_{\frac{1}{2}B_{ R},B_{R}}(u_{-})\right),
\end{align}%
thus showing condition \eqref{sEH} for $u\in \mathcal{F^{\prime }}\cap L^{\infty }$ by renaming $r/2$ by $r$. For general $u\in \mathcal{F^{\prime }}$, the proof follows an argument similar to the one presented at the end of the proof of Theorem \ref{t.main}-\eqref{e.main-wEH}. Hence, the implication \eqref{308} follows.
\end{proof}
\begin{proof}[Proof of Theorem \ref{t.main}-\eqref{e.main-EHI}]
We know that under \eqref{VD} and \eqref{RVD}, 
\begin{align}
\eqref{CS}+\eqref{PI}+\eqref{TJ}\ &\Longrightarrow \eqref{wEH}%
\text{\quad (by Theorem \ref{t.main}-\eqref{e.main-wEH}),} \\
\eqref{CS}+\eqref{PI}+\eqref{TJ}+\eqref{UJS} &\Longrightarrow \eqref{MV}\text{\quad (by Proposition \ref{p.PI=>FK} and Lemma }\ref{lemma:MV2}),
\end{align}%
combining this with \eqref{308}, we obtain \eqref{sEH}. 
\end{proof}

\section{Examples}\label{s.examples}
In this section, we give three examples to illustrate Theorem \ref{t.main}. Before we start, we give a proposition on constructing a purely nonlocal $p$-energy form.

\begin{proposition}\label{p.NL=>Sub}
Let $(M,d,m)$ be a metric measure space. Suppose we given a kernel $J(x,E)$ for $x\in M$, $E\in\sB(M)$, on $M\times \sB(M)$ that satisfies the following two conditions\begin{enumerate}[label=\textup{({\alph*})},align=left,leftmargin=*,topsep=5pt,parsep=0pt,itemsep=2pt]
	\item\label{it.NLJ1} For any $r\in(0,\infty)$, the function $x\mapsto J(x,B(x,r)^{c})$ is in $L^{1}_{\loc}(M,m)$;
	\item\label{it.NLJ2} For any two non-negative Borel measurable functions $u$ and $v$, \begin{equation}\label{e.JSymUlt}
			\int_{M}\int_{M}u(x)v(y)J(x,\dif y)\dif m(x)=\int_{M}\int_{M}u(y)v(x)J(x,\dif y)\dif m(x).
		\end{equation}
\end{enumerate}
For such $J$, we define 
	\begin{equation}
		\sE(u):=\int_{M}\int_{M}\abs{u(x)-u(y)}^{p}J(x,\dif y)\dif m(x),\ \forall u\in L^{p}(M,m),\label{e.BDNL}
	\end{equation}
 and \begin{align}
		\sD(\sE)&:=\Sett{u\in L^{p}(M,m)}{\int_{M}\int_{M}\abs{u(x)-u(y)}^{p}J(x,\dif y)\dif m(x)<\infty}\\
		\sF&:=\text{the closure of }\{u\in C_{c}(M):\mathcal{E}(u)<\infty \}\text{ under the norm }\mathcal{E}_{1}^{1/p}.
	\end{align} Then we have 
	\begin{enumerate}[label=\textup{({\arabic*})},align=left,leftmargin=*,topsep=5pt,parsep=0pt,itemsep=2pt]
\item \label{cl-1} $\sF\subset \sD(\sE)$, and the pairs $(\sE,\sD(\sE))$ and $(\sE,\sF)$ are closed Markovian $p$-energy forms on $(M,m)$ that satisfies \eqref{Sub}. 
\item \label{cl-2} If we further assume that $J$ satisfies that \begin{equation}\label{e.NLJ3}
	\text{there exists $r>0$ such that the function }	x\mapsto \int_{B(x,r)}{d(x,y)^{p}J(x,\dif y)}\text{ is in $L^{1}_{\loc}(M,m)$ }
	\end{equation} then $(\sE,\sF)$ is regular. In particular, $(\sE,\sF)$ is a purely-nonlocal $p$-energy form on $(M,m)$.
\end{enumerate}
\end{proposition}
\begin{proof}
\begin{enumerate}[label=\textup{({\arabic*})},align=left,leftmargin=*,topsep=5pt,parsep=0pt,itemsep=2pt]
	\item[\ref{cl-1}] 
	That $(\sE,\sD(\sE))$ being a well-defined, closed Markovian  $p$-energy form on $(M,m)$ follows from a direct modification of \cite[Example 1.2.4]{FOT11}. To show \eqref{Sub}, we only need to prove
\begin{equation}
  |u\vee v(x) - u\vee v(y)|^p+ |u\wedge v(x) - u\wedge v(y)|^p\leq |u(x)-u(y)|^p  + |v(x)-v(y)|^p \label{e.E-14} 
\end{equation}
for any two functions $u,v\in \sD(\sE)$ and any two points $x,y\in M$.

 Without loss of generality, assume $ u(x)=\max\{u(x),v(x),u(y),v(y)\}$.
 \begin{enumerate}[label=\textup{({\roman*})},align=left,leftmargin=*,topsep=5pt,parsep=0pt,itemsep=2pt]
	\item If $u(y) \geq v(x)$, then both sides of \eqref{e.E-14} are equal.
	\item If $ u(y) < v(x) $, and let \( a = u(x), b = u(y), c = v(x), d = v(y). \)
    Specifically, it suffices to show
\[ |a-d|^p + |c-b|^p \leq |a-b|^p + |c-d|^p \qquad (a\geq c\vee d,\; b<d), \]
that is 
\begin{equation}
  f(c)\leq f(a) \text{ for $a\geq c\vee d,\; b<d$}, \label{e.E-15}
\end{equation}
where $f(x):=|x-b|^p-|x-d|^p$. Since $f^{\prime}(x)=p|x-b|^{p-1} \sgn(x-b)-p|x-d|^{p-1} \sgn(x-d)\geq 0$, we know that $f$ is non-decreasing on $\mathbb{R}$. Therefore \eqref{e.E-15} holds, which implies \eqref{e.E-14}.
\end{enumerate}

Since $(\sD(\sE),\sE_{1}^{1/p})$ is a Banach space by the closeness of $(\sE,\sD(\sE))$, we see that $\sF$ is a linear subspace of $\sD(\sE)$. Therefore $(\sE,\sF)$ is also closed and Markovian, and satisfies \eqref{Sub}, thus showing \ref{cl-1}. 

\item[\ref{cl-2}]  Suppose that $J$ further satisfies \eqref{e.NLJ3}. We are going to prove the regularity of $(\sE,\sF)$. It suffices to prove that $\mathcal{F}\cap C_{c}(M)$ is dense in $C_{c}(M)$ under the uniform norm. For this, for $w\in M$, $r\in (0,1)$ and $R\in (0,\infty)$, define
\begin{equation}
\phi_{w,r,R} (z)\coloneqq\frac{2}{r}\left( R+\frac{r}{2}-d(z,w) \right)
_{+}\wedge 1,\ z\in M.  \label{phi1-}
\end{equation}
We will prove that \begin{equation}\label{e.ex3-co-}
	\Sett{\phi_{w,r,R}}{w\in M,\ r\in (0,1),\ R\in (0,\infty)}\subset \cutoff(B(w,R),B(w,R+r)).
\end{equation}
Fix $w\in M$, $r\in (0,1)$ and $R\in (0,\infty)$. For the simplicity of notation, we denote $\phi:=\phi_{w,r,R}$.
Clearly, $\restr{\phi}{B(w,R)} =1$, $\restr{\phi }{B(w,R+r/2)^{c}}=0$, $0\leq \phi
\leq 1$ in $M$, and $|\phi (x)-\phi (y)|\leq 2d(x,y)/r$ for all $x,y\in M$.

Moreover,
\begin{align}
	\int_M|\phi (x)-\phi (y)|^pJ(x,\dif y)&=\int_{B(x,r)}|\phi (x)-\phi (y)|^pJ(x,\dif y)+\int_{B(x,r)^{c}}|\phi (x)-\phi (y)|^pJ(x,\dif y)\\
	&\leq
\frac{2^p}{r^p} \int_{B(x,r)}{d(x,y)^{p}J(x,\dif y)}+J(x,B(x,r)^{c}).
\end{align}
Therefore \begin{align}
\mathcal{E}(\phi)&=\int_{M\times M}|\phi (x)-\phi (y)|^pJ(x,\dif y)\dif m(x) \\
&\leq  2\int_{B(w,R+r/2)}\int_{ M}|\phi (x)-\phi (y)|^pJ(x,\dif y)\dif m(x)\\
&\leq C\frac{2^p}{r^p} \int_{B(w,R+r/2)}\int_{B(x,r)}{d(x,y)^{p}J(x,\dif y)}\dif m(x)+\int_{B(w,R+r/2)}J(x,B(x,r)^{c})\dif m(x)\\
& <\infty \text{\ (by assumption \ref{it.NLJ1} and \eqref{e.NLJ3})},
\label{e.E-13-}
\end{align}
which implies that $\phi \in \sD\cap C_{c}(M)=\sF\cap C_{c}(M)$, thus showing \eqref{e.ex3-co-}. Let $\sA$ be the finite linear combination of $\Sett{\phi_{w,r,R}}{w\in M,\ r\in (0,1),\ R\in (0,\infty)}$ over $\bR$. Then by a similar calculation as in \eqref{e.E-13-}, it can be proved that $\sA$ is a linear subspace of $\sF\cap C_{c}(M)$ and forms an algebra under pointwise multiplication. For every two different points $x$ and $y$, the function $\phi_{x,\frac{\abs{x-y}}{4},\frac{\abs{x-y}}{8}\wedge\frac{1}{2}}$ separates $x$ any $y$. An application of the
Stone-Weierstrass theorem gives that $\sA$ is dense in $C_{c}(M)$ under uniform norm. This shows the regularity of $(\sE,\sF)$.
\end{enumerate}
\end{proof}

\subsection{Euclidean space}
	In this subsection, we study the simplest mixed local and nonlocal equation on Euclidean space $\bR^{n}$. Let $\rL^{n}$ be the $n$-dimensional Lebesgue measure on $\bR^{n}$. Let $s\in(0,1)$ and $p\in(1,\infty)$. Let $J$ be a symmetric kernel in $x$ and $y$ such that, for some $C\in(1,\infty)$, \begin{equation}
		\frac{C^{-1}}{\abs{x-y}^{n+ps}}\leq J(x,y)\leq\frac{C}{\abs{x-y}^{n+ps} },\ \forall (x,y)\in (\bR^{n}\times\bR^{n})\setminus {\diag}.\label{e.Rn-J}
	\end{equation} Define 
	\begin{align}
		\sF&:=W^{1,p}(\bR^{n})\ \text{and }\sE(u):=\int_{\bR^{n}}\abs{\nabla u}^{p}\dif x+\int_{\bR^{n}}\int_{\bR^{n}}\abs{u(x)-u(y)}^{p}J(x,y)\dif x\dif y,\ \forall u\in\sF.
	\end{align}
	where $\nabla u$ is the weak derivative for $u\in \sF$. 
	We first show that $(\sE,\sF)$ is a mixed local and nonlocal $p$-energy form on $(\bR^{n},\rL^{n})$. \begin{enumerate}[label=\textup{({\arabic*})},align=left,leftmargin=*,topsep=5pt,parsep=0pt,itemsep=2pt]
	\item By \cite[Proposition 2.2]{DNPV12}, there exists $C>1$ such that \begin{align}
		\int_{\bR^{n}}\abs{\nabla u}^{p}\dif x\leq \sE(u)\leq C\int_{\bR^{n}}\abs{\nabla u}^{p}\dif x, \forall u\in\sF.
	\end{align}
	Therefore, \begin{equation}
		\norm{u}_{W^{1,p}(\bR^{n})}\leq \sE_{1}(u)^{1/p}\leq C\norm{u}_{W^{1,p}(\bR^{n})},
	\end{equation} so by the completeness of $(W^{1,p}(\bR^{n}),\norm{\wcdot}_{W^{1,p}(\bR^{n})})$, we obtain that $(\sF,\sE_{1}^{1/p})$ is a Banach space.
	\item The set $C_{c}^{\infty}(\bR^{n})$ is clearly dense in $(C_{c}(\bR^{n}),\norm{\wcdot}_{\sup})$. By the definition of $W^{1,p}(\bR^{{n}})$, the set $C_{c}(\bR^{n})$ is dense in $(\sF,\sE_{1}^{1/p})$. Therefore $(\sE,\sF)$ is regular.
\item Let $A_1=\{x\in M: u(x)\geq v(x)\}$ and $A_2=\{x\in M: v(x)> u(x)\}$. Then $M=A_1\cup A_2$ and
\begin{align}
&\int_{M}(|\nabla (u\vee v)(x)|^p+|\nabla (u\wedge v)(x)|^p)\dif x \\
&=\int_{A_1}(|\nabla u(x)|^p+|\nabla v(x)|^p)\dif x +\int_{A_2}(|\nabla v(x)|^p+|\nabla u(x)|^p)\dif x  \\
&=\int_{M}( |\nabla u(x)|^p+|\nabla v(x)|^p )\dif x,
\end{align}
which combining with \eqref{e.E-14} shows that  $(\sE,\sF)$ satisfies \eqref{Sub}.

	\item Let \begin{equation}
		\sE^{(L)}(u):=\int_{\bR^{n}}\abs{\nabla u}^{p}\dif x,\ \forall u\in\sF.
	\end{equation} Let $u,v\in\sF\cap C_{c}(\bR^{n})$ such that $u$ is constant over a neighborhood $U\subset G$ of $\supp(v)$. By \cite[Theorem 4.4-(iv)]{EG15}, $\nabla u=0$ $\rL^{n}$-a.e. on $U$. Since $\restr{v}{\bR^{n}\setminus U}=0$, we have $\nabla v=0$ $\rL^{n}$-a.e. on $\bR^{n}\setminus U$. Thus \begin{align}
		\int_{\bR^{n}}\abs{\nabla (u+v)}^{p}\dif x&=\int_{\bR^{n}\setminus U}\abs{\nabla u+\nabla v}^{p}\dif x+\int_{U}\abs{\nabla u+\nabla v}^{p}\dif x\\
		&=\int_{\bR^{n}\setminus U}\abs{\nabla u}^{p}\dif x+\int_{U}\abs{\nabla v}^{p}\dif x=\int_{\bR^{n}}\abs{\nabla u}^{p}\dif x+\int_{\bR^{n}}\abs{\nabla v}^{p}\dif x.
	\end{align}
	Therefore $(\sE^{(L)},\sF\cap C_{c}(\bR^{n}))$ is strongly local, this gives \ref{lb.EFJP} with $\dif j(x,y)=J(x,y)\dif x\dif y$.
	\item The Markovian property of $(\sE^{(L)},\sF\cap C_{c}(\bR^{n}))$ follows from the chain rule of weak derivatives for Lipschitz functions, see e.g. \cite{LM07}. The Clarkson inequality \eqref{Cla} follows by applying the Clarkson inequality \eqref{e.B3-1} and \eqref{e.B3-2} for $L^{p}$ spaces to the weak derivatives. 
	\end{enumerate}
So $(\sE,\sF)$ is a mixed local and nonlocal $p$-energy form on $(\bR^{n},\rL^{n})$. Let $W(x,r):=\min\{r^p,r^{ps}\}$, $\forall (x,r)\in\bR^{n}\times(0,\infty)$. The verifications of \eqref{TJ} and \eqref{UJS} for $(\sE,\sF)$ are straightforward. 
Next, we show Poincar\'e inequality \eqref{PI}. By \cite[Theorem 2 in p.~291]{Eva98},
for any ball $B=B(x_0,r)$ and any $u \in \sF$, we have
\begin{equation}
	\int_B\left|u-u_B\right|^p \dif x \leq C_1 r^p\int_{B} \abs{\nabla u(x)}^{p}\dif x. \label{e.PI1} 
\end{equation}
Also, by H\"older's inequality, 
\begin{align}
\int_B\left|u-u_B\right|^p \dif x & \leq \int_B\left(\fint_{B}\abs{u(x)-u(y)}\dif y\right)^p \dif x  \leq \int_B\left(\fint_B \abs{u(x)-u(y)}^p\dif y\right)\dif x \\
   & \leq  C_2\int_{B\times B} \frac{\abs{u(x)-u(y)}^p}{|x-y|^n} \dif x\dif y \\
  & \leq C_3r^{ps} \int_{B\times B}\abs{u(x)-u(y)}^p J(x,y) \dif x\dif y. \label{e.PI2}
\end{align}
Since $\dif \Gamma_{B}\la u\ra(x)=\abs{\nabla\phi(x)}^{p}\dif x+\big(\int_{y\in B}\abs{u(x)-u(y)}^p J(x,y) \dif y\big)\dif x$, by \eqref{e.PI1} and \eqref{e.PI2}, we have
\begin{equation}
	\int_B\left|u-u_B\right|^p \dif x \leq C W(x,r)\int_{ B} \dif \Gamma_{B}\la u\ra, 
\end{equation}
with $C=\max(C_1,C_3)$, thus showing \eqref{PI}. 

Let $(x_{0},R,r)\in \bR^{n}\times(0,\infty)\times(0,\infty) $. Let $B_{1}=B(x_{0},R)$, $B_{2}=B(x_{0},R+r)$, $B_3=B(x_{0},R+2r)$. Define \begin{equation}
	\phi(x):=1\wedge\frac{(R+r-d(x_0,x))_{+}}{r},\ \forall x\in \bR^{n}.
	\end{equation}
	Then $\abs{\nabla\phi(x)}=r^{-1}\one_{B_2\setminus B_1}(x)$ $\rL^{n}$-a.e. $x\in \Omega$ and
\begin{equation}\label{e.CS1}
  \int_{B_3}\abs{u(x)}^{p}\abs{\nabla\phi(x)}^{p}\dif x\leq \frac{1}{r^{p}}\int_{B_3}\abs{u}^{p}\dif x.
\end{equation}
Let \begin{align}
	I_1(x):=\int_{B(x,r)}\abs{\phi(x)-\phi(y)}^{p}J(x,y)\dif y\text{ and }I_2(x):=\int_{B(x,r)^{c}}\abs{\phi(x)-\phi(y)}^{p}J(x,y)\dif y
\end{align}
and define $B_k=B(x,2^{-k}r)$. Then 
\begin{align}
  I_1 (x)& \leq C_4\int_{B(x,r)}\frac{\abs{x-y}^{p}}{r^p}\frac{1}{|x-y|^{n+ps}}\dif y  = \frac{C_4}{r^p}\sum_{k=0}^\infty \int_{B_k\setminus B_{k+1}} \frac{1}{|x-y|^{n+ps-p}}  \\
  & \leq \frac{C_5}{r^p}\sum_{k=0}^\infty \frac{(2^{-k}r)^{n+p}}{(2^{-k-1}r)^{n+ps}}=\frac{2^{n+sp}C_5}{r^{sp}}\sum_{k=0}^\infty\frac{1}{2^{kp(1-s)}}=\frac{2^{n+sp}C_5}{1-2^{s-1}}\cdot\frac{1}{r^{ps}}. \label{e.CS2}
\end{align}
Note that $\abs{\phi(x)-\phi(y)}\leq 1$ for all $x,y$, we have $I_2(x)  \leq J(x,B(x,r)^{c}) {\leq} {C'}{r^{-ps}}$ by \eqref{TJ}, 
which combines with \eqref{e.CS1} and \eqref{e.CS2} gives
\begin{align}
		\int_{B_3}\abs{u}^{p}\dif\Gamma_{B_3}(\phi)&=\int_{B_3}\abs{u(x)}^{p}\abs{\nabla\phi(x)}^{p}\dif x+\int_{B_3}\int_{B_3}\abs{u(x)}^{p}\abs{\phi(x)-\phi(y)}^{p}J(x,y)\dif x\dif y\\
&\leq \int_{B_3}\abs{u(x)}^{p}\abs{\nabla\phi(x)}^{p}\dif x+ \int_{B_3}\abs{u(x)}^{p}(I_1(x)+I_2(x))\dif x \\
		&\leq \frac{1}{r^{p}}\int_{B_3}\abs{u}^{p}\dif x+\frac{C_6}{r^{ps}}\int_{B_3}\abs{u}^{p}\dif x\leq \sup_{x\in B_3}\frac{C_6+1}{W(x,r)}\int_{B_3}\abs{u}^{p}\dif x,
	\end{align}
so \eqref{CS} holds. By Theorem \ref{t.main}, \eqref{wEH} and \eqref{sEH} hold for $(\sE,\sF)$.

\subsection{Ultrametric space}

We say that $(X,d)$ is \emph{ultra-metric} if 
\begin{equation}
		d(x,y)\leq \max(d(x,z),d(z,y)),\ \forall x,y,z\in X.
\end{equation}
For example, given a prime number $p$, the field $\bQ_{p}$ of $p$-adic numbers with $p$-adic distance is an ultrametric space. A notable property of ultrametric space is that, every point inside an open ball is the center of this ball, i.e.,\begin{equation}
		\label{e.CenUlt}
		\text{for all $x_0\in X$, $r>0$ and all $x\in B(x,r)$, we have $B(x_0,r)=B(x,r)$.}
	\end{equation}
	Let $(M_{i},d_{i})$ ($i=1,2$) be two complete separable ultra-metric spaces. Let $m_i$ be a Radon measure on $M_i$ with full support ($i=1,2$). Suppose $(M_{i},d_{i},m_{i})$ is $\alpha_{i}$-Ahlfors regular, that is, there is $C>1$ such that \begin{equation}
		C^{-1}r^{\alpha_{i}}\leq m_{i}(B_{i}(x_{i},r))\leq Cr^{\alpha_{i}},\ \forall (x_i,r)\in M_{i}\times (0,\diam(M_{i},d_{i})),\  i=1,2.
	\end{equation} Let $M=M_{1}\times M_{2}$ be the product space equipped with the product measure $m=m_1\otimes m_2$ and the product metric \begin{equation}d((x_1,x_2),(y_1,y_2)):=\max_{i=1,2}d_{i}(x_i,y_i),\ \forall (x_i,y_i)\in M, \ i=1,2.
	\end{equation} Then $(M,d,m)$ is also an ultrametric space and is $\alpha_1+\alpha_2$ regular. Let $\beta>0$ and let $J_{i}$ be a function on $M_{\od}^{2}$\begin{equation}
		J_{i}(x_i,y_i)=\frac{1}{d_{i}(x_i,y_i)^{\alpha_{i}+\beta}}.
	\end{equation}
	We define a kernel \begin{equation}
		J(x,\dif y)=J_1(x_1,y_1)\dif m_{1}(y_1)\dif\delta_{x_2}(y_2)+J_2(x_2,y_2)\dif m_{2}(y_2)\dif\delta_{x_1}(y_1)
	\end{equation} for $(x,y)=((x_1,x_2),(y_1,y_2))\in M\times M$. By the calculation in \cite[(7.7)]{HY23}, we see that $J$ satisfies \eqref{TJ} with $W(x,r):=r^{\beta}$. 
Therefore, the kernel $J(x,E)$ on $M\times \sB(M)$ that satisfies Proposition \ref{p.NL=>Sub}-\ref{it.NLJ1} and \ref{it.NLJ2}. Next, we construct a regular, Markovian closed purely non-local $p$-energy form $(\sE,\sF)$ on $(M,m)$ by using Proposition \ref{p.NL=>Sub}-\ref{cl-1}. Consider the form defined by \begin{equation}
		\begin{dcases}
		\sE(u):=\int_{M^{2}_{\od}}\abs{u(x)-u(y)}^{p} J(x,\dif y)\dif m(x)\\
		\sF_{\max} :=\Sett{u\in L^{p}(M,m)}{\text{$u$ is Borel measurable and }\sE(u)<\infty}.
		\end{dcases}	
	\end{equation}

	Recall that any metric ball in an ultrametric space is closed and open. Let $\sD$ be the space of all locally constant functions on $M$ with compact supports, that is \begin{equation}
		\sD:=\Sett{\sum_{j=1}^{n}c_{j}\one_{B_{j}}}{n\in\bN, c_{j}\in\bR, B_{j}\text{ is a compact ball}}.
	\end{equation}
	By the proof of \cite[Lemma 4.1]{BGHH21}, $\one_{B}\in\sF_{\max}$ for any compact metric ball $B$, and therefore $\sD\subset\sF_{\max}$. Moreover, $\sD$ is dense in $(C_{c}(M),\norm{\cdot}_{\sup})$ and in $(L^{p}(M,m),\norm{\cdot}_{L^{p}(M,m)})$. Let $\sF$ be the closure of $\sD$ in the norm $\sE_{1}^{1/p}:=(\sE+\norm{\cdot}_{L^{p}(M,m)})^{1/p}$, then $(\sE,\sF)$ is regular. By Proposition \ref{p.NL=>Sub}-\ref{cl-1}, $(\sE,\sF)$ is a Markovian closed purely nonlocal $p$-energy form that satisfies \eqref{Sub}.
	
	We show that \eqref{CS} holds. Let $(x_{0},R,r)\in M\times (0,\infty)\times (0,\infty)$. Let \begin{equation}
		B_{0}=B(x_{0},R),\ B_{1}=B(x_{0},R+r)\ \text{and}\ \Omega=B(x_{0},R+2r).
	\end{equation}
Since $B_0$ is both open and closed, we have $\one_{B_0}\in\cutoff(B_0,B_1)$. For any $u\in\sF^{\prime}\cap L^{\infty}$, \begin{align}
		&\phantom{\ \leq}\int_{\Omega}\abs{u}^{p}\dif\Gamma_{\Omega}\la \one_{B_0}\ra =\int_{\Omega}\int_{\Omega}\abs{u(x)}^{p}\abs{\one_{B_{0}}(x)-\one_{B_{0}}(y)}^{p}J(x,\dif y)\dif m(x)\\
		&=\int_{B_{0}}\abs{u(x)}^{p}\bigg(\int_{\Omega\setminus B_0}J(x,\dif y)\bigg)\dif m(x)+\int_{\Omega\setminus B_0}\abs{u(x)}^{p}\bigg(\int_{B_0}J(x,\dif y)\bigg)\dif m(x)\\
		&=2\int_{B_{0}}\abs{u(x)}^{p}\bigg(\int_{\Omega\setminus B_0}J(x,\dif y)\bigg)\dif m(x)=2\int_{B(x_{0},R)}\abs{u(x)}^{p}\bigg(\int_{\Omega\setminus B(x_{0},R)}J(x,\dif y)\bigg)\dif m(x)\\
		&\overset{\eqref{e.CenUlt}}{=}2\int_{B(x_{0},R)}\abs{u(x)}^{p}\bigg(\int_{\Omega\setminus B(x,R)}J(x,\dif y)\bigg)\dif m(x)\overset{\eqref{TJ},\eqref{eq:vol_0}}{\leq}\frac{C}{W(x_0,R)}\int_{B(x_{0},R)}\abs{u(x)}^{p}\dif m(x),
\end{align}
which gives \eqref{CS}. The Poincar\'e inequality \eqref{PI} also holds by an analogue proof of \cite[p.~823-824]{HY23}. Therefore \eqref{wEH} holds by Theorem \ref{t.main}-\eqref{e.main-wEH}.

\subsection{Blow-up of a Cantor set}\label{s.cantor}
In this section, we give an example of a purely nonlocal $p$-energy form on a fractal set which satisfies \eqref{TJ} and \eqref{UJS}, while the canonical upper bound of the jump measure 
\begin{equation}\label{J<}\tag{$\operatorname{J}_{\leq}$}
	J(x,y)\leq \frac{C}{m(B(x,d(x,y))W(x,d(x,y))},\ m\otimes m\text{-a.e.} (x,y)\in M^{2}_{\od}
\end{equation}
fails. 

Let $\mathcal{C}$ be the standard  Cantor set in $[0,1]$ and let $M:=\mathcal{C}+\mathbb{Z}:=\{ x=y+k: \ y\in \mathcal{C}\text{ and } k\in \mathbb{Z}\}$. Let $m$ be the Hausdorff measure of the metric space $(M,\abs{\cdot -\cdot})$ where $\abs{\cdot -\cdot}$ is the Euclidean distance restricted on $M$. By the self-similarity of $\sC$, there exists $C>1$ such that  
\begin{equation}
 C^{-1}V(r)\leq  m(B(x,r))\leq C V(r)  \ \text{for all }x\in M \ \text{and }r\in(0,\infty), \label{e.ar}
\end{equation}
where $V(r):=r^{\alpha}\vee r$ for $r\in(0,\infty)$ and $\alpha=\log_{3}2$. We know by \eqref{e.ar} that \eqref{VD} and \eqref{RVD} hold for the metric measure space $(M,\abs{\cdot -\cdot},m)$. 

Let $B_k=B(2k,1)\cap M$, $k\in \mathbb{Z}$, be the balls in $M$ centered at $2k$ with radius $1$.
Let $\beta>0$. Define 
\begin{equation}
	J_1(x,y):=\frac{1}{V(|x-y|)\abs{x-y}^{\beta} },\ 
	J_2(x,y):=\sum_{n\in\bZ\setminus\{0\}} \frac{1}{n^{\beta} }\one_{B_n\times B_{-n}}(x,y),\ \forall (x,y)\in M^{2}_{\od},
\end{equation}
and \begin{equation}
	J(x,y):=J_1(x,y)+J_2(x,y),\ \forall (x,y)\in M^{2}_{\od}.
\end{equation}

\begin{proposition}
	 Let $p\in(1,\infty)$ and $\beta\in(0,p)$. Let $(M,d,m)$ and $J$ be defined as above. Define 
	\begin{align}
		\mathcal{E}(u) &:=\int_{M_{\od}^{2}
}\abs{u(x)-u(y)}^{p}J(x,y)\dif m(y)\dif m(x),  \notag \\
\mathcal{F} &:=\text{the closure of }\{u\in C_{c}(M):\mathcal{E}(u)<\infty \}\text{ under the norm }\mathcal{E}_{1}^{1/p}.  \label{F-0}
	\end{align}
	Then \begin{enumerate}[label=\textup{({\arabic*})},align=left,leftmargin=*,topsep=5pt,parsep=0pt,itemsep=2pt]
	\item\label{it.ex3-1} The pair $(\sE,\sF)$ is a purely nonlocal $p$-energy form on $(M,d,m)$.
	\item\label{it.ex3-2} \eqref{TJ}, \eqref{UJS} holds with $\ol{R}=1/2$, while \eqref{J<} fails.
	\item\label{it.ex3-3} The cutoff Sobolev inequality \eqref{CS} and the Poincar\'e inequality  \eqref{PI} hold. Therefore the Harnack inequalities \eqref{wEH} and \eqref{sEH} hold.
\end{enumerate}
\end{proposition}
\begin{proof}
	\begin{enumerate}[label=\textup{({\arabic*})},align=left,leftmargin=*,topsep=5pt,parsep=0pt,itemsep=2pt]
	\item[\ref{it.ex3-1}] 
A straightforward check confirms Proposition 6.1-\ref{it.NLJ1} and \ref{it.NLJ2}. Note that 
\begin{equation}\label{e.Eft}
(x,y)\notin \bigcup_{k\in \mathbb{Z}\setminus \{0\}} B_k\times B_{-k} \text{ for any $x,y\in M$ with $|x-y|<1$}.
\end{equation}
For $r\in (0,1)$, we have
$J(x,y)=J_1(x,y)=|x-y|^{-\alpha-\beta}$ for $y\in B(x,r)$. By \eqref{e.ar}, 
\begin{align}
 \int_{B(x,r)}|x-y|^pJ(x,\dif y) = 
 \int_{B(x,r)}|x-y|^{p-\alpha-\beta}\dif m(y)
  \leq C\int_{0}^{r}s^{p-1-\beta}\dif s=\frac{Cr^{p-\beta}}{p-\beta},\label{e.E-12}
\end{align}
where in the last inequality, the restriction $\beta<p$ and a calculation similar to \eqref{e.CS2} are used. Therefore, \eqref{e.NLJ3} holds. The result follows from Proposition \ref{p.NL=>Sub}.

\item[\ref{it.ex3-2}] By a cake-layer decomposition and \eqref{e.ar},\begin{equation}
  \int_{B(x,r)^c}J_1(x,y)\dif m(y)=  \int_{B(x,r)^c}\frac{1}{V(|x-y|)\abs{x-y}^{\beta} }\dif m(y)\leq Cr^\beta. \label{e.E-7}
\end{equation}
Note that $\int_{B(x,r)^c}J_2(x,y)\dif m(y)=0$ if $x\in (-1,1)\cup \{2k+1: k\in \mathbb{Z}\}$, otherwise $x\in B_{k}$ for some $k\in \mathbb{Z}\setminus \{0\}$, and then
\begin{equation}
  \int_{B(x,r)^c}J_2(x,y)\dif m(y)=\frac{1}{|k|^\beta} m (B_{-k}\cap B(x,r)^c).\label{e.E-6}
\end{equation}
If $r> 4|k|+2$, then $m (B_{-k}\cap B(x,r)^c)=0$, and if $r\leq 4|k|+2(\leq 6|k|)$, then
$\frac{1}{|k|^\beta} m (B_{-k}\cap B(x,r)^c)\leq \frac{1}{|k|^\beta}m (B_{-k})= \frac{2}{|k|^\beta}\leq 2\cdot {6^\beta}r^{-\beta}$. By \eqref{e.E-6}, we have
\begin{equation}
  \int_{B(x,r)^c}J_2(x,y)\dif m(y)\leq \frac{2\cdot 6^\beta}{r^\beta}. \label{e.E-8}
\end{equation}
Therefore, condition \eqref{TJ} holds by \eqref{e.E-7} and \eqref{e.E-8}. 

Next, we show \eqref{UJS} holds with $\overline{R}=1/2$. For any $x,y\in M$ satisfies $|x-y|\geq 2r$ with $r>0$, then $|y-z|\leq |x-y|+|x-z|\leq |x-y|+r\leq 3|x-y|/2$ for all $z\in B(x,r)$, thus by noting that $(R/r)^\alpha\leq V(R)/V(r)\leq R/r$ for $0<r\leq R<\infty$, we have by \eqref{e.ar} that
\begin{equation}\label{e.E-9}
  \fint_{B(x,r)}J_1(z,y)\dif m(z)=\fint_{B(x,r)}\frac{\dif m(z)}{V(|z-y|)|z-y|^{\beta}}
  \geq  \frac{(2/3)^{\alpha+\beta}}{V(|x-y|) |x-y|^{\beta}}=(2/3)^{\alpha+\beta}J_1(x,y).
\end{equation}

Let $r\in (0,1/2)$. If $(x,y)\in B_k\times B_{-k}$ for some $k\in \mathbb{Z}\setminus \{0\}$, then $d(x,y)>4|k|-2\geq 2>2r$ and
\begin{equation}
  \fint_{B(x,r)}J_2(z,y)\dif m(z)= \frac{1}{k^\beta}\frac{m(B(x,r)\cap B_k)}{m(B(x,r))}\asymp \frac{1}{k^\beta}=J_2(x,y). \label{e.E-10}
\end{equation}
Otherwise, \eqref{e.E-10} holds trivially since $J_2(x,y)=0$ for $(x,y)\notin \bigcup_{k\in \mathbb{Z}\setminus \{0\}}B_k\times B_{-k}$. However,
\begin{equation}
\text{for $n$ large enough and $(x,y)\in B_n\times B_{-n}$, }  J(x,y)=\frac{1}{n^{\beta}}\gg \frac{C}{n^{1+\beta}}\asymp\frac{C}{V(|x-y|)|x-y|^{\beta}}.
\end{equation}
So \eqref{J<} cannot hold.
\item[\ref{it.ex3-3}] Let $x_0\in M$, and $R,r>0$ such that $R+2r<\ol{R}=1/2$ and let 
\begin{equation}
		B_0:=B(x_0,R),\ B:=B(x_0,R+r),\ \text{and }\Omega:=B(x_0,R+2r).
\end{equation} 
Define $\phi=\phi_{x_{0},R,r}\in \cutoff(B_0,B)$ as in \eqref{phi1-}. By \eqref{e.E-13-}, \eqref{e.E-12} and \eqref{TJ}, we have \begin{equation}
  \int_{M}|\phi (x)-\phi (y)|^pJ(x,y)\dif m(y)\leq \frac{C}{r^{\beta}}.\label{e.E-13}
\end{equation} Thus for any $u\in\sF^{\prime}\cap L^{\infty}(M,m)$,
\begin{align}
  \int_{\Omega}\abs{\wt{u}}^{p}\dif\Gamma_{\Omega}\la \phi\ra  & \leq \int_\Omega \abs{\wt{u}(x)}^{p}\left(\int_M
   \abs{\phi (x)-\phi (y)}^pJ(x,y)\dif m(y)\right)\dif m(x)   \leq \frac{C}{r^{\beta}} \int_\Omega \abs{u(x)}^{p} \dif m(x).
\end{align}

Let $B$ be a ball with radius $r \in\left(0, 1/2\right)$ and $u \in \sF$. Note that $d(x,y)< 1$ for any $x,y\in B$, we have from \eqref{e.Eft} that
$J(x,y)=|x-y|^{-\alpha-\beta}\geq (2r)^{-\alpha-\beta}$. Therefore, similar to \eqref{e.PI2},
\begin{align}
 \int_B\left|u-u_B\right|^p \dif m \leq C r^\beta \int_B\int_B |u(x)-u(y)|^pJ(x,y)\dif m(y) \dif m(x),
\end{align}

which implies \eqref{PI} with $\overline{R}=1/2$. Conditions \eqref{wEH} and \eqref{sEH} follow by using Theorem \ref{t.main}.
\end{enumerate}
\end{proof}

\begin{appendices}
\section{Properties of \texorpdfstring{$p$}{\em{p}}-energy form}\label{A.propEF}
\subsection{Basic properties of \texorpdfstring{$p$}{\em{p}}-energy form}
\begin{lemma}\label{l.cuofbdd}
 Let $(\sE,\sF)$ be a Markovian $p$-energy form that satisfies \eqref{Cla}. For any $u\in\sF\cap L^{\infty}(M,m)$ and any sequence $u_{n}\to u$ in $\sE$, we have \begin{equation}\label{e.cuofbdd}
 	(-\norm{u}_{L^{\infty}(M,m)}\vee u_n)\wedge \norm{u}_{L^{\infty}(M,m)}\to u \text{\  in $\sE$.}
 \end{equation}
\end{lemma}
\begin{proof}
   Taking $\varphi(s):=(-\norm{u}_{L^{\infty}(M,m)}\vee s)\wedge \norm{u}_{L^{\infty}(M,m)}$ and applying \cite[Corollary 3.19-(b)]{KS25}, we obtain \eqref{e.cuofbdd}.
\end{proof}
\begin{lemma}[\text{\cite[Lemma 3.17 and Proposition 3.18-(a)]{KS25}}]
\label{lem:A1} Let $(\sE,\sF)$ is a closed $p$-energy form on $(M,m)$ that satisfies \eqref{Cla}. Let $\{u_{n}\}_{n\in\bN}\subset \mathcal{F}$ and $u\in L^{p}(M,m)$. If 
\begin{equation}
\lim_{n\to\infty}\norm{u_{n}-u}_{L^{p}(M,m)}=0 \text{ and } \sup_{n\in\bN}\mathcal{E}(u_{n})<\infty ,
\end{equation}%
then $u\in \mathcal{F}$, $\mathcal{E}(u)\leq \liminf_{n\rightarrow \infty }\mathcal{E}(u_{n})$ and $u_{n}\to u$ weakly in $(\sF,\sE_{1}^{1/p})$.  
\end{lemma}  

 \begin{proposition}\label{p.mar}
   Let $(\mathcal{E},\mathcal{F})$ be a Markovian $p$-energy form on $(M,m)$. Let $\psi\in C^{1}(\bR)$ with $\psi(0)=0$. Then $\psi\circ u\in \sF\cap L^\infty$ for any $u\in \sF\cap L^\infty$. Also, \begin{equation}
   	\sE(u\cdot v)^{1/p}\leq 2(\norm{u}_{L^{\infty}(M,m)}+\norm{v}_{L^{\infty}(M,m)})(\sE(u)^{1/p}+\sE(v)^{1/p}),\ \forall u,v\in\sF\cap L^\infty(M,m).
   \end{equation}
 \end{proposition}
 \begin{proof}
 Define $f(s):=(s\wedge \norm{u}_{L^\infty})\vee (-\norm{u}_{L^\infty}),\ s\in\bR$. Then $\psi\circ u=(\psi\circ f)\circ u$. 	Let \[L:=\sup_{-\norm{u}_{L^{\infty}(M,m)}\leq t\leq \norm{u }_{L^{\infty}(M,m)}} \abs{\psi^{\prime}(t)}.\]  Note that $L^{-1}\psi\circ f$ is $1$-Lipschitz, we have by Markovian property that $(\psi\circ f)\circ u\in\sF$ and $\sE(\psi\circ u)^{1/p}\leq L\sE(\psi)^{1/p}$. In particular, if $\psi(s)={s}^{2}$, then $L=2\norm{u}_{L^{\infty}(M,m)}$ so we have \begin{equation}\label{e.Alg1}
 	\sE(u^{2})^{1/p}\leq 2\norm{u}_{L^{\infty}(M,m)}\sE(u)^{1/p},\ \forall u\in \sF\cap L^\infty.
 \end{equation}
 If $u,v\in\sF\cap L^{\infty}(M,m)$, then since $(u+v)^{2}$, $u^{2}$ and $v^{2}$ are all in $\sF\cap L^{\infty}(M,m)$, we have $2u v=(u+v)^{2}-u^{2}-v^{2}\in\sF\cap L^{\infty}(M,m)$ and \begin{align}
 	\sE(u v)^{1/p}&=\frac{1}{2}\sE(2u v)^{1/p}=\frac{1}{2}\sE((u+v)^{2}-u^{2}-v^{2})^{1/p}\\
 	&\leq \frac{1}{2}\big(\sE((u+v)^{2})^{1/p}+\sE(u^{2})^{1/p}+\sE(v^{2})^{1/p}\big)\\
 	&\overset{\eqref{e.Alg1}}{\leq }\norm{u+v}_{L^{\infty}(M,m)}\sE(u+v)^{1/p}+\norm{u}_{L^{\infty}(M,m)}\sE(u)^{1/p}+\norm{v}_{L^{\infty}(M,m)}\sE(v)^{1/p}\\
 	&\leq 2(\norm{u}_{L^{\infty}(M,m)}+\norm{v}_{L^{\infty}(M,m)})(\sE(u)^{1/p}+\sE(v)^{1/p}).
 \end{align}
 The proof is complete.
\end{proof}

\begin{proposition}[{\cite[Proposition 3.28]{KS25}}] \label{p.cut}
Let $(\mathcal{E},\mathcal{F})$ be a regular and Markovian $p$-energy form. Then for any two non-empty open subsets $U\Subset V$ of $M$, the set $\cutoff(U,V)$ is non-empty. 
\end{proposition}

\begin{proof}
The assumptions in \cite[formulas (3.28) and (3.29)]{KS25} are ensured by the Markovian property and Proposition \ref{p.mar}, respectively. If $V$ has a compact closure, then the conclusion follows from \cite[Proposition 3.28]{KS25}. Otherwise, since $M$ is locally compact Hausdorff, by \cite[Proposition 4.31]{Fol99}, we can take a relative compact open set that contains $U$ and its closure is contained in $V$, and then apply \cite[Proposition 3.28]{KS25}.
\end{proof}

\subsection{Capacity and relative capacity of \texorpdfstring{$p$}{\em{p}}-energy form}\label{ss.capacity}
In this subsection, we assume that $(\sE,\sF)$ is a closed, Markovian $p$-energy form on $(M,m)$ that satisfies \eqref{Cla} and \eqref{Sub}. Let $\sO$ be the collection of open subsets of $M$.  For $A\in\sO$, define 
\begin{equation}
	\sL_{A}:=\Sett{u\in\sF}{u\geq1\ m\text{-a.e. on }A}\label{e.defLAop}
\end{equation}
and\footnote{We adopt the convention that $\inf\emptyset=\infty$.}
\begin{equation}
	\capacity_{1}(A):=\inf\Sett{\sE_{1}(u)}{u\in\sL_{A}}\in[0,\infty] .
\end{equation}
\begin{proposition}[{\cite[Lemma 8.3]{Yan25c}}]Let $(\sE,\sF)$ be a $p$-energy form on $(M,m)$.
	\begin{enumerate}[label=\textup{({\arabic*})},align=left,leftmargin=*,topsep=5pt,parsep=0pt,itemsep=2pt]
	\item For $A,B\in\sO$ with $A\subset B$, we have $\capacity_{1}(A)\leq \capacity_{1}(B)$.
	\item For $A,B\in\sO$, we have 
		$\capacity_{1}(A\cap B)+\capacity_{1}(A\cup B)\leq \capacity_{1}(A)+\capacity_{1}(B)$.
	\item For any $\set{A_n}\subset \sO$ with $A_n \subset A_{n+1}$ for all $n$, we have \begin{equation}
		\capacity_{1}\big(\bigcup_{n}A_n)=\sup_n\capacity_{1}(A_n).\label{e.CapUni}
	\end{equation}
	\item For any subset $A\subset M$, define its \emph{capacity} \begin{equation}
		\capacity_{1}(A):=\inf\Sett{\capacity_{1}(G)}{G\in\sO\text{ and }A\subset G}.\label{e.defcap1}
	\end{equation}
	Then $\capacity_{1}$ is a \emph{Choquet capacity}. Moreover, if $A\subset M$ is Borel, then \begin{equation}\label{e.defcap2}
		\capacity_{1}(A)=\sup\Sett{\capacity_{1}(K)}{K\text{ is compact and }K\subset A}.
	\end{equation}
	\item If $A\subset M$ is Borel and if $\capacity_{1}(A)=0$ then $m(A)=0$.
\end{enumerate}
\end{proposition}
{Let $N\subset M$ be a subset. If $\capacity_{1}(N)=0$, then we say $N$ is exceptional. By \eqref{e.defcap1}, for any exceptional $N$, there is a $G_{\delta}$ set that contains $N$ and has zero capacity. In particular, and subset of zero capacity is contained in a Borel subset of zero capacity.}

Let $A\subset M$ be any subset. A statement depending on  $x\in A$ is said to hold $\sE$-quasi-everywhere (abbreviated $\sE$-q.e.) on $A$ if there exists a Borel exceptional set $N\subset A$ with $\capacity_{1}(N)=0$ such that the statement holds for any $x\in A\setminus N$. 

A numerical function $u:M\to[-\infty,\infty]$ is called \emph{$\sE$-quasi-continuous}, if for any $\epsilon>0$, there exists an open subset $G\in\sO$ with $\capacity_{1}(G)<\epsilon$ and such that $\restr{u}{M\setminus G}$ is finite-valued and continuous. Clearly, any continuous function on $M$ is $\sE$-quasi-continuous.
\begin{proposition}[{\cite[Proposition 8.5]{Yan25c}}]\label{p.QC-ver}
	Any $u\in\sF$ admits a $\sE$-quasi-continuous version $\widetilde{u}$, that is, $\widetilde{u}$ is $\sE$-quasi-continuous and $\widetilde{u}=u$ $m$-a.e. on $M$.
\end{proposition}
For any $A\subset M$, define \begin{equation}
	\sL_{A}:=\Sett{u\in\sF}{\wt{u}\geq1\ \sE\text{-q.e. on }A}. \label{e.defLAge}
\end{equation}
By \cite[Lemma 8.4]{Yan25c}, the definition in \eqref{e.defLAge} is compatible with \eqref{e.defLAop} when $A$ is open.

\begin{lemma}\label{l.equi}
For any subset $A$ with $\sL_{A}\neq\emptyset$, there exists a unique element $e_{A}\in \sL_{A}$ such that $\capacity_{1}(A)=\sE_{1}(e_{A})$. Moreover, $0 \leq e_A \leq 1$ $m$-a.e. on $M$, $\wt{e_A} = 1$ $\sE$-q.e. on $A$, and $e_A$ is the unique element $u\in \sF$ satisfying $\wt{u}=1$ $\sE$-q.e. on $A$ and $\sE_1(u;v)\geq0$ for any $v\in\sF$ with $\wt{v} \geq 0$ $\sE$-q.e. on $A$.
\end{lemma}
The element $e_{A}$ in Lemma \ref{l.equi} is called the \emph{equilibrium potential} of $A$.

\begin{lemma}[{\cite[Corollary 8.7]{Yan25c}}]\label{l.QEconv}
	Let $u_{n}\to u$ in $(\sF,\sE_{1})$. Then there exists a subsequence $\{n_{k}\}$ such that $\wt{u_{n_{k}}}\to \wt{u}$ $\sE$-q.e. as $k\to\infty$.
\end{lemma}

Let $G\in\sO$ be an open subset of $M$. Define a $p$-energy form $(\sE_{G},\sF(G))$ on $(G,\restr{m}{G})$ by \begin{equation}\label{e.partEF}
\begin{dcases}
	\sF(G):=\{u\in\sF: \text{$\widetilde{u}=0$ $\sE$-q.e. on $M\setminus G$}\},\\ \sE_{G}(u):=\sE(u),\ \forall u\in \sF({G}).
	\end{dcases}
\end{equation}
We shall identify $\sF(G)$ as a linear subspace of $\sF$.

\begin{lemma}\label{l.partEF}
	$(\sE_{G},\sF({G}))$ is a $p$-energy form on $(G,\restr{m}{G})$ with the following properties:\begin{enumerate}[label=\textup{({\arabic*})},align=left,leftmargin=*,topsep=5pt,parsep=0pt,itemsep=2pt]
	\item\label{lb.part-csd} $(\sE_{G},\sF({G}))$ is closed, that is $(\sF(G),\sE_{G,1}^{1/p})$ is a Banach space, where $\sE_{G,1}(\wcdot):=\sE_{G}(\wcdot)+\norm{\wcdot}^{p}_{L^{p}(G,m)}$.
	\item\label{lb.part-mar} $(\sE_{G},\sF({G}))$ is Markovian and satisfies \eqref{Cla} and \eqref{Sub}.
	\item\label{lb.part-reg} If $\overline{G}$ is compact in $M$, then $(\sE_{G},\sF({G}))$ is regular, that is, $\sF(G)\cap C_{c}(G)$ is both dense in $(\sF(G),\sE_{G,1}^{1/p})$ and in $(C_{c}(G),\norm{\cdot}_{\sup})$.
\end{enumerate}
\end{lemma}
\begin{proof}
	\ref{lb.part-csd} and \ref{lb.part-mar} directly follow from the closedness and Markovian property of $(\sE,\sF)$, respectively. To show \ref{lb.part-reg}, we shall use \cite[Lemma 6.2]{Yan25b} (which proof actually requires the completeness of $M$ to ensure any bounded closed subset being compact), and states that for any open subset $G$ with $\ol{G}$ compact, there holds that
\begin{equation}
\sF(G)=\{\text{the }\sE_1\text{-closure of }\sF\cap C_c(G)\}, \label{e.sF}
\end{equation}
which implies that $\sF(G)\cap C_{c}(G)$ is dense in $(\sF(G),\sE_{G,1}^{1/p})$. We then show that $\sF(G)\cap C_{c}(G)$ is dense in $(C_{c}(G),\norm{\cdot}_{\sup})$. Let $u\in C_{c}(G)$ and $F:=\supp(u)\subset G$. By the regularity of $(\sE,\sF)$ on $(M,m)$, there is $u_{n}\in\sF\cap C_{c}(M)$ such that $\norm{u_{n}-u}_{\sup}<n^{-1}$. Let $\phi_{n}(t):=t-(-n^{-1}\vee t)\wedge n^{-1}$, $t\in\bR$, then $\phi_{n}$ is Lipschitz and $\phi_{n}(0)=0$. By the Markovian property of $(\sE,\sF)$, $\phi_{n}\circ u_{n}\in\sF\cap C_{c}(M)$. For $x\in M\setminus \supp(u)$, $u(x)=0$ so $\phi\circ u_{n}(x)=0$, which gives that $\supp(\phi\circ u_{n})\subset G$. By noting that $\sup_{x\in G}\abs{\phi_{n}(u_{n}(x))-u(x)}\leq \sup_{x\in F}\abs{\phi_{n}(u_{n}(x))-u_{n}(x)}+\sup_{x\in F}\abs{u_{n}(x)-u(x)}\leq 2n^{-1}\to 0$, we obtain the density of $\sF(G)\cap C_{c}(G)$ in $(C_{c}(G),\norm{\cdot}_{\sup})$.
\end{proof}

Let $\capacity_{1}^{G}$ be the capacity relative to $(\sE_{G},\sF(G))$. 
\begin{proposition}\label{p.ComEqi}
	Let $A\subset G$ be an open set in $G$. Let $e_{A}\in\sF$ and $e^{G}_{A}\in\sF(G)$ be equilibrium potentials of $A$ under $(\sE,\sF)$ and $(\sE_{G},\sF(G))$, respectively. Then $e^{G}_{A}\leq e_{A}$ $m$-a.e. on $G$.
\end{proposition}
\begin{proof}
	It suffices to show that $e^{G}_{A}=e^{G}_{A}\wedge e_{A}$. In fact, note that $e^{G}_{A}\wedge e_{A}\in\sF(G)$, \begin{align}
		\sE_{G,1}(e^{G}_{A}\wedge e_{A})=\sE_{1}(e^{G}_{A}\wedge e_{A})&\overset{\eqref{Sub}}{\leq}-\sE_{1}(e^{G}_{A}\vee e_{A})+\sE_{1}(e^{G}_{A})+\sE_{1}(e_{A})\\
		&=-\sE_{1}(e^{G}_{A}\vee e_{A})+\capacity_{1}^{G}(A)+\capacity_{1}(A).\label{e.ComEqi1}
	\end{align}
	Since $e^{G}_{A}\vee e_{A}=1$ $m$-a.e. on $A$, we have $\sE_{1}(e^{G}_{A}\vee e_{A})\geq\capacity_{1}(A)$, which combines with \eqref{e.ComEqi1} gives $\sE_{G,1}(e^{G}_{A}\wedge e_{A})\leq \capacity_{1}^{G}(A)$. On the other hand, as $e^{G}_{A}\wedge e_{A}=1$ $m$-a.e. on $A$ and by the uniqueness of $e^{G}_{A}\in\sF(G)$, we have $e^{G}_{A}=e^{G}_{A}\wedge e_{A}$. 
	\end{proof}
\begin{proposition}\label{p.ShaQua}
	Let $A$ be a subset of $G$, then \begin{equation}
		\capacity_{1}(A)=0\text{ if and only if }\capacity_{1}^{G}(A)=0.
	\end{equation}
	Therefore any $\sE$-quasi-continuous function on $G$ is also $\sE_{G}$-quasi-continuous.
\end{proposition}
\begin{proof}
By the definition of capacity for any set, it suffices to show that for any decreasing open sets $\{A_{n}\}_{n\in\bN}$ contained in $G$ (that is $A_{n+1}\subset A_{n}\subset G$ for all $n\in\bN$), we have\begin{equation}
		\lim_{n\to\infty}\capacity_{1}(A_{n})=0\text{ if and only if }\lim_{n\to\infty}\capacity_{1}^{G}(A_{n})=0.
	\end{equation}
	The ``if'' part is easy to be seen, as $\capacity_{1}^{G}(A_n)\geq \capacity_{1}(A_n)$ for all $n\in\bN$. The ``only if'' part is proved by comparing the equilibrium potentials. By \eqref{e.CapUni}, we may assume that $A$ is contained in an open set in $G$ whose $\sE_G$-capacity is finite, so $\capacity_{1}^{G}(A)<\infty$ and thus we may further assume that $\capacity_{1}^{G}(A_{n})<\infty$. Let $e_{n}^{G}$ and $e_{n}$ be the equilibrium potentials of $A_{n}$ in $(\sE_{G},\sF(G))$ and in $(\sE,\sF)$, respectively. Then $\capacity_{1}(A_{n})\to0$ is equivalent to $\sE_{1}(e_{n})\to0 $, which implies that $e_{n}\to 0$ $m$-a.e. on $M$. By Proposition \ref{p.ComEqi}, $e_{n}^{G}\to0$ $m$-a.e. on $G$. Since $\set{\capacity_{1}^{G}(A_{n})}\subset[0,\infty)$ is a sequence of decreasing real finite numbers, it is a Cauchy sequence and $\sup_{n}\sE_{G,1}(e_{n}^{G})<\infty$. By the reflexivity of the Banach space $(\sF(G),\sE_{G,1}^{1/p})$ and the Banach-Alaoglu theorem, there exists $e^{G}\in\sF(G)$ and there is a subsequence of $\set{e^{G}_{n}}$, still denoted by $\set{e_{n}^G}$, such that $e^G_{n}\to e^G$ weakly in $(\sF(G),\sE_{G,1}^{1/p})$. However, since $e_{n}^{G}\to0$ $m$-a.e. on $G$, we must have $e^G=0$ $m$-a.e. on $G$. By Mazur's lemma \cite[Theorem 2 in Section V.1]{Yos95}, for each $n\in\bN$, there is a $N_{n}\geq n$ and a convex combination $\{\lambda_{k}^{(n)}\}_{k=n}^{N_n}\subset[0,1]$ with $\sum_{k=n}^{N_n}\lambda_{k}^{(n)}=1$, such that $u_{n}:=\sum_{k=n}^{N_n}\lambda_{k}^{(n)}e_{k}^{G}\to0$ in $(\sF(G),\sE_{G,1}^{1/p})$. Note that $u_{n}=1$ $m$-a.e. on $A_{N_{n}}$, so $\capacity_{1}^{G}(A_{N_{n}})\leq \sE_{G,1}(u_{n})\to0$. Since $\set{\capacity_{1}^{G}(A_{n})}\subset[0,\infty)$ is a Cauchy sequence, we have $\capacity_{1}^{G}(A_{n})\to0$.
\end{proof}
\subsection{Properties of mixed local and nonlocal \texorpdfstring{$p$}{\em{p}}-energy form}

 \begin{proposition}\label{p.A3-1}
Let $(\mathcal{E},\mathcal{F})$ be a mixed local and nonlocal $p$-energy form on $(M,m)$. Let $u\in \sF^{\prime}\cap L^\infty(M,m)$ and $\Omega$ be an open subset of $M$. If $v\in\sF\cap L^\infty(M,m)$ and $\widetilde{u}\cdot\widetilde{v}=0$ $\sE$-q.e. in $M\setminus \Omega$, then $u\cdot v\in \sF(\Omega)\cap L^\infty(M,m)$. In particular, if $v\in\sF(\Omega)\cap L^\infty(M,m)$, then $u\cdot v\in \sF(\Omega)\cap L^\infty(M,m)$.
 \end{proposition}
\begin{proof}
The proof is the same as in \cite[Proposition 15.1-(iii)]{GHH24}.
\end{proof}

\begin{lemma}\label{l.j-UNQ}
Let $(\mathcal{E},\mathcal{F})$ be a mixed local and nonlocal $p$-energy form on $(M,m)$. Then the positive Radon measure $j$ on $M^{2}_{\od}$ satisfying \ref{lb.EFJP} is uniquely determined by 
\begin{equation}
j(E\times F)=\frac{1}{4}\inf\left\{\sE(u)+\sE(v)-\sE(u+v): 
\begin{array}{l}
 u,v\in \sF\cap C_{c}(M) \text{ with disjoint supports},\\
u|_E=1, v|_F=1, \ 0\leq u,v\leq 1
\end{array}
 \right\}
\end{equation}
for any two disjoint compact sets $E$ and $F$.
 \end{lemma}
 \begin{proof}
  Assume that there exists jump measure $j$ such that
  \begin{equation}
     \sE^{(L)}(u):=\sE(u)-\int_{M_{\od}^{2}}\abs{u(x)-u(y)}^{p}\dif j(x,y),\ u\in\sF\cap C_{c}(M)
  \end{equation}
   is strongly local. By \ref{lb.EFJP}, for any functions $u,v\in \sF\cap C_{c}(M)$ with disjoint supports, we have $\sE^{(L)}(u+v)=\sE^{(L)}(u)+\sE^{(L)}(v)$,
that is
   \begin{align}
    &\phantom{\ \leq} \sE(u+v)-\int_{M_{\od}^{2}}\abs{u(x)+v(x)-u(y)-v(y)}^{p}\dif j(x,y)  \\   
    &= \sE(u)-\int_{M_{\od}^{2}}\abs{u(x)-u(y)}^{p}\dif j(x,y)+\sE(v)-\int_{M_{\od}^{2}}\abs{v(x)-v(y)}^{p}\dif j(x,y),
   \end{align}
   which implies 
      \begin{align}
     &\phantom{\ \leq}\sE(u)+\sE(v)-\sE(u+v)\\
     &=\sum_{w\in\{u,v\}}\int_{M_{\od}^{2}}\abs{w(x)-w(y)}^{p}\dif j(x,y)- \int_{M_{\od}^{2}}\abs{u(x)+v(x)-u(y)-v(y)}^{p}\dif j(x,y)\\
    &=2\int_{\supp(u)\times \supp(v)} (\abs{u(x)}^{p}+\abs{v(y)}^{p}-\abs{u(x)-v(y)}^{p})\dif j(x,y) ,\label{e.j-uniq}
   \end{align}
where in the second equality we have used the symmetric of $j$. Assume that $0\leq u,v\leq 1$ with $u|_E=1, v|_F=1$, we have
\begin{equation}\label{e.uv1}
  0\leq \abs{u(x)}^{p}+\abs{v(y)}^{p}-\abs{u(x)-v(y)}^{p}\leq 2,
\end{equation}
thus
\begin{align}
\sE(u)+\sE(v)-\sE(u+v)\geq 2\int_{E\times F} (\abs{u(x)}^{p}+\abs{v(y)}^{p}-\abs{u(x)-v(y)}^{p})\dif j(x,y)    
  =4j(E\times F).
\end{align}
Therefore, 
\begin{equation}
j(E\times F)\leq \frac{1}{4}\inf\left\{\sE(u)+\sE(v)-\sE(u+v): 
\begin{array}{l}
 u,v\in \sF\cap C_{c}(M) \text{ with disjoint supports},\\
u|_E=1, v|_F=1, \ 0\leq u,v\leq 1.
\end{array}
 \right\}.
\end{equation}
On the other hand, for any open sets $U$, $V$ satisfy $U\supset E$ and $V\supset F$, by Proposition \ref{p.cut}, there exists $u,v\in \sF\cap C_c(M)$ such that $u,v\geq 0$, $u=1$ in $E$, $u=0$ in $M\setminus U$ and $v=1$ in $F$, $u=0$ in $M\setminus V$. Note that $\supp(u)\times \supp(v)\subset U\times V$, we know by \eqref{e.j-uniq} and \eqref{e.uv1} that
\begin{equation}
 \sE(u)+\sE(v)-\sE(u+v)\leq 2\int_{U\times V} (\abs{u(x)}^{p}+\abs{v(y)}^{p}-\abs{u(x)-v(y)}^{p})\dif j(x,y)\leq 4j(U\times V),
\end{equation}
which implies 
\begin{align}
&\phantom{\ \leq}\inf\left\{\sE(u)+\sE(v)-\sE(u+v): 
\begin{array}{l}
 u,v\in \sF\cap C_{c}(M) \text{ with disjoint supports},\\
u|_E=1, v|_F=1, \ 0\leq u,v\leq 1.
\end{array}
 \right\} \\
 &\leq 4\inf\{j(U\times V): \ E\times F\subset U\times V,\ U\text{ and }V\ \text{open}\}=4j(E\times F),
\end{align}
where in the last equality we have used the outer regularity of the Radon measure $j$.
 \end{proof}
  
\begin{proposition}\label{p.j-smooth} Let $(\mathcal{E},\mathcal{F})$ be a mixed local and nonlocal $p$-energy form on $(M,m)$. Let $j$ be a jump measure in \ref{lb.EFJP}.
	Let $E\subset M^{2}_{\od}$ be a Borel set. Let \begin{align}
		E_{1}&:=\Sett{x\in M}{\text{there exists } y\in M\text{ such that }(x,y)\in E},\\
		E_{2}&:=\Sett{y\in M}{\text{there exists } x\in M\text{ such that }(x,y)\in E}
	\end{align}
	be the projections of $E$ on the first coordinate and the second coordinate, respectively. If $\capacity_{1}(E_1)=0$ or if $\capacity_{1}(E_2)=0$, then we have $j(E)=0$. In other words, $j$ charges no subset of $M^{2}_{\od}$ whose projection on the factor $M$ is exceptional.
\end{proposition}
\begin{proof}
We prove this result only for $E_{1}$. The proof for $E_2$ is similar. By the regularity of the measure $j$, we may assume that $E_1$ is contained in a relative compact open set.

Let $U\subset{M}$ be a fixed relative compact open set with $(\ol{U})^{c}\neq\emptyset$. By Proposition \ref{p.ShaQua}, we have $\capacity_{1}^{U}(E_1\cap U)=0$ and by definition there is an open set $G_{\epsilon}\subset G$ such that $E_{1}\cap U\subset G_{\epsilon}$ and $\capacity_{1}^{U}(G_{\epsilon})<\epsilon$.

 Let $D\subset (\ol{U})^{c}$ be a relative compact open set. By the regularity of $(\sE,\sF)$ and Proposition \ref{p.cut}, there exists a non-negative $v\in\sF\cap C_{c}(M)$ with $\supp(v)\subset (\ol{U})^{c}$ and $\restr{v}{D}=1$. Let $K$ be any non-empty compact set in $U$. Let $e^{U}_{K}$ be the equilibrium potential of $K$ with respect to $(\sE_{U,1},\sF(U))$, and we always represent $e^{U}_{K}$ by its $\sE$-quasi-continuous version. Then $0\leq {e^U_{K}}\leq1$ $\sE$-q.e. on $U$, $e_{K}^{U}=0$ $\sE$-q.e. on $M\setminus U$, ${e^U_{K}}=1$ $\sE$-q.e. on $K$ and $\capacity^{U}_{1}(K)=\sE_{U,1}(e^{U}_K)$. For $\delta\in(0,1)$, since $K$ is compact and $\wt{e_{K}}$ is $\sE$-quasi-continuous, there is a relative compact open set $\Omega$ such that $K\subset\Omega\subset U$ and $\wt{e_{K}}>1-\delta$ $\sE$-q.e. on $\Omega$. {By the regularity of $(\sE_{U},\sF(U))$} in Lemma \ref{l.partEF}-\ref{lb.part-reg} and Proposition \ref{p.cut}, there is $\phi\in\cutoff(K,\Omega)\cap \sF(U)$. Define $u:=((1-\delta)^{-1}{e_{K}})\wedge1$, then $u\in\sF(U)$, $0\leq u\leq1$ $\sE$-q.e. on $M$, $u=1$ $\sE$-q.e. on $\Omega$. So we have by the definition of $\capacity^{U}_{1}(K)$ that 
 \begin{equation}
	\capacity_{1}^{U}(K)\leq \sE_{U,1}(u)=\sE_{U,1}\Big(\frac{{e_{K}}}{1-\delta}\wedge1\Big)\leq \frac{1}{(1-\delta)^{p}}\sE_{U,1}({e_{K}})=\frac{1}{(1-\delta)^{p}}\capacity^{U}_{1}(K).  \label{e.cap1}
\end{equation}
By the regularity of $(\sE_{U},\sF(U))$ in Lemma \ref{l.partEF}-\ref{lb.part-reg}, there exists $\{u_{n}\}_{n\in\bN}\subset\sF\cap C_{c}(U)$ such that $u_{n}\to u$ in $\sE_{U,1}$. By \cite[Corollary 3.19-(b)]{KS25}, we may assume $0\leq u_n\leq1$ by replacing $u_{n}$ with $(0\vee u_{n})\wedge1$. Let $v_{n}:=\phi+(1-\phi)u_{n}\in\sF\cap C_{c}(U)$. It can be verified that $v_{n}\geq0$ on $M$ and $\restr{v_{n}}{K}=1$. Since $\phi u_n\to \phi u$ in $L^{p}(U,\restr{m}{U})$ and $\sup_{n}\sE_{U,1}(\phi u_{n})<\infty$ by Proposition \ref{p.mar}, Lemma \ref{lem:A1} implies that $\phi u_{n}\to\phi u$ weakly in $(\sF(U),\sE_{U,1}^{1/p})$. By Mazur's lemma \cite[Theorem 2 in Section V.1]{Yos95}, for each $n\in\bN$, there is a convex combination $\{\lambda_{k}^{(n)}\}_{k=1}^{n}\in[0,1]$ with $\sum_{k=1}^{n}\lambda_{k}^{(n)}=1$ such that $\sum_{k=1}^{n}\lambda_{k}^{(n)}\phi u_{k}\to \phi u$ in $\sE_1$. Now let $w_{n}:=\sum_{k=1}^{n}\lambda_{k}^{(n)}v_{k}$, so $w_{n}\geq0$ on $M$, $\restr{w_{n}}{K}=1$ and $w_{n}\to \phi+u-\phi u=u$ in $\sE_1$. Note that $\supp(w_{n})\subset U$ and is disjoint with $\supp(v)$, we have by the symmetry of $j$ that
\begin{align}
	\sE(w_{n};v)&=\int_{M_{\od}^{2}}\abs{w_n(x)-w_{n}(y)}^{p-2}(w_n(x)-w_{n}(y))(v(x)-v(y))\dif j(x,y)\\
	&=-\int_{\supp(w_n)\times \supp(v)}w_n(x)^{p-1}v(y)\dif j(x,y)-\int_{\supp(v)\times \supp(w_n)}w_n(y)^{p-1}v(x)\dif j(x,y)\\
	&=-2\int_{\supp(w_n)\times \supp(v)}w_n(x)^{p-1}v(y)\dif j(x,y)\label{e.jsm1}.
\end{align}
Therefore, note that $\restr{v}{D}=1$ and $\restr{w_{n}}{K}=1$, 
 \begin{align}
	j(K\times D)
	&\leq \liminf_{n\to\infty}\int_{K\times D}w_{n}(x)^{p-1}v(y)\dif j(x,y)\overset{\eqref{e.jsm1}}{\leq}\frac{1}{2}\liminf_{n\to\infty}\abs{\sE(w_{n};v)}\\
	&\overset{\eqref{e.sE}}{\leq}\frac{1}{2}\liminf_{n\to\infty}\sE(w_{n})^{(p-1)/p}\sE(v)^{1/p}= \frac{1}{2}\liminf_{n\to\infty}\sE_U(w_{n})^{(p-1)/p}\sE(v)^{1/p}\\
	&= \frac{1}{2}\sE_U(u)^{(p-1)/p}\sE(v)^{1/p}\overset{\eqref{e.cap1}}{\leq} \frac{1}{2}(1-\delta)^{1-p}\capacity^{U}_{1}(K)^{(p-1)/p}\sE(v)^{1/p}.
\end{align}
By letting $\delta\to0$, we obtain that \begin{equation}\label{e.jsmooth}
	j(K\times D)\leq \frac{1}{2}\capacity^U_{1}(K)^{(p-1)/p}\sE(v)^{1/p}.
\end{equation}
By the regularity of $j$, we have \begin{align}
	j(G_\epsilon\times D)&=\sup\Sett{j(K\times D)}{\text{$K$ is a compact set in $G_\epsilon$}}\\
	&\overset{\eqref{e.jsmooth}}{\leq} \sup\Sett{\frac{1}{2}\capacity^{U}_{1}(K)^{(p-1)/p}\sE(v)^{1/p}}{\text{$K$ is a compact set in $G_\epsilon$}}\\
	&\leq \frac{1}{2}\capacity_{1}^{U}(G_\epsilon)^{(p-1)/p}\sE(v)^{1/p}\leq \frac{1}{2}\epsilon^{(p-1)/p}\sE(v)^{1/p}.
\end{align}
Letting $\epsilon\to0$ and recall that $E\cap(U\times D)\subset  (E_1\cap U)\times D\subset   G_\epsilon\times D$, we obtain that $j(E\cap (U\times D) )=0$. Since $D$ is an arbitrary relative compact open subset in $(\ol{U})^{c}$, by the sub-additive of the measure $j$, we see that $j(E\cap (U\times (\ol{U})^{c}) )=0$. By taking a countable topological base $\sU$ of $M$, and noting that $E=\bigcup_{U\in\sU}E\cap (U\times (\ol{U})^{c})$, using the sub-additivity of $j$ again, we finally obtain that $j(E)=0$.
\end{proof}

\begin{proposition}\label{p.SL}
	Let $(\sE,\sF)$ be a mixed local and nonlocal $p$-energy form on $(M,m)$ and $u,v\in\sF\cap L^{\infty}(M,m)$ with $\widetilde{u}(\widetilde{v}-a)=0$ $\sE$-q.e. for some $a\in \bR$, then $\sE^{(L)}(u+v)=\sE^{(L)}(u)+\sE^{(L)}(v)$.
\end{proposition}
\begin{proof}
Since $\widetilde{u}(\widetilde{v}-a)=0$ $\sE$-q.e., by definition there exists a Borel exceptional set $N$ such that $\widetilde{u}(\widetilde{v}-a)=0$ in $M\setminus N$. 
Let $A_1=\{x\in M\setminus N: \widetilde{v}(x)=a\}$ and $A_2=\{x\in M\setminus N: \widetilde{v}(x)\neq a\}$. Note that
$\restr{\left((\widetilde{u}+\widetilde{v})-\widetilde{u}\right)}{A_{1}}=\restr{\widetilde{v}}{A_{1}}=a$, by Theorem \ref{t.sasEM}-\ref{lb.M-local}, $\Gamma^{(L)}\la u+v\ra (A_1)=\Gamma^{(L)}\la u\ra (A_1)$. Also, we have $\Gamma^{(L)}\la v\ra (A_1)=0$ by Theorem \ref{t.sasEM}-\ref{lb.M-densi}. Note that $A_2\subset\{x\in M\setminus N: \widetilde{u}(x)=0\}$ since $\widetilde{u}(\widetilde{v}-a)=0$ in $M\setminus N$, similarly we have $\Gamma^{(L)}\la u+v\ra (A_2)=\Gamma^{(L)}\la v\ra (A_2)$ and $\Gamma^{(L)}\la u\ra (A_2)=0$.
Therefore, by Theorem \ref{t.sasEM}-\ref{lb.M-smtt},
\begin{align}
  \sE^{(L)}(u+v) & =\int_M \dif \Gamma^{(L)}\la u+v\ra=\int_{A_1} \dif \Gamma^{(L)}\la u+v\ra+\int_{A_2} \dif \Gamma^{(L)}\la u+v\ra \\
   & =\int_{A_1} \dif \Gamma^{(L)}\la u\ra+\int_{A_2} \dif \Gamma^{(L)}\la v\ra  = \sE^{(L)}(u)+\sE^{(L)}(v).
\end{align}
\end{proof}

\begin{proposition}\label{P61}
Let $(\mathcal{E},\mathcal{F})$ be a mixed local and nonlocal $p$-energy form on $(M,m)$. Assume that a function $f\in C^{2}(\mathbb{R})$
satisfies
\begin{equation}
f^{\prime }\geq 0,~f^{\prime \prime }\geq 0,~\sup_{\mathbb{R}}f^{\prime }<\infty ,~\sup_{%
\mathbb{R}}f^{\prime \prime }<\infty .
\end{equation}%
Let $u,\varphi \in \mathcal{F}^{\prime }\cap L^{\infty }$. If 
\begin{equation}\label{e.conf}
\inf\Sett{f^{\prime}(t)}{-\norm{u}_{L^{\infty}(M,m)}\leq t\leq \norm{u}_{L^{\infty}(M,m)}}>0,
\end{equation}
then $f(u),f^{\prime }(u)^{p-1},f^{\prime }(u)^{p-1}\varphi $ are in $\mathcal{F}^{\prime }\cap L^{\infty }$. Moreover, if further $\varphi \geq 0$ in $M$,
then
\begin{equation}
\mathcal{E}(f(u);\varphi )\leq \mathcal{E}(u;f^{\prime }(u)^{p-1}\varphi ).
\label{eq61}
\end{equation}
\end{proposition}
\begin{proof}
The proof is similar to \cite[Proposition 9.1]{GHH24}.
We see from $\sup_{\mathbb{R}}f^{\prime }<\infty $ and $\sup_{\mathbb{R}}f^{\prime\prime }<\infty $ that both $f$ and $f^{\prime}$ are Lipschitz, then $f(u),f^{\prime}(u)\in \mathcal{F}^{\prime }\cap L^{\infty }$ by Proposition \ref{p.mar}. Let $L:=\inf_{-\norm{u}_{L^{\infty}(M,m)}\leq t\leq \norm{u}_{L^{\infty}(M,m)}} f^{\prime}(t)$. Since the function $g(x):=x^{p-1}$ has bounded Lipschitz constant on the bounded set $[L,\sup_{\mathbb{R}}f^{\prime }] $, we know that $f^{\prime }(u)^{p-1}\in \mathcal{F}
^{\prime }\cap L^{\infty } $ by a similar proof of Proposition \ref{p.mar}. By the proof of \cite[Lemma 4.4]{Yan25a}, for any $f\in C^2(\mathbb{R})$ such that $f^{\prime},f^{\prime\prime}\geq 0$, and for any non-negative 
$\varphi\in \sF^{\prime}\cap L^{\infty}$,
\begin{equation}
\mathcal{E}^{(L)}(f(u);\varphi )\leq \mathcal{E}^{(L)}(u;f^{\prime }(u)^{p-1}\varphi ). \label{e.LG5}
\end{equation}
We then prove 
\begin{equation}
\mathcal{E}^{(J)}(f(u);\varphi )\leq \mathcal{E}^{(J)}(u;f^{\prime }(u)^{p-1}\varphi ). \label{e.LG4}
\end{equation}
To do this, we prove that, for any $a,b\geq 0$ and any $X,Y\in\bR$, there holds
\begin{equation}
|f(X)-f(Y)|^{p-2}(f(X)-f(Y))(a-b)\leq |X-Y|^{p-2}(X-Y)(f^{\prime}(X)^{p-1}a-f^{\prime}(Y)^{p-1}b). \label{e.LG3}
\end{equation}
Indeed, without loss of generality, we assume  $X >Y$ (otherwise switch the role of $a$ and $b$ and of $X$ and $Y$). By the differential mean value theorem, there exists $Z\in (Y,X)$ such that $f(X)-f(Y)=f^{\prime}(Z)(X-Y)$. Since $f^{\prime \prime }\geq 0$, we know that $ f^{\prime}(X)\geq  f^{\prime}(Z)\geq f^{\prime}(Y)$ and therefore
\begin{equation}
 f^{\prime}(X)^{p-1}a-f^{\prime}(Y)^{p-1}b\geq f^{\prime}(Z)^{p-1}a-f^{\prime}(Z)^{p-1}b=(f(X)-f(Y))^{p-1}(X-Y)^{-(p-1)}(a-b),
\end{equation}
which implies \eqref{e.LG3}. Applying \eqref{e.LG3} with $X=u(x),Y=u(y),a=\varphi(x)$, and $b=\varphi(y)$, we obtain \eqref{e.LG4}. Combining \eqref{e.E-2}, \eqref{e.LG5} and \eqref{e.LG4}, we obtain \eqref{eq61}.
\end{proof}
\begin{remark}
The assumption \eqref{e.conf} is required in the proof, as the function $g(x):=x^{p-1}$ is not Lipschitz near $0$ when $p\in(1,2)$. If $p\geq 2$, then Proposition \ref{P61} holds without the restriction \eqref{e.conf}.
\end{remark}

\section{Useful inequalities}\label{A.ineqs}
\begin{lemma}\label{l.a-blem}
	For any $a,b\in \mathbb{R}$, $t\in (0,1]$ and $p\in(1,\infty)$, we have
\begin{equation}\label{e.E-4}
 \abs{\frac{\abs{a+tb}^p-\abs{a}^p}{t}}\leq 2^{p-1}(p+1)(|a|^p+|b|^p).
\end{equation}
\end{lemma}
\begin{proof}
	Since the function $f:x\mapsto |x|^p$, $x\in\bR$ is in $C^{1}(\bR)$ and $|f^{\prime}(x)|=p|x|^{p-1}$. By the differential mean value theorem, there exists some $z\in [a\wedge (a+tb),a\vee (a+tb)]$ such that $f(a+tb)-f(a)=f^{\prime}(z)tb$. Therefore
\begin{align}
 \abs{\abs{a+tb}^p-\abs{a}^p}&\leq tp|z|^{p-1}|b|\leq tp|b|(|a|+|b|)^{p-1}\leq  2^{p-1}tp|b|(|a|^{p-1}+|b|^{p-1})\\
 &\leq 2^{p-1}tp\left(\frac{p-1}{p}|a|^p+\frac{1}{p}|b|^p+|b|^{p}\right) \ \text{(by Young's inequality)}\\
 &\leq 2^{p-1}(p+1)(|a|^{p}+|b|^{p}),
\end{align}
which implies \eqref{e.E-4}.
\end{proof}

\begin{proposition}[{\cite[Proposition 15.4]{GHH24}}]
\label{prop:A3} Let $\{a_{k}\}_{k=0}^{\infty }$ be a sequence of
non-negative numbers such that 
\begin{equation}
a_{k}\leq D\lambda ^{k}a_{k-1}^{1+\nu },\ \forall k\in\bN
\label{eq:A3-1}
\end{equation}%
for some constants $D,\nu >0$ and $\lambda \geq 1$. Then, 
\begin{equation}
a_{k}\leq D^{-\frac{1}{\nu }}\left( D^{\frac{1}{\nu }}\lambda ^{\frac{1+\nu 
}{\nu ^{2}}}a_{0}\right) ^{(1+\nu )^{k}},\ \forall k\in\{0\}\cup\bN.  \label{eq:A3-2}
\end{equation}
\end{proposition}

\begin{lemma}[{\cite[Lemma 3.1]{DCKP16}}]\label{l.p-dif}
	Let $p\in[1,\infty)$ and $\epsilon\in(0,1]$, then 
	\begin{equation}
		\abs{a}^{p}\leq (1+c_{p}\epsilon)\big(\abs{b}^{p}+\epsilon^{1-p}\abs{a-b}^{p}\big),\ \forall a,b\in\bR,
	\end{equation}
	where $c_{p}:=(p-1)\Gamma(\max(1,p-2))$.
\end{lemma}
\begin{lemma}\label{l.LL}
	Let $p\in(1,\infty)$, then for all $s\in(0,1)$,
	\begin{equation}\label{e.LL0}
		(1-s)^{p}\bigg(\frac{p-1}{2^{p+1}}+\frac{1-s^{1-p}}{1-s}\bigg)\leq-\frac{(p-1)^2}{p2^{p+1}}\big(\log s^{-1}\big)^{p}.
	\end{equation}
\end{lemma}
	\begin{proof}
		Note that \begin{equation}\label{e.LL1}
			\frac{1-s^{1-p}}{1-s}\leq -\frac{p-1}{2^{p}}\frac{s^{1-p}}{1-s}\one_{(0,\frac{1}{2}]}(s)-(p-1)\one_{(\frac{1}{2},1)}(s).
		\end{equation}
		Suppose $s\in(0,\frac{1}{2}]$, then $(s^{-1}-1)^{p-1}\geq 1\geq (1-s)^{p} $. We have
\begin{align}
			(1-s)^{p}\bigg(\frac{p-1}{2^{p+1}}+\frac{1-s^{1-p}}{1-s}\bigg)&\overset{\eqref{e.LL1}}{\leq} (1-s)^{p}\bigg(\frac{p-1}{2^{p+1}}-\frac{p-1}{2^{p}}\frac{s^{1-p}}{1-s}\bigg)\\
			&= \frac{p-1}{2^{p+1}}\big((1-s)^{p}-2(s^{-1}-1)^{p-1}\big)\leq-\frac{p-1}{2^{p+1}}(s^{-1}-1)^{p-1}.\label{e.LL2}
		\end{align}
Next, we show that for all $p>1$ and all $t\in[2,\infty)$,
\begin{equation}
  \frac{(t-1)^{p-1}}{(\log t)^{p}}> \frac{p-1}{p}. \label{e.LL4}
\end{equation}
Let $f(t):=p(t-1)^{p-1}-(p-1)(\log t)^{p}$. Since $t-1>\log t$ for all $t\in[2,\infty)$, we have
\begin{equation}
 f^{\prime}(t)=p(p-1)\left((t-1)^{p-2}-(\log t)^{p-1}t^{-1}\right)\geq p(p-1)t^{-1}\left((t-1)^{p-1}-(\log t)^{p-1}\right)\geq 0,
\end{equation}
therefore $f(t)\geq f(2)=p-(p-1)(\log 2)^{p}>1>0$, which implies \eqref{e.LL4}, thus showing 
 \begin{equation}
	(1-s)^{p}\bigg(\frac{p-1}{2^{p+1}}+\frac{1-s^{1-p}}{1-s}\bigg)\leq -\frac{(p-1)^2}{p2^{p+1}} \big(\log(s^{-1}))^{p} \label{e.LL5}
\end{equation}
for $s\in(0,\frac{1}{2}]$ by using \eqref{e.LL2}. On the other hand, suppose $s\in(\frac{1}{2},1)$. Then 
		\begin{align}
			(1-s)^{p}\bigg(\frac{p-1}{2^{p+1}}+\frac{1-s^{1-p}}{1-s}\bigg)&\overset{\eqref{e.LL1}}{\leq}  -(p-1)\bigg(1-\frac{1}{2^{p+1}}\bigg)(1-s)^{p}\\
			&\leq-(p-1)\bigg(1-\frac{1}{2^{p+1}}\bigg)\frac{(1-s)^{p}}{2^{p}s^{p}}\\
			&\leq-\frac{p-1}{2^{p+1}}\big(\log(s^{-1})\big)^{p}\leq-\frac{(p-1)^2}{p2^{p+1}}\big(\log(s^{-1})\big)^{p}.\label{e.LL3}
		\end{align}
Then \eqref{e.LL5} and \eqref{e.LL3} gives \eqref{e.LL0}.
	\end{proof}
	
\begin{lemma}\label{l.4real}
	For any $p\in[1,\infty)$ and any four real numbers $a,b,c,d\in\bR$, we have \begin{equation}\label{e.4real}
		(\abs{a}^{p}+\abs{b}^{p})\abs{c-d}^{p}\leq 4\abs{c-d}^{p-2}(c-d)(\abs{a}^{p}c-\abs{b}^{p}d)+2^{p^{2}+p+1}p^{p}(\abs{c}^{p}+\abs{d}^{p})
\big||a| - |b|\big|^{p}
	\end{equation}
\end{lemma}
\begin{proof}
	If $c = d$, then the inequality obviously holds. By symmetry, we may assume $c > d$. By homogeneity, we may assume that $c - d = 1$ (otherwise we may consider $c/(c-d)$ and $d/(c-d)$). Then the inequality becomes
\begin{equation}\label{e.B5-0.1}
|a|^p + |b|^p \leq 4\left((d + 1)|a|^p - d|b|^p\right) + 2^{p^2+p+1} p^p (|d|^p + |d+1|^p)
\big||a| - |b|\big|^p.
\end{equation}
If $|a| \geq |b|$, then since
\begin{equation*}
4\left((d + 1)|a|^p - d|b|^p\right)\geq 4 |a|^p\geq |a|^p+|b|^p,
\end{equation*}
 we obtain \eqref{e.4real}. Now we assume $|a| < |b|$. By homogeneity, we may assume that $|b|=|a|+1$. It suffices to prove
\begin{equation}\label{e.B5-0}
(4d + 1)\left((|a| + 1)^p - |a|^p\right) \leq 2|a|^p + 2^{p^2+p+1} p^p (|d|^p + |d+1|^p).
\end{equation}
Note that
\begin{equation*}
  |4d+1|\leq 3|d|+|d+1|\leq 4(|d|^p + |d+1|^p)^{{1}/{p}}.
\end{equation*}
Therefore, to prove \eqref{e.B5-0}, we only need to prove 
\begin{equation}\label{e.B5-1}
2(|d|^p + |d+1|^p)^{{1}/{p}}\left((|a| + 1)^p - |a|^p\right) \leq |a|^p + 2^{p^2+p} p^p (|d|^p + |d+1|^p).
\end{equation}
\begin{enumerate}[label=\textit{Case {\arabic*}.},align=left,leftmargin=*,topsep=5pt,parsep=0pt,itemsep=2pt]
	\item If $|a| \leq 1$, then $(|a| + 1)^p - |a|^p\leq 2^p$, we have
\begin{equation*}
 2(|d|^p + |d+1|^p)^{{1}/{p}}\left((|a| + 1)^p - |a|^p\right)\leq 2^{p+1}(|d|^p + |d+1|^p)^{{1}/{p}}
 \leq 2^{p^2+p} p^p (|d|^p + |d+1|^p),
\end{equation*}
thus showing \eqref{e.B5-1}.
\item  If $|a|>1$. By the elementary inequality $(1+t)^p\leq 1+p2^{p-1}t$ for any $p\in[1,\infty)$ and $t\in [0,1)$, we have
\begin{align}
&\phantom{\ \leq}2(|d|^p + |d+1|^p)^{\frac{1}{p}}\left(\left(1 + \frac{1}{|a|}\right)^p - 1\right)\\
&\leq 2^p(|d|^p + |d+1|^p)^{\frac{1}{p}} \frac{p}{|a|} \leq \max\left(1, 2^p(|d|^p + |d+1|^p)^{\frac{1}{p}} \frac{p}{|a|} \right)^{p}\\
&\leq 1+ 2^{p^2+p} p^p \frac{(|d|^p + |d + 1|^p)}{|a|^p}.
\end{align}
Multiply both sides by $|a|^p$, this is \eqref{e.B5-1}.
\end{enumerate}
This completes the proof of \eqref{e.4real}.
\end{proof}

\begin{lemma}[\text{\cite[p.~462 (4), (6)]{Bre11}}]Let $(X,\mu)$ be a $\sigma$-finite measure space. Then for any $f,g\in L^p(X,\mu)$, 
\begin{align}
  \norm{f+g}_{L^p(X,\mu)}^p+\norm{f-g}_{L^p(X,\mu)}^p & \leq 2\left(\norm{f}_{L^p(X,\mu)}^{\frac{p}{p-1}}+\norm{g}_{L^p(X,\mu)}^{\frac{p}{p-1}}\right)^{p-1},\  \text{if }p\in [2,\infty) \label{e.B3-1} \\
  \norm{f+g}_{L^p(X,\mu)}^{\frac{p}{p-1}}+\norm{f-g}_{L^p(X,\mu)}^{\frac{p}{p-1}}   & \leq 2\left(\norm{f}_{L^p(X,\mu)}^{p}+\norm{g}_{L^p(X,\mu)}^{p}\right)^{\frac{1}{p-1}},\ \text{if }p\in (1,2) \label{e.B3-2}.
\end{align}
\end{lemma}
\end{appendices}

\vspace{10pt}
\noindent Aobo Chen

\vspace{3pt}
\noindent Department of Mathematical Sciences, Tsinghua University, Beijing 100084, China

\vspace{3pt}
\noindent \texttt{cab21@mails.tsinghua.edu.cn},  \texttt{aobochen.math@hotmail.com}

\vspace{10pt}
\noindent Zhenyu Yu

\vspace{3pt}
\noindent College of Science, National University of Defense Technology, Changsha 410073, China

\vspace{3pt}
\noindent \texttt{yuzy23@nudt.edu.cn}

\end{document}